%% file: multipliers.tex
\documentclass[a4paper,10pt,twoside]{report}
\newcommand{\titolofont}{\sffamily}
\usepackage{ifpdf}
\ifpdf
  \usepackage{mmap}
\fi
\usepackage{fix-cm}

\usepackage[english]{babel}
\usepackage{amsmath,amsthm,amssymb,amsfonts}

\usepackage{makeidx}

\usepackage{multirow}

\def\clap#1{\hbox to 0pt{\hss#1\hss}}

\newcommand{\TS}{\rule{0pt}{2.6ex}}
\newcommand{\BS}{\rule[-1.2ex]{0pt}{0pt}}

\usepackage{graphicx}
\usepackage[all]{xy}

\ifpdf
  \usepackage{hyperref}
  \usepackage[miktex]{pdftricks}
  \begin{psinputs}
    \usepackage{pstricks,pst-plot}
  \end{psinputs}
\else
  \usepackage[hypertex]{hyperref}
  \usepackage{pstricks,pst-plot}
  \newenvironment{pdfpic}{}{}
\fi

\usepackage[justification=centering]{caption}

\usepackage{fancyhdr}
\makeatletter
\newcommand*\cleardoublepageeven{\clearpage\if@twoside
  \ifodd\c@page \hbox{}\newpage\if@twocolumn\hbox{}%
  \newpage\fi\fi\fi}
\makeatother

\newcommand{\clearemptydoublepage}{\clearpage{\pagestyle{empty}\cleardoublepage}}
\newcommand{\clearemptydoublepageeven}{\clearpage{\pagestyle{empty}\cleardoublepageeven}}

\input{notation}

\numberwithin{equation}{section}

\newtheorem{thm}{Theorem}[section]
\newtheorem{prp}[thm]{Proposition}
\newtheorem{lem}[thm]{Lemma}
\newtheorem{cor}[thm]{Corollary}
\newtheorem{thmi}{Theorem}

\theoremstyle{remark}
\newtheorem{rem}{Remark}[section]

\title{Algebras of differential operators\\
and joint spectral multipliers\\
on homogeneous Lie groups}
\author{Alessio Martini}

\makeindex

\begin{document}
\pagestyle{empty}
\pagenumbering{alph}
\input{frontespizio}
\input{abstract}

\cleardoublepage\pagenumbering{roman}
\pagestyle{plain}
\setcounter{tocdepth}{1}
\tableofcontents

\input{intro}
\clearpage
\input{notationlist}

\clearemptydoublepage
\pagenumbering{arabic}
\pagestyle{fancy}
\input{operators}

\clearemptydoublepage
\input{conditions}
\clearemptydoublepage
\input{plancherel}
\clearemptydoublepage
\input{weighted}

\clearemptydoublepage
\input{mihlin}

\input{marcinkiewicz}
\input{applications}

\appendix
\clearemptydoublepage
\input{appendix}

\clearemptydoublepage
\phantomsection
\pagestyle{plain}
\bibliographystyle{abbrvalpha}
\fontsize{10}{12}\selectfont
\addcontentsline{toc}{chapter}{Bibliography}
\bibliography{../multipliers}

\clearemptydoublepageeven
\phantomsection
\index{convolution!kernel|see{Schwartz kernel theorem}}
\addcontentsline{toc}{chapter}{Index}
\printindex

\end{document}

%% file: notation.tex
\newcommand{\tc}{\,:\,}                             

\newcommand{\id}{\mathrm{id}}                       
\newcommand{\chr}{\chi}                             
\DeclareMathOperator{\supp}{\mathrm{supp}}          

\newcommand{\R}{\mathbb{R}}                         
\newcommand{\N}{\mathbb{N}}                         
\newcommand{\Z}{\mathbb{Z}}                         
\newcommand{\C}{\mathbb{C}}                         
\newcommand{\Q}{\mathbb{Q}}                         
\newcommand{\T}{\mathbb{T}}                         

\newcommand{\D}{\mathcal{D}}                        
\newcommand{\E}{\mathcal{E}}                        
\newcommand{\Sz}{\mathcal{S}}                       
\newcommand{\loc}{\mathrm{loc}}                     
\newcommand{\Cv}{Cv}                                

\newcommand{\Diff}{\mathfrak{D}}                    
\newcommand{\dstar}{\circ}                          
\newcommand{\dtwist}{\bullet}                       
\newcommand{\dconj}{\diamond}                       

\newcommand{\LA}{\mathrm{L}}
\newcommand{\la}{\mathrm{l}}
\newcommand{\RA}{\mathrm{R}}
\newcommand{\ra}{\mathrm{r}}

\newcommand{\lie}{\mathfrak}                        
\DeclareMathOperator{\Ad}{\mathrm{Ad}}              
\DeclareMathOperator{\ad}{\mathrm{ad}}              
\DeclareMathOperator{\Aut}{\mathrm{Aut}}            
\DeclareMathOperator{\tr}{\mathrm{tr}}              
\DeclareMathOperator{\Span}{\mathrm{span}}          

\DeclareMathOperator{\spec}{\mathrm{Spec}}          
\newcommand{\Bdd}{\mathcal{B}}                      
\newcommand{\HS}{\mathrm{HS}}							          

\newcommand{\UEnA}{\mathrm{U}}                      

\newcommand{\ptimes}{\mathop{\hat\otimes_\pi}}      
\newcommand{\htimes}{\mathop{\hat\otimes_\mathrm{H}}}        
\newcommand{\etimes}{\mathop{\hat\otimes_\epsilon}} 

\newcommand{\Gelf}{\mathcal{G}}                     
\newcommand{\Four}{\mathcal{F}}                     
\newcommand{\Kern}{\mathcal{K}}                     

\newcommand{\GS}{\mathfrak{G}}                      

\newcommand{\VV}{\mathcal{V}}                       
\newcommand{\WW}{\mathcal{W}}                       
\newcommand{\HH}{\mathcal{H}}                       
\newcommand{\Alg}{\mathcal{A}}                      

\newcommand{\JJ}{\mathcal{J}}                       
\newcommand{\II}{\Gamma^1}                          
\newcommand{\IS}{\Gamma^2}                          
\newcommand{\PP}{\mathcal{P}}                       

\newcommand{\HP}[2]{\textup{(I$_{{#1},{#2}}$)}}     
\newcommand{\HPK}[2]{\textup{(J$_{{#1},{#2}}$)}}    

\newcommand{\strong}{\mathrm{strong}}               

\newcommand{\evmap}{\vartheta}
\renewcommand{\ell}{\varrho}

%% file: frontespizio.tex
\begin{titlepage}
\begin{center}
\includegraphics[width=2.0cm]{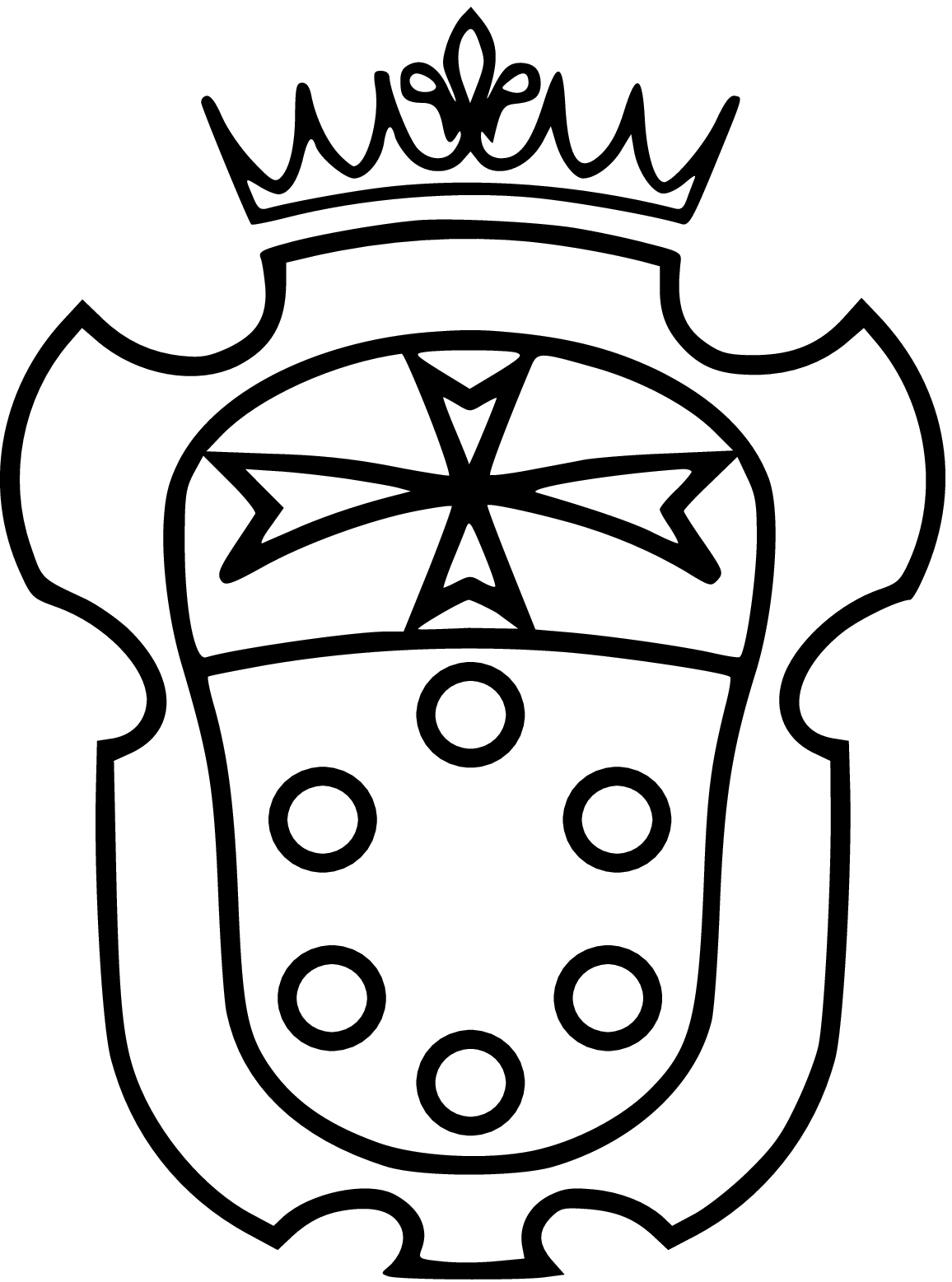}\\
\vspace{0.3cm}

\textsc{\large Scuola Normale Superiore}\\
\vspace{0.05cm}
\textsc{\normalsize Pisa}\\
\vspace{2.0cm}

\textsc{\LARGE Tesi di Perfezionamento}\\
\vspace{0.2cm}
\textsc{\LARGE in Matematica}\\
\vspace{0.2cm}
\textsc{\large triennio 2006-2009}

\vspace{5cm}

\Huge{\textbf{\titolofont Algebras of differential\\operators on Lie groups\\and spectral multipliers}}\\
\vspace{1.5cm}

{\titolofont \Large candidato\\} \vspace{-0.37cm} \textbf{{\titolofont\LARGE Alessio Martini}}\\
\vspace{0.3cm}

{\titolofont \Large relatore}\\ \vspace{-0.37cm} \textbf{{\titolofont\LARGE prof.\ Fulvio Ricci}}
\end{center}

\newpage
\thispagestyle{empty} \phantom{}
\end{titlepage}

%% file: abstract.tex
\section*{Abstract}

Let $(X,\mu)$ be a measure space, and let $L_1,\dots,L_n$ be (possibly unbounded) self-adjoint operators on $L^2(X,\mu)$, which commute strongly pairwise, i.e., which admit a joint spectral resolution $E$ on $\R^n$.
A joint functional calculus is then defined via spectral integration: for every Borel function $m : \R^n \to \C$,
\[m(L) = m(L_1,\dots,L_n) = \int_{\R^n} m(\lambda) \,dE(\lambda)\]
is a normal operator on $L^2(X,\mu)$, which is bounded if and only if $m$ --- called the \emph{joint spectral multiplier} associated to $m(L)$ --- is ($E$-essentially) bounded. However, the abstract theory of spectral integrals does not tackle the following problem: to find conditions on the multiplier $m$ ensuring the boundedness of $m(L)$ on $L^p(X,\mu)$ for some $p \neq 2$.

We are interested in this problem when the measure space is a connected Lie group $G$ with a right Haar measure, and $L_1,\dots,L_n$ are left-invariant differential operators on $G$. In fact, the question has been studied quite extensively in the case of a single operator, namely, a sublaplacian or a higher-order analogue. On the other hand, for multiple operators, only specific classes of groups and specific choices of operators have been considered in the literature.

Suppose that $L_1,\dots,L_n$ are formally self-adjoint, left-invariant differential operators on a connected Lie group $G$, which commute pairwise (as operators on smooth functions). Under the assumption that the algebra generated by $L_1,\dots,L_n$ contains a weighted subcoercive operator --- a notion due to \cite{ter_elst_weighted_1998}, including positive elliptic operators, sublaplacians and Rockland operators --- we prove that $L_1,\dots,L_n$ are (essentially) self-adjoint and strongly commuting on $L^2(G)$. Moreover, we perform an abstract study of such a system of operators, in connection with the algebraic structure and the representation theory of $G$, similarly as what is done in the literature for the algebras of differential operators associated with Gelfand pairs.

Under the additional assumption that $G$ has polynomial volume growth, weighted $L^1$ estimates are obtained for the convolution kernel of the operator $m(L)$ corresponding to a compactly supported multiplier $m$ satisfying some smoothness condition. The order of smoothness which we require on $m$ is related to the degree of polynomial growth of $G$. Some techniques are presented, which allow, for some specific groups and 
operators, to lower the smoothness requirement on the multiplier.

In the case $G$ is a homogeneous Lie group and $L_1,\dots,L_n$ are homogeneous operators, a multiplier theorem of Mihlin-H\"ormander type is proved, extending the result for a single operator of \cite{christ_multipliers_1991} and \cite{mauceri_vectorvalued_1990}. Further, a product theory is developed, by considering several homogeneous groups $G_j$, each of which with its own system of operators; a non-conventional use of transference techniques then yields a multiplier theorem of Marcinkiewicz type, not only on the direct product of the $G_j$, but also on other (possibly non-homogeneous) groups, containing homomorphic images of the $G_j$. Consequently, for certain non-nilpotent groups of polynomial growth and for some distinguished sublaplacians, we are able to improve the general result of \cite{alexopoulos_spectral_1994}.

%% file: intro.tex
\chapter*{Introduction}
\addcontentsline{toc}{chapter}{Introduction}

\section*{Fourier and spectral multipliers}

\paragraph{Fourier multipliers.}
One of the classical problems of harmonic analysis is the study of operators of the form
\[T_m : f \mapsto \Four^{-1} (m \, \Four f),\]
where $\Four$ is the Fourier transform on $\R^n$, and $m : \R^n \to \C$ is a suitable (measurable) function, which is called the \emph{Fourier multiplier} associated to the operator $T_m$. In fact, the Fourier transform intertwines the operator $T_m$ with the operator of multiplication by $m$.

Since the Fourier transform is (modulo a constant factor) an isometry of $L^2(\R^n)$, boundedness of $T_m$ on $L^2(\R^n)$ is equivalent to (essential) boundedness of the multiplier $m$, and
\[\|T_m\|_{2 \to 2} = \|m\|_{\infty}.\]
On the other hand, boundedness of $T_m$ on $L^p(\R^n)$ for $p \neq 2$ is much more difficult to characterize.

Sufficient conditions for boundedness of $T_m$ on $L^p$ can be expressed in terms of smoothness conditions on the multiplier $m$. For instance, given a family of dilations
\[\delta_t(x_1,\dots,x_n) = (t^{\lambda_1} x_1,\dots,t^{\lambda_n} x_n)\]
on $\R^n$ (for some $\lambda_1,\dots,\lambda_n > 0$), we say that $m$ satisfies a \emph{Mihlin-H\"ormander condition}\footnote{Starting from the work of Marcinkiewicz \cite{marcinkiewicz_sur_1939} about multipliers of Fourier series, Mihlin \cite{mihlin_multipliers_1956}, \cite{mihlin_fourier_1957} obtained the sufficient condition
\begin{equation*}\label{eq:mihlin}
\sup_{\alpha \in \{0,1\}^n} \sup_{x \neq 0} |x|^{|\alpha|} |\partial^\alpha m(x)| < \infty
\tag{**}
\end{equation*}
for the operator $T_m$ to be bounded on $L^p$ for $1 < p < \infty$ (an account of the results of Marcinkiewicz and Mihlin can be found in the appendix of \cite{mihlin_multidimensional_1965}). Subsequently, H\"ormander \cite{hrmander_estimates_1960} reduced to $\lfloor n/2 \rfloor + 1$ the maximal order of derivatives of $m$ to be controlled, and replaced the supremum with an $L^2$ norm, thus obtaining substantially the same condition as \eqref{eq:mihlinhoermander} in the case of isotropic dilations $\delta_t(x) = tx$. The non-isotropic case was eventually considered by Kr\'ee \cite{kre_sur_1966} and Fabes and Rivi\`ere \cite{fabes_singular_1966}. An extensive bibliography about Fourier multipliers in the 1960s may be found in \S2.2.4 of \cite{triebel_interpolation_1978} and in Chapter IV of \cite{stein_singular_1970}.} of order $s \in \R$, adapted to the dilations $\delta_t$, if
\begin{equation*}\label{eq:mihlinhoermander}
\sup_{t > 0} \| (m \circ \delta_t) \eta \|_{H^s(\R^n)} < \infty,
\tag{*}
\end{equation*}
where $\eta$ is a non-negative smooth cut-off function supported on an annulus of $\R^n$ centered at the origin, and
\[\| f \|_{H^s(\R^n)}^2 = \int_{\R^n} |\Four f(\xi)|^2 \,(1 + |\xi|)^{2s} \,d\xi\]
is the $L^2$ Sobolev norm of fractional order $s$. Correspondingly, we have the Mihlin-H\"ormander multiplier theorem:

\begin{thmi}\label{thmi:euclmihlin}
If $m$ satisfies a Mihlin-H\"ormander condition of order 
\[s > \frac{n}{2},\]
adapted to some family of dilations on $\R^n$, then $T_m$ is of weak type $(1,1)$ and bounded on $L^p(\R^n)$ for $1 < p < \infty$.
\end{thmi}

A multi-variate analogue of the previous condition is the following: we say that $m$ satisfies a \emph{Marcinkiewicz condition}\footnote{The original result of Marcinkiewicz \cite{marcinkiewicz_sur_1939} was about Fourier series. The expression ``Marcinkiewicz multiplier theorem'' refers commonly to a transposition of Marcinkiewicz's result into the non-periodic setting, which is stated in \S XI.11.31 of \cite{dunford_linear_1963} (see also Theorem IV.6' of \cite{stein_singular_1970}, or Theorem~5.2.4 of \cite{grafakos_classical_2008}), and involves $L^1$ norms of mixed partial derivatives of the multiplier $m$. One feature of this condition on $m$ is its invariance by independent dilations of the components of $\R^n$; this feature disappears in Mihlin's pointwise condition \eqref{eq:mihlin}, but is recovered, e.g., by Lizorkin \cite{lizorkin_multipliers_1963} with the condition
\begin{equation*}\label{eq:lizorkin}
\sup_{\alpha \in \{0,1\}^n} \sup_{x_1,\dots,x_n \neq 0} |x_1^{\alpha_1} \cdots x_n^{\alpha_n} \, \partial^\alpha m(x)| < \infty.
\tag{\dag\dag}
\end{equation*}
Condition \eqref{eq:marcinkiewicz}, involving a multi-parameter $L^2$ Sobolev norm, is stated in \cite{carbery_homogeneous_1995} as the hypothesis of a ``multiparameter version of the H\"ormander-Marcinkiewicz multiplier theorem''. Subsequently, conditions such as \eqref{eq:marcinkiewicz} and \eqref{eq:lizorkin} have been referred to as ``Marcinkiewicz (\mbox{-type}) conditions'' in the context of joint spectral multipliers for left-invariant differential operators on Heisenberg-type groups (see \cite{mller_marcinkiewicz_1995}, \cite{mller_marcinkiewicz_1996}, \cite{fraser_marcinkiewicz_1997}, \cite{veneruso_marcinkiewicz_2000}, \cite{fraser_convolution_2001}, \cite{fraser_fold_2001}).} of order $\vec{s} = (s_1,\dots,s_n) \in \R^n$ if
\begin{equation*}\label{eq:marcinkiewicz}
\sup_{t_1,\dots,t_n > 0} \| m_{\vec{t}} \,\, \eta \otimes \dots \otimes \eta \|_{S^{\vec{s}} H(\R^n)} < \infty,
\tag{\dag}
\end{equation*}
where
\[m_{\vec{t}} \, (x_1,\dots,x_n) = m(t_1 x_1,\dots,t_n x_n),\]
$\eta$ is a smooth cut-off function as before, but on $\R^1$, and
\[\| f \|_{S^{\vec{s}} H(\R^n)}^2 = \int_{\R^n} |\Four f(\xi)|^2 \,(1 + |\xi_1|)^{2s_1} \cdots (1 + |\xi_n|)^{2s_n} \,d\xi\]
is the multi-parameter $L^2$ Sobolev norm of order $\vec{s}$. Then we have

\begin{thmi}\label{thmi:euclmarcinkiewicz}
If $m$ satisfies a Marcinkiewicz condition of order $\vec{s}$, with
\[s_1 > \frac{1}{2}, \quad \dots, \quad s_n > \frac{1}{2},\]
then $T_m$ is bounded on $L^p(\R^n)$ for $1 < p < \infty$.
\end{thmi}

The first result is proved via the Calder\'on-Zygmund theory of singular integral operators, while the second exploits also Littlewood-Paley theory. Their applications include a-priori estimates and regularity results for partial differential equations (see, e.g., \cite{hrmander_analysis_1990}). Notice that, in both theorems, the regularity threshold (i.e., the minimum number of derivatives to be controlled in order to obtain boundedness on $L^p$ for $1 < p < \infty$) is half the dimension of the environment space: in the former, the threshold is $n/2$, which is half of $\dim \R^n$, whereas in the latter it is $1/2$ in each component, that is half the dimension of each factor $\R^1$ of $\R^n$.

Generalizations of this kind of results are given at least in two directions. First of all, one can consider a setting where some sort of Fourier transform is defined, e.g., locally compact groups\footnote{See \cite{coifman_analyse_1971} for $SU_2$, \cite{rubin_multipliersrigid_1976} for the plane motion group, \cite{de_michele_heisenberg_1979} for the Heisenberg group; see also \cite{weiss_estimates_1972} for central multipliers on a compact Lie group.} or symmetric spaces. In particular, for non-compact symmetric spaces, Clerc and Stein \cite{clerc_multipliers_1974} proved that --- in contrast with the Euclidean case, where a finite order of differentiability of the multiplier is sufficient --- a necessary condition\footnote{For other results in the context of non-compact symmetric spaces, including sufficient conditions for boundedness on $L^p$, see \cite{anker_fourier_1990}, \cite{giulini_multipliers_1997} and references therein.} for a function $m$ to define a bounded operator $T_m$ on $L^p$ for some $p \neq 2$ is the existence of a holomorphic extension of $m$.

\paragraph{Spectral multipliers.}
The second direction develops from the observation that, if $P : \R^n \to \R$ is a polynomial, then the differential operator 
\[L = P(-i\partial_1,\dots,-i\partial_n)\]
is (essentially) self-adjoint on $L^2(\R^n)$, thus admits a spectral resolution
\[L = \int_{\R} \lambda \,dE(\lambda),\]
and moreover, for a bounded Borel function $m : \R \to \C$, we have
\[T_{m \circ P} = m(L) = \int_{\R} m(\lambda) \,dE(\lambda);\]
an easy consequence of Theorem~\ref{thmi:euclmihlin} is then the following

\begin{thmi}\label{thmi:euclmihlinop}
Let $L = P(-i\partial_1,\dots,-i\partial_n)$, where $P : \R^n \to \R$ is a polynomial which is homogeneous with respect to some family of dilations on $\R^n$ and nowhere null off the origin. If $m : \R \to \C$ satisfies a Mihlin-H\"ormander condition of order
\[s > \frac{n}{2},\]
then $m(L)$ is of weak type $(1,1)$ and bounded on $L^p(\R^n)$ for $1 < p < \infty$.
\end{thmi}

The function $m$ defining the operator $m(L)$ is called a \emph{spectral multiplier} for $L$. A typical example for the operator $L$ is the Laplacian $-(\partial_1^2 + \dots + \partial_n^2)$, which corresponds to the polynomial $P(x) = x_1^2 + \dots + x_n^2$, and therefore to radial Fourier multipliers $m \circ P$.

In a setting where an analogue $L$ of the Laplacian is defined, the spectral theorem gives a functional calculus for $L$ such that, if $m : \R \to \C$ is a bounded Borel function, then $m(L)$ is a bounded operator on $L^2$; therefore one can again look for sufficient conditions on the spectral multiplier $m$ ensuring boundedness of $m(L)$ on $L^p$ for some $p \neq 2$.

This question has been explored particularly in the case of a connected Lie group $G$ with a hypoelliptic left-invariant self-adjoint\footnote{Since we consider left-invariant differential operators on $G$, it is understood that the measure defining the Lebesgue spaces $L^p(G)$ is a right Haar measure.} differential operator $L$, typically a sublaplacian. Notice that, for a sublaplacian $L$ on a connected Lie group $G$, a general result of Stein (see \cite{stein_topics_1970}, \S IV.6) yields boundedness of the operator $m(L)$ on $L^p(G)$ for $1 < p < \infty$ whenever the multiplier $m$ is of Laplace transform type; this rather strong condition on the multiplier --- it implies analyticity --- can be considerably weakened on specific classes of groups.

For a homogeneous sublaplacian $L$ on a stratified Lie group $G$, there is a result analogous to Theorem~\ref{thmi:euclmihlinop} due to\footnote{This problem has a long history: starting from Theorem~6.25 of \cite{folland_hardy_1982}, due to Hulanicki and Stein, the smoothness requirement on the multiplier has been lowered and lowered (see \cite{mauceri_maximal_1987}, \cite{de_michele_mulipliers_1987}), culminating with the above-mentioned result.} Christ \cite{christ_multipliers_1991} and Mauceri and Meda \cite{mauceri_vectorvalued_1990}, where a condition of order
\[s > \frac{Q}{2},\]
i.e., half the homogeneous dimension $Q$ of the stratified group $G$, is required. This result in general is not sharp: although on $\R^n$ (with isotropic dilations) it is $Q= n$, on more general stratified groups $G$ the homogeneous dimension $Q$ is greater than the topological dimension $\dim G$, and in the case of a Heisenberg-type group $G$ it has been proved (by M\"uller and Stein \cite{mller_spectral_1994} and Hebisch \cite{hebisch_multiplier_1993} for a sublaplacian, and by Hebisch and Zienkiewicz \cite{hebisch_multiplier_1995} for a more general positive Rockland operator) that a condition of order
\[s > \frac{\dim G}{2}\]
is sufficient.

An analogous result, due to Alexopoulos \cite{alexopoulos_spectral_1994}, is known for a sublaplacian $L$ on a connected Lie group $G$ of polynomial volume growth. In this case, the condition on the multiplier is expressed in terms of an $L^\infty$ Besov norm, and the regularity threshold involves both a local dimension $Q_0$ (related to the Carnot-Carath\'eodory distance associated to the sublaplacian $L$) and a dimension at infinity $Q_\infty$ (that is, the degree of polynomial growth of $G$), so that a condition of order
\[s > \frac{\max\{Q_0,Q_\infty\}}{2}\]
is sufficient; notice that, for a homogeneous sublaplacian on a stratified group, it is $Q_0 = Q_\infty = Q$. An extension of this result to higher-order operators is contained in \cite{duong_plancherel-type_2002} (where generalizations to different settings, such as fractals, are also discussed). Moreover, in the case of a distinguished sublaplacian on the compact group $G = SU_2$, for which $Q_0 = 4 > 3 = \dim G$, Cowling and Sikora \cite{cowling_spectral_2001} have shown that a condition of order
\[s > \frac{\dim G}{2}\]
(with an $L^2$ Sobolev norm) is sufficient.

In the context of Lie groups with exponential volume growth, two different behaviours may appear: in some cases (similarly to what happens on non-compact symmetric spaces) a multiplier $m$ defining a bounded operator $m(L)$ on $L^p$ for some $p \neq 2$ necessarily extends to a holomorphic function (see \cite{christ_spectral_1996}, \cite{ludwig_sub-laplacians_2000}), whereas in other cases a finite order of differentiability is sufficient to ensure boundedness on $L^p(G)$ for $1 < p < \infty$ (see \cite{hebisch_subalgebra_1993}, \cite{cowling_spectral_1994}); the distinction between the two behaviours is somehow connected with the symmetry of the Banach $*$-algebra $L^1(G)$.

\paragraph{Joint spectral multipliers.}
Suppose now that, instead of a single operator, we have a family $L_1,\dots,L_n$ of (essentially) self-adjoint operators, which moreover commute strongly pairwise, i.e., admit a joint spectral resolution $E$ on $\R^n$:
\[L_j = \int_{\R^n} \lambda_j \,dE(\lambda)\]
for $j=1,\dots,n$. This defines a joint functional calculus for $L_1,\dots,L_n$, which yields, for every bounded Borel $m : \R^n \to \C$, a bounded operator
\[m(L) = m(L_1,\dots,L_n) = \int_{\R^n} m(\lambda) \,dE(\lambda)\]
on $L^2$; the function $m$ is said to be a \emph{joint spectral multiplier} for the system $L_1,\dots,L_n$, and again one can look for sufficient conditions on $m$ ensuring boundedness of $m(L)$ on $L^p$ for some $p \neq 2$.

Notice that Fourier multipliers on $\R^n$ fall into this context: in fact, a Fourier multiplier on $\R^n$ can be thought of as a joint spectral multiplier for the system $-i\partial_1,\dots,-i\partial_n$ of differential operators on $\R^n$. However, for instance, on a non-commutative Lie group, two left-invariant differential operators do not necessarily commute, so that joint spectral multipliers can be considered only for particular systems of left-invariant differential operators.

Certainly, if the environment space is a direct product, and each $L_j$ operates on a different factor, then strong commutativity is ensured; this case is considered in \cite{sikora_multivariable_2009}, where results of Mihlin-H\"ormander type are obtained, and applications to products of fractals are given.

Another extensively studied context is that of Heisenberg-type groups:
\begin{itemize}
\item Mauceri \cite{mauceri_zonal_1981} obtained a result of Mihlin-H\"ormander type for joint functions of the sublaplacian $L$ and the central derivative $-iT$ on Heisenberg groups;
\item M\"uller, Ricci and Stein \cite{mller_marcinkiewicz_1995}, \cite{mller_marcinkiewicz_1996} proved results of Marcinkiewicz type for the sublaplacian $L$ and the central derivatives $-iT_1,\dots,-iT_n$ on Heisenberg-type groups;
\item Fraser \cite{fraser_marcinkiewicz_1997}, \cite{fraser_fold_2001}, \cite{fraser_convolution_2001} and Veneruso \cite{veneruso_marcinkiewicz_2000} obtained results of Marcinkiewicz type for the partial sublaplacians $L_1,\dots,L_n$ and the central derivative $-iT$ on Heisenberg groups.
\end{itemize}
In some of these works, which are concerned with specific groups and operators, sharp thresholds are obtained.

However, to our knowledge, in the literature there is no result for multiple operators with the same generality as in the above-presented results for a single operator.

\section*{Our results}

One of the aims of this work is to partially fill the mentioned gap in the literature, by obtaining joint spectral multiplier results in some generality. Therefore, one of the problems which we had to face was that of identifying the right object to be studied, i.e., a class of systems of left-invariant differential operators on a connected Lie group for which joint multiplier theorems were likely to hold. The desiderata were, on one hand, the applicability of the spectral theorem for defining the joint functional calculus, and consequently the self-adjointness on $L^2(G)$ and the strong commutativity of the operators under consideration; on the other hand, the validity of some analogue of the ``heat kernel estimates'', which are an essential tool in the proof of multiplier theorems for a single operator. 

Our solution involves the \emph{weighted subcoercive operators} of ter Elst and Robinson \cite{ter_elst_weighted_1998}, which are a rather wide class of left-invariant differential operators on connected Lie groups, including positive elliptic operators, sublaplacians, and also positive Rockland operators on homogeneous groups. Correspondingly, we define a system $L_1,\dots,L_n$ of left-invariant differential operators on a connected Lie group $G$ to be a \emph{weighted subcoercive system} if the $L_j$ are formally self-adjoint and pairwise commuting, and if moreover the algebra generated by them contains a weighted subcoercive operator. We then prove that, if $L_1,\dots,L_n$ is a weighted subcoercive system on the group $G$, then the operators $L_j$ are (essentially) self-adjoint on $L^2(G)$ and commute strongly pairwise; moreover, the presence of a weighted subcoercive operator in the algebra generated by $L_1,\dots,L_n$ provides, by the results of \cite{ter_elst_weighted_1998}, the required heat kernel estimates.

A somehow abstract study of these weighted subcoercive systems is subsequently carried out, exploring their relationships with the algebraic structure and the representation theory of the environment group, analogously as what is done in the literature for the algebras of differential operators associated with Gelfand pairs. In particular, we show that, to every weighted subcoercive system $L_1,\dots,L_n$ on a group $G$, a Plancherel measure on the joint $L^2$ spectrum of the operators can be associated, so that the $L^2$ norm of the convolution kernel of the operator $m(L)$ on $G$ coincides with the $L^2$ norm of the multiplier $m$ with respect to the Plancherel measure.

Next, for a weighted subcoercive system $L_1,\dots,L_n$ on a Lie group $G$ with polynomial growth, we obtain weighted $L^1$ estimates for the convolution kernel of the operator $m(L)$ corresponding to a joint multiplier $m$ with compact support, satisfying some smoothness condition. This implies in particular that the convolution kernel corresponding to a smooth and compactly supported multiplier $m$ is integrable, so that the operator $m(L)$ is bounded on $L^p(G)$ for $1 \leq p \leq \infty$; in fact, this result is shown to hold also for multipliers $m$ in the Schwartz class of rapidly decreasing smooth functions.

The weighted estimates, together with the Calder\'on-Zygmund singular integral theory, then yield, in the case of a homogeneous group $G$ with homogeneous operators $L_1,\dots,L_n$, a joint multiplier theorem of Mihlin-H\"ormander type, with a condition of order
\[s > \frac{Q_G}{2} + \frac{n-1}{q}\]
in terms of an $L^q$ Besov norm with $q \in \left[2,\infty\right]$, where $Q_G$ is the degree of polynomial growth of $G$. Notice that this result extends the multiplier theorem of \cite{christ_multipliers_1991} and \cite{mauceri_vectorvalued_1990} for a sublaplacian on a stratified group.

For a homogeneous group $G$ with a positive Rockland operator $L$, the homogeneous dimension of $G$ corresponds to the local dimension $Q_0$ associated to $L$ as in \cite{alexopoulos_spectral_1994}, whereas $Q_G = Q_\infty$ is the dimension at infinity, and it is always $Q_0 \geq Q_G$, with equality if $G$ is stratified. Since most of the results in the literature about spectral multipliers on homogeneous Lie groups are stated in the context of stratified groups, the different role of these two quantities generally remains in the shade. However, on a fixed nilpotent Lie group $G$, there may be several homogeneous structures --- as in $\R^n$ there are plenty of families of non-isotropic dilations --- with possibly different homogeneous dimensions, whereas the degree of polynomial growth $Q_G$ is intrinsic to the algebraic structure of $G$. In our multiplier theorem, differently from \cite{alexopoulos_spectral_1994}, but analogously as in Theorem~\ref{thmi:euclmihlin} on $\R^n$, the homogeneous dimension $Q_0$ does not appear in the regularity threshold, which depends only on $Q_G$.

Some methods are presented which allow, for certain groups and systems of operators, to improve the multiplier theorem by lowering the regularity threshold. In particular, the technique of Hebisch and Zienkiewicz \cite{hebisch_multiplier_1995} is extended to our multi-variate setting, enabling us to replace, e.g, in the case of a Heisenberg-type group $G$, the degree $Q_G$ of polynomial growth with the topological dimension $\dim G$ in the threshold. The summand $(n-1)/q$ in the threshold can be lowered too, through an analysis of the Plancherel measure associated to a specific weighted subcoercive system, which sometimes can be explicitly computed.

Finally, a sort of product theory is developed, yielding a multiplier theorem of Marcinkiewicz type involving several homogeneous groups $G_j$, each of which with its own weighted subcoercive system. A non-conventional use of transference techniques allows us to obtain this multiplier theorem not only on the direct product of the $G_j$, but also on other groups (containing homomorphic images of the $G_j$) which need not be nilpotent. In this way, we are also able to improve, for certain non-nilpotent groups of polynomial growth (including the plane motion group, the oscillator groups and the diamond groups) and for some distinguished sublaplacians, the general result of \cite{alexopoulos_spectral_1994}.

\section*{Structure of the work}

This thesis is divided into five chapters. The first two chapters are of introductory character, and are aimed to establish the language which is used throughout the work.

Specifically, the first chapter is a brief summary of the basic definitions and results related to Lie groups, Lie algebras and translation-invariant differential operators. Several classes of Lie groups are considered, focusing particularly on nilpotent and homogeneous Lie groups. Moreover, some of the results of ter Elst and Robinson \cite{ter_elst_weighted_1998} about weighted subcoercive operators are discussed and slightly amplified.

The second chapter introduces the instruments used to measure smoothness of functions, and particularly of joint multipliers, i.e., Sobolev and Besov spaces on $\R^n$. Special attention is payed to Sobolev and Besov spaces with dominating mixed smoothness, which allow to prescribe different orders of differentiability on different directions. Mihlin-H\"ormander and Marcinkiewicz conditions are then studied abstractly (i.e., without immediate reference to multiplier theorems), by investigating their behaviour under a change of variables, along with mutual implications.

The remaining three chapters contain the above-mentioned results. Namely, in the third chapter weighted subcoercive systems of operators on a connected Lie group are introduced, together with the associated Plancherel measure, and relationships with the algebraic structure and the representation theory of the group are examined. In the fourth chapter, weighted estimates for convolution kernels of operators in the functional calculus are obtained, and a few examples are worked out. In the fifth chapter, our two multiplier theorems of Mihlin-H\"ormander and Marcinkiewicz type are proved, and some applications are discussed.

Finally, an appendix collects some known results about spectral integration and Banach $*$-algebras which are extensively used in the work.

\section*{Acknowledgements}

My deepest gratitude goes to my supervisor, prof.\ Fulvio Ricci, for his great expertise and helpfulness, for his ability in posing the right questions, and for his enormous patience, without which the present work would not have been possible.

I would also like to thank A.\ Carbonaro, M.\ Cowling, E.\ David-Guillou and C.\ Molitor-Braun for interesting discussions, A.~F.~M.\ ter Elst for the exchange of comments on his work, and L.\ Ambrosio, P.\ Ciatti, G.~M.~Dall'Ara and G.~Mauceri for pointing out some references to the literature.

Last, but not least, I thank my family and my friends for their support, their listening ears and their encouragement.

%% file: notationlist.tex
\section*{List of commonly used notation}
\addcontentsline{toc}{chapter}{List of commonly used notation}

\begin{tabular}{@{}lll}
\TS $\N$              & natural numbers (including $0$) & \\
\TS $\Z$, $\Q$, $\R$, $\C$ & integral, rational, real and complex numbers & \\
\TS $\R^+$            & positive real numbers (excluding $0$) & \\ 
\TS $\T$              & $1$-dimensional torus & \\
\TS $S^n$             & $n$-dimensional sphere & \\
\\
\TS $C(X)$            & continuous (complex-valued) functions & \\
\TS $C_b(X)$          & bounded continuous functions & \\
\TS $C_0(X)$          & continuous functions vanishing at infinity & \\
\TS $C_c(X)$          & compactly supported continuous functions & \\
\TS $C_{lub}(G)$      & left uniformly continuous and bounded functions & \\
\TS $C_{rub}(G)$      & right uniformly continuous and bounded functions & \\
\TS $C_{ub}(G)$       & left and right uniformly continuous and bounded functions & \\
\TS $\E(X)$           & smooth functions & \\
\TS $\E'(X)$          & compactly supported distributions & \\
\TS $\D(X)$           & compactly supported smooth functions & \\
\TS $\D'(X)$          & distributions & \\
\TS $\Sz(G)$          & Schwartz functions & \\
\TS $\Sz'(G)$         & tempered distributions & \\
\TS $L^p(X,\mu)$      & Lebesgue spaces & \\
\TS $L^p(G)$          & Lebesgue spaces with respect to a right Haar measure & \\
\TS $L^{p;\infty}(G)$ & smooth functions which are in $L^p(G)$ & \\
                      & together with all their left-invariant derivatives & \\
\TS $C^\infty_0(G)$   & smooth functions which vanish at infinity & \\
                      & together with all their left-invariant derivatives & \\
\TS $W^{k,p}(\R^n)$   & classical $L^p$ Sobolev spaces & \\
\TS $H^s(\R^n)$       & $L^2$ Sobolev spaces of fractional order & \\
\TS $B_{p,q}^s(\R^n)$ & Besov spaces & \\
\TS $S_{p,q}^{\vec{s}}B(\R^{\vec{n}})$ & Besov spaces with dominating mixed smoothness & \\
\\
\TS $\LA_x$, $\RA_x$  & left and right regular representations & \\
\TS $\Diff(G)$        & left-invariant differential operators & \\
\TS $D^+$             & formal adjoint of $D$ & \\
\\
\TS $|\cdot|_p$       & $p$-norm on $\R^n$ & \\
\TS $|\cdot|_\delta$  & homogeneous norm with respect to the dilations $\delta_t$ & \\
\TS $Q_\delta$        & homogeneous dimension with respect to the dilations $\delta_t$ & \\
\TS $Q_G$             & degree of polynomial growth (dimension at infinity) & \\
\TS $\dim G$          & topological dimension & \\
\end{tabular}

%% file: operators.tex
\chapter{Differential operators on Lie groups}

This chapter is an introduction to our main objects of study and to the environment in which they live.

Starting from general Lie groups, we then consider more restrictive conditions (type I, amenability, polynomial growth, ...), eventually reaching nilpotent and homogeneous Lie group, on which the fundamental class of Rockland operators is defined.

In the last part, we move back to general, by showing how homogeneous Lie groups and Rockland operators may serve as a model for studying some invariant differential operators (the so-called weighted subcoercive operators) on general Lie groups.

\section{Lie groups and differential operators}

Here we recall the basic definitions and results involving Lie groups, Lie algebras and differential operators, which will be freely used in the following. This brief exposition also allows us to set out the notation.

Most of the theory presented here for Lie groups is in fact valid in the more general context of locally compact topological groups; for these results, we refer mainly to the treatises of Hewitt and Ross \cite{hewitt_abstract_1979} and Folland \cite{folland_course_1995}. For the parts which instead are specific to Lie groups, our primary references are the books of Helgason \cite{helgason_differential_1962}, \cite{helgason_differential_1978}, \cite{helgason_groups_1984} and Varadarajan \cite{varadarajan_lie_1974}. A discussion of distributions on smooth manifolds and Lie groups, including the Schwartz kernel theorem, can be found in the treatise of Dieudonn\'e \cite{dieudonn_treatise_analysis_1972}, \cite{dieudonn_treatise_analysis_1974}, \cite{dieudonn_treatise_analysis_1988}; however, most results are straightforward generalizations of the corresponding results for $\R^n$ given, e.g., in the book of Tr\`eves \cite{treves_topological_1967}.

Notice that, since we focus mainly on left-invariant differential operators, the ``standard'' measure on a Lie group will be a right Haar measure, and the definition of convolution of functions will be formulated accordingly. In some of the references, the choice is instead for right-invariant operators and/or left Haar measures, so that the results summarized here might take a slightly different form in those works.

\subsection{Lie groups and their Lie algebras}
In the following, a \emph{Lie group}\index{Lie group} will be a second-countable smooth manifold $G$ endowed with a group structure such that the maps
\[G \times G \ni (x,y) \mapsto xy \in G, \qquad G \ni x \mapsto x^{-1} \in G\]
(multiplication and inversion) are smooth. We set
\[\la_g(x) = gx, \qquad \ra_g(x) = x g^{-1}\]
(left and right translations) and correspondingly
\[\LA_g f = f \circ \la_{g}^{-1}, \qquad \RA_g f = f \circ \ra_g^{-1}\]
(left and right regular representations).

On a Lie group, we have a \emph{left uniform structure}, generated by entourages of the form
\[\{(x,y) \in G \times G \tc x^{-1}y \in U\}\]
where $U$ is a neighborhood of the identity $e \in G$, and a \emph{right uniform structure}, generated by entourages of the form
\[\{(x,y) \in G \times G \tc yx^{-1} \in U\}.\]
These two structures do not necessarily coincide (they do coincide, e.g., when $G$ is compact or abelian), however both structures are compatible with the topology of $G$. Correspondingly, we have the notions of \emph{left (right) uniformly continuous}\index{function!uniformly continuous} (complex-valued) functions on $G$; it can be shown that $f : G \to \C$ is left (right) uniformly continuous if and only if
\[\lim_{g \to e} \|\RA_g f - f\|_\infty = 0 \qquad \text{(}\lim_{g \to e} \|\LA_g f - f\|_\infty = 0\text{).}\]

A left-invariant vector field on $G$ is a (smooth) vector field $X$ such that
\[d(\la_g)_x(X_x) = X_{gx} \qquad\text{for all $g,x \in G$}\]
or, equivalently, thinking of $X$ as an operator on (smooth) functions,
\[X \LA_g = \LA_g X \qquad\text{for all $g \in G$.}\]
Right-invariant vector fields are defined analogously, using right translations in place of left translations. From the definition, it is clear that a left- (or right-) invariant vector field is uniquely determined by its value at the identity $e \in G$. Moreover, left- (right-) invariant vector fields constitute a Lie subalgebra of the Lie algebra of smooth vector fields on $G$.

The Lie algebra\index{Lie algebra!of a Lie group} $\lie{g}$ of the Lie group $G$ is the Lie subalgebra of left-invariant vector fields on $G$; as we said before, $\lie{g}$ can be identified with the tangent space $T_e G$ of $G$ at the identity, so that in particular $\lie{g}$ is finite-dimensional and its dimension coincides with the topological dimension $\dim G$ of the Lie group.

If $\phi : G \to G'$ is a Lie group homomorphism (i.e., $\phi$ is smooth and a group homomorphism), then it maps left-invariant vector fields on $G$ to left-invariant vector fields on $G'$; the correspondence $\phi' : \lie{g} \to \lie{g}'$ thus established is a Lie algebra homomorphism, which is called the \emph{derivative} of $\phi$.

The map
\[\la_x \circ \ra_x : G \ni g \mapsto x g x^{-1} \in G\]
is an (inner) automorphism of $G$; its derivative is then an automorphism of $\lie{g}$, which is denoted by $\Ad(x)$. The map $\Ad : G \to \Aut(\lie{g})$ is a Lie group homomorphism, which is called the \emph{adjoint action} of $G$ on its Lie algebra $\lie{g}$. The derivative of $\Ad$, denoted by $\ad$, is the adjoint action of the Lie algebra $\lie{g}$ on itself, which maps $\lie{g}$ onto the inner derivations of $\lie{g}$:
\[\ad(X)(Y) = [X,Y].\]

\subsection{Exponential map}
The \emph{exponential map}\index{exponential map}
\[\exp : \lie{g} \to G\]
is a smooth map which is uniquely determined by the condition that
\begin{equation}\label{eq:1parsg}
\R \ni t \mapsto \exp(tX) \in G
\end{equation}
gives the flow curve of $X$ through the identity ($\exp(0) = e$). The exponential map is a local diffeomorphism at $0$ (its differential $d(\exp)_0 : \lie{g} \to \lie{g}$ is the identity map). The curves \eqref{eq:1parsg} are $1$-parameter subgroups of $G$ and in fact every $1$-parameter subgroup of $G$ can be written in this form.

In exponential coordinates, the product law of $G$ can be written as a power series in a neighborhood of the identity, through the \emph{Baker-Campbell-Hausdorff formula}\index{Baker-Campbell-Hausdorff formula}:
\begin{equation}\label{eq:bch}
\exp(X) \exp(Y) = \exp(b(X,Y)),
\end{equation}
where
\[b(X,Y) = X + Y + \sum_{n \geq 2} b_n(X,Y),\]
$b_n(X,Y)$ is a linear combination of $n$-times-iterated commutators of $X,Y$ (whose coefficients do not depend on the Lie group $G$) and the series converges for sufficiently small $X,Y$ (see \cite{varadarajan_lie_1974}, Theorem~2.15.4).

\subsection{Translation-invariant differential operators}
A (smooth) differential operator $D$ on $G$ (with complex coefficients) is said to be left-invariant\index{differential operator!left- (right-) invariant} if
\[\LA_g D = D \LA_g \qquad\text{for all $g \in G$.}\]
Right-invariant differential operators are defined analogously.

Left-invariant vector fields are left-invariant differential operators of order $1$ and in fact they generate the whole algebra of left-invariant differential operators; more precisely, the algebra $\Diff(G)$ of left-invariant differential operators on $G$ can be identified with the universal enveloping algebra\index{Lie algebra!universal enveloping algebra of} $\UEnA(\lie{g}_\C)$ of the complexification of $\lie{g}$.

\subsection{Haar measures}
A \emph{left (right) Haar measure}\index{Haar measure} on $G$ is a non-null regular positive Borel measure which is invariant under left (right) translations. Since a Lie group $G$ is a locally compact group, a left (right) Haar measure on $G$ exists and is unique, up to multiplication by a positive constant. A Haar measure on a Lie group is smooth, i.e., in every local coordinate system it is absolutely continuous with respect to the Lebesgue measure, with a density function which is smooth and nowhere null.

Let $\mu_G$ be a right Haar measure on $G$. For every $x \in G$, the push-forward $\la_x(\mu_G)$ is still a right Haar measure, so that
\[\la_x(\mu_G) = \Delta_G(x) \mu_G\]
for some $\Delta_G(x) > 0$. The function $\Delta_G : G \to \R^+$ is called the \emph{modular function} of the group $G$ and is a (smooth) homomorphism of Lie groups, which can be expressed in terms of the adjoint action $\Ad$ of $G$ on $\lie{g}$ as
\[\Delta_G(x) = \left|\det \Ad(x^{-1})\right|.\]
(cf.\ \cite{helgason_differential_1962}, \S X.1.1). If $\nu_G(A) = \mu_G(A^{-1})$, then $\nu_G$ is a left Haar measure on $G$. In fact $\nu_G = \Delta_G \mu_G$, i.e., $\nu_G$ is absolutely continuous with respect to $\mu_G$, with density $\Delta_G$.

$G$ is said to be \emph{unimodular}\index{Lie group!unimodular} if $\Delta_G \equiv 1$. In a unimodular group, a right Haar measure is also a left Haar measure, and vice versa. Compact groups and abelian groups are unimodular; more generally, if the left and right uniform structures of $G$ coincide, then $G$ is unimodular (see \cite{hewitt_abstract_1979}, \S 19.28).

\subsection{Function spaces}
In the following, let a right Haar measure $\mu_G$ be fixed on $G$, so that expressions of the form
\[\int_G f(x) \,dx\]
will be always understood as integrals with respect to $\mu_G$. The chosen measure determines also the Lebesgue spaces $L^p(G) = L^p(G,\mu_G)$ for $1 \leq p \leq \infty$; if $G$ is not unimodular, these spaces differ from the corresponding Lebesgue spaces $L^p(G,\Delta_G \mu_G)$ with respect to the left Haar measure $\Delta_G \mu_G$. For a Borel function $f : G \to \C$, we set for short
\[\|f\|_p = \|f\|_{L^p(G)} = \|f\|_{L^p(G,\mu_G)}, \qquad \|f\|_{\hat p} = \|f\|_{L^p(G,\Delta_G \mu_G)}.\]
We will also use the notation $C(G)$ for (the space of) continuous functions on $G$, $C_b(G)$ for bounded continuous functions, $C_{lub}(G)$ for left uniformly continuous and bounded functions, $C_{rub}(G)$ for right uniformly continuous and bounded functions, $C_0(G)$ for continuous functions vanishing at infinity, $C_c(G)$ for compactly supported continuous functions, $\E(G)$ for smooth functions, $\D(G)$ for compactly supported smooth functions. Moreover, $\D'(G)$ will denote the space of distributions on $G$, whereas $\E'(G)$ will be the space of compactly supported distributions. As usual, the space $L^1_\loc(G)$ of locally $\mu_G$-integrable functions on $G$ is naturally embedded into $\D'(G)$, by identifying a function $f \in L^1_\loc(G)$ with the Radon measure $f \mu_G$.

We will use the notation
\begin{equation}\label{eq:pairing}
\langle f, g \rangle = \int_G f \, \overline{g} \,d\mu_G,
\end{equation}
with the obvious extension when one of $f,g$ is a distribution\footnote{In fact, if distributions are defined as conjugate-linear functionals on the space of test functions, then the pairing $\langle f, g \rangle$ for $f \in \D'(G)$ and $g \in \D(G)$ is simply the evaluation of $f$ on $g$, whereas \eqref{eq:pairing} determines the embedding of $L^1_\loc(G)$ in $\D'(G)$.}. In particular, the pairing $\langle \cdot, \cdot \rangle$ will always be linear in the first argument and conjugate-linear in the second, in order to extend the inner product of $L^2(G)$.

\subsection{Formal adjoint of a differential operator}
The \emph{formal adjoint}\index{differential operator!formal adjoint of} of a (smooth) differential operator $D$ on $G$ (with respect to the right Haar measure $\mu_G$) is the differential operator $D^+$ on $G$ such that
\[\langle D^+ \phi, \psi \rangle = \langle \phi, D \psi\rangle \qquad \text{for all $\phi,\psi \in \D(G)$.}\]
Existence and uniqueness of the formal adjoint are guaranteed by the properties of smoothness of the measure $\mu_G$.

If $D$ is left- (right-) invariant, then also $D^+$ is, and the map $D \mapsto D^+$ is an involutive conjugate-linear anti-automorphism of the algebra $\Diff(G)$, which makes it into a $*$-algebra. If $X \in \lie{g}$ is a left-invariant vector field, then its flow is given by right translations, so that the measure $\mu$ is invariant under the flow of $X$, but this implies that
\[\langle X \phi, \psi \rangle + \langle \phi, X \psi \rangle = 0 \qquad\text{for all $\phi,\psi \in \D(G)$,}\]
i.e., $X^+ = -X$.

A differential operator $D$ on $G$ is said to be \emph{formally self-adjoint} if $D^+ = D$; in particular, if $X \in \lie{g}$, then $-iX$ is formally self-adjoint.

\subsection{Strongly continuous representations, unitary representations, smooth vectors}
A \emph{representation}\index{representation!of a Lie group} $\pi$ of $G$ on a topological vector space $\VV$ is a map from $G$ to the bounded linear operators on $\VV$ such that
\[\pi(e) = \id_\VV, \qquad \pi(xy) = \pi(x) \pi(y);\]
we will always suppose that a representation is \emph{strongly continuous}, i.e., for every $v \in \VV$, the map
\[G \ni x \mapsto \pi(x) v \in \VV\]
is continuous. A \emph{unitary representation} of $G$ is a representation $\pi$ of $G$ on a Hilbert space $\HH$, such that the operators $\pi(x)$ for $x \in G$ are unitary operators on $\HH$.

If $\pi$ is a representation of $G$ on a Fr\'echet space $\VV$ and $\phi \in L^1(G)$ has compact support, then the expression
\begin{equation}\label{eq:representationoffunctions}
\pi(\phi) v = \int_G \pi(x^{-1}) v \, \phi(x) \,dx \qquad\text{for $v \in \VV$}
\end{equation}
defines a bounded linear operator $\pi(\phi)$ on $\VV$. If $\pi$ is a \emph{uniformly bounded} representation (i.e., $\sup_{x \in G} p(\pi(x) v) < \infty$ for all seminorms $p$ of $\VV$ and all $v \in \VV$), then the definition of $\pi(\phi)$ can be extended to all $\phi \in L^1(G)$. Notice that, if $\pi$ is unitary, then $\|\pi(\phi)\| \leq \|\phi\|_1$.

Given a representation $\pi$ of $G$ on a Fr\'echet space $\VV$, we say that $v \in \VV$ is a \emph{smooth vector} of the representation $\pi$ if the map
\[G \ni x \mapsto \pi(x)v \in \VV\]
is smooth. The set $\VV^\infty$ of the smooth vectors of $\pi$ is a dense linear subspace of $\VV$ and moreover, if $\phi \in \D(G)$, then $\pi(\phi) v \in \VV^\infty$ for all $v \in \VV$; in fact, by a theorem of Dixmier and Malliavin's (see \cite{dixmier_factorisations_1978}, Th\'eor\`eme~3.3), the space $\VV^\infty$ coincides with the \emph{G\aa rding space} of the representation, i.e., the set of finite sums of vectors of the form $\pi(\phi) v$ for $v \in \VV$, $\phi \in \D(G)$.

For every $v \in \VV^\infty$, we can consider the differential of $\pi(\cdot) v$ at the identity $e \in G$, which will be denoted by $d\pi(\cdot)v$. It can be shown that, if $X \in \lie{g}$ and $v \in \VV^\infty$, then also $d\pi(X)v \in \VV^\infty$, so that $d\pi(X) : v \mapsto d\pi(X)v$ can be considered as a linear operator on $\VV^\infty$. Since
\[d\pi([X,Y]) = d\pi(X) d\pi(Y) - d\pi(Y) d\pi(X),\]
then $d\pi$ can be extended to a homomorphism of unital algebras from $\Diff(G)$ to the linear operators on $\VV^\infty$, and one can prove that
\[d\pi(D) \pi(\phi) v = \pi(D\phi) v\]
for all $v \in \VV$, $\phi \in \D(G)$ and $D \in \Diff(G)$. If $\pi$ is a unitary representation, it can be shown that, for all $D \in \Diff(G)$,
\begin{equation}\label{eq:adjointrepresentation}
d\pi(D^+) \subseteq d\pi(D)^*,
\end{equation}
and in particular a formally self-adjoint operator $D$ is mapped to a symmetric (densely defined) operator on $\VV$.

\subsection{Regular representations.}\label{subsection:regularrepresentations}
The \emph{(right) regular representation}\index{representation!of a Lie group!regular} of $G$ on $L^p(G)$ ($1 \leq p < \infty$) is the representation $\pi$ given by $\pi(g) = \RA_g$. Since $\mu_G$ is right-invariant, for every $g \in G$, $\pi(g)$ is an isometry of $L^p(G)$; in particular, if $p = 2$, $\pi$ is a unitary representation. An element $f \in L^p(G)$ is a smooth vector of $\pi$ if and only if, for all $D \in \Diff(G)$, $Df \in L^p(G)$ (where $Df$ is meant in the sense of distributions); moreover, in this case $d\pi(D) f = Df$.

We denote by $L^{p;\infty}(G)$ the space of smooth vectors of the regular representation of $G$ on $L^p(G)$; more generally, for a measurable function $u : G \to \C$, we denote by $L^{p;\infty}(G, u(x) \,dx)$ the Fr\'echet space of the functions $f$ on $G$ which belong to the weighted Lebesgue space $L^p(G,u(x)\,dx)$ together with all their right-invariant derivatives.

An extension of the previous considerations is obtained by observing that a homomorphism $\gamma : G \to G'$ of Lie groups induces\index{representation!of a Lie group!induced by a homomorphism}, by composition, a representation $\pi$ of $G$ on $L^p(G')$ given by $\pi(g) = \RA_{\gamma(g)}$. Since the derivative $\gamma' : \lie{g} \to \lie{g'}$ of $\gamma$ is a Lie algebra homomorphism, it extends to a homomorphism of unital algebras $\gamma' : \Diff(G) \to \Diff(G')$. We then have that an element $f \in L^p(G')$ is a smooth vector of $\pi$ if and only if $\gamma'(D)f \in L^p(G')$ for all $D \in \Diff(G)$, and in this case $d\pi(D) f = \gamma'(D) f$.

Up to this point, for every representation $\pi$ of $G$ on a Fr\'echet space $\VV$, we have always considered the operators $d\pi(D)$ for $D \in \Diff(G)$ as defined on the common domain $\VV^\infty$. In fact, for some particular representations, we can restrict to certain subspaces $\WW$ of $\VV^\infty$ without losing essentially any information:

\begin{prp}
Let $\pi$ be a representation of $G$ on a Fr\'echet space $\VV$, and let $\WW$ be a linear subspace of $\VV^\infty$, which is dense in $\VV$ and such that $\pi(\phi) \WW \subseteq \WW$ for all $\phi \in \D(G)$. Then, for all $D \in \Diff(G)$,
\[\overline{d\pi(D)|_\WW} = \overline{d\pi(D)}.\]
\end{prp}
\begin{proof}
See \cite{nelson_representation_1959}, Theorem~1.1.
\end{proof}

In particular, if $\pi$ is the representation induced by a Lie group homomorphism $\gamma : G \to G'$ on $L^p(G')$ ($1 \leq p < \infty$), then we can take $\WW = \D(G')$. This will be understood in the following, without further mention.

\subsection{Convolution}
\index{convolution|(}Let $M(G)$ denote the space of complex finite regular Borel measures on $G$, i.e., the dual of the space $C_0(G)$ of continuous functions on $G$ vanishing at infinity. For every $\sigma,\tau \in M(G)$, the \emph{convolution}\index{convolution} of $\sigma$ and $\tau$ is the unique element $\sigma * \tau \in M(G)$ such that
\[\int_G f \,d(\sigma * \tau) = \int_G \int_G f(xy) \,d\sigma(x) \,d\tau(y) \qquad\text{for all $f \in C_0(G)$.}\]
Convolution, together with the involution $\sigma \mapsto \sigma^*$ defined by
\[\sigma^*(E) = \overline{\sigma(E^{-1})},\]
makes $M(G)$ into a Banach $*$-algebra.

The space $L^1(G)$ can be identified with the subspace of $M(G)$ of absolutely continuous measures with respect to the right Haar measure $\mu$. In fact, $L^1(G)$ is a closed $*$-subalgebra of $M(G)$, where involution is given by
\[f^*(x) = \Delta_G(x) \overline{f(x^{-1})},\]
whereas convolution $f*g$ is given by
\[f*g(x) = \int_G f(xy^{-1}) g(y) \,dy = \int_G f(y^{-1}) g(yx) \,dy \qquad\text{for $\mu_G$-a.e.\ $x \in G$,}\]
or, equivalently, by
\begin{equation}\label{eq:convolutionregularrepresentation}
f*g = \int_G g(y) \RA_{y^{-1}} f \,dy = \int_G f(y^{-1}) \LA_{y^{-1}} g \,dy.
\end{equation}

These various formulations allow one to extend the definition of convolution to other spaces of functions, and also to distributions (see \cite{dieudonn_treatise_analysis_1972}, \S17.11). For instance, we have \emph{Young's inequalities}\index{Young's inequalities}:
\[\|(f \Delta_G^{-1/q'}) * g\|_r \leq \|f\|_p \|g\|_q\]
for $p,q,r \in \left[1,\infty\right]$ with $1 + 1/r = 1/p + 1/q$ (see \cite{klein_sharp_1978}, Lemma~2.1). For $r = \infty$, this inequality gives
\[\|f * g\|_\infty \leq \|f\|_{\hat p} \|g\|_{p'};\]
in fact, if $f \in L^p(G,\Delta_G \mu)$ and $g \in L^{p'}(G,\mu)$, then $f * g$ is continuous and bounded on $G$, and if moreover $1 < p < \infty$, then $f * g \in C_0(G)$ (see \cite{hewitt_abstract_1979}, Theorem~20.16).

Other (continuous) inclusions involving convolution are the following:
\[\D(G) * \D'(G) \subseteq \E(G), \qquad \E(G) * \E'(G) \subseteq \E(G), \qquad \E'(G) * \D'(G) \subseteq \D'(G).\]
The ``smoothing'' property of convolution, i.e., the fact that convolution has at least the (differential) regularity of each of its factors, is due to the behaviour of convolution under translations:
\[\LA_x(f * g) = (\LA_x f) * g, \qquad \RA_x(f * g) = f * (\RA_x g) \qquad\text{for all $x \in G$,}\]
which implies that, for every left-invariant differential operator $D$ and right-invariant differential operator $D'$ on $G$,
\[D(f * g) = f * Dg, \qquad D'(f * g) = (D'f) * g.\]
Finally, we have that
\[\supp (f*g) \subseteq \overline{(\supp f) \cdot (\supp g)},\]
and the following duality relations hold:
\[\langle f * g, h \rangle = \langle f, h * g^* \rangle = \langle g, (f^* \Delta_G^{-1}) * h \rangle.\]

By comparing \eqref{eq:representationoffunctions} and \eqref{eq:convolutionregularrepresentation}, we see that, for a fixed $f \in L^1(G)$, the map $\phi \mapsto \phi * f$ is the linear operator associated to $f$ via the regular representation:
\[\phi * f = \RA(f) \phi.\]
Associativity of convolution then gives $\RA(f * g) = \RA(g) \RA(f)$. More generally, for a (uniformly bounded) representation $\pi$ of $G$, we have
\[\pi(f * g) = \pi(g) \pi(f)\]
(i.e., $\pi$ is an anti-representation of $L^1(G)$) and
\[\pi(\RA_x f) = \pi(x) \pi(f);\]
moreover, if $\pi$ is unitary, then
\[\pi(f^*) = \pi(f)^*.\]
\index{convolution|)}

\subsection{Approximate identities}\label{subsection:approximateidentities}
The convolution algebra $L^1(G)$ is not unital, unless $G$ is discrete (cf.\ \cite{hewitt_abstract_1979}, Theorem~20.25). However, a valid and widely used replacement for an identity element of $L^1(G)$ is given by the following statement, which summarizes well known results (cf.\ \cite{grafakos_classical_2008}, \S1.2.4).

\begin{prp}\label{prp:approximateidentity}
Let $G$ be a Lie group and $A$ be a directed set. Suppose that $(\eta_\alpha)_{\alpha \in A}$ is a family of elements of $L^1(G)$ such that:
\begin{itemize}
\item $\limsup_{\alpha \in A} \|\eta_\alpha \|_1 < \infty$;
\item $\lim_{\alpha \in A} \int_{G \setminus U} |\eta_\alpha(x)| \,dx = 0$ for all neighborhoods $U$ of the identity $e \in G$;
\item $\lim_{\alpha \in A} \int_G \eta_\alpha(x) \,dx = c$ for some $c \in \C$.
\end{itemize}
Then, for all $f \in C_b(G)$,
\[\lim_{\alpha \in A} f * \eta_\alpha = c f \qquad\text{uniformly on compacta.}\]
Moreover, for every uniformly bounded representation $\pi$ of $G$ on a Fr\'echet space $\VV$ and every $v \in \VV$,
\[\lim_{\alpha \in A} \pi(\eta_\alpha) v = c v \qquad\text{in $\VV$.}\]
In particular, for all $f \in C_{lub}(G)$,
\[\lim_{\alpha \in A} f * \eta_\alpha = c f \qquad\text{uniformly,}\]
and, if $1 \leq p < \infty$, for all $f \in L^p(G)$,
\[\lim_{\alpha \in A} f * \eta_\alpha = c f \qquad\text{in $L^p(G)$.}\]
\end{prp}

A family $(\eta_\alpha)_{\alpha \in A}$ satisfying the hypotheses of the previous proposition with $c = 1$ will be called an \emph{approximate identity}\index{approximate identity} on $G$ (along the directed set $A$). Approximate identities exist in great abundance: if one takes a basis $\{U_n\}_{n \in \N}$ of neighborhoods of $e \in G$, and if one chooses non-negative functions $\eta_n \in L^1(G)$ with $\supp \eta_n \subseteq U_n$ and $\int_G \eta_n \,d\mu = 1$, then it is easily checked that $(\eta_n)_{n \in \N}$ is an approximate identity (for $n \to +\infty$). In particular, if the $\eta_n$ are chosen in $\D(G)$, then the corresponding approximate identity gives a ``uniform'' method of approximation by smooth vectors/functions.

\subsection{Direct products}\label{subsection:directproducts}
Suppose that $G_1,\dots,G_n$ are Lie groups. Then the differential structure of direct product of smooth manifolds and the algebraic structure of direct product of groups are compatible, and define a structure of Lie group on
\[G^\times = G_1 \times \dots \times G_n.\]
Correspondingly, the Lie algebra $\lie{g}^\times$ of the product is canonically identified with the direct product (or the direct sum) of the Lie algebras of the factors:
\[\lie{g}^\times \cong \lie{g}_1 \times \dots \times \lie{g}_n \cong \lie{g}_1 \oplus \dots \oplus \lie{g}_n,\]
and this in turn gives a canonical identification of the universal enveloping algebra of $(\lie{g}^\times)_{\C}$ with the algebraic tensor product of the universal enveloping algebras of $(\lie{g}_1)_\C,\dots,(\lie{g}_n)_\C$:
\[\Diff(G^\times) \cong \Diff(G_1) \otimes \dots \otimes \Diff(G_n)\]
(see \cite{dieudonn_treatise_analysis_1974}, \S19.7.2).

If $\mu_{G_1}, \dots, \mu_{G_n}$ are right Haar measures on the factors, then their product
\[\mu_{G^\times} = \mu_{G_1} \times \dots \times \mu_{G_n}\]
is a right Haar measure on $G^\times$. Consequently, the algebraic tensor product $L^p(G_1) \otimes \dots \otimes L^p(G_n)$ is identified with a subspace of $L^p(G^\times)$, and
\[\|f_1 \otimes \dots \otimes f_n\|_{L^p(G^\times)} = \|f_1\|_{L^p(G_1)} \cdots \|f_n\|_{L^p(G_n)};\]
moreover, for $1 \leq p < \infty$, the algebraic tensor product is dense in $L^p(G^\times)$, so that $L^p(G^\times)$ is the completion of $L^p(G_1) \otimes \dots \otimes L^p(G_n)$ with respect to this norm (see \cite{defant_tensor_1993}, \S7.2 and \S15.10, Corollary~2). In particular, for $p=1$ we have a projective tensor product
\[L^1(G^\times) \cong L^1(G_1) \ptimes \cdots \ptimes L^1(G_n),\]
whereas for $p = 2$ we have a Hilbert tensor product
\[L^2(G^\times) \cong L^2(G_1) \htimes \cdots \htimes L^2(G_n),\]
so that the inner product in $L^2(G^\times)$ satisfies
\[\langle f_1 \otimes \dots \otimes f_n, g_1 \otimes \dots \otimes g_n \rangle = \langle f_1, g_1 \rangle \cdots \langle f_n, g_n \rangle.\]
Moreover, for the Banach spaces of continuous functions vanishing at infinity, and for the Fr\'echet spaces of continuous functions (with the topology of uniform convergence on compacta) we have injective tensor products:
\[C_0(G^\times) \cong C_0(G_1) \etimes \cdots \etimes C_0(G_n), \qquad C(G^\times) \cong C(G_1) \etimes \cdots \etimes C(G_n).\]

About smooth functions, we have that $\D(G_1) \otimes \dots \otimes \D(G_n)$ is dense in $\D(G^\times)$, and also in $\E(G^\times)$ (cf.\ \S17.10.2 of \cite{dieudonn_treatise_analysis_1972}, or Theorem~39.2 of \cite{treves_topological_1967}). Moreover, the LF-space $\D(G_j)$ and the Fr\'echet space $\E(G_j)$ are nuclear (cf.\ \cite{treves_topological_1967}, proof of Corollary of Theorem~51.5), and we have
\[\D(G^\times) \cong \D(G_1) \ptimes \cdots \ptimes \D(G_n) = \D(G_1) \etimes \cdots \etimes \D(G_n),\]
\[\E(G^\times) \cong \E(G_1) \ptimes \cdots \ptimes \E(G_n) = \E(G_1) \etimes \cdots \etimes \E(G_n).\]

The action of a left-invariant vector field $(X_1,\dots,X_n) \in \lie{g}_1 \times \dots \times \lie{g_n}$ on smooth functions on $G^\times = G_1 \times \dots \times G_n$ is determined by
\[(X_1,\dots,X_n) (f_1 \otimes \dots \otimes f_n) = (X_1 f_1) \otimes \dots \otimes (X_n f_n).\]
More generally, if $D \in \Diff(G_j)$, and $D^\sharp$ is the image of $D$ via the derivative of the canonical inclusion $G_j \to G^\times$, then
\[D^\sharp (f_1 \otimes \dots \otimes f_n) = f_1 \otimes \dots \otimes f_{j-1} \otimes (D f_j) \otimes f_{j+1} \otimes \dots \otimes f_n.\]
In this case, we say that $D^\sharp$ is a differential operator \emph{along the $j$-th factor} of the product $G^\times = G_1 \times \dots \times G_n$, which corresponds to the differential operator $D$ on $G_j$.

A consequence of nuclearity of the spaces of smooth functions is the \emph{Schwartz kernel theorem}\index{Schwartz kernel theorem|(}: for every continuous linear map
\[T : \D(G_1) \to \D'(G_2),\]
there exists a unique $K \in \D'(G_1 \times G_2)$ such that, for all $f \in \D(G_1)$ and $g \in \D(G_2)$,
\[\langle T f, g \rangle = \langle K, \overline{f} \otimes g\rangle\]
(see Theorem~51.7 of \cite{treves_topological_1967}, or \S23.9.2 of \cite{dieudonn_treatise_analysis_1988}). Although this result (together with most of the previous results about function spaces) is valid for more general (second countable) smooth manifolds, the following particular instance --- which will be of great use in the following --- is specific to the algebraic structure of a Lie group $G$:

\begin{thm}\label{thm:schwartzkerneld}
For every continuous linear operator $T : \D(G) \to \D'(G)$ which is left-invariant, i.e.,
\[T \LA_x = \LA_x T \qquad \text{for all $x \in G$,}\]
there exists a unique $k \in \D'(G)$, called the \emph{(convolution) kernel} of $T$, such that
\begin{equation}\label{eq:convolutionkernel}
T \phi = \phi * k \qquad\text{for all $\phi \in \D(G)$.}
\end{equation}
In particular, such an operator maps $\D(G)$ into $\E(G)$.
\end{thm}
\index{Schwartz kernel theorem|)}

\section{Some classes of Lie groups}

The following diagram (which is an abridged version of \cite{palmer_banach_2001}, Diagram~3, p.\ 1487) shows the inclusions between some categories of connected Lie groups:
\[
\xymatrix@!C@C=-35pt@R=10pt{
                       &                                          & \text{nilpotent} \ar[ld]\ar[rd]  &                                 &              \\
                       & \text{polynomial growth} \ar[ld]\ar[rdd] &                                  & \text{CCR} \ar[ldd]\ar[rdd]     &              \\
\text{hermitian} \ar[d]&                                          &                                  &                                 &              \\
\text{amenable}        &                                          & \text{unimodular}                &                                 & \text{type I}
}
\]
Although the multiplier theorems of Chapter~\ref{chapter:multipliers} are essentially limited to homogeneous nilpotent Lie groups, most of the results about commutative algebras of differential operators of Chapter~\ref{chapter:plancherel} require weaker hypotheses.

In this section we summarize the main definitions and results about the classes of Lie groups which will be considered in the following (except for nilpotent and homogeneous Lie groups, to which \S\ref{section:nilpotent} is committed). In doing so, we also briefly discuss irreducible representations, functions of positive type, transference methods, Schwartz functions and tempered distributions.

\subsection{Irreducible representations and functions of positive type}\label{subsection:positivetype}
Let $G$ be a Lie group. If $\pi$ is a unitary representation of $G$ on a Hilbert space $\HH$, then an element $v \in \HH$ is called a \emph{cyclic vector} for the representation $\pi$ if
\[\Span \{\pi(x)v \tc x \in G\}\]
is dense in $G$; the representation $\pi$ is said to be \emph{irreducible}\index{representation!of a Lie group!irreducible} if $\HH \neq \{0\}$ and every non-null element of $\HH$ is a cyclic vector for $\pi$.

If $\pi_1$ and $\pi_2$ are unitary representations of $G$ on Hilbert spaces $\HH_1$ and $\HH_2$ respectively, then $\pi_1$ is said to be \emph{(unitarily) equivalent} to $\pi_2$ if there exists an isometric isomorphism $U : \HH_1 \to \HH_2$ which \index{intertwining operator}intertwines $\pi_1$ and $\pi_2$, i.e.,
\[U \pi_1(x) = \pi_2(x) U \qquad\text{for all $x \in G$.}\]

A \emph{function of positive type}\index{function!of positive type} on $G$ is a continuous function $\phi : G \to \C$ such that, for every choice of $n \in \N$, $x_1,\dots,x_n \in G$ and $\xi_1,\dots,\xi_n \in \C$,
\[\sum_{j,k=1}^n \phi(x_k^{-1} x_j) \xi_j \overline{\xi_k} \geq 0;\]
in fact, every function $\phi$ of positive type is bounded, and we have
\[\|\phi\|_\infty = \phi(e), \qquad \phi(x^{-1}) = \overline{\phi(x)}.\]
For every unitary representation $\pi$ of $G$ on a Hilbert space $\HH$ and every $v \in \HH$, the diagonal coefficient
\[\phi_{\pi,v}(x) = \langle \pi(x) v, v\rangle_\HH\]
is a function of positive type on $G$; in fact, every function of positive type can be written in this form (for suitable $\pi$ and $v$). For $j=1,2$, if $\pi_j$ is a unitary representation of $G$ on $\HH_j$ and $v_j \in \HH_j$ is a cyclic vector for $\pi_j$, then we have $\phi_{\pi_1,v_1} = \phi_{\pi_2,v_2}$ if and only if there exists a (unique) isometric isomorphism $U : \HH_1 \to \HH_2$ which intertwines $\pi_1$ and $\pi_2$ and maps $v_1$ to $v_2$ (see \cite{folland_course_1995}, \S3.3).

The set $\PP$ of the functions of positive type on $G$ is a weakly-$*$ closed convex cone in $L^\infty(G)$. In particular, the subset $\PP_0 = \{ \phi \in \PP \tc \phi(e) \leq 1\}$ is convex and weakly-$*$ compact, and its extreme points are $0$ and the extreme points of $\PP_1 = \{ \phi \in \PP \tc \phi(e) = 1\}$, i.e., the functions of the form $\phi_{\pi,v}$ for $\pi$ irreducible and $v$ of unit norm (\cite{folland_course_1995}, Theorem~3.25 and Lemma~3.26). From the Krein-Milman theorem, it follows that the convex hull of the extreme points of $\PP_1$ is weakly-$*$ dense in $\PP_1$ (\cite{folland_course_1995}, Theorem~3.27). Moreover, the weak-$*$ topology induced by $L^\infty(G)$ on $\PP_1$ coincides with the topology of uniform convergence on compacta induced by $C(G)$ (\cite{folland_course_1995}, Theorem~3.31).

A consequence of the previous results is the \emph{Gelfand-Raikov theorem}, i.e., the fact that the irreducible unitary representations of $G$ separate points on $G$: for every $x,y \in G$, if $x \neq y$ then there exists an irreducible unitary representation $\pi$ of $G$ such that $\pi(x) \neq \pi(y)$ (\cite{folland_course_1995}, Theorem~3.34). There is a corresponding result for $L^1(G)$, which reads: if $f \in L^1(G)$ is not null, then there exists an irreducible unitary representation $\pi$ of $G$ such that $\pi(f) \neq 0$. In particular, the expression
\[\|f\|_{*} = \sup_\pi \|\pi(f)\|,\]
(where $\pi$ can indifferently range either over the unitary representations of $G$, or over the irreducible ones) defines a norm on $L^1(G)$ (\cite{folland_course_1995}, Proposition~7.1 and Corollary~7.2); the completion $C^*(G)$ of the $*$-algebra $L^1(G)$ with respect to this norm is in fact a C$^*$-algebra, which is called the \emph{group C$^*$-algebra}\index{Lie group!group C$^*$-algebra of} of $G$.

\subsection{CCR groups and type I groups, Plancherel formula}\label{subsection:typeI}
A Lie group $G$ is said to be \emph{CCR}\index{Lie group!CCR} (which stands for ``completely continuous representations'')
if, for every irreducible unitary representation $\pi$ of $G$ and every $f \in L^1(G)$, the operator $\pi(f)$ is compact. Compact groups and abelian groups are CCR, because their irreducible representations are finite-dimensional.

Let $\pi$ be a unitary representation of a Lie group $G$ on a Hilbert space $\HH$, and let $I(\pi)$ be the algebra of bounded linear operators on $\HH$ which \index{intertwining operator}intertwine $\pi$ with itself; the representation $\pi$ is said to be \emph{primary} if the center of $I(\pi)$ consists of scalar multiples of the identity. The group $G$ is said to be \emph{type I}\index{Lie group!type I} if every primary representation $\pi$ of $G$ is the direct sum of copies of some irreducible representation.

Let $\widehat G$ denote the set of equivalence classes of irreducible unitary representations of the Lie group $G$. For every function $\phi$ of positive type which is an extremal point of $\PP_1$, there is an irreducible unitary representation $\pi_\phi$ (unique up to equivalence) for which $\phi$ is a diagonal coefficient; if the set of extreme points of $\PP_1$ is endowed with the weak-$*$ topology of $L^\infty(G)$, the quotient topology induced on $\widehat G$ via the map $\phi \mapsto [\pi_\phi]$ is called the \emph{Fell topology} on $\widehat G$. It can be shown that $G$ is CCR if and only if the Fell topology is $T_1$, and that $G$ is type I if and only if the Fell topology is $T_0$ (see \cite{dixmier_algebras_1982}, Propositions~3.1.6 and 3.4.11, \S4.7.15, Theorem~9.1 and \S13.9.4; see also \cite{folland_course_1995}, \S7.2); in particular, CCR groups are type I.

One of the consequences of the type I condition involves irreducible representations of direct products: if $G_1$ and $G_2$ are Lie groups, and at least one of them is type I, then a unitary representation $\pi$ of $G_1 \times G_2$ is irreducible if and only if $\pi$ is equivalent to $\pi_1 \otimes \pi_2$ for some irreducible unitary representations $\pi_1$ of $G_1$ and $\pi_2$ of $G_2$ (see \cite{folland_course_1995}, \S7.3).

If the Lie group $G$ is unimodular and type I, then there exists a unique positive regular Borel measure on $\widehat G$, which is called the \emph{group Plancherel measure}\index{Plancherel measure!group Plancherel measure} of $G$, such that (with a slight abuse of notation with respect to equivalence classes of irreducible representations)
\[\|f\|_2^2 = \int_{\widehat G} \|\pi(f)\|_{\HS}^2 \,d\pi \qquad\text{for every $f \in L^1 \cap L^2(G)$}\]
(see \cite{dixmier_algebras_1982}, Theorem~18.8.2, and also \cite{folland_course_1995}, \S7.5).

\subsection{Amenable groups and transference}\label{subsection:amenability}
There are many equivalent definitions of amenability for a Lie group (or more generally, a locally compact topological group), for which we refer to \cite{greenleaf_invariant_1969} and \cite{pier_amenable_1984}. For instance, one can say that a Lie group $G$ is \emph{amenable}\index{Lie group!amenable} if it admits a left-invariant mean on $L^\infty(G)$, i.e., a linear functional $M : L^\infty(G) \to \C$ such that
\[M(\LA_x f) = M f \qquad\text{for all $f \in L^\infty(G)$ and $x \in G$,}\]
\[Mf \geq 0 \qquad\text{if $f \geq 0$, and}\qquad M 1 = 1.\]

A characterization of amenability is related to convolution operators. For $1 \leq p < \infty$, let $\Cv^p(G)$ the set of the distributions $u \in \D'(G)$ such that the associated left-invariant operator
\[\phi \mapsto \phi * u\]
extends to a bounded operator on $L^p(G)$. By identifying left-invariant operators with their kernels (cf.\ Theorem~\ref{thm:schwartzkerneld}), $\Cv^p(G)$ inherits a structure of Banach algebra from the algebra of bounded linear operators on $L^p(G)$. Moreover, by Young's inequality, we have that $L^1(G) \subseteq \Cv^p(G)$, with continuous inclusion, and in fact
\[\|f\|_{\Cv^1} = \|f\|_{1},\]
whereas for $1 < p < \infty$ we only know 
\begin{equation}\label{eq:cvl1}
\|f\|_{\Cv^p} \leq \|f\|_1.
\end{equation}
Then we have that, for all $p \in \left]1,\infty\right[$, $G$ is amenable if and only if
\[\|f\|_{\Cv^p} = \|f\|_1 \qquad\text{for all $f \in L^1(G)$ with $f \geq 0$}\]
(see \cite{pier_amenable_1984}, Theorem~9.6).

Another characterization of amenability can be given in terms of unitary representations. Recall that, for any unitary representations $\pi$ and $\pi'$ of $G$, $\pi$ is said to be \emph{weakly contained} in $\pi'$ if
\[\|\pi(f)\| \leq \|\pi'(f)\| \qquad\text{for all $f \in L^1(G)$.}\]
Then $G$ is amenable if and only if every irreducible unitary representation of $G$ is weakly contained in the right regular representation of $G$ on $L^2(G)$ (see \cite{greenleaf_invariant_1969}, \S3.5, and also \cite{pier_amenable_1984}, Theorem~8.9).

The last characterization can be rephrased in terms of the $C^*(G)$-norm: the group $G$ is amenable if and only if
\[\|f\|_{*} = \|f\|_{\Cv^2} \qquad\text{for all $f \in L^1(G)$,}\]
so that, in this case, $C^*(G)$ is isometrically isomorphic to the closure of $L^1(G)$ in $\Cv^2(G)$.

This property of amenable groups --- i.e., the fact that the norm of a convolution operator on $L^2(G)$ dominates the norm of the corresponding operator in every unitary representation --- can be extended to more general representations (and operators) by \emph{transference methods}\index{transference} (see \cite{coifman_transference_1976}, and also the appendix of \cite{cowling_herzs_1997}, for a discussion of the main ideas and contributions). For instance, we have

\begin{thm}\label{thm:transference}
Let $X$ be a $\sigma$-finite measure space, $\pi$ be a (strongly continuous) uniformly bounded representation of an amenable group $G$ on $L^p(X)$, where $1 < p < \infty$, and set
\[c_\pi = \sup_{x \in G} \|\pi(x)\|_{p \to p}.\]
Then, for every $f \in L^1(G)$, we have
\[\|\pi(f)\|_{p \to p} \leq c_\pi^2 \|f\|_{\Cv^p(G)}.\]
\end{thm}
\begin{proof}
See \cite{coifman_transference_1976}, Theorem~2.4, and also \cite{berkson_transference_1996}, Theorem~2.7.
\end{proof}

By strengthening the hypothesis on the representation, it is possible to obtain transference results also for the so-called \emph{maximal operators}. Recall that a bounded linear operator $P$ on $L^p(X)$ is said to be \emph{positivity preserving} if $Pf \geq 0$ whenever $f \geq 0$.

\begin{thm}\label{thm:transferencemaximal}
Under the hypotheses of Theorem~\ref{thm:transference}, suppose moreover that, for every $x \in X$, there exists a positivity preserving operator $P_x$ on $L^p(X)$ such that
\[|\pi(x) f| \leq P_x |f| \qquad\text{for all $f \in L^p(X)$,}\]
and that
\[c_P = \sup_{x \in G} \|P_x\|_{p \to p} < \infty.\]
Let $\{f_j\}_{j \in \N}$ be a sequence in $L^1(G)$, and set
\[M\phi = \sup_{j \in \N} |\phi * f_j| \qquad\text{for $\phi \in L^p(G)$,}\]
\[M_\pi \eta = \sup_{j \in \N} |\pi(f_j)\eta| \qquad\text{for $\eta \in L^p(X)$}.\]
If $M$ is bounded on $L^p(G)$, i.e.,
\[\|M \phi\|_p \leq C_M \|\phi\|_p \qquad\text{for all $\phi \in L^p(G)$,}\]
then $M_\pi$ is bounded on $L^p(X)$, and we have
\[\|M_\pi \eta\|_p \leq c_\pi c_P C_M \|\eta\|_p \qquad\text{for all $\eta \in L^p(X)$.}\]
\end{thm}
\begin{proof}
See \cite{berkson_transference_1996}, Theorem~2.11.
\end{proof}

\subsection{Hermitian groups}
A Lie group $G$ will be called \emph{hermitian}\index{Lie group!hermitian} if the Banach $*$-algebra $L^1(G)$ is hermitian (see \S\ref{subsection:hermitianalgebras}). Thanks to Raikov's criterion, this property can be stated in terms of the $C^*(G)$-norm: namely, if $\rho(f)$ denotes the spectral radius of an element $f$ of $L^1(G)$, then $G$ is hermitian if and only if
\[\|f\|_* = \sqrt{\rho(f^* f)} \qquad\text{for all $f \in L^1(G)$.}\]
(see also \cite{palmer_classes_1978}, p.\ 695). It is known that 
connected hermitian Lie groups are amenable (\cite{palmer_banach_2001}, Theorem~12.5.18(e)).

\subsection{Volume growth and connected moduli}
Let $G$ be a connected Lie group. Then, for every compact neighborhood $V$ of the identity of $G$, we have that
\[V^n \subseteq V^{n+1} \text{ for all $n \in \N \setminus \{0\}$, and } \bigcup_{n \in \N \setminus \{0\}} V^n = G.\]
If $\mu_G$ is a (left or right) Haar measure on $G$, we can then consider the growth of the increasing sequence of real numbers $\mu_G(V^n)$ as an index of the \emph{volume growth} of the group $G$. In fact, by a result of Guivarc'h (\cite{guivarch_croissance_1973}, Th\'eor\`eme~II.3; see also \cite{jenkins_growth_1973}), we know that the volume growth of a connected Lie group $G$ can be either \emph{strictly polynomial}\index{Lie group!of polynomial growth}, i.e.,
\[\mu_G(V^n) \asymp n^{Q_G}\]
for some $Q_G \in \N$ (which is called the \emph{degree of growth}\index{degree of (polynomial) growth}), or \emph{exponential}, i.e.,
\[e^{\alpha n} \lesssim \mu_G(V^n) \lesssim e^{\beta n}\]
for some real $\beta \geq \alpha > 0$. Notice that this classification (and also the degree $Q_G$ in the case of polynomial growth) does not depend on the choice of the compact neighborhood $V$ of the identity, or even on the (left or right) Haar measure $\mu_G$.

Another way to express the volume growth of a Lie group $G$ is via a \emph{connected modulus}\index{connected modulus}, i.e., a continuous proper function $|\cdot| : G \to \left[0,+\infty\right[$ which is symmetric, subadditive and separating, i.e., for every $x,y \in G$,
\[|x^{-1}| = |x|, \qquad |xy| \leq |x|+|y|,\]
\[|x| = 0 \qquad\text{if and only if}\quad x = e,\]
and which moreover satisfies the following condition: there exist constants $C_1,C_2 > 0$ such that, for every $x \in G$ with $|x| \geq C_1$, there exist elements
\[e = x_0, x_1, \dots, x_{n-1}, x_n = x\]
in $G$ with $|x_j x_{j-1}^{-1}| \leq C_1$ for $j=1,\dots,n$ and $n \leq C_2 |x|$.

To every connected modulus $|\cdot|$ on $G$, we associate a left-invariant distance $d_l$ and a right-invariant distance $d_r$ on $G$, which are defined by
\[d_l(x,y) = |x^{-1} y|, \qquad d_r(x,y) = |y x^{-1}|;\]
each of these distances induces on $G$ the same topology as the manifold structure, and moreover
\[|x| = d_l(e,x) = d_r(e,x).\]

Examples of connected moduli are those associated to the so-called Carnot-Carath\'eodory distances on $G$ (see \cite{varopoulos_analysis_1992}, \S III.3.3), and also to the more general control distances considered in the following \S\ref{section:weightedsubcoercive}.

If $|\cdot|$ and $|\cdot|'$ are both connected moduli on $G$, then they are equivalent at infinity (cf.\ \cite{varopoulos_analysis_1992}, Proposition~III.4.2), i.e., there exists $C > 0$ such that
\[C^{-1} |x| \leq |x|' \leq C |x| \quad\text{when $|x| \geq 1$ or $|x|' \geq 1$.}\]
Moreover, if $|\cdot|$ is a connected modulus, then the growth of the (Haar) measure of the associated balls $B_r = \{ x \in G \tc |x| < r\}$ corresponds to the volume growth of the group: namely, if $G$ has polynomial growth, then
\[\mu_G(B_r) \asymp r^{Q_G} \qquad\text{for $r$ large}\]
(where $Q_G \in \N$ is the degree of growth previously defined), whereas, if $G$ has exponential growth, then
\[e^{\alpha r} \lesssim \mu_G(B_r) \lesssim e^{\beta r} \qquad\text{for $r$ large,}\]
for some real $\beta \geq \alpha > 0$.

Connected Lie groups with polynomial growth have many of the previously examined properties: they are in fact unimodular (\cite{guivarch_croissance_1973}, Lemme I.3), hermitian (see \cite{ludwig_class_1979}), and therefore amenable (see also \cite{guivarch_croissance_1973}, p.~338). Notice however that there exist connected Lie groups with polynomial growth which are not type I (see, e.g., \cite{baggett_representations_1978} or \cite{cowling_plancherel_1978}).

\subsection{Schwartz functions and tempered distributions}\label{subsection:schwartz}

In the context of connected Lie groups with polynomial growth, it is possible to define the class of \emph{Schwartz functions}\index{function!Schwartz} in a natural way\footnote{About the definition of Schwartz functions for other classes of Lie groups, see, e.g., \cite{schweitzer_dense_1993} and \cite{david-guillou_schwartz_2010}.}. Let $G$ be a connected Lie group with polynomial growth, $|\cdot|_G$ be a connected modulus, and set
\[\langle x \rangle_G = 1 + |x|_G.\]
Then, for every $p \in [1,\infty]$, the seminorms
\[\eta_{p,D,k}(f) = \| \langle \cdot \rangle_G^k D f \|_p  \qquad\text{for $D \in \Diff(G)$ and $k \in \N$}\]
define a Fr\'echet structure on a linear subspace $\Sz(G)$ of $\E(G)$, which contains $\D(G)$, and which in fact does not depend on $p \in [1,\infty]$ (see \cite{schweitzer_dense_1993}, Theorems~1.2.21 and 6.8). Moreover, $\Sz(G)$ is a translation-invariant $*$-subalgebra of $L^1(G)$ (\cite{schweitzer_dense_1993}, Theorems~1.3.13 and 1.4.3), and is also a nuclear Fr\'echet space (\cite{schweitzer_dense_1993}, Theorem~6.24). Therefore, if we define the class $\Sz'(G)$ of \emph{tempered distributions} as the dual of $\Sz(G)$, then we have the (continuous) inclusions
\[\Sz(G) * (\Sz(G) + \E'(G)) \subseteq \Sz(G), \qquad \Sz(G) * \Sz'(G) \subseteq \E(G) \cap \Sz'(G).\]
Moreover, analogously as in \S\ref{subsection:directproducts}, there is a version of the Schwartz kernel theorem involving Schwartz functions and tempered distributions, whose translation-invariant version is

\begin{thm}\label{thm:schwartzkernels}\index{Schwartz kernel theorem}
For every bounded linear operator $T : \Sz(G) \to \Sz'(G)$ which is left-invariant, there exists a unique $k \in \Sz'(G)$ such that
\[T \phi = \phi * k \qquad\text{for all $\phi \in \Sz(G)$.}\]
In particular, such an operator maps $\Sz(G)$ into $\E(G) \cap \Sz'(G)$.
\end{thm}

\section{Nilpotent and homogeneous groups, Rockland operators}\label{section:nilpotent}
Nilpotent Lie groups are, in a sense, the slightest non-abelian generalization of $\R^n$: in fact, a (connected, simply connected) nilpotent Lie group is a real vector space, endowed with a polynomial group law, which is ``eventually abelian'' (i.e., iterated commutators of sufficiently high order equal the identity\footnote{Notice that, by a result of Lazard \cite{lazard_nilpotence_1955}, every polynomial group law on $\R^n$ is nilpotent.}). Due to these facts, most of the classical instruments and techniques of analysis on Euclidean spaces can be also applied to nilpotent Lie groups. On the other hand, the non-commutative algebraic structure brings completely new phenomena, which may be non-trivial at all (e.g., the theory of irreducible representations).

Homogeneous Lie groups are nilpotent Lie groups with a fixed family of automorphic dilations. This homogeneous structure allows to define a particularly interesting class of left-invariant differential operators, the so-called Rockland operators; moreover, dilations are an essential tool in the proof of the multiplier theorems of Chapter~\ref{chapter:multipliers}.

Our main references for nilpotent and homogeneous Lie groups are the books of Corwin and Greenleaf \cite{corwin_representations_1990}, Folland and Stein \cite{folland_hardy_1982}, and Goodman \cite{goodman_nilpotent_1976}. Moreover, a rich source of examples is the list \cite{nielsen_unitary_1983} of all the nilpotent Lie groups of dimension up to $6$ (with their irreducible representations).

\subsection{Nilpotent Lie groups}
A Lie algebra $\lie{g}$ is \emph{nilpotent}\index{Lie algebra!nilpotent} if the \emph{descending central series}\index{descending central series}
\[\lie{g}_{[1]} = \lie{g}, \qquad \lie{g}_{[n+1]} = [\lie{g},\lie{g}_{[n]}]\]
is eventually null. If $n$ is such that $\lie{g}_{[n]} \neq 0$ and $\lie{g}_{[n+1]} = 0$, then $\lie{g}$ is said to be $n$-step. In a nilpotent Lie algebra, the formal power series $b(X,Y)$ in the Baker-Campbell-Hausdorff formula \eqref{eq:bch} becomes a finite sum, which defines a polynomial group law $\lie{g} \times \lie{g} \to \lie{g}$, such that the identity is $0$ and the inverse of $v \in \lie{g}$ is $-v$; with this structure, $\lie{g}$ becomes a Lie group, whose Lie algebra is isomorphic to $\lie{g}$ itself, in such a way that the exponential map becomes the identity map.

More generally, if $G$ is a connected Lie group such that its Lie algebra $\lie{g}$ is nilpotent, and if we consider on $\lie{g}$ the group law previously defined, then the exponential map $\exp : \lie{g} \to G$ becomes a Lie group homomorphism, and in fact $\exp$ is a universal covering map of $G$; in particular, if $G$ is simply connected, then $\exp$ is an isomorphism.

A \emph{nilpotent Lie group}\index{Lie group!nilpotent} is a Lie group whose Lie algebra is nilpotent; unless otherwise specified, we always suppose that a nilpotent Lie group is connected and simply connected. Therefore, if $G$ is a nilpotent Lie group, then $G$ can be identified with its Lie algebra $\lie{g}$ via the exponential map. In this way, the Lie subgroups of $G$ are identified with the Lie subalgebras of $\lie{g}$ (so that in particular they are linear subspaces of $\lie{g}$ and are closed). Moreover, a Lie group homomorphism between nilpotent Lie group is identified with its derivative, which is a Lie algebra homomorphism and in particular it is a linear map.

If $G$ is nilpotent, then the push-forward of the Lebesgue measure on $\lie{g}$ via the exponential map is both left- and right-invariant, so that it is a left and right Haar measure; in particular, nilpotent Lie groups are unimodular.

In fact, every nilpotent Lie group has polynomial growth; more precisely:

\begin{prp}\label{prp:nilpotentgrowth}
Suppose that $G$ is $n$-step and let $V_j$ be a supplement of $\lie{g}_{[j+1]}$ in $\lie{g}_{[j]}$ for $j=1,\dots,n$. Choose moreover norms $|\cdot|_j$ on the $V_j$ and set
\begin{equation}\label{eq:growthnorm}
|x| = \sum_{j=1}^n |x_j|_j^{1/j},
\end{equation}
where $x = x_1 + \dots + x_n$ is the decomposition of $x \in \lie{g} = V_1 \oplus \dots \oplus V_n$. Then every connected modulus on $G$ is equivalent, in the large, to $|\cdot|$. In particular, $G$ has polynomial growth of degree
\[Q_G = \sum_{j=1}^n j \dim V_j = \sum_{j=1}^n \dim \lie{g}_{[j]}.\]
\end{prp}
\begin{proof}
Cf.\ the proofs of \cite{guivarch_croissance_1973}, Th\'eor\`eme~II.1 and Lemme~II.1.
\end{proof}

In particular, on a nilpotent Lie group we can consider the classes $\Sz(G)$ and $\Sz'(G)$ of Schwartz functions\index{function!Schwartz} and tempered distributions which has been defined for groups of polynomial growth; in fact, it can be shown that these classes coincide, in exponential coordinates, with the ``classical'' spaces of Schwartz functions and tempered distributions on $\lie{g}$ (see also \cite{ludwig_algebre_1995}).

About Schwartz functions, we also recall the following result from the representation theory of nilpotent Lie groups, which implies that a nilpotent Lie group is CCR.

\begin{thm}
Let $\pi$ be an irreducible unitary representation of $G$ on a Hilbert space $\HH$. If $f \in \Sz(G)$, then $\pi(f)$ is a trace class operator on $\HH$. If $f \in L^1(G)$, then $\pi(f)$ is a compact operator on $\HH$.
\end{thm}
\begin{proof}
See \cite{kirillov_unitary_1962}, Theorem 7.3 and Corollary, or also \cite{corwin_representations_1990}, Theorem 4.2.1.
\end{proof}

\subsection{Homogeneous Lie groups}\label{subsection:homogeneousgroups}
An expanding automorphism on a Lie algebra $\lie{g}$ is a Lie algebra automorphism $\xi$ of $\lie{g}$ which is diagonalizable and whose eigenvalues are all greater than $1$; such an automorphism can be written in the form $\xi = e^A$, where $A$ is a diagonalizable derivation of $\lie{g}$ whose eigenvalues are all positive.

Starting from such a derivation $A$, we define a family of \emph{dilations}\index{dilations}
\[\delta_t = e^{A \log t} \qquad \text{for $t > 0$,}\]
which are automorphisms of $\lie{g}$ which share the eigenspaces $W_\lambda$:
\[\delta_t(v) = t^\lambda v \qquad\text{for $v \in W_\lambda$.}\]
The map $t \mapsto \delta_t$ is a homomorphism $\R^+ \to \Aut(\lie{g})$.

A \emph{homogeneous Lie algebra}\index{Lie algebra!homogeneous} is a Lie algebra with a fixed family of dilations. The automorphic dilations $\delta_t$ of a homogeneous Lie algebra $\lie{g}$ extend to automorphisms $\delta_t$ of the universal enveloping algebra $\UEnA(\lie{g}_\C)$, which are simultaneously diagonalizable. An element $D \in \UEnA(\lie{g}_\C)$ is said to be homogeneous of degree $\lambda$ if
\[\delta_t(D) = t^\lambda D \qquad\text{for all $t > 0$.}\]
The \emph{homogeneous dimension}\index{homogeneous dimension} of the homogeneous Lie algebra $\lie{g}$ is the sum $Q_\delta$ of the degrees of the elements a homogeneous basis of $\lie{g}$, i.e., $Q_\delta = \tr A$.

By replacing the derivation $A$ with $cA$ for some $c > 0$, we can rescale all the degrees in $\lie{g}$ and in $\UEnA(\lie{g}_\C)$, and also the homogeneous dimension $Q_\delta$, by the same factor $c$; in this case, we say that the new homogeneous structure on $\lie{g}$, obtained by rescaling, is \emph{equivalent} to the original one. In the following, unless otherwise specified, we always suppose that the least eigenvalue of the derivation $A$ defining a family of dilations is not less than $1$. Under this hypothesis, the degree of a non-constant homogeneous element of $\UEnA(\lie{g}_\C)$ is not less than $1$, and moreover $Q_\delta \geq \dim \lie{g}$.

A family of automorphic dilations $\delta_t = e^{A \log t}$ on a Lie algebra $\lie{g}$ determines a direct-sum decomposition
\begin{equation}\label{eq:gradation}
\lie{g} = W_{\lambda_1} \oplus \dots \oplus W_{\lambda_k},
\end{equation}
for some $k \in \N$ and real numbers $\lambda_k > \dots > \lambda_1 \geq 1$ --- which are called \emph{weights} of the homogeneous structure --- such that, if we set $W_\lambda = 0$ whenever $\lambda \notin \{\lambda_1,\dots,\lambda_k\}$, then
\begin{equation}\label{eq:gradation2}
[W_\lambda,W_{\lambda'}] \subseteq W_{\lambda+\lambda'} \qquad\text{for all $\lambda,\lambda' \geq 1$}\
\end{equation}
(the spaces $W_\lambda$ are the eigenspaces of the derivation $A$, and $\lambda_1,\dots,\lambda_k$ are its eigenvalues). Vice versa, such a direct-sum decomposition determines a family of automorphic dilations $\delta_t$ on $\lie{g}$ by setting
\[\delta_t(v) = t^\lambda v \qquad \text{for $v \in W_\lambda$.}\]
The existence of such a decomposition implies that a homogeneous Lie algebra is nilpotent. Notice that, conversely, not every nilpotent Lie algebra admits expanding automorphisms (see \cite{dyer_nilpotent_1970}).

In the case the weights $\lambda_1,\dots,\lambda_k$ are all integers, then the direct-sum decomposition is called a \emph{gradation} of $\lie{g}$, and a \emph{graded Lie algebra} is a Lie algebra with a fixed gradation (i.e., with a fixed homogeneous structure with integral weights).

In general, if $\lambda_1,\dots,\lambda_k$ have a common rational multiple, then $\lie{g}$ has an equivalent homogeneous structure which is graded. If this is not the case, it is nevertheless possible to re-index the subspaces $W_\lambda$ of the decomposition \eqref{eq:gradation}, replacing $\lambda_1,\dots,\lambda_n$ with positive integers, in such a way that \eqref{eq:gradation2} continues to hold (see \cite{miller_parametrices_1980}, Proposition~1.1). This means that, if a (nilpotent) Lie algebra admits automorphic dilations, then it admits also a gradation. However, it should be noticed that this re-indexing may give an homogeneous structure not equivalent to the original one, and in general it does not preserve the homogeneous decomposition of the universal enveloping algebra.

A \emph{stratification} of a Lie algebra $\lie{g}$ is a gradation $\lie{g} = W_1 \oplus \dots \oplus W_n$ such that $W_1$ generates $\lie{g}$ as a Lie algebra, i.e., such that
\[[W_1,W_j] = W_{j+1}\]
for $j=1,\dots,n-1$. A \emph{stratified Lie algebra}\index{Lie algebra!stratified} is a Lie algebra with a fixed stratification (i.e., it is a graded Lie algebra whose associated gradation is a stratification). There exist Lie algebras which have expanding automorphisms (and therefore admit a gradation), but which are not stratifiable, i.e., they do not admit a stratification (see, e.g., the group $G_{6,23}$ in \S\ref{subsection:example3step}).

A \emph{homogeneous} (resp.\ \emph{graded}, \emph{stratified}) \emph{Lie group}\index{Lie group!homogeneous}\index{Lie group!stratified}\index{Lie group!graded} is a connected, simply connected Lie group $G$ whose Lie algebra $\lie{g}$ is homogeneous (resp.\ graded, stratified). If $G$ is a homogeneous Lie group, then the automorphic dilations $\delta_t$ on $\lie{g}$ corresponds, via the exponential map, to automorphisms of $G$, which will be also denoted by $\delta_t$. The Haar measure $\mu_G$ on $G$ is homogeneous with respect to this family of dilations:
\[\mu_G(\delta_t(B)) = t^{Q_\delta} \mu_G(B) \qquad \text{for all Borel $B \subseteq G$ and $t > 0$.}\]

A \emph{homogeneous norm}\index{homogeneous norm} on a homogeneous Lie group $G$ with dilations $\delta_t$ is a continuous function $|\cdot|_\delta : G \to \left[0,+\infty\right[$ such that, for all $x \in G$:
\begin{itemize}
\item $|x|_\delta = 0$ if and only if $x$ is the identity of $G$;
\item $|x^{-1}|_\delta = |x|_\delta$;
\item $|\delta_t(x)|_\delta = t|x|_\delta$ for all $t > 0$.
\end{itemize}
A homogeneous norm $|\cdot|_\delta$ is quasi-subadditive:
\begin{equation}\label{eq:quasisubadditive}
|xy|_\delta \leq c (|x|_\delta + |y|_\delta) \qquad \text{for all $x,y \in G$,}
\end{equation}
for some constant $c > 0$; moreover, two homogeneous norms $|\cdot|_\delta$, $|\cdot|_\delta'$ on the same homogeneous Lie group $G$ are equivalent:
\[C^{-1} |x|_\delta \leq |x|_\delta' \leq C |x|_\delta \qquad\text{for all $x \in G$,}\]
for some constant $C \geq 1$ (see \cite{goodman_filtrations_1977}, \S3,  or \cite{goodman_nilpotent_1976}, \S1.2). In particular, each homogeneous norm on $G$ is equivalent to the following ones:
\begin{equation}\label{eq:modelhomogeneousnorms}
|x|_{\delta,1} = \sum_{j=1}^k |x_j|_j^{1/\lambda_j}, \qquad |x|_{\delta,\infty} = \max_{1 \leq j \leq k} |x_j|_j^{1/\lambda_j},
\end{equation}
where $x = x_1 + \dots + x_k$ is the decomposition of $x \in \lie{g}$ according to \eqref{eq:gradation}, and $|\cdot|_j$ is a norm on $W_{\lambda_j}$. On the other hand, on every homogeneous group there exists a homogeneous norm which is subadditive (i.e., \eqref{eq:quasisubadditive} holds with $c=1$) and smooth off the origin (see \cite{hebisch_smooth_1990}).

Notice that, for every nilpotent Lie group $G$, the function $|\cdot|$ given by \eqref{eq:growthnorm} can be considered as a homogeneous norm, if $G$ is thought of as an abelian group (with the additive structure of $\lie{g}$) with a suitable gradation (defined by the supplementaries $V_j$); the homogeneous dimension associated to this (abelian) gradation coincides with the degree $Q_G$ of polynomial growth of $G$ (with respect to the original algebraic structure). If $G$ is stratified, then (for a suitable choice of the $V_j$) the function $|\cdot|$ is also a homogeneous norm with respect to the original (possibly non-abelian) structure of $G$, so that the degree of polynomial growth of $G$ coincides with its homogeneous dimension; moreover, every subadditive homogeneous norm on $G$ is also a connected modulus on $G$. For a general homogeneous Lie group, we have the following (cf.\ also \cite{jenkins_dilations_1979}):

\begin{prp}\label{prp:quasiequivalence}
Let $G$ be a homogeneous Lie group, with dilations $\delta_t$ and homogeneous dimension $Q_\delta$, and let $|\cdot|_\delta$ be a homogeneous norm on $G$. Let $|\cdot| : G \to \left[0,+\infty\right[$ be defined as in \eqref{eq:growthnorm}, and let $Q_G$ be the degree of polynomial growth of $G$.
\begin{itemize}
\item[(i)] One has $Q_\delta \geq Q_G$, with equality if and only if $G$ is stratified\footnote{We remark that ``stratified'' for a homogeneous group means not only that the Lie algebra of the group admits a stratification (which is the meaning of ``stratifiable''), but also that the fixed homogeneous structure on the Lie algebra yields a stratification.}.
\item[(ii)] There exist $a,b,c > 0$ such that
\begin{equation}\label{eq:quasiequivalence}
c^{-1} |x|^a_\delta \leq |x| \leq c |x|_\delta^b
\end{equation}
for $x \in G$ large (i.e., off a compact neighborhood of the identity). The same inequality holds with $|\cdot|$ replaced by any connected modulus on $G$. Moreover, we can take $a = b = 1$ if and only if $G$ is stratified.
\end{itemize}
\end{prp}
\begin{proof}
(i) Decompose $\lie{g}$ as in \eqref{eq:gradation}. Notice that the subspaces $\lie{g}_{[s]}$ composing the descending central series are characteristic ideals of $\lie{g}$; since the dilations $\delta_t$ are automorphisms, the $\lie{g}_{[s]}$ are homogeneous. A homogeneous element of $\lie{g}_{[s]}$, being the sum of $s$-fold iterated commutators of homogeneous elements of $\lie{g}$, has a homogeneity degree which must be the sum of $n$ of the homogeneity degrees $\lambda_1 < \dots < \lambda_k$ of the elements of $\lie{g}$; since all these degrees are not less than $1$, the sum is not less than $n$, therefore $\lie{g}_{[s]} \cap W_\lambda = \{0\}$ if $\lambda < s$, so that
\begin{equation}\label{eq:homogeneouscontainment}
\lie{g}_{[s]} \subseteq \bigoplus_{\lambda \geq s} W_\lambda.
\end{equation}
In particular, if $G$ is $n$-step,
\begin{equation}\label{eq:homogeneousdimchain}
Q_G = \sum_{s = 1}^n \dim \lie{g}_{[s]} \leq \sum_{s=1}^n \sum_{\lambda \geq s} \dim W_\lambda \leq \sum_{j=1}^k \lfloor \lambda_j \rfloor \dim W_{\lambda_j} \leq Q_\delta.
\end{equation}

We already know that, if $G$ is stratified, then $Q_G = Q_\delta$. Conversely, if $Q_G = Q_\delta$, then all the inequalities in \eqref{eq:homogeneousdimchain} must be equalities; this means first of all that the degrees $\lambda_1,\dots,\lambda_k$ are integers --- i.e., $G$ is graded --- and secondly that the inclusion \eqref{eq:homogeneouscontainment} is an equality, so that $W_s \subseteq \lie{g}_{[s]}$, but then necessarily $W_1$ generates $\lie{g}$ --- i.e., $G$ is stratified.

(ii) By the explicit definition of $|\cdot|$ in Proposition~\ref{prp:nilpotentgrowth} and the equivalence of $|\cdot|_\delta$ with one of the homogeneous norms in \eqref{eq:modelhomogeneousnorms}, the inequality \eqref{eq:quasiequivalence} follows easily. Since moreover, by Proposition~\ref{prp:nilpotentgrowth}, $|\cdot|$ is equivalent in the large to any connected modulus on $G$, clearly a connected modulus can replace $|\cdot|$ in \eqref{eq:quasiequivalence}.

If $G$ is stratified, then $|\cdot|_{\delta}$ is (modulo equivalence of homogeneous norms) of the form \eqref{eq:growthnorm}, with a choice of the supplements $V_j$ possibly different to the one defining $|\cdot|$; by Proposition~\ref{prp:nilpotentgrowth}, $|\cdot|_\delta$ is equivalent in the large to any connected modulus, and then also to $|\cdot|$. Conversely, since
\[\mu_G(\{x \in G \tc |x| < r\}) \sim r^{Q_G}, \qquad \mu_G(\{x \in G \tc |x|_\delta < r\}) \sim r^{Q_\delta}\]
for $r$ large, if \eqref{eq:quasiequivalence} holds with $a = b = 1$, then necessarily $Q_G = Q_\delta$, and the conclusion follows by (i).
\end{proof}

On a homogeneous Lie group $G$ with a homogeneous norm $|\cdot|_\delta$, the Schwartz space $\Sz(G)$ can be characterized as follows: a function $f \in \E(G)$ belongs to the Schwartz class\index{function!Schwartz} if and only if, for some $p \in \left[1,\infty\right]$,
\[(1+|\cdot|_\delta)^m D f \in L^p(G)\]
for all $m \in \N$ and $D \in \Diff(G)$.

\subsection{Rockland operators}
Let $G$ be a homogeneous Lie group with dilations $\delta_t$. A \emph{Rockland operator}\index{differential operator!Rockland} on $G$ is a $\delta_t$-homogeneous left-invariant differential operator $D \in \Diff(G)$ such that, for every non-trivial unitary representation $\pi$ of $G$, $d\pi(D)$ is injective (as an operator on the smooth vectors of the representation).

Notice that, in the abelian case ($G = \R^n$), left-invariant differential operators coincide with constant-coefficient differential operators, i.e., operators of the form
\[p\left(-i\frac{\partial}{\partial x_1},\dots,-i\frac{\partial}{\partial x_n}\right)\]
for some polynomial $p$ with complex coefficients, which is called the \emph{symbol} of the operator. In this case, a differential operator with symbol $p$ is Rockland if and only if $p$ is $\delta_t$-homogeneous and such that
\[p(x) \neq 0 \qquad\text{for $x \neq 0$.}\]
If the dilations are isotropic ($\delta_t(x) = tx$), the notion of Rockland operator then coincides with that of elliptic homogeneous operator with constant coefficients on $\R^n$.

It can be shown that the existence of a Rockland operator on a homogeneous Lie group $G$ implies that the degrees of homogeneity have a common rational multiple, so that we can suppose (modulo rescaling) that $G$ is graded (see \cite{miller_parametrices_1980} and \cite{ter_elst_spectral_1997}). This allows us to state in the context of homogeneous groups the Helffer-Nourrigat theorem (see \cite{helffer_caracterisation_1979}, and also \cite{helffer_hypoellipticit_1985}), which is a characterization of Rockland operators in terms of hypoellipticity. Recall that a differential operator $L$ on $G$ is said to be \emph{hypoelliptic}\index{differential operator!hypoelliptic} if, for every $u \in \D'(G)$ and every open set $\Omega \subseteq G$,
\[(Lu)|_\Omega \in \E(\Omega) \quad\Longrightarrow\quad u|_\Omega \in \E(\Omega).\]

\begin{thm}
Let $G$ be a homogeneous Lie group, and let $L \in \Diff(G)$.
\begin{itemize}
\item If $L$ is homogeneous, then $L$ is Rockland if and only if it is hypoelliptic.
\item If the principal part of $L$ (with respect to the homogeneous structure) is Rockland, then $L$ is hypoelliptic.
\end{itemize}
\end{thm}

Several regularity results have been proved for this class of operators, e.g., heat kernel estimates (see \cite{auscher_positive_1994}). We will consider these properties in the more general framework of weighted subcoercive operators.

\section{Weighted subcoercive operators}\label{section:weightedsubcoercive}

As we have just seen, the class of Rockland operators contains, as a particular case, the constant-coefficient homogeneous elliptic operators on $\R^n$. The notion of (strongly) elliptic operator can in fact be extended to manifolds, and most of the properties valid in $\R^n$ can be obtained also in this wider setting, where $\R^n$ plays the role of ``local approximation'', and constant-coefficient operators are somehow the model case. In particular, for connected Lie groups, a thorough analysis of left-invariant elliptic operators can be found in \cite{robinson_elliptic_1991}, where ``Gaussian'' estimates are obtained for the heat kernel of a positive elliptic operator, in terms of an invariant Riemannian metric.

On the other hand, on manifolds, and in particular on Lie groups, non-commutativity of vector fields naturally arises, and makes it possible to consider classes of differential operators which, although not elliptic, share nevertheless some of the good properties of elliptic operators. For instance, if one considers a system $X_1,\dots,X_n$ of vector fields which satisfy the \emph{H\"ormander condition} (i.e., for some $m \in \N$, the span of $X_1,\dots,X_n$ and their iterated commutators up to order $m$ is the whole tangent space in each point of the manifold), then the associated \emph{sublaplacian}\index{sublaplacian}
\[-(X_1^2 + \dots + X_n^2)\]
is hypoelliptic (see \cite{hrmander_hypoelliptic_1967}). Moreover, in the case of a left-invariant sublaplacian on a connected Lie group, good estimates for the heat kernel can also be proved (see, e.g., \cite{varopoulos_analysis_1992}), where now the Riemannian metric is replaced by a Carnot-Carath\'eodory (or subriemannian, or control) distance associated to the system $X_1,\dots,X_n$ (see also \cite{nagel_balls_1985}).

At some point (see \cite{folland_applications_1977}, \cite{rothschild_hypoelliptic_1976}, \cite{goodman_nilpotent_1976}), it has been realized that the study of such differential operators could be performed by the use of a ``local approximation'' in terms of left-invariant homogeneous operators on a homogeneous nilpotent Lie group, which may be a simpler setting than the original one, but at the same time allows to preserve some of the non-commutative structure (in contrast to the Euclidean local approximation used for elliptic operators).

There are several approaches for defining such a local approximation. Here we will consider the procedure of \emph{contraction} of a Lie algebra $\lie{g}$, which produces a homogeneous Lie algebra $\lie{g}_*$ with the same linear dimension as $\lie{g}$; notice that a Lie algebra $\lie{g}$ admits several non-equivalent contractions, and one of them is Euclidean (i.e., abelian and isotropic). For the study of left-invariant differential operators on connected Lie groups, this technique has been exploited, among others (see \cite{nagel_harmonic_1990}, \cite{hebisch_estimatessemigroups_1992}), by ter Elst and Robinson \cite{ter_elst_weighted_1998}, who have introduced the notion of \emph{weighted subcoercive operator}.

As we will see, a positive left-invariant differential operator on a connected Lie group $G$ is weighted subcoercive if and only if it is ``locally Rockland'', i.e., it corresponds on a suitable contraction of the Lie algebra $\lie{g}$ to an operator whose principal part is Rockland. In particular, every elliptic positive left-invariant differential operator on $G$ is weighted subcoercive (with respect to the Euclidean contraction), but also every left-invariant sublaplacian on $G$ is weighted subcoercive (with respect to a different, possibly non-commutative contraction).

In this section, we summarize and slightly amplify some of the results of \cite{ter_elst_weighted_1998} on weighted subcoercive operators, in particular the Gaussian heat kernel estimates, which are crucial for our multiplier theorems.

\subsection{Weighted algebraic bases and contraction of a Lie algebra}\label{subsection:noncommutativemultiindex}
An \emph{algebraic basis}\index{Lie algebra!algebraic basis of|(} of a Lie algebra $\lie{g}$ is a system $A_1,\dots,A_d$ of linearly independent elements of $\lie{g}$ which generate $\lie{g}$ as a Lie algebra. A \emph{weighted algebraic basis} is an algebraic basis $A_1,\dots,A_d$ together with an assignment of a weight $w_j \in \left[1,+\infty\right[$ to each $A_j$ ($j=1,\dots,d$).

Fix a weighted algebraic basis on $\lie{g}$. We introduce now a notation, which is analogous to the multi-index notation for partial derivatives on $\R^n$, but which takes care of the non-commutative structure. Let $J(d)$ be the set of finite sequences of elements of $\{1,\dots,d\}$, and $J_+(d)$ be the subset of non-empty sequences. For every $\alpha = (\alpha_1,\dots,\alpha_k) \in J(d)$, let $|\alpha|$ denote the length $k$ of $\alpha$, and set
\[\|\alpha\| = \sum_{j=1}^k w_{\alpha_j},\]
\[A^\alpha = A_{\alpha_1} A_{\alpha_2} \cdots A_{\alpha_n} \text{ (as an element of $\UEnA(\lie{g})$),}\]
\[A_{[\alpha]} = [[\dots[A_{\alpha_1},A_{\alpha_2}],\dots],A_{\alpha_k}] \qquad\text{if $\alpha \in J_+(d)$.}\]

The fixed weighted algebraic basis defines an (increasing) \emph{filtration} on $\lie{g}$:
\[F_\lambda = \Span \{A_{[\alpha]} \tc \alpha \in J_+(d), \, \|\alpha\| \leq \lambda\} \qquad \text{for $\lambda \in \R$;}\]
we have in fact
\[[F_{\lambda},F_{\mu}] \subseteq F_{\lambda+\mu}, \qquad F_\lambda = \bigcap_{\mu > \lambda} F_\mu, \qquad \bigcup_{\lambda \in \R} F_\lambda = \lie{g}.\]
We can then consider the \emph{associated homogeneous Lie algebra}: in fact, such a filtration determines a finite set of weights $\lambda_1,\dots,\lambda_k$, with
\[1 \leq \lambda_1 < \dots < \lambda_k,\]
defined by the condition $F_{\lambda_j} \neq \bigcup_{\mu < {\lambda_j}} F_\mu$ for $j=1,\dots,k$; if we put
\[F_\lambda^- = \bigcup_{\mu < {\lambda}} F_\mu, \qquad W_\lambda = F_\lambda / F_\lambda^-,\]
then
\[\lie{g}_* = \bigoplus_{\lambda \in \R} W_\lambda = W_{\lambda_1} \oplus \dots \oplus W_{\lambda_k}\]
is a homogeneous Lie algebra, with weights $\lambda_1, \dots, \lambda_k$. The homogeneous Lie algebra $\lie{g}_*$, which has the same dimension as $\lie{g}$, is said to be the \emph{contraction}\index{Lie algebra!contraction of} of $\lie{g}$ with respect to the fixed weighted algebraic basis.

The weighted algebraic basis is called a \emph{reduced basis}\index{Lie algebra!reduced basis of} of $\lie{g}$ if\footnote{Our definition of reduced basis is more restrictive than the definition given in \S2 of \cite{ter_elst_weighted_1998}, where it is only required that $A_j \notin F_{w_j}^-$; however, without our restriction, the fundamental Lemma~2.2 of \cite{ter_elst_weighted_1998}, which allows to extend the reduced basis to a linear basis compatible with the associated filtration $F_\lambda$, is false, as it is shown by the following example. On the free $3$-step nilpotent Lie algebra on two generators, defined by
\[[X_1,X_2] = Y, \qquad [X_1,Y] = T_1, \qquad [X_2,Y] = T_2,\]
the weighted algebraic basis $X_1,X_2,Y+T_1,T_1,T_2$, with weights $1,1,3,3,3$, is reduced according to \cite{ter_elst_weighted_1998}, but it not compatible with the associated filtration, and cannot be extended since it is already a linear basis.}
\begin{equation}\label{eq:reducedbasis}
\Span \{ A_j \tc w_j = \lambda\} \cap F_\lambda^- = \{0\} \qquad\text{for all $\lambda$.}
\end{equation}
Given a weighted algebraic basis, it is always possible to remove some elements from it, in order to obtain a reduced basis of $\lie{g}$ which defines the same filtration.

Notice that, if the Lie algebra $\lie{g}$ is homogeneous as in \eqref{eq:gradation}, every algebraic basis $A_1,\dots,A_d$ of $\lie{g}$ made of homogeneous elements, with the weights equal to the respective homogeneity degrees, is a reduced basis; such a basis is said to be \emph{adapted}\index{Lie algebra!homogeneous!adapted basis of} to the homogeneous structure of $\lie{g}$. In this case we have $F_\lambda = W_\lambda \oplus F_\lambda^-$, so that in particular the corresponding homogeneous contraction $\lie{g}_*$ is canonically isomorphic to $\lie{g}$.\index{Lie algebra!algebraic basis of|)}

Consider now a general Lie algebra $\lie{g}$, with a fixed reduced basis. Then the weights $w_1,\dots,w_d$ of the basis are among the weights $\lambda_1,\dots,\lambda_k$ of the filtration; moreover, if $\bar A_j$ is the element of the quotient $W_{w_j}$ corresponding to $A_j \in V_{w_j}$, then $\bar A_1,\dots,\bar A_k$ is an (adapted) reduced basis of $\lie{g}_*$, with the same weights $w_1,\dots,w_k$ (cf.\ \cite{ter_elst_weighted_1998}, Lemma~2.2 and Proposition~3.1). As we will see in the following, the correspondence between the two reduced bases $A_1,\dots,A_k$ of $\lie{g}$ and $\bar A_1,\dots,\bar A_k$ of $\lie{g}_*$ determines in turn a correspondence between $\UEnA(\lie{g}_\C)$ and $\UEnA((\lie{g}_*)_\C)$, i.e., between invariant differential operators on a Lie group and invariant differential operators on a contraction.

\subsection{Control distance}\label{subsection:controldistance}

Let $G$ be a connected Lie group, with Lie algebra $\lie{g}$. Fix a reduced basis $A_1,\dots,A_k$ of $\lie{g}$, with weights $w_1,\dots,w_k$. For $s \in \{0,\infty,*\}$ and $\varepsilon > 0$, let $C_s(\varepsilon)$ be the set of absolutely continuous curves $\gamma : [0,1] \to G$ such that
\[\gamma'(t) = \sum_{j=1}^k \phi_j(t) \, A_j|_{\gamma(t)} \qquad\text{for a.e.\ $t \in [0,1]$,}\]
where
\begin{equation}\label{eq:mincondition}
|\phi_j(t)| < \begin{cases}
\varepsilon^{w_j} & \text{if $s = 0$,}\\
\varepsilon       & \text{if $s = \infty$,}\\
\min\{\varepsilon,\varepsilon^{w_j}\} & \text{if $s = *$,}
\end{cases} \qquad\text{for $t \in [0,1]$, $j=1,\dots,k$;}
\end{equation}
for $x,y \in G$, we define then
\[d_s(x,y) = \inf \{\varepsilon > 0 \tc \exists \gamma \in C_s(\varepsilon) \text{ with } \gamma(0) = x, \, \gamma(1) = y\}.\]

It is not difficult to show that $d_0$, $d_\infty$ and $d_*$ are left-invariant distances on $G$, compatible with the topology of $G$. In fact, $d_\infty$ is the classical ``unweighted'' Carnot-Carath\'eodory distance associated to the H\"ormander system $A_1,\dots,A_k$ (cf.\ \cite{varopoulos_analysis_1992}, \S III.4), whereas $d_0$ is a ``weighted'' Carnot-Carath\'eodory distance (similar to the ones studied in \cite{nagel_balls_1985}). Moreover, for $x,y \in G$, we have
\[d_0(x,y) \leq 1 \quad\iff\quad d_\infty(x,y) \leq 1 \quad\iff\quad d_*(x,y) \leq 1,\]
and the same holds with strict inequalities. Finally,
\[d_*(x,y) = \begin{cases}
d_0(x,y) &\text{for $d_*(x,y) \leq 1$,}\\
d_\infty(x,y) &\text{for $d_*(x,y) \geq 1$.}
\end{cases}\]
We call $d_*$ the \emph{control distance} associated to the fixed reduced basis.\footnote{Notice that the definition of the control distance by ter Elst and Robinson in \S6 of \cite{ter_elst_weighted_1998} (see also \cite{ter_elst_weighted_1994}) is different from the one given here, and coincides with our distance $d_0$. Their definition has the advantage that, in the case of a homogeneous group with an adapted basis, the modulus $|\cdot|_0$ induced by $d_0$ is a homogeneous norm; on the other hand, this shows that in general $|\cdot|_0$ is not a connected modulus (a homogeneous norm on a homogeneous group is a connected modulus if and only if the group is stratified, see Proposition~\ref{prp:quasiequivalence}). Nevertheless, in the whole papers \cite{auscher_positive_1994}, \cite{ter_elst_weighted_1994}, \cite{ter_elst_weighted_1998} it is understood that $|\cdot|_0$ is a connected modulus.

By a careful examination of the proofs, one sees that the specific properties of $d_0$ are used only for small distances, whereas in the large only connectedness is used. Therefore, our modified definition of the control distance $d$ fixes the problem (as it has been confirmed to us by ter Elst in a private communication). As a side-effect, since $d_* \geq d_0$ everywhere, the heat kernel estimates obtained with this modification (see Theorem~\ref{thm:robinsonterelst}(e)) are stronger than the ones claimed by ter Elst and Robinson (which are therefore true \emph{a posteriori}).}.

The control distance $d_*$ induces a \emph{control modulus}\index{control modulus} $|\cdot|_*$ on $G$, given by
\[|g|_* = d_*(e,g),\]
which is in fact a connected modulus (since $d_*$ coincides with $d_\infty$ in the large). Moreover, if $B_r$ denotes the $d_*$-ball with radius $r$ centered at the identity of $G$, then we have
\[\mu(B_r) \sim r^{Q_*} \qquad\text{for $r \leq 1$,}\]
where $Q_*$ is the homogeneous dimension of the contraction $\lie{g}_*$ (see \cite{ter_elst_weighted_1998}, Proposition~6.1). On the other hand, since $|\cdot|_*$ is a connected modulus, the growth rate of $\mu(B_r)$ for $r$ large coincides with the (intrinsic) volume growth of the group $G$; in particular, if $G$ has polynomial growth of degree $Q_G$, then
\[\mu(B_r) \sim r^{Q_G} \qquad\text{for $r \geq 1$.}\]

\subsection{Weighted subcoercive forms and operators}

Let $G$ be a connected Lie group, and fix a reduced basis $A_1,\dots,A_d$ of its Lie algebra $\lie{g}$, with weights $w_1,\dots,w_d$. In this context, a \emph{form} is an element of the free (non-commutative associative unital) algebra over $\C$ on $d$ indeterminates $X_1,\dots,X_d$; in other words, a form is a function $C : J(d) \to \C$ null off a finite subset of $J(d)$, which can be thought of as the non-commutative polynomial
\[\sum_{\alpha \in J(d)} C(\alpha) X^\alpha.\]
The \emph{degree} of the form $C$ is the number
\[\max \{\|\alpha\| \tc \alpha \in J(d),\,C(\alpha) \neq 0\}.\]
If $C$ is a form of degree $m$, its \emph{principal part} is the form $P : J(d) \to \C$ which is given by the sum of the terms of $C$ of degree $m$:
\[P(\alpha) = \begin{cases}
C(\alpha) &\text{if $\|\alpha\| = m$,}\\
0 &\text{otherwise.}
\end{cases}\]
A form is said to be \emph{homogeneous} if it equals its principal part. The \emph{adjoint} of a form $C$ is the form $C^+$ defined by
\[C^+(\alpha) = (-1)^{|\alpha|} \overline{C(\alpha_*)},\]
where $\alpha_* = (\alpha_k,\dots,\alpha_1)$ if $\alpha = (\alpha_1,\dots,\alpha_k)$.

To each form $C$, we associate the differential operator $d\RA_G(C) \in \Diff(G)$ defined by
\[d\RA_G(C) = \sum_{\alpha \in J(d)} C(\alpha) A^\alpha.\]
Notice that, by definition, we have clearly
\[d\RA_G(C^+) = d\RA_G(C)^+.\]
More generally, if $\pi$ is a representation of $G$, we define
\[d\pi(C) = d\pi(d\RA_G(C)) = \sum_{\alpha \in J(d)} C(\alpha) d\pi(A)^\alpha.\]

If $\pi$ is a representation of $G$ on a Banach space $\VV$, we define seminorms and norms on (subspaces of) $\VV$ by
\[N_{\pi,s}(x) = \max_{\substack{\alpha \in J(d)\\ \|\alpha\| = s}} \|dU(X^\alpha) x\|_\VV,\]
\[\|x\|_{\pi,s} = \max_{\substack{\alpha \in J(d)\\ \|\alpha\| \leq s}} \|dU(X^\alpha) x\|_\VV,\]
for $s \in \R$, $s \geq 0$. If $\pi$ is the right regular representation of $G$ on $L^p(G)$, we use the alternative notation $N_{p;s}$, $\|\cdot\|_{p;s}$.

A form $C$ of degree $m$ is said to be \emph{weighted subcoercive} on $G$ if $m/w_i \in 2\N$ for $i=1,\dots,d$ and if moreover the corresponding operator satisfies a local \emph{G\aa rding inequality}: there exist $\mu > 0$, $\nu \in \R$ and an open neighborhood $V$ of the identity $e \in G$ such that
\[\Re \langle \phi, d\RA_G(C) \phi\rangle \geq \mu (N_{2;m/2}(\phi))^2 - \nu \|\phi\|_2^2\]
for all $\phi \in \D(G)$ with $\supp \phi \subseteq V$. In this case, the operator $d\RA_G(C)$ is called a \emph{weighted subcoercive operator}\index{differential operator!weighted subcoercive}.

Let $\lie{g}_*$ be the homogeneous contraction of $\lie{g}$, and $G_*$ be the connected, simply connected Lie group corresponding to $\lie{g}_*$. Since $A_1,\dots,A_d$ induces a reduced weighted algebraic basis $\bar A_1,\dots,\bar A_d$ on $\lie{g}_*$ (with the same weights), we can associate to a form $C$ both a differential operator $d\RA_G(C)$ on $G$ and a differential operator $d\RA_{G_*}(C)$ on $G_*$, which is a homogeneous group: in some sense, $d\RA_{G_*}(C)$ is the ``local counterpart'' of the operator $d\RA_G(C)$. The next theorem, which is the fundamental result of ter Elst and Robinson \cite{ter_elst_weighted_1998}, clarifies the relationship between the two operators, and substantiates the ideas previously described.

\begin{thm}\label{thm:robinsonterelst}
Let $C$ be a form of degree $m$, whose principal part is $P$, such that $m/w_i \in 2\N$ for $i=1,\dots,d$. The following are equivalent:
\begin{itemize}
\item[(i)] $C$ is a weighted subcoercive form on $G$;
\item[(ii)] $P$ is a weighted subcoercive form on $G$;
\item[(iii)] $C$ is a weighted subcoercive form on $G_*$;
\item[(iv)] $d\RA_{G_*}(P + P^+)$ is a positive Rockland operator on $G_*$;
\item[(v)] there are constants $\mu > 0$, $\nu \in \R$ such that, for every unitary representation $\pi$ of $G$ on a Hilbert space $\HH$,
\[\Re \langle x, d\pi(C) x \rangle \geq \mu \|x\|_{\pi,m/2}^2 - \nu \|x\|_\HH^2\]
for all $x \in \HH^\infty$;
\item[(vi)] there is a constant $\mu > 0$ such that, for every unitary representation $\pi$ of $G_*$ on a Hilbert space $\HH$,
\[\Re \langle x, d\pi(P) x \rangle \geq \mu (N_{\pi,m/2}(x))^2\]
for all $x \in \HH^\infty$.
\end{itemize}
Moreover, if these conditions are satisfied, for every representation $\pi$ of $G$ on a Banach space $\VV$, we have:
\begin{itemize}
\item[(a)] the closure of $d\pi(C)$ generates a continuous semigroup $\{S_t\}_{t \geq 0}$ on $\VV$;
\item[(b)] for $t > 0$, $S_t(\VV) \subseteq \VV^\infty$, and moreover
\[\VV^\infty = \bigcap_{n = 1}^\infty D(\overline{d\pi(C)}^n);\]
\item[(c)] if $\pi$ is unitary, then $\overline{d\pi(C)} = d\pi(C^+)^*$;
\item[(d)] there exists a representation-independent kernel $k_t \in L^{1;\infty} \cap C_0^\infty(G)$ (for $t > 0$) such that
\[d\pi(X^\alpha) S_t x = \pi(A^\alpha k_t) x = \int_G (A^\alpha k_t)(g) \pi(g^{-1})x \,dg\]
for all $\alpha \in J(d)$, $t > 0$, $x \in \VV$;
\item[(e)] the kernel satisfies the following ``Gaussian'' estimates: for all $\alpha \in J(d)$ there exist $b,c,\omega > 0$ such that
\[|A^\alpha k_t(g)| \leq c t^{-\frac{Q_* + \|\alpha\|}{m}} e^{\omega t} e^{-b \left(\frac{|g|_*^m}{t}\right)^{1/(m-1)}}\]
for all $t >0$ and $g \in G$, where $Q_*$ is the homogeneous dimension of $\lie{g}_*$ and $|\cdot|_*$ is the control modulus;
\item[(f)] for all $\rho \geq 0$, the map $t \mapsto k_t$ is continuous $\left]0,\infty\right[ \to L^{1;\infty}(G, e^{\rho |x|_*} \,dx)$ and, for all $\alpha \in J(d)$, there exist $c,\omega > 0$ such that
\[\|A^\alpha k_t\|_{L^1(G, e^{\rho |x|_*} \,dx)} \leq c t^{-\frac{\|\alpha\|}{m}} e^{\omega t};\]
\item[(g)] the function
\[k(t,x) = \begin{cases}
0 &\text{for $t \leq 0$,}\\
k_t(x) &\text{for $t > 0$,}
\end{cases}\]
on $\R \times G$ satisfies
\[\left( \frac{\partial}{\partial t} + d\RA_G(C) \right) k = \delta\]
in the sense of distributions, where $\delta$ is the Dirac delta at the identity of $\R \times G$.
\end{itemize}
\end{thm}
\begin{proof}
This theorem is a summary of results of \cite{ter_elst_weighted_1998}.

More specifically: the equivalence of (i), (ii), (iii) and (iv) follows from Theorem~10.1 of \cite{ter_elst_weighted_1998}; (v) implies (i) by choosing as $\pi$ the regular representation on $L^2(G)$; vice versa, (iii) implies (v) by Theorem~9.2 of \cite{ter_elst_weighted_1998}; (vi) corresponds to (v) with $C = P$, $\nu = 0$, $G = G_*$, so that (vi) implies that $P$ is weighted subcoercive on $G_*$, but then, by the previous equivalences, we get (i)-(iv); conversely, (vi) can be obtained from (v) with $C = P$, $G = G_*$ by the use of homogeneity and dilations, but in fact it is obtained directly from (iii) in the proof of Lemma~5.1 of \cite{ter_elst_weighted_1998}.

The remaining statements are contained in Theorems~1.1, 7.2, 8.2 and Corollary 8.3 of \cite{ter_elst_weighted_1998}, except for (f), since in Theorem~7.2 of \cite{ter_elst_weighted_1998} it is only stated that the map $t \mapsto k_t$ is continuous $\left]0,\infty\right[ \to L^1(G, e^{\rho |x|_*} \,dx)$. However, by integration of the Gaussian estimates (e), we for all $\alpha \in J(d)$ we obtain
\[\begin{split}
\|A^\alpha k_t\|_{L^1(G, e^{\rho |x|_*} \,dx)} &\leq c t^{-\frac{Q_* + \|\alpha\|}{m}} e^{\omega t} \int_0^\infty e^{-b \left(\frac{r^m}{t}\right)^{1/(m-1)}} e^{\sigma r} r^{Q_*-1} \,dr \\
&= c t^{-\frac{\|\alpha\|}{m}} e^{\omega t} \int_0^\infty e^{-br^{m/(m-1)}} e^{\sigma t^{1/m} r} r^{Q_*-1} \,dr
\end{split}\]
for some $b,c,\omega,\sigma > 0$; for $t \leq 1$, the last integral is less than
\[\int_0^\infty e^{-br^{m/(m-1)}} e^{\sigma r} r^{Q_*-1} \,dr,\]
which is finite, whereas, for $t \geq 1$, the integral is less than
\begin{multline*}
\int_0^\infty e^{-b \left(\frac{r^m}{t}\right)^{1/(m-1)}} e^{\sigma r} r^{Q_*-1} \,dr = \int_0^{(2\sigma/b)^{m-1} t} + \int_{(2\sigma/b)^{m-1} t}^\infty \\
\leq \int_0^{(2\sigma/b)^{m-1} t} e^{\sigma r} r^{Q_*-1} \,dr + \int_0^\infty e^{-\sigma r} r^{Q_*-1} \,dr \leq c' e^{\sigma' t},
\end{multline*}
for some $c',\sigma' > 0$, and putting all together we obtain the required estimates. Moreover, by the semigroup property, we have
\begin{equation}\label{eq:semigroup}
A^{\alpha} (k_{t+s}) = k_t * (A^\alpha k_s)
\end{equation}
and, since $A^\alpha k_s \in L^1(G, e^{\rho |x|_*} \,dx)$ by the Gaussian estimates (e), the required continuity follows from the properties of convolution.
\end{proof}

\begin{cor}\label{cor:hypoelliptic}
With the notation of the previous theorem, if $C$ is a weighted subcoercive form on $G$, then the function $k(t,x) = k_t(x)$ is smooth off the identity of $\R \times G$ and the operator $d\RA_G(C)$ is hypoelliptic.
\end{cor}
\begin{proof}
From Theorem~\ref{thm:robinsonterelst}(g) we deduce that, for every $r \in \N \setminus \{0\}$, the distribution
\[(\partial_t^r - (- d\RA_G(C))^r) k\]
is supported in the origin of $\R \times G$. In particular, if $\phi \in \D(\left]0,+\infty\right[)$ and $\psi \in \D(G)$, we get
\[\langle (\partial_t^r - (- d\RA_G(C))^r) k, \phi \otimes \psi \rangle = 0,\]
which can be rewritten as
\[(-1)^r \int_0^\infty \langle k_t, \psi \rangle \, \overline{\phi^{(r)}(t)} \,dt = \int_0^\infty \langle (-d\RA_G(C))^r k_t, \psi \rangle \, \overline{\phi(t)} \,dt.\]
Since both $t \mapsto k_t$ and $t \mapsto (-d\RA_G(C))^r k_t$ are continuous $\left]0,+\infty\right[ \to L^1(G)$ by Theorem~\ref{thm:robinsonterelst}(f), this identity holds also for all $\psi \in C_0(G)$. In other words, for all $\psi \in C_0(G)$, the $r$-th distributional derivative of the function
\[t \mapsto \langle k_t, \psi \rangle\]
on $\left]0,+\infty\right[$ is the map
\[t \mapsto \langle (-d\RA_G(C))^r k_t, \psi \rangle;\]
since all these derivatives are continuous, the function $t \mapsto \langle k_t, \psi \rangle$ is smooth on $\left]0,+\infty\right[$, so that also the map $t \mapsto k_t$ is smooth $\left]0,+\infty\right[ \to L^1(G)$. But then from \eqref{eq:semigroup} it follows easily that, for all $\alpha \in J(d)$, the map $t \mapsto A^\alpha k_t$ is smooth $\left]0,+\infty\right[ \to L^1(G)$, i.e., that $t \mapsto k_t$ is smooth $\left]0,+\infty\right[ \to L^{1;\infty}(G)$. By Sobolev's embedding, we then get that $t \mapsto k_t$ is smooth $\left]0,+\infty\right[ \to \E(G)$; this gives that $k$ is smooth on $\left]0,+\infty\right[ \times G$, and the Gaussian estimates of Theorem~\ref{thm:robinsonterelst}(e) show that $k$ can be extended smoothly by zero to the whole $\R \times G \setminus \{(0,e)\}$.

Notice that $k_t^*$ is the kernel of $d\RA_G(C^+)$, which is also a weighted subcoercive operator. If we put
\[\tilde k(t,x) = \begin{cases}
0 &\text{if $t \geq 0$},\\
k_{-t}^* &\text{if $t \leq 0$,}
\end{cases}\]
we then have that $\tilde k$ is smooth on $\R \times G \setminus \{(0,e)\}$ and satisfies
\[(-\partial_t + d\RA_G(C^+)) \tilde k = \delta\]
in the sense of distributions. By arguing analogously as in the proof of Theorem~52.1 of \cite{treves_topological_1967}, we obtain that $\partial_t + d\RA_G(C)$ is hypoelliptic on $\R \times G$, and the hypoellipticity of $d\RA_G(C)$ on $G$ follows immediately.
\end{proof}

\begin{cor}\label{cor:kernelapproximateidentity}
With the notation of Theorem~\ref{thm:robinsonterelst}, if $C$ is a weighted subcoercive form on $G$, then, for every $D \in \Diff(G)$, $\beta \geq 0$ and every neighborhood $U$ of the identity $e \in G$,
\begin{equation}\label{eq:kerneloutofaneighborhood}
\lim_{t \to 0^+} t^{-\beta} \int_{G \setminus U} |Dk_t(x)| \,dx = 0.
\end{equation}
Moreover, $(k_t)_{t > 0}$ is an approximate identity on $G$ for $t \to 0^+$.
\end{cor}
\begin{proof}
If $R > 0$ is such that
\[\{x \in G \tc |x|_* < R \} \subseteq U,\]
then, by Theorem~\ref{thm:robinsonterelst}(e), for $t \leq 1$ we have
\[t^{-\beta} \int_{G \setminus U} |Dk_t(x)| \,dx \leq c t^{-\gamma} \int_R^{+\infty} e^{-b(r^m/t)^{1/(m-1)}} e^{\sigma r} \,dr\]
for some $c,b,\sigma,\gamma > 0$. On the other hand, for $t \leq 1$ and $r \geq R$,
\[t^{-\gamma} e^{-b(r^m/t)^{1/(m-1)}} e^{\sigma r} \leq e^{-b(r^{\frac{m}{m-1}} - R^{\frac{m}{m-1}}) + \sigma r} e^{- \gamma \log t -b R^{\frac{m}{m-1}} t^{-\frac{1}{m-1}} },\]
where the first factor on the right-hand side is integrable on $\left]R,+\infty\right[$ and does not depend on $t$, whereas the second factor is infinitesimal for $t \to 0^+$ and does not depend on $r$; the limit \eqref{eq:kerneloutofaneighborhood} then follows by dominated convergence.

In particular, we have
\[\lim_{t \to 0^+} \int_{G \setminus U} |k_t(x)| \,dx = 0,\]
and moreover, by Theorem~\ref{thm:robinsonterelst}(f), the norms $\|k_t\|_1$ are uniformly bounded for $t$ small. Finally, if $\pi$ is the trivial representation of $G$ on $\C$ and if
\[c = d\pi(C) 1,\]
then by Theorem~\ref{thm:robinsonterelst}(d) we have
\[\int_G h_t(x) \,dx = \pi(h_t) 1 = e^{-tc},\]
which tends to $1$ as $t \to 0^+$.
\end{proof}

\subsection{Examples}

Theorem~\ref{thm:robinsonterelst}, which contains the announced relationship between Rockland operators and weighted subcoercive operators, however does not imply immediately that every positive Rockland operator is weighted subcoercive. Nevertheless, this is true:

\begin{prp}\label{prp:rocklandwsub}
Suppose that $G$ is a homogeneous Lie group. Then there exists an adapted basis $A_1,\dots,A_d$ of $\lie{g}$ such that
\[\Span\{A_1,\dots,A_d\} \cap [\lie{g},\lie{g}] = 0.\]
With respect to such a basis, for every positive Rockland operator $L$ of degree $m$, there exists a homogeneous form of degree $m$ such that $d\RA_G(P) = L$; moreover, such a form is weighted subcoercive.
\end{prp}
\begin{proof}
See \cite{ter_elst_spectral_1997}, Lemmata 2.2 and 2.4, and Theorem 2.5; see also \cite{ter_elst_weighted_1998}, Example 4.4.
\end{proof}

On the other hand, it is easy to check that, for every choice of an algebraic basis $A_1,\dots,A_d$ of a Lie algebra $\lie{g}$, the assignment of weights all equal to $1$ always gives a reduced basis, and that the corresponding contraction $\lie{g}_*$ is stratified. In particular, the sublaplacian\index{sublaplacian} $L = -(A_1^2 + \dots + A_d^2)$ corresponds to a homogeneous sublaplacian on the stratified contraction, which is Rockland, and therefore $L$ is weighted subcoercive. Moreover, if $A_1,\dots,A_d$ linearly generate $\lie{g}$, then the contraction $\lie{g}_*$ is Euclidean (abelian and isotropic), and it is not difficult to prove that positive elliptic operators are weighted subcoercive with respect to this contraction.

%% file: conditions.tex
\chapter{Smoothness conditions}\label{chapter:conditions}

While the previous chapter was devoted to the ``group side'' --- i.e., to the notions and results related to Lie groups and differential operators --- the present chapter focuses on the ``spectral side'' --- i.e., on the environment where the (joint) multiplier function is defined. Although one could consider the joint spectrum of a system of operators as an abstract space (namely, the Gelfand spectrum of some commutative C$^*$-algebra), generally an embedding of the spectrum into some $\R^n$ will be fixed, so that the multiplier will be thought of as a Borel function $\R^n \to \C$. Consequently, in order to impose smoothness conditions on the multiplier, we can use all the machinery of classical real and harmonic analysis, particularly the \index{transform!Fourier!on $\R^n$}Fourier transform
\[\Four f(\xi) = \int_{\R^n} f(x) \, e^{-i x \cdot \xi}  \,dx.\]

Smoothness conditions will be in fact expressed in terms of \emph{Besov spaces}. This choice is certainly not new in the context of spectral multiplier theorems (see, e.g., \cite{de_michele_mulipliers_1987} or \cite{alexopoulos_spectral_1994}). Besov spaces $B_{p,q}^s(\R^n)$ allow us to consider smoothness conditions of fractional order $s \in \R$ and of $L^p$ flavour for $1 \leq p \leq \infty$; in particular, for $p=2$, they include the classical Sobolev (or Bessel potential) spaces, whereas for $p=\infty$ they include the H\"older spaces. Moreover, they have a very good behaviour under (real and complex) interpolation, which is a particularly useful tool in lowering regularity thresholds of multiplier theorems.

In order to formulate conditions of Marcinkiewicz type, we need a way of requiring different orders of smoothness independently on the different factors of a product
\[\R^{\vec n} = \R^{n_1} \times \dots \times \R^{n_\ell}.\]
This is achieved by considering Besov spaces with \emph{dominating mixed smoothness}\footnote{For spaces with dominating mixed smoothness, we prefer the notation $S^{\vec{s}}_{p,q} B(\R^{\vec{n}})$, with an $S$ prefixed, to the notation $B_{p,q}^{\vec{s}}(\R^{\vec{n}})$, since the latter is used, in our main references, to denote a different class of spaces, the anisotropic spaces (see, e.g., \cite{triebel_theory_1983}, \S10.1, or \cite{schmeisser_topics_1987}, \S2.2.2).} $S^{\vec{s}}_{p,q} B(\R^{\vec{n}})$, in which the order of smoothness is a vector $\vec s = (s_1,\dots,s_\ell)$ of real numbers (the idea is that, if $s_1,\dots,s_\ell > 0$, the maximum order of differentiability $s_1+ \dots + s_\ell$ is reached only by mixed derivatives).

Whereas an enormous amount of literature exists about Besov spaces, their dominating-mixed-smoothness variant has not been treated to the same extent, and usually not in the generality which is required here; moreover, not every property of classical Besov spaces has a straightforward multi-parameter extension (e.g., Besov spaces of dominating mixed smoothness are less flexible with respect to interpolation). For these reasons, in the following, a sketch of the proofs of the main properties of Besov spaces will be given, so that their (possible) generalization to the dominating-mixed-smoothness case will be apparent. It is interesting to notice that some properties of Besov spaces (e.g., the lifting properties) follow from an elementary version of a Fourier multiplier theorem.

Finally, conditions of Mihlin-H\"ormander and Marcinkiewicz type (i.e., local Besov conditions which are invariant with respect to some family of dilations) will be introduced, and their change-of-variable properties and mutual implications will be studied.

First of all, however, we need to recall some definitions and results about interpolation of Banach spaces, and some elementary properties of the classical Sobolev spaces.

Notice that $\R^n$ is a nilpotent Lie group, therefore the contents of the previous chapter apply, particularly the notions about families of dilations and homogeneous norms. On the other hand, here we do not need the involved non-commutative version of the multi-index notation for translation-invariant derivatives described in \S\ref{subsection:noncommutativemultiindex}, thus we adhere to the simpler traditional commutative notation.

\section{Interpolation}

This section is simply a collection of well-known results from interpolation theory. For a more extensive presentation see, e.g., \cite{bergh_interpolation_1976}, \cite{triebel_interpolation_1978}, \cite{brudny_interpolation_1991}.

\subsection{Banach couples and interpolation functors}
A \emph{Banach couple}\index{Banach couple} is a pair $\overline{X} = (X_0,X_1)$ of Banach spaces, such that there exists a Hausdorff topological vector space $V$ in which both $X_0,X_1$ are continuously included; in this case, we can define
\[\Delta(\overline{X}) = X_0 \cap X_1, \qquad \Sigma(\overline{X}) = X_0 + X_1\]
as subspaces of $V$, which are in fact Banach spaces with norms
\[\|x\|_{\Delta(\overline{X})} = \max\{\|x\|_{X_0},\|x\|_{X_1}\}, \qquad \|x\|_{\Sigma(\overline{X})} = \inf_{x = x_0 + x_1} (\|x_0\|_{X_0} + \|x_1\|_{X_1}).\]

A morphism of Banach couples $T : \overline{X} \to \overline{Y}$ is a pair $(T_0,T_1)$ of bounded linear maps $T_j : X_j \to Y_j$, whose restrictions to $\Delta(\overline{X})$ coincide, and define then a bounded linear map $T_{\Delta} : \Delta(\overline{X}) \to \Delta(\overline{Y})$. In this case, the maps $T_0,T_1$ can be pasted together to obtain a bounded linear map $T_{\Sigma} : \Sigma(\overline{X}) \to \Sigma(\overline{Y})$.

An \emph{intermediate space} for a Banach couple $\overline{X}$ is a Banach space $X$ such that
\[\Delta(\overline{X}) \subseteq X \subseteq \Sigma(\overline{X}),\]
with continuous inclusions.

An \emph{interpolation functor}\index{interpolation!functor} is a functor $F$ from the category of Banach couples to the category of Banach spaces such that:
\begin{itemize}
\item for all Banach couples $\overline{X}$, $F(\overline{X})$ is an intermediate space for $\overline{X}$;
\item for all morphisms $T : \overline{X} \to \overline{Y}$ of Banach couples, we have
\[F(T) = T_{\Sigma}|_{F(\overline{X})}.\]
\end{itemize}
In other words, an interpolation functor $F$ is a choice of an intermediate space $F(\overline{X})$ for each Banach couple $\overline{X}$, such that, for every morphism of Banach couples $T : \overline{X} \to \overline{Y}$, the restriction $T_\Sigma|_{F(\overline{X})}$ is bounded $F(\overline{X}) \to F(\overline{Y})$.

For $\theta \in \left[0,1\right]$, an interpolation functor $F$ is said to be \emph{exact of exponent $\theta$} if, for every morphism $T : \overline{X} \to \overline{Y}$ of Banach couples,
\[\|F(T)\|_{F(\overline{X}) \to F(\overline{Y})} \leq \|T_0\|_{X_0 \to Y_0}^{1-\theta} \|T_1\|_{X_1 \to Y_1}^\theta.\]

\subsection{Retracts}
Suppose that $T : \overline{X} \to \overline{Y}$ and $S : \overline{Y} \to \overline{X}$ are morphisms of Banach couples such that $S \circ T = \id_{\overline{X}}$ (in this case, we say that $\overline{X}$ is a \emph{retract}\index{retract} of $\overline{Y}$ via the maps $T$ and $S$). Then, for every interpolation functor $F$ and every intermediate space $X$ of $\overline{X}$, we have that $F(\overline{X}) = A$ (with equivalent norms) if and only if the maps
\[T_\Sigma|_A : A \to F(\overline{Y}) \qquad\text{and}\qquad S_\Sigma|_{F(\overline{Y})} : F(\overline{Y}) \to A\]
are continuous; in other words, $F(\overline{X})$ is the (unique) retract of $F(\overline{Y})$ via (restrictions of) the maps $T$ and $S$ (cf.\ \cite{bergh_interpolation_1976}, Theorem~6.4.2, and \cite{triebel_interpolation_1978}, \S1.2.4).

\subsection{Real interpolation}
The \emph{real interpolation method}\index{interpolation!real} gives one of the most important examples of interpolation functors. Recall that, for a Banach couple $\overline{X}$, the \emph{Peetre functional} is defined by
\[K(t,x;\overline{X}) = \inf_{x = x_0 + x_1} (\|x\|_{X_0} + t \|x||_{X_1}).\]
$K(t,\cdot;\overline{X})$ is an equivalent norm on $\Sigma(\overline{X})$ for every $t > 0$. We then set
\[\|x\|_{\overline{X}_{\theta,q}} = \left( \int_0^{+\infty} (t^{-\theta} K(t,x;\overline{X}))^q \,\frac{dt}{t} \right)^{1/q}\]
for $0 < \theta < 1$, $1 \leq q < \infty$, and also
\[\|x\|_{\overline{X}_{\theta,\infty}} = \sup_{t > 0} t^{-\theta} K(t,x;\overline{X})\]
for $0 \leq \theta \leq 1$, $q = \infty$. By defining
\[\overline{X}_{\theta,q} = \{ x \in \Sigma(\overline{X}) \tc \|x\|_{\overline{X}_{\theta,q}} < \infty\},\]
we get that $\overline{X}_{\theta,q}$ is an intermediate space for $\overline{X}$, and moreover the correspondence
\[\overline{X} \mapsto \overline{X}_{\theta,q}\]
gives an exact interpolation functor of exponent $\theta$ (\cite{bergh_interpolation_1976}, Theorem 3.1.2).

\subsection{Complex interpolation}
The other important example of interpolation functors is given by the \emph{complex interpolation method}\index{interpolation!complex}. Set 
\[S = \{\theta + it \in \C \tc 0 \leq \theta \leq 1, \, t \in \R\}\]
and, for a Banach couple $\overline{X}$, let $\mathcal{F}(\overline{X})$ be the set of the continuous and bounded functions $f : S \to \Sigma(\overline{X})$ which are holomorphic in the interior of $S$ and such that, for $\theta=0,1$, the function $t \mapsto f(\theta + it)$ is continuous $\R \to X_\theta$ and vanishes at infinity; in fact, $\mathcal{F}(\overline{X})$ is a Banach space, with norm
\[\|f\|_{\mathcal{F}(\overline{X})} = \max\left\{\sup_{t \in \R} \|f(it)\|_{X_0}, \, \sup_{t \in \R} \|f(1+it)\|_{X_1}\right\}.\]
Then, for $\theta \in [0,1]$, the space
\[\overline{X}_{[\theta]} = \{f(\theta) \tc f \in \mathcal{F}(\overline{X})\},\]
with norm
\[\|x\|_{\overline{X}_{[\theta]}} = \inf \{ \|f\|_{\mathcal{F}(\overline{X})} \tc f(\theta) = a\},\]
is a Banach space, which is an intermediate space for the couple $\overline{X}$; moreover, the correspondence
\[\overline{X} \mapsto \overline{X}_{[\theta]}\]
gives an exact interpolation functor of exponent $\theta$ (\cite{bergh_interpolation_1976}, Theorem~4.1.2).

\subsection{Interpolation of Lebesgue spaces}

In the following, equalities between Banach spaces are meant in the sense of equivalence of norms.

The basic result about interpolation of Lebesgue spaces is

\begin{thm}\label{thm:lebesgueinterpolation}
If $(U,\mu)$ is a measure space, $1 \leq p,p_0,p_1 \leq \infty$, $0 < \theta < 1$, and
\[\frac{1}{p} = \frac{1-\theta}{p_0} + \frac{\theta}{p_1},\]
then
\[(L^{p_0}(U,\mu),L^{p_1}(U,\mu))_{\theta,p} = (L^{p_0}(U,\mu),L^{p_1}(U,\mu))_{[\theta]} = L^p(U,\mu).\]
\end{thm}
\begin{proof}
See \cite{bergh_interpolation_1976}, Theorems~5.1.1 and 5.2.1.
\end{proof}

For weighted Lebesgue spaces, we have the following extension of the Stein-Weiss interpolation theorem.

\begin{thm}\label{thm:steinweiss}
Let $(U,\mu)$ be a measure space and $w_0,w_1: U \to \C$ be measurable and non-negative. If $1 \leq p,p_0,p_1 < \infty$, $0 < \theta < 1$, and
\[\frac{1}{p} = \frac{1-\theta}{p_0} + \frac{\theta}{p_1}, \qquad w = w_0^{p(1-\theta)/p_0} w_1^{p\theta/p_1},\]
then
\[(L^{p_0}(U,w_0 \mu),L^{p_1}(U,w_1 \mu))_{\theta,p} = (L^{p_0}(U,w_0 \mu),L^{p_1}(U,w_1 \mu))_{[\theta]} = L^p(U,w \mu).\]
\end{thm}
\begin{proof}
See \cite{bergh_interpolation_1976}, Theorems~5.4.1, 5.5.1 and 5.5.3, or \cite{triebel_interpolation_1978}, \S1.18.5.
\end{proof}

More extensive results hold for vector-valued sequence spaces. For every Banach space $A$, for $1 \leq p \leq \infty$, $s \in \R$, we set
\[l^p_s(A) = l^p_s(\N;A) = \{ (a_m)_m \in A^{\N} \tc \| (2^{ms} \|a_m\|_A)_m \|_p < \infty \}.\]

\begin{thm}\label{thm:sequenceinterpolation}
Let $1 \leq p_0, p_1, p \leq \infty$, $s_0,s_1 \in \R$, $0 < \theta < 1$.
\begin{itemize}
\item[(i)] For every Banach space $A$, if $s_0 \neq s_1$ then we have
\[(l^{p_0}_{s_0}(A),l^{p_1}_{s_1}(A))_{\theta,p} = l^p_{(1-\theta)s_0 + \theta s_1}(A).\]
\item[(ii)] For every Banach couple $(A_0,A_1)$, if $p_0,p_1 < \infty$ and
\[\frac{1}{p} = \frac{1-\theta}{p_0} + \frac{\theta}{p_1},\]
then
\[(l^{p_0}_{s_0}(A_0),l^{p_1}_{s_1}(A_1))_{\theta,p} = l^p_{(1-\theta)s_0 + \theta s_1}((A_0,A_1)_{\theta,p}).\]
\item[(iii)] For every Banach couple $(A_0,A_1)$, if $p_0 < \infty$ and
\[\frac{1}{p} = \frac{1-\theta}{p_0} + \frac{\theta}{p_1},\]
then
\[(l^{p_0}_{s_0}(A_0),l^{p_1}_{s_1}(A_1))_{[\theta]} = l^p_{(1-\theta)s_0 + \theta s_1}((A_0,A_1)_{[\theta]}).\]
\end{itemize}
\end{thm}
\begin{proof}
See \cite{bergh_interpolation_1976}, Theorems 5.6.1, 5.6.2 and 5.6.3, or \cite{triebel_interpolation_1978}, \S\S1.18.1 and 1.18.2.
\end{proof}

\section{Sobolev spaces}

\subsection{Basic properties}

For $1 \leq p \leq \infty$, $k \in \N$, the \emph{$L^p$ Sobolev space of order $k$} is defined by
\[W^{k,p}(\R^n) = \{ f \in L^p(\R^n) \tc \partial^\alpha f \in L^p(\R^n) \text{ for $|\alpha| \leq k$} \}\]
(where derivatives $\partial^\alpha f$ are meant in the sense of distributions). $W^{k,p}(\R^n)$ is a Banach space, with norm
\[\|f\|_{W^{k,p}} = \sum_{|\alpha| \leq k} \|\partial^\alpha f\|_{p}.\]

For $p=2$, $W^{k,p}(\R^n)$ is (modulo equivalence of norms) an Hilbert space, which can be characterized in terms of the Fourier transform. In fact, if we set
\[\langle x \rangle = (1 + |x|_2^2)^{1/2},\]
where $|\cdot|_2$ is the Euclidean norm on $\R^n$, and, for $s \in \R$, we introduce the \emph{Bessel-potential space}
\[H^s(\R^n) = \{f \in \Sz'(\R^n) \tc \Four f \in L^2(\R^n,\langle \xi \rangle^{2s} \,d\xi) \},\]
then $H^s(\R^n)$ is a Hilbert space, with inner product
\[\langle f,g\rangle_{H^s} = \int_{\R^n} \hat f(\xi) \, \overline{\hat g(\xi)} \, \langle \xi \rangle^{2s} \,d\xi,\]
and moreover, by the properties of the Fourier transform, it is easily seen that
\[W^{k,2}(\R^n) = H^k(\R^n) \qquad\text{for $k \in \N$.}\]
That is why, for a general $s \in \R$, the space $H^s(\R^n)$ is also called $L^2$ Sobolev space of fractional order $s$.

Since the spaces $H^s(\R^n)$ correspond, via the Fourier transform, to weighted $L^2$ spaces, from Theorem~\ref{thm:steinweiss} we get immediately
\begin{prp}\label{prp:sobolevinterpolation}
For $s_0,s_1 \in \R$, $0 < \theta < 1$, we have
\[(H^{s_0}(\R^n),H^{s_1}(\R^n))_{[\theta]} = (H^{s_0}(\R^n),H^{s_1}(\R^n))_{\theta,2} = H^{(1-\theta)s_0 + \theta s_1}(\R^n).\]
\end{prp}

Some elementary embedding results for Sobolev spaces are collected in the following

\begin{prp}\label{prp:sobolevembedding}
\begin{itemize}
\item[(i)] If $s > n/2$, then
\[\|\Four f\|_1 \leq C_s \|f\|_{H^s}.\]
\item[(ii)] If $k \in \N$ and $s > k+n/2$, then, for every $f \in H^s(\R^n)$, we have $f \in C^k_0(\R^n)$ and
\[\|f\|_{C^k_0} \leq C_{k,s} \|f\|_{H^s}.\]
\item[(iii)] The continuous embedding $W^{n,1}(\R^n) \subseteq C_b(\R^n)$ holds.
\end{itemize}
\end{prp}
\begin{proof}
(i) follows immediately from H\"older's inequality. From the Riemann-Lebesgue lemma and (i), we get easily (ii). Finally, if $f \in \D(\R^n)$, then
\[f(x_1,\dots,x_n) = \int_{-\infty}^{x_1} \cdots \int_{-\infty}^{x_n} \partial_1 \cdots \partial_n f(t_1,\dots,t_n) \,dt_1 \dots dt_n,\]
so that
\[\|f\|_\infty \leq \|f\|_{W^{n,1}},\]
and (iii) follows by a density argument (cf.\ Proposition~\ref{prp:approximateidentity}).
\end{proof}

\subsection{Sobolev spaces with dominating mixed smoothness}

If $\vec n = (n_1,\dots,n_\ell)$, for $l=1,\dots,\ell$, let $\partial_{l,1},\dots,\partial_{l,n_l}$ denote the canonical basis of translation-invariant vector fields on the $l$-th factor of
\[\R^{\vec{n}} = \R^{n_1} \times \dots \times \R^{n_\ell}.\]
Partial derivatives on $\R^{\vec{n}}$ will be denoted by a multi-multi-index notation: if $\vec{\alpha} = (\alpha_1,\dots,\alpha_\ell) \in \N^{\vec{n}} = \N^{n_1} \times \dots \times \N^{n_\ell}$, then we set
\[\partial^{\vec{\alpha}} = \partial_1^{\alpha_1} \cdots \partial_n^{\alpha_n},\]
where in turn
\[\partial_l^{\alpha_l} = \partial_{l,1}^{\alpha_{l,1}} \cdots \partial_{l,n_l}^{\alpha_{l,n_l}}.\]

For $1 \leq p \leq \infty$, $\vec{k} \in \N^\ell$, we then define the $L^p$ Sobolev space with dominating mixed smoothness of order $\vec{k}$ as
\[S^{\vec{k},p} W(\R^{\vec{n}}) = \{ f \in L^p(\R^{\vec{n}}) \tc \partial^{\vec\alpha} f \in L^p(\R^{\vec{n}}) \text{ for $|\alpha_1| \leq k_1, \dots, |\alpha_\ell| \leq k_\ell$}\}.\]
Moreover, for $\vec{s} \in \R^\ell$, we introduce the notation
\[\langle \xi \rangle^{\vec{s}} = \langle \xi_1 \rangle^{s_1} \cdots \langle \xi_\ell \rangle^{s_\ell} \qquad\text{for $\xi = (\xi_1,\dots,\xi_\ell) \in \R^{\vec{n}}$,}\]
and then we define the $L^2$ Sobolev (or Bessel-potential) space with dominating mixed smoothness of fractional order $\vec{s}$ as
\[S^{\vec{s}} H(\R^{\vec{n}}) = \{f \in \Sz'(\R^{\vec{n}}) \tc \Four f \in L^2(\R^{\vec{n}}, \langle \xi \rangle^{2\vec{s}} \,d\xi)\}.\]

Analogously as before, the $S^{\vec k,p}(\R^{\vec{n}})$ are Banach spaces, the $S^{\vec{s}}H(\R^{\vec{n}})$ are Hilbert spaces, and
\[S^{\vec{k},2} W(\R^{\vec{n}}) = S^{\vec{k}} H(\R^{\vec{n}}) \qquad\text{for $\vec{k} \in \N^\ell$.}\]
Moreover, the space $S^{\vec{s}} H(\R^{\vec{n}})$ can be characterized as a Hilbert tensor product of Bessel-potential spaces on the factors of $\R^{\vec{n}}$:
\begin{equation}\label{eq:sobolevtensorproduct}
S^{\vec{s}}H(\R^{\vec{n}}) \cong H^{s_1}(\R^{n_1}) \htimes \cdots \htimes H^{s_\ell}(\R^{n_\ell}).
\end{equation}

\section{Besov spaces}

We now introduce the class of Besov spaces, which contains and extends the previously considered scale of $L^2$ Sobolev spaces with fractional order.

Our main references for Besov spaces are the books by Peetre \cite{peetre_new_1976}, Bergh and L\"ofstr\"om \cite{bergh_interpolation_1976}, Triebel \cite{triebel_interpolation_1978}, \cite{triebel_spaces_1978}, \cite{triebel_theory_1983}, \cite{triebel_structure_2001}, Edmunds and Triebel \cite{edmunds_function_1996}, Runst and Sickel \cite{runst_sobolev_1996}. In fact, results about Besov spaces are somewhat scattered through the literature. Moreover, there are several equivalent definitions of Besov spaces, which use quite different techniques (Lipschitz-H\"older conditions, atomic decompositions, wavelets...).

Here we use the characterization of the Besov spaces $B_{p,q}^s(\R^n)$ by dyadic decompositions via the Fourier transform. Since we are only concerned with the ranges $1 \leq p \leq \infty$, $1 \leq q \leq \infty$, and since moreover we are not interested in the simultaneous introduction of Triebel-Lizorkin spaces, the exposition can be kept to quite an elementary level.

About Besov spaces with dominating mixed smoothness, the standard (and almost unique) reference is the monograph by Schmeisser and Triebel \cite{schmeisser_topics_1987}; a survey with recent developments can be found in \cite{schmeisser_recent_2007}.

\subsection{Basic definitions, interpolation and embeddings between Besov spaces}\label{subsection:besovbasic}
Let $(\phi_0,\phi_1) \in \D(\R^n)$ be non-negative, and set
\[\phi_k(\xi) = \phi_1(2^{1-k} \xi) \qquad\text{for $\xi \in \R^n$, $k > 1$.}\]
The pair $(\phi_0,\phi_1)$ is said to be \emph{admissible} if $\supp \phi_1 \subseteq \R^n \setminus\{0\}$ and
\[\bigcup_{k \geq 0} \{\phi_k \neq 0\} = \R^n;\]
the system $(\phi_k)_{k \in \N}$ is said to be \emph{generated} by the pair $(\phi_0,\phi_1)$.

If $(\phi_0,\phi_1)$ is an admissible pair, then we define, for $1 \leq p,q \leq \infty$, $s \in \R$, the \emph{Besov space} $B_{p,q}^s(\R^n)$ as the set of tempered distributions $f \in \Sz'(\R^n)$ such that the quantity
\[\|f\|_{B_{p,q}^s} = \|(2^{ks} \|(\Four^{-1} \phi_k) * f\|_p)_k\|_{l^q}\]
is finite. Using Proposition~\ref{prp:sobolevembedding}(i), partition-of-unity techniques and the fact that the quantity $\|\Four^{-1} \psi\|_1$ is invariant by dilations of $\psi$, it is not difficult to prove that, by replacing $(\phi_0,\phi_1)$ with another admissible pair, we obtain an equivalent norm on $B_{p,q}^s(\R^n)$.

One way of obtaining an admissible pair is by starting from a non-negative $\phi \in \D(\R^n)$ such that
\[\supp \phi \subseteq \{ \xi \tc 2^{-1} \leq |\xi| \leq 2 \},\]
(where $|\cdot|$ is some norm on $\R^n$) and
\[\sum_{k \in \Z} \phi(2^k \xi) = 1 \qquad\text{for every $\xi \neq 0$,}\]
and then setting
\[\phi_k(\xi) = \phi(2^{-k} \xi) \quad\text{for $k > 0$, and}\quad \phi_0 = 1 - \sum_{k > 0} \phi_k.\]
Thus, we have $\phi_k \in \D(\R^n)$,
\[0 \leq \phi_k \leq 1 \qquad\text{for $k \geq 0$, and}\qquad \sum_{k \geq 0} \phi_k \equiv 1;\]
\begin{equation}\label{eq:besovsupports}
\supp \phi_k \subseteq \begin{cases}
\{ \xi \tc |\xi| \leq 2 \} &\text{for $k = 0$,}\\
\{ \xi \tc 2^{k-1} \leq |\xi| \leq 2^{k+1} \} &\text{for $k > 0$,}
\end{cases}
\end{equation}
so that in particular $(\phi_0,\phi_1)$ is admissible (a pair obtained in this way will be called a \emph{standard admissible pair} with respect to the norm $|\cdot|$). Moreover, if we set
\[\tilde \phi_k = \begin{cases}
\phi_0 + \phi_1 &\text{if $k = 0$,}\\
\phi_{k-1} + \phi_k + \phi_{k+1} &\text{if $k > 0$,}
\end{cases}\]
then we have
\[\phi_k \tilde \phi_k = \phi_k \qquad\text{for $k \geq 0$,}\]
\begin{equation}\label{eq:fl1boundedness}
\sup_{k \geq 0} \| \Four^{-1} \phi_k \|_1 < \infty \qquad\text{and}\qquad \sup_{k \geq 0} \|\Four^{-1} \tilde\phi_k \|_1 < \infty.
\end{equation}
With this choice of an admissible pair, it is not too difficult to obtain the following properties:
\begin{enumerate}
\item for $f \in \Sz(\R^n)$ (resp., $f \in \Sz'(\R^n)$), we have
\begin{equation}\label{eq:besovdecomposition}
f = \sum_{k \geq 0} \phi_k f \qquad\text{and}\qquad f = \sum_{k \geq 0} (\Four^{-1} \phi_k) * f,
\end{equation}
with convergence in $\Sz(\R^n)$ (resp., in $\Sz'(\R^n)$);
\item for $1 \leq p,q \leq \infty$, $s \in \R$, the following H\"older-type inequality holds:
\[| \langle f, g \rangle | \leq C_{p,q,s} \|f\|_{B_{p,q}^s} \|g\|_{B_{p',q'}^{-s}}\]
for $f \in \Sz'(\R^n)$, $g \in \Sz(\R^n)$;
\item for $1 \leq p,q \leq \infty$, $s \in \R$, we have the continuous inclusions
\[\Sz(\R^n) \subseteq B_{p,q}^s(\R^n) \subseteq \Sz'(\R^n);\]
\item $B_{p,q}^s(\R^n)$ is a retract of $l^q_s(L^p(\R^n))$ via the maps
\[f \mapsto ((\Four^{-1} \phi_k) * f)_k, \qquad (g_k)_k \mapsto \sum_{k \geq 0} (\Four^{-1} \tilde\phi_k)  * g_k,\]
and in particular its completeness as a normed space follows from the completeness of $l^q_s(L^p(\R^n))$.
\end{enumerate}
Moreover, from Theorems~\ref{thm:lebesgueinterpolation} and \ref{thm:sequenceinterpolation}, we obtain immediately the following fundamental interpolation properties (see also \cite{bergh_interpolation_1976}, Theorem~6.4.5):

\begin{prp}\label{prp:besovinterpolation}
Let $1 \leq p,p_0,p_1,q,q_0,q_1 \leq \infty$, $s_0,s_1 \in \R$, $0 < \theta < 1$.
\begin{itemize}
\item[(i)] If $s_0 \neq s_1$, then
\[(B_{p,q_0}^{s_0}(\R^n),B_{p,q_1}^{s_1}(\R^n))_{\theta,q} = B_{p,q}^{(1-\theta)s_0 + \theta s_1}(\R^n).\]
This identity holds also if $s_0 = s_1$, provided that $q_0,q_1 < \infty$ and
\[\frac{1}{q} = \frac{1-\theta}{q_0} + \frac{\theta}{q_1}.\]
\item[(ii)] If $q_0,q_1 < \infty$ and
\[\frac{1-\theta}{p_0} + \frac{\theta}{p_1} = \frac{1}{p} = \frac{1-\theta}{q_0} + \frac{\theta}{q_1},\]
then
\[(B_{p_0,q_0}^{s_0}(\R^n),B_{p_1,q_1}^{s_1}(\R^n))_{\theta,p} = B_{p,p}^{(1-\theta)s_0 + \theta s_1}(\R^n).\]
\item[(iii)] If $q_0 < \infty$ and
\[\frac{1}{p} = \frac{1-\theta}{p_0} + \frac{\theta}{p_1}, \qquad  \frac{1}{q} = \frac{1-\theta}{q_0} + \frac{\theta}{q_1},\]
then
\[(B_{p_0,q_0}^{s_0}(\R^n),B_{p_1,q_1}^{s_1}(\R^n))_{[\theta]} = B_{p,q}^{(1-\theta)s_0 + \theta s_1}(\R^n).\]
\end{itemize}
\end{prp}

There are several inclusions between Besov spaces. In order to give a more detailed formulation, we introduce now the little Besov space $b_{p,q}^s(\R^n)$, which is the (closed) subspace of $B_{p,q}^s(\R^n)$ made of the elements $f$ such that
\[\lim_{m \to +\infty} 2^{ms} \|(\Four^{-1} \phi_m) * f\|_p = 0.\]
Clearly $b_{p,q}^s(\R^n) = B_{p,q}^s(\R^n)$ for $q < \infty$, whereas, for $q = \infty$, the difference between Besov and little Besov corresponds to the one between $l^\infty$ and $c_0$.

\begin{prp}\label{prp:besovembeddings}
Let $1 \leq p,p_1,p_2,q,q_1,q_2 \leq \infty$, $s_1,s_2 \in \R$.
\begin{itemize}
\item[(i)] If $s_1 > s_2$, or otherwise if $s_1 = s_2$ and $q_1 > q_2$, then we have the continuous inclusion
\[B_{p,q_1}^{s_1}(\R^n) \subseteq b_{p,q_2}^{s_2}(\R^n).\]
\item[(ii)] If $p_1 \leq p_2$ and $s_1 - n/p_1 \geq s_2 - n/p_2$, then we have the continuous inclusions
\[B_{p_1,q}^{s_1}(\R^n) \subseteq B_{p_2,q}^{s_2}(\R^n), \qquad b_{p_1,q}^{s_1}(\R^n) \subseteq b_{p_2,q}^{s_2}(\R^n).\]
\end{itemize}
\end{prp}
\begin{proof}
It is an immediate consequence of inclusions between weighted $l^p$ and $c_0$ spaces, Young's inequality and the estimate
\[\|\Four^{-1} \tilde \phi_k\|_p \leq C_p 2^{kn/p'}\]
for $k \geq 0$, $1 \leq p \leq \infty$ (cf.\ \cite{bergh_interpolation_1976}, Theorems~6.2.4 and 6.5.1).
\end{proof}
In particular, $B_{p,q}^s(\R^n)$ and $b_{p,q}^s(\R^n)$ increase in $q$ and decrease in $s$.

\subsection{Fourier multipliers, lifting properties, comparison with other spaces}
Up to now, the properties which we have considered are almost entirely ``internal'' to the scale of Besov spaces, so that they do not show particular connections between Besov and other known spaces; moreover, it is not clear how Besov spaces allow to measure smoothness. These questions will be faced by the use of Fourier multiplier results for Besov spaces.

In order to state these results, we introduce, for $k \in \N \cup \{\infty\}$ and $\sigma \in \R$, the ``class of symbols'' $S^k_\sigma$:
\[S^k_\sigma = \{ f \in C^k(\R^n) \tc \sup_{x \in \R^n} \langle x \rangle^{|\alpha| - \sigma} |\partial^\alpha f(x)| < \infty \text{ for $|\alpha| \leq k$}\}.\]
Clearly $S^k_\sigma$ is a vector subspace of $C^k(\R^n)$ containing $\Sz(\R^n)$, and moreover it is increasing in $\sigma$ and decreasing in $k$; apart from this, the main properties of these classes of symbols are summarized in the following lemma, which is easily proved by induction.

\begin{lem}\label{lem:symbols}
Let $k,k_1,k_2 \in \N$, $\sigma,\sigma_1,\sigma_2 \in \R$, $\alpha \in \N^n$.
\begin{itemize}
\item[(i)] If $f \in S^k_\sigma$, then $\partial^\alpha f \in S^{k-|\alpha|}_{\sigma-|\alpha|}$.
\item[(ii)] If $f_1 \in S^{k_1}_{\sigma_1}$ and $f_2 \in S^{k_2}_{\sigma_2}$, then $f_1 f_2 \in S^{\min\{k_1,k_2\}}_{\sigma_1 + \sigma_2}$.
\item[(iii)] The function $\langle \cdot \rangle^\sigma$ belongs to $S^\infty_\sigma$; moreover, if $p$ is a polynomial of degree not greater than $\sigma$, then $p \in S^\infty_\sigma$.
\item[(iv)] If $f \in S^k_\sigma$ and $|f| \geq c \langle \cdot \rangle^\sigma$ for some $c > 0$, then $1/f \in S^k_{-\sigma}$.
\end{itemize}
\end{lem}

Here is the announced Fourier multiplier result for Besov spaces:
\begin{prp}\label{prp:besovfouriermultiplier}
If $m \in S^k_\sigma$ for some $k \in \N$ with $k > n/2$, then the operator
\[f \mapsto (\Four^{-1} m) * f\]
is continuous $B_{p,q}^s(\R^n) \to B_{p,q}^{s-\sigma}(\R^n)$ and $b_{p,q}^s(\R^n) \to b_{p,q}^{s-\sigma}(\R^n)$.
\end{prp}
\begin{proof}[Proof]
By Young's inequality, it will be sufficient to prove that
\[\|\Four^{-1} (m \tilde\phi_j) \|_1 \leq C 2^{j\sigma}.\]
This can be obtained, for $j > 1$, thanks to Proposition~\ref{prp:sobolevembedding}(i), from the condition $m \in S^k_\sigma$, by estimating the $C^k$ norm of
\[\xi \mapsto m(2^{j-2} \xi) \tilde\phi_2(\xi),\]
whose support is contained in a fixed annulus of $\R^n$.
\end{proof}

From this, we obtain the fundamental \emph{lifting properties} of Besov spaces:

\begin{prp}\label{prp:besovlifting}
Let $1 \leq p,q \leq \infty$, $s \in \R$.
\begin{itemize}
\item[(i)] For every $\sigma \in \R$, the \emph{Bessel potential} operator
\[J_\sigma : f \mapsto \Four^{-1} ( \langle \cdot \rangle^\sigma (\Four f))\]
is an isomorphism $B_{p,q}^s(\R^n) \to B_{p,q}^{s-\sigma}(\R^n)$, with inverse $J_{-\sigma}$.
\item[(ii)] For every $k \in \N$, the three conditions
\begin{itemize}
\item[(a)] $f \in B_{p,q}^s(\R^n)$,
\item[(b)] $\partial^\alpha f \in B_{p,q}^{s-k}(\R^n)$ for $|\alpha| \leq k$,
\item[(c)] $f,\partial_1^k f, \dots, \partial_n^k f \in B_{p,q}^{s-k}(\R^n)$
\end{itemize}
are equivalent; moreover, the quantities
\[\sum_{|\alpha| \leq k} \|\partial^\alpha f\|_{B_{p,q}^{s-k}} \qquad\text{and}\qquad \|f\|_{B_{p,q}^{s-k}} + \sum_{j=1}^n \|\partial_j^k f\|_{B_{p,q}^{s-k}}\]
are equivalent norms of $f \in B_{p,q}^s(\R^n)$.
\end{itemize}
The same results hold also for little Besov spaces.
\end{prp}
\begin{proof}
(i) The conclusion follows immediately from Proposition~\ref{prp:besovfouriermultiplier} and the fact that the function $\langle \cdot \rangle^\sigma$ belongs to $S^\infty_\sigma$.

(ii) Since the polynomials $\xi^\alpha$, for $|\alpha| \leq k$, belong to $S^\infty_k$, (b) follows from (a) and Proposition~\ref{prp:besovfouriermultiplier}, whereas (b) trivially implies (c). In order to obtain (a) from (c), we claim that there exist functions $\rho_j \in S^\infty_0$ such that
\[\eta(\xi) = 1 + \sum_{j=1}^n \rho_j(\xi) \xi_j^k \geq c \langle \xi \rangle^k\]
for some $c > 0$; in fact, this inequality gives $1/\eta \in S^\infty_{-k}$ by Lemma~\ref{lem:symbols}, and since, by (c) and Proposition~\ref{prp:besovfouriermultiplier}, we have $(\Four^{-1} \eta) * f \in B_{p,q}^{s-k}(\R^n)$, by Proposition~\ref{prp:besovfouriermultiplier} again we get $f \in B_{p,q}^{s}(\R^n)$. From this argument we also obtain the equivalence of the norms.

If $k$ is even, we can simply take $\rho_j \equiv 1$. If instead $k$ is odd, we put
\[\rho_j(\xi) = \beta(\xi_j/|\xi|) \gamma(|\xi|),\]
where $\beta,\gamma \in \E(\R)$ are weakly increasing smooth bounded functions, with $\beta$ odd and such that $\beta(t) = 0$ for $|t| \leq (2n)^{-1/2}$, $\beta(t) = 1$ for $t \geq n^{-1/2}$, whereas $\gamma(t) = 0$ for $t \leq 1/2$, $\gamma(t) = 1$ for $t \geq 1$. It is not difficult to show that these $\rho_j$ are in $S^\infty_0$ and satisfy the required inequality.
\end{proof}

The lifting properties show that in fact the index $s$ in $B_{p,q}^s(\R^n)$ expresses a level of differentiability. In order to complete the picture, we need a better understanding of the Besov spaces at level $s = 0$, and this will be achieved by showing some inclusions involving other known spaces.

\begin{prp}\label{prp:besovzeroinclusions}
The following are continuous inclusions:
\begin{itemize}
\item[(i)] $B_{1,1}^0(\R^n) \subseteq L^1(\R^n) \subseteq M(\R^n) \subseteq B_{1,\infty}^0(\R^n)$,
\item[(ii)] $B_{p,1}^0(\R^n) \subseteq L^p(\R^n) \subseteq b_{p,\infty}^0(\R^n)$ for $1 \leq p < \infty$,
\item[(iii)] $B_{\infty,1}^0(\R^n) \subseteq C_{ub}(\R^n) \subseteq b_{\infty,\infty}^0(\R^n)$,
\item[(iv)] $B_{\infty,1}^0(\R^n) \subseteq C_{ub}(\R^n) \subseteq L^\infty(\R^n) \subseteq B_{\infty,\infty}^0(\R^n)$.
\end{itemize}
Moreover, we have $B_{2,2}^0(\R^n) = L^2(\R^n)$.
\end{prp}
\begin{proof}
The left-hand-side inclusions are due simply to \eqref{eq:besovdecomposition}, to the completeness of $L^p(\R^n)$ and $C_{ub}(\R^n)$, and to the fact that $\Sz * L^\infty \subseteq C_{ub}$. The right-hand-side inclusions follow instead from Young's inequalities, from \eqref{eq:fl1boundedness} and, for (ii) and (iii), by Proposition~\ref{prp:approximateidentity}, since $\int_{\R^n} (\Four^{-1} \phi_1)(x) \,dx = 0$.

Finally, it is not difficult to see that
\[\|f\|_{B_{2,2}^0}^2 = \sum_{m \geq 0} \|(\Four^{-1} \phi_m) * f\|_2^2 \sim \|f\|_2^2,\]
which gives the identity $B_{2,2}^0 = L^2$.
\end{proof}

Notice that, by putting together these inclusions with those of Proposition~\ref{prp:besovembeddings}, we immediately obtain that, for $s > 0$, the elements of $B_{p,q}^s(\R^n)$ are $L^p$ functions (and not simply distributions); moreover, for $s > n/p$, they are also uniformly continuous and bounded.

By ``lifting'' the inclusions of Proposition~\ref{prp:besovzeroinclusions}, we then obtain further evidence for the interpretation of the index $s$ in $B_{p,q}^s$ as an ``order of differentiability''.

\begin{cor}\label{cor:besovinclusions}
For every $k \in \N$, the following are continuous inclusions:
\begin{itemize}
\item[(i)] $B_{p,1}^k(\R^n) \subseteq W^{k,p}(\R^n) \subseteq b_{p,\infty}^k(\R^n)$ for $1 \leq p < \infty$,
\item[(ii)] $B_{\infty,1}^k(\R^n) \subseteq C_{ub}^k(\R^n) \subseteq b_{\infty,\infty}^k(\R^n)$,
\item[(iii)] $B_{\infty,1}^k(\R^n) \subseteq W^{k,\infty}(\R^n) \subseteq B_{\infty,\infty}^k(\R^n)$.
\end{itemize}
Moreover, we have $B_{2,2}^s(\R^n) = H^s(\R^n)$ for all $s \in \R$.
\end{cor}

These inclusions, together with the results of Proposition~\ref{prp:besovinterpolation}, lead to the characterization of Besov spaces as real interpolation spaces between Sobolev spaces $W^{k,p}$ (or spaces $C^k_b$ of bounded continuously differentiable functions) of different orders:
\begin{cor}
For $1 \leq p,q \leq \infty$, $m \in \N \setminus \{0\}$, $0 < \theta < 1$, we have
\[(L^p(\R^n),W^{m,p}(\R^n))_{\theta,q} = B_{p,q}^{\theta m}(\R^n);\]
for $p = \infty$ we have moreover
\[(C_b(\R^n),C_b^m(\R^n))_{\theta,q} = (C_{ub}(\R^n),C_{ub}^m(\R^n))_{\theta,q} = B_{\infty,q}^{\theta m}(\R^n).\]
\end{cor}

\subsection{Approximation by smooth functions and little Besov spaces}
By the use of a suitable approximate identity (see \S\ref{subsection:approximateidentities}), one can hope to approximate the elements of $B_{p,q}^s$ by smooth functions in the corresponding Besov norm. Unfortunately, some problems arise in the case $q = \infty$, which force us to restrict to little Besov spaces.

\begin{prp}\label{prp:besovapproximation}
Let $1 \leq p,q \leq \infty$, $s \in \R$, $f \in B_{p,q}^s(\R^n)$. If $u \in \Sz(\R^n)$, then $f * u \in W^{\infty,p}(\R^n)$. Moreover, if $f \in b_{p,q}^s(\R^n)$, and if $u_h \in \Sz(\R^n)$ is an approximate identity for $h \to \infty$, then $f * u_h \to f$ in $B_{p,q}^s(\R^n)$.
\end{prp}
\begin{proof}
Since a Schwartz function belongs to the symbol class $S^\infty_\sigma$ for every $\sigma \in \R$, the first part of the conclusion follows immediately from Propositions~\ref{prp:besovfouriermultiplier}, \ref{prp:besovembeddings} and Corollary~\ref{cor:besovinclusions}.

Suppose now that $u_h \in \Sz(\R^n)$ is an approximate identity. We then have, by Proposition~\ref{prp:approximateidentity},
\[\|(\Four^{-1}\phi_k) * (f * u_h - f)\|_p = \|((\Four^{-1}\phi_k) * f) * u_h - (\Four^{-1}\phi_k) * f\|_p \to 0,\]
since $(\Four^{-1}\phi_k) * f \in L^p(\R^n)$ for $p < \infty$, and $(\Four^{-1} \phi_k) * f \in C_{ub}(\R^n)$ for $p = \infty$. The conclusion then follows by dominated convergence (which can be applied also in the case $q = \infty$ due to the restriction to little Besov spaces).
\end{proof}

The previous result in fact clarifies the role of little Besov spaces, and moreover leads to an alternative characterization of them:

\begin{cor}
Let $1 \leq p,q,q_1 \leq \infty$, $s,s_1 \in \R$. Then
\[b_{p,q}^s(\R^n) = \overline{B_{p,q_1}^{s_1}(\R^n)}^{B_{p,q}^s(\R^n)} \qquad\text{for $s_1 > s$,}\]
and
\[b_{p,q}^s(\R^n) = \overline{B_{p,q_1}^s(\R^n)}^{B_{p,q}^s(\R^n)} \qquad\text{for $q_1 < q$.}\]
\end{cor}

\subsection{Pointwise multiplication}
In the previous paragraphs, we have considered the convolution product in Besov spaces. Now we are interested in results about pointwise multiplication (which are a particular case of the ones presented in Chapter~4 of \cite{runst_sobolev_1996}).

First of all, one should notice that the pointwise product is not defined in general between tempered distributions. However, if $f,g \in \Sz'(\R^n)$, and one decomposes
\[f = \sum_{h \geq 0} f_h, \qquad g = \sum_{k \geq 0} g_k,\]
where
\begin{equation}\label{eq:factordecomposition}
f_h = (\Four^{-1} \phi_h) * f, \qquad g_k = (\Four^{-1} \phi_k) * g,
\end{equation}
then $f_h, g_k$ are smooth (in fact, analytic) functions which have polynomial growth together with all their derivatives (since their Fourier transforms are compactly supported). The \emph{paraproduct} of $f$ and $g$ is then defined as the limit
\[f g = \lim_{N \to +\infty} \left(\sum_{h=0}^N f_h\right) \left(\sum_{k=0}^N g_k\right) \qquad\text{in $\Sz'(\R^n)$,}\]
whenever it exists.

Notice that, when $f \in \Sz'(\R^n)$ and $g \in \Sz(\R^n)$, then the product $fg$ is already defined. However, since the map
\[\Sz'(\R^n) \times \Sz(\R^n) \ni (f,g) \mapsto fg \in \Sz'(\R^n)\]
is jointly sequentially continuous by the uniform boundedness principle, it is immediate to check that the definition of the paraproduct in fact extends the ``classical'' product.

An important tool for proving inequalities involving paraproducts is the \emph{Fatou property} of Besov spaces (cf.\ \cite{franke_spaces_1986}):

\begin{lem}\label{lem:besovfatou}
Suppose that $(f^{(r)})_r$ is a sequence in $B_{p,q}^s(\R^n)$, converging in $\Sz'(\R^n)$ to a tempered distribution $f$. If
\[\liminf_{r \to +\infty} \|f^{(r)}\|_{B_{p,q}^s} < +\infty,\]
then $f \in B_{p,q}^s(\R^n)$, and
\[\|f\|_{B_{p,q}^s} \leq \liminf_{r \to +\infty} \|f^{(r)}\|_{B_{p,q}^s}.\]
\end{lem}
\begin{proof}
Let $f^{(r)}_h = (\Four^{-1} \phi_h) * f^{(r)}$, $f_h = (\Four^{-1} \phi_h) * f$. Since $f^{(r)} \to f$ in $\Sz'(\R^n)$ and $\Four^{-1} \phi_h \in \Sz(\R^n)$, we have that $f^{(r)}_h \to f_h$ in $\E(\R^n)$, and in particular $|f_h^{(r)}|$ converges to $|f_h|$ uniformly on compacta. But then, by Fatou's lemma,
\[\|f_h\|_p \leq \liminf_{r \to +\infty} \|f^{(r)}_h\|_p,\]
so that, again by Fatou's lemma,
\[\left\|(\|f_h\|_p)_h \right\|_{l^q_s} \leq \liminf_{r \to +\infty} \left\| (\|f^{(r)}_h\|_p)_h \right\|_{l^q_s},\]
which is the conclusion.
\end{proof}

We then have the following

\begin{lem}\label{lem:besovproductestimate}
Let $f,g \in \Sz'(\R^n)$, and suppose that, for some $p,q \in [1,\infty]$ and $s \in \R$,
\begin{equation}\label{eq:serieprodottifinita}
\left\| \left(\sum_{h \geq 0} \sum_{k \geq 0} \|(\Four^{-1} \phi_m) * (f_h g_k)\|_p \right)_m\right\|_{l^q_s} < \infty,
\end{equation}
where $f_h,g_k$ are defined as in \eqref{eq:factordecomposition}. Then the limit
\[fg = \lim_{N \to +\infty} \left(\sum_{h=0}^N f_h\right) \left(\sum_{k=0}^N g_k\right)\]
exists in $\Sz'(\R^n)$ (and also in $B_{p,q}^s(\R^n)$ if $q < \infty$), and belongs to $B_{p,q}^s(\R^n)$. Moreover
\[\|fg\|_{B_{p,q}^s} \leq \left\| \left(\sum_{h \geq 0} \sum_{k \geq 0} \|(\Four^{-1} \phi_m) * (f_h g_k)\|_p \right)_m\right\|_{l^q_s}.\]
\end{lem}
\begin{proof}
Let
\[P_N = \left(\sum_{h=0}^N f_h\right) \left(\sum_{k=0}^N g_k\right).\]
Then, for $N' \geq N$,
\[\|(\Four^{-1} \phi_m) * (P_{N'} - P_N)\|_p \leq \sum_{\max\{h,k\} > N} \|(\Four^{-1} \phi_m) * (f_h g_k)\|_p.\]
By \eqref{eq:serieprodottifinita}, the right-hand side is infinitesimal for $N \to +\infty$, and in fact, if $q < \infty$, by dominated convergence in $l^q_s$ we get that $P_N$ is a Cauchy sequence in $B_{p,q}^s(\R^n)$ and therefore converges (to the paraproduct $fg$). If $q = \infty$, by the previous argument one gets convergence of $P_N$ in every $B_{p,\tilde q}^{\tilde s}(\R^n)$ with $\tilde s < s$, and in particular in $\Sz'(\R^n)$. Finally, the bound on the norm of the paraproduct follows by Lemma~\ref{lem:besovfatou}.
\end{proof}

The result which we present here is a H\"older-type inequality, in which one takes into account also the order of differentiability. The following inequalities giving upper bounds on a norm of a paraproduct are to be meant in the following sense: if the right-hand side of the inequality is finite, then the paraproduct exists and its norm satisfies the inequality.

\begin{prp}\label{prp:besovproduct}
Let $1 \leq p, p_1, p_2, q \leq \infty$, $s \in \R$, with
\[\frac{1}{p} = \frac{1}{p_1} + \frac{1}{p_2}.\]
\begin{itemize}
\item[(i)] For every $\sigma > |s|$, we have
\[\| f g \|_{B^s_{p,q}} \leq C_{p_1,p_2,q,\sigma,s} \|f\|_{B^\sigma_{p_1,\infty}} \|g\|_{B^s_{p_2,q}}.\]
\item[(ii)] If $s > 0$, then we have also the more precise inequality
\[\| f g \|_{B^s_{p,q}} \leq C_{p_1,p_2,q,s} \|f\|_{B^s_{p_1,q}} \|g\|_{B^s_{p_2,q}}.\]
\end{itemize}
\end{prp}
\begin{proof}
Define $f_h,g_k$ as in \eqref{eq:factordecomposition}, and notice that
\[\supp \Four(f_h g_k) \subseteq \supp \phi_h + \supp \phi_k.\]
Therefore, if
\begin{multline*}
I_m = \{(h,k) \in \N^2 \tc m-2 \leq \max\{h,k\} \leq m+2 \text{ and } |h-k| \geq 3\} \\
\cup \{(h,k) \in \N^2 \tc m-2 \leq \max\{h,k\} \text{ and } |h-k| \leq 2\},
\end{multline*}
then by \eqref{eq:besovsupports} it is easy to see that $(\Four^{-1} \phi_m) * (f_h g_k) = 0$ for $(h,k) \notin I_m$. Consequently, by Lemma~\ref{lem:besovproductestimate}, H\"older's inequality and \eqref{eq:fl1boundedness}, we have
\[\|fg\|_{B_{p,q}^s} \leq \left\| \left(\sum_{(h,k) \in I_m} \|f_h\|_{p_1} \|g_k\|_{p_2} \right)_m\right\|_{l^q_s}.\]

(i) The conclusion will follow from the inequality
\[ \left\| \left(\sum_{(h,k) \in I_m} a_h b_k\right)_m \right\|_{l^q_s} \leq C_{q,\sigma,s} \| (a_h)_h \|_{l^\infty_\sigma} \| (b_k)_k \|_{l^q_s}\]
for sequences $(a_h)_{h \in \N}$, $(b_k)_{k \in \N}$ of non-negative real numbers.

In fact
\[I_m \subseteq I_{m,1} \cup I_{m,2} \cup I_{m,3},\]
where
\begin{align*}
I_{m,1} &= \{(h,k) \tc h \geq m-2 \text{ and } 0 \leq k \leq h\},\\
I_{m,2} &= \{(h,k) \tc m-2 \leq k \leq m+2 \text{ and } 0 \leq h \leq k-3\},\\
I_{m,3} &= \{(h,k) \tc k \geq m-2 \text{ and } k-2 \leq h \leq k-1\}.
\end{align*}
Since $\sigma > |s|$, we may choose $\tilde\sigma$ such that
\[\max\{0,-s\} < \tilde\sigma < \min\{\sigma,\sigma-s\},\]
and we have
\[\begin{split}
\sum_{(h,k) \in I_{m,1}} a_h b_k &\leq \sum_{h \geq (m-2)_+} 2^{-h(\sigma-\tilde\sigma)} 2^{h\sigma} a_h \sum_{k = 0}^h 2^{-k\tilde\sigma} b_k \\
&\leq C_{\sigma,s} 2^{-m(\sigma-\tilde\sigma)} \|(a_h)_h\|_{l^\infty_\sigma} \|(b_k)_k\|_{l^1_{-\tilde\sigma}},
\end{split}\]
so that
\begin{multline*}
\left\|\left(\sum_{(h,k) \in I_{m,1}} a_h b_k\right)_m\right\|_{l^q_s} \leq C_{\sigma,s} \|(a_h)_h\|_{l^\infty_\sigma} \|(b_k)_k\|_{l^1_{-\tilde\sigma}} \left\| (2^{-m(\sigma-\tilde\sigma-s)})_m \right\|_{l^q} \\
\leq C_{q,\sigma,s} \|(a_h)_h\|_{l^\infty_\sigma} \|(b_k)_k\|_{l^q_{s}}.
\end{multline*}
On the other hand,
\[\sum_{(h,k) \in I_{m,2}} a_h b_k \leq \|(a_h)_h\|_{l^1} \sum_{k=-2}^2 b_{m+k}\]
(where $b_{-2} = b_{-1} = 0$), thus
\begin{multline*}
\left\|\left(\sum_{(h,k) \in I_{m,2}} a_h b_k\right)_m\right\|_{l^q_s} \leq \|(a_h)_h\|_{l^1} \sum_{k=-2}^2 \|(b_{m+k})_m\|_{l^q_s} \\
\leq C_{q,\sigma,s} \|(a_h)_h\|_{l^\infty_\sigma} \|(b_k)_k\|_{l^q_s}.
\end{multline*}
Finally, by H\"older's inequality and since $\sigma + s > 0$,
\[\begin{split}
\sum_{(h,k) \in I_{m,3}} a_h b_k &\leq 4 \sum_{k \geq (m-2)_+} 2^{-k(\sigma+s)} 2^{ks} b_k \sum_{h=k-2}^{k-1} 2^{h\sigma} a_h \\
&\leq C_{q,\sigma,s} 2^{-m(\sigma+s)} \|(a_h)_h\|_{l^\infty_\sigma} \|(b_k)\|_{l^q_s},
\end{split}\]
so that again, since $\sigma > 0$,
\[\left\|\left(\sum_{(h,k) \in I_{m,3}} a_h b_k\right)_m\right\|_{l^q_s} \leq C_{q,\sigma,s} \|(a_h)_h\|_{l^\infty_\sigma} \|(b_k)\|_{l^q_s}.\]
By putting all together, we obtain the claimed inequality.

(ii) The conclusion will follow from the inequality
\[ \left\| \left(\sum_{(h,k) \in I_m} a_h b_k\right)_m \right\|_{l^q_s} \leq C_{q,s} \| (a_h)_h \|_{l^q_s} \| (b_k)_k \|_{l^q_s}\]
for sequences $(a_h)_{h \in \N}$, $(b_k)_{k \in \N}$ of non-negative real numbers.
Set
\[I_{m,+} = I_m \cap \{(h,k) \tc h > k\}, \qquad I_{m,-} = I_m \cap \{(h,k) \tc h \leq k\}.\]
Since $s > 0$, we may choose $t$ such that $0 < t < s$, therefore
\[\begin{split}
\sum_{(h,k) \in I_{m,+}} a_h b_k &\leq \sum_{h \geq \max\{m-2,1\}} 2^{-ht} 2^{ht} a_h \sum_{k = 0}^{h-1} b_k \\
&\leq C_{q,s} 2^{-mt} \|(b_k)_k\|_{l^1_0} \left(\sum_{h \geq \max\{m-2,1\}} (2^{ht} a_h)^q \right)^{1/q},
\end{split}\]
so that, for $q < \infty$,
\begin{multline*}
\sum_{m \geq 0} \left(2^{ms} \sum_{(h,k) \in I_{m,+}} a_h b_k\right)^{q} \\
\leq  C_{q,s} \|(b_k)_k\|_{l^1_0}^q \sum_{m \geq 0} 2^{qm(s-t)}  \sum_{h \geq \max\{m-2,1\}} (2^{ht} a_h)^q \\
\leq C_{q,s} \|(b_k)_k\|_{l^1_0}^q \sum_{h \geq 1} 2^{hsq} a_h^q,
\end{multline*}
which implies
\[ \left\| \left(\sum_{(h,k) \in I_{m,+}} a_h b_k\right)_m \right\|_{l^q_s} \leq C_{q,s} \| (a_h)_h \|_{l^q_s} \| (b_k)_k \|_{l^q_s}.\]
The case $q = \infty$ and the sum for $(h,k) \in I_{m,-}$ are handled analogously.
\end{proof}

Proposition~\ref{prp:besovproduct} has as a consequence a sort of ``continuous inclusion'' between Besov spaces, different from the previously obtained ones, which holds under the hypothesis of compact support.

\begin{cor}\label{cor:besovcompactsupportinclusion}
For every compact $K \subseteq \R^n$, $s \in \R$, $1 \leq p_1,p_2,q \leq \infty$ with $p_1 \leq p_2$, we have
\[\|f\|_{B_{p_1,q}^s} \leq C_{K,p_1,p_2,q,s} \|f\|_{B_{p_2,q}^s} \qquad\text{if $\supp f \subseteq K$.}\]
\end{cor}
\begin{proof}
Since $p_1 \leq p_2$, we can find $p_3 \in [1,\infty]$ such that $1/p_1 = 1/p_2 + 1/p_3$. Therefore, it is sufficient to choose a function $\eta \in \D(\R^n)$ such that $\eta|_K \equiv 1$, and to apply Proposition~\ref{prp:besovproduct} to the product $\eta f$, noticing that it coincides with $f$ when $\supp f \subseteq K$.
\end{proof}

The aforementioned threshold $n/p$ on the order $s$ to ensure continuity of the elements of $B_{p,q}^s$ can be shown to be optimal.

\begin{prp}\label{prp:besovoptimalembedding}
Let $1 \leq p,q \leq \infty$, $s \in \R$. Suppose that there exist a constant $C > 0$, a point $x \in \R^n$ and an open neighborhood $U$ of $x$ such that, for every $f \in \D(\R^n)$ with $\supp f \subseteq U$, we have
\[|f(x)| \leq C \|f\|_{B_{p,q}^s}.\]
Then either $s \geq n/p$ and $q = 1$, or $s > n/p$, and in particular $B_{p,q}^s(\R^n)$ is continuously embedded in $C_{ub}(\R^n)$.
\end{prp}
\begin{proof}
By choosing $\eta \in \D(\R^n)$ with $\supp \eta \subseteq U$ and $\eta(x) = 1$, we obtain that, for all $f \in \Sz(\R^n)$,
\[|f(x)| = |(\eta f)(x)| \leq C \|\eta f\|_{B_{p,q}^s} \leq C_{\eta,p,q,s} \|f\|_{B_{p,q}^s},\]
by Proposition~\ref{prp:besovproduct}. Since the Besov norms are invariant by translations, it follows easily that, by possibly enlarging the constant $C$,
\[\|f\|_\infty \leq C \|f\|_{B_{p,q}^s}\]
for all $f \in \Sz(\R^n)$. This proves in particular that the closure of $\Sz(\R^n)$ in $B_{p,q}^s(\R^n)$ is contained in $L^\infty(\R^n)$. The conclusion then follows by Theorem~1 in \S2.6.2 of \cite{triebel_spaces_1978}.
\end{proof}

\subsection{Change of variables}

In the following, by the use of elementary Sobolev embeddings and interpolation, we will obtain change-of-variable properties for Besov spaces of positive order (see also \cite{triebel_theory_1983}, \S2.10).

\begin{lem}\label{lem:elementarychange}
Let $\Omega \subseteq \R^{n_1}$ be open, $\Phi : \Omega \to \R^{n_2}$ be smooth, $s \in \N$, $\psi \in \D(\Omega)$.
\begin{itemize}
\item[(i)] For every $f \in C^s_b(\R^{n_2})$,
\[\| (f \circ \Phi) \psi \|_{C^s_b(\R^{n_1})} \leq C_{\Phi,\psi,s} \|f\|_{C^s_b(\R^{n_2})}.\]
\item[(ii)] Suppose that the differential of $\Phi$ has constant rank $r$. Then, for every $f \in W^{s+n_2-r,1}(\R^{n_2})$,
\[\| (f \circ \Phi) \psi \|_{W^{s,1}(\R^{n_1})} \leq C_{\Phi,\psi,s} \|f\|_{W^{s+n_2-r,1}(\R^{n_2})}.\]
\end{itemize}
\end{lem}
\begin{proof}
Let $K = \supp \psi$. Notice that, if $f$ is modified outside the compact $\Psi(K)$, then $(f \circ \Phi) \psi$ does not change. Therefore, by density arguments (see Proposition~\ref{prp:approximateidentity}), we may suppose that $f \in \D(\R^{n_2})$.

Let $\Phi = (\Phi_1,\dots,\Phi_{n_2})$ and, for $l=1,2$, let $\partial_{l,1},\dots,\partial_{l,n_l}$ be a basis of (invariant vector fields) of $\R^{n_l}$. Then we have, for $j=1,\dots,n_1$,
\[\partial_{1,j}(f \circ \Phi) = \sum_{k=1}^{n_2} ((\partial_{2,k} f) \circ \Phi) \cdot \partial_{1,j} \Phi_k,\]
so that, inductively, for every $\alpha \in \N^{n_1}$,
\[\partial_1^\alpha((f \circ \Phi) \psi) = \sum_{\beta \tc |\beta| \leq |\alpha|} ((\partial_2^\beta f) \circ \Phi) \cdot \Psi_{\alpha,\beta},\]
where the $\Psi_{\alpha,\beta}$ are linear combinations of products of derivatives of the functions $\Phi_1,\dots,\Phi_n,\psi$, with supports contained in $K = \supp \psi$. Since $K$ is compact, we then get
\[\|\partial_1^\alpha ((f \circ \Phi) \psi) \|_\infty \leq C_{\Phi,\psi,\alpha} \sum_{\beta \tc |\beta| \leq |\alpha|} \|\partial_2^\beta f\|_\infty,\]
which gives immediately (i), and also
\[\|\partial_1^\alpha ((f \circ \Phi) \psi) \|_1 \leq C_{\Phi,\psi,\alpha} \sum_{\beta \tc |\beta| \leq |\alpha|} \|((\partial_2^\beta f) \circ \Phi) \chr_K\|_1.\]
Notice that, in the case $\Phi$ is a diffeomorphism with its image, (ii) follows immediately from the last inequality by a change of variable in the integrals on the right-hand side, since $\det d\Phi$ is bounded on the compact $K$.

Suppose now that the differential of $\Phi$ has constant rank $r \in \N$. Then, by the rank theorem (see \cite{lee_smooth_2003}, Theorem~7.8), for all $x \in \Omega$ we can find a coordinate chart $\phi_x : U_x \to \R^{n_1}$ of $\Omega$ centered at $x$ and a coordinate chart $\zeta_x : V_x \to \R^{n_2}$ of $\R^{n_2}$ centered at $\Phi(x)$ such that $\Phi(U_x) \subseteq V_x$ and
\[\zeta_x \circ \Phi \circ \phi_x^{-1}(t_1,\dots,t_{n_1}) = (t_1,\dots,t_r,0,\dots,0)\]
for all $(t_1,\dots,t_{n_1}) \in \phi_x(U_x)$. If $a_x > 0$ is such that $[-a_x,a_x]^{n_1} \subseteq \phi_x(U_x)$, then, by compactness of $K$, we can find a finite subset $\{x_1,\dots,x_k\}$ of $\Omega$ such that
\[K \subseteq \phi_{x_1}^{-1}(\left]-a_{x_1},a_{x_1}\right[^{n_1}) \cup \dots \cup \phi_{x_k}^{-1}(\left]-a_{x_k},a_{x_k}\right[^{n_1}).\]
Choose moreover $\eta_j \in \D(\zeta_{x_j}(V))$ such that $\eta_j|_{[-a_{x_j},a_{x_j}]^r \times \{0\}} \equiv 1$. 

For every $g \in \D(\R^{n_2})$, if $g_j = (g \circ \zeta_{x_j}^{-1})\eta_j$, then, by Proposition~\ref{prp:sobolevembedding}(iii),
\[|g_j(y',0)| \leq \|g_j(y',\cdot)\|_{W^{n_2-r,1}(\R^{n_2-r})}\]
for $j=1,\dots,k$ and $y' \in \R^{r}$, thus
\[\begin{split}
\|(g \circ \Phi) \chr_K\|_1 &\leq \sum_{j=1}^k \int_{\phi_{x_j}^{-1}([-a_{x_j},a_{x_j}]^{n_1})} |g(\Phi(x))| \,dx \\
&\leq C_{\Phi,\psi}  \sum_{j=1}^k \int_{[-a_{x_j},a_{x_j}]^{n_1}} |g(\zeta_{x_j}^{-1}(t_1,\dots,t_r,0,\dots,0))| \,dt_1 \dots dt_{n_1} \\
&= C_{\Phi,\psi}  \sum_{j=1}^k (2a_{x_j})^{n_1 - r} \int_{[-a_{x_j},a_{x_j}]^{r}} |g_j(y',0)| \,dy' \\
&\leq C_{\Phi,\psi}  \sum_{j=1}^k \sum_{|\gamma| \leq n_2 - r} \int_{\R^{n_2}} |\partial_2^\gamma g_j(y)| \,dy \leq C_{\Phi,\psi} \sum_{|\gamma| \leq n_2 - r} \|\partial_2^\gamma g\|_1,
\end{split}\]
where the last inequality follows from (ii) applied to the diffeomorphism $\zeta_{x_j}^{-1}$ and to the smooth function $\eta_j$.

Putting all together, we get
\[\|\partial_1^\alpha ((f \circ \Phi) \psi) \|_2 \leq C_{\Phi,\psi,\alpha} \sum_{\beta \tc |\beta| \leq |\alpha| + n_2 - r} \|\partial_2^\beta f\|_2,\]
which gives (ii) in the general case.
\end{proof}

\begin{prp}\label{prp:besovchange}
Let $\Omega \subseteq \R^{n_1}$ be open, $\Phi : \Omega \to \R^{n_2}$ be smooth, $\psi \in \D(\Omega)$, $s > 0$, $p,q \in [1,\infty]$. If either $p = \infty$ or the differential of $\Phi$ has constant rank $r \in \N$, then,
\[\| (f \circ \Phi) \psi \|_{B^s_{p,q}(\R^{n_1})} \leq C_{\Phi,\psi,p,q,s} \|f\|_{B^{s+ (n_2-r)/p}_{p,q}(\R^{n_2})}.\]
for all $f \in B^{s+ (n_2-r)/p}_{p,q}(\R^{n_2})$.
\end{prp}
\begin{proof}
Thanks to the inclusions of Corollary~\ref{cor:besovinclusions}, for $p=1$ and for $p=\infty$ the conclusion follows immediately from Lemma~\ref{lem:elementarychange} by real interpolation (see Proposition~\ref{prp:besovinterpolation}(i)). For $1 < p < \infty$ and $1 \leq q < \infty$, we then use complex interpolation (see Proposition~\ref{prp:besovinterpolation}(iii)) with
\[p_0 = 1, \qquad p_1 = \infty, \qquad q_0 = q_1 = q, \qquad \theta = (p-1)/p,\]
by choosing $s_0,s_1 > 0$ such that $s = (1-\theta)s_0 + \theta s_1$. Finally, for the remaining cases with $q = \infty$, we use real interpolation again.
\end{proof}

\subsection{Local conditions, spaces on manifolds}

By Proposition~\ref{prp:besovproduct}, if $m \in B^s_{p,q}(\R^n)$ and $\psi \in \D(\R^n)$, then $m \psi \in B^s_{p,q}(\R^n)$. This suggests a way of introducing a ``local'' version of Besov spaces: for every open $\Omega \subseteq \R^n$, $s \in \R$, $1 \leq p,q \leq \infty$, let $B^s_{p,q,\loc}(\Omega)$ be the set of distributions $m \in \D'(\Omega)$ such that
\[m \psi \in B^s_{p,q}(\R^n) \qquad\text{for all $\psi \in \D(\Omega)$}\]
(where $m\psi$, which in principle is in $\E'(\Omega)$, is extended by zero to an element of $\E'(\R^n)$). Clearly, if $s \geq 0$, the elements of $B_{p,q,\loc}^s(\R^n)$ are functions. Moreover, via partition-of-unity techniques, it is not difficult to prove the following

\begin{prp}\label{prp:besovloc}
Let $s \in \R$, $1 \leq p,q \leq \infty$, $\Omega \subseteq \R^n$ be open, and let $\{\psi_\alpha\}_{\alpha \in A} \subseteq \D(\Omega)$ such that
\[\psi_\alpha \geq 0 \qquad\text{and}\qquad \bigcup_\alpha \{\psi_\alpha \neq 0\} = \Omega.\]
For all $m \in \D'(\Omega)$, we have that $m \in B_{p,q,\loc}^s(\R^n)$ if and only if
\[m \psi_\alpha \in B_{p,q}^s(\R^n) \qquad\text{for all $\alpha \in A$.}\]
\end{prp}

By Proposition~\ref{prp:besovchange}, the local spaces $B_{p,q,\loc}^s(\R^n)$ are invariant by diffeomorphisms, at least for $s > 0$, so that one can define, via local coordinates and partitions of unity, $B_{p,q,\loc}^s(M)$ for every smooth manifold $M$. For $p=q=2$, we have in particular the local $L^2$ Sobolev space $H^s_\loc(M)$, which can be clearly considered also for $s = 0$ (since $H^0 = L^2$).

In the case of a compact manifold $M$, we set $H^s(M) = H^s_\loc(M)$. In fact, since $M$ admits a finite atlas, we have that $H^s(M)$ is a Banach space, with norm
\[\|m\|_{H^s(M)} = \max_j \| (m \eta_j) \circ \psi_j^{-1}\|_{H^s(\R^n)},\]
where $\{(U_j,\psi_j)\}_j$ is a finite atlas for $M$ and $\{\eta_j\}_j$ is a partition of unity subordinated to the open cover $\{U_j\}_j$. Moreover, if we choose non-negative functions $\tilde \eta_j \in \D(M)$ with $\supp \tilde\eta_j \subseteq U_j$, such that $\tilde\eta_j \eta_j = \eta_j$, then the space $H^s(M)$ is a retract of $H^s(\R^n) \times \dots \times H^s(\R^n)$ via the maps
\[m \mapsto ( (m\eta_j) \circ \psi_j^{-1} )_j \qquad\text{and}\qquad (f_j)_j \mapsto \sum_j (f_j \circ \psi_j) \tilde \eta_j.\]
This immediately gives interpolation properties for $H^s(M)$:

\begin{prp}\label{prp:compactsobolevinterpolation}
For every compact manifold $M$, $s_0,s_1 \geq 0$, $0 < \theta <1$, we have
\[(H^{s_0}(M),H^{s_1}(M))_{[\theta]} = (H^{s_0}(M),H^{s_1}(M))_{\theta,2} = H^{(1-\theta)s_0 + \theta s_1}(M).\]
\end{prp}

Among compact manifolds, we have the $n$-dimensional torus $\T^n$. In this case, the Sobolev spaces $H^s(\T^n)$ can be characterized in terms of Fourier series:

\begin{prp}\label{prp:sobolevtorus}
For every $s \geq 0$ and Borel $m : \T^n \to \C$, we have that $m \in H^s(\T^n)$ if and only if
\[\left(\sum_{k \in \N^n} |\hat f(k)|^2 (1+|k|^2)^{s}\right)^{1/2} < \infty.\]
In fact, the left-hand side of the inequality gives an equivalent norm on $H^s(\T^n)$.
\end{prp}
\begin{proof}
For $s \in \N$, by the properties of Fourier series it is easy to see that
\[\left(\sum_{k \in \N^n} |\hat f(k)|^2 (1+|k|^2)^{s}\right)^{1/2} \sim \max_{\alpha \tc |\alpha| \leq s} \|D^\alpha f\|_{L^2(\T^n)},\]
where $D_1,\dots,D_n$ is a basis of translation-invariant vector fields on $\T^n$. From this and the previous results, the conclusion follows easily for $s \in \N$.

For a general $s \geq 0$, it is then sufficient to use interpolation (see Propositions~\ref{prp:sobolevinterpolation} and \ref{prp:compactsobolevinterpolation}).
\end{proof}

\subsection{Trace theorems}
Trace theorems are related to what happens to Sobolev or Besov spaces (and norms) when one considers the restriction (i.e., the trace) of a function to a subspace, or more generally to a submanifold.

Notice that Proposition~\ref{prp:besovchange} contains a trace theorem: if $M$ is a smooth $k$-dimensional submanifold of $\R^n$, $1 \leq p,q \leq \infty$, $s > (n-k)/p$, then
\[f \in B_{p,q}^s(\R^n) \quad\Rightarrow\quad f|_M \in B_{p,q,\loc}^{s-(n-k)/p}(M).\]

The following is a much more refined result, which allows to consider traces on subsets of $\R^n$ which are far from being submanifolds.

\begin{thm}\label{thm:triebeltrace}
Let $\mu$ be a positive regular Borel measure on $\R^n$ with compact support, such that, for some $d > 0$,
\[\mu(B(x,r)) \leq C r^d\]
for all $x \in \R^n$ and $r > 0$. If $s > (n-d)/2$, then the identity operator on $\Sz(\R^n)$ extends to a bounded operator $H^s(\R^n) \to L^2(\mu)$.
\end{thm}
\begin{proof}
It is a particular case of \cite{triebel_structure_2001}, Corollary 9.8(ii), since the Sobolev space $H^s(\R^n)$ coincides with the Triebel-Lizorkin space $F_{22}^s(\R^n)$ (cf.\ \S1.2 of \cite{triebel_structure_2001}).
\end{proof}

Let $A \subseteq \R^n$ be open and $d \geq 0$. We say that a positive regular Borel measure $\mu$ on $A$ is \emph{locally $d$-bounded}\index{locally $d$-bounded measure} if, for every compact $K \subseteq A$ and $\varepsilon > 0$, there exist $C, \bar r > 0$ such that
\[\mu(B(x,r)) \leq C r^{d-\varepsilon} \qquad\text{for $x \in K$ and $0 < r \leq \bar r$.}\]
Notice that local $0$-boundedness is an empty condition.

\begin{cor}\label{cor:triebeltrace}
Let $\mu$ be a locally $d$-bounded positive regular Borel measure on an open $A \subseteq \R^n$. If $s > (n-d)/2$, then, for every compact $K \subseteq A$,
\[\|f\|_{L^2(\mu)} \leq C_{K,s} \|f\|_{H^s(\R^n)}\]
for every $f \in \D(\R^n)$ with $\supp f \subseteq K$.
\end{cor}
\begin{proof}
Let $K \subseteq A$ be compact and $\varepsilon > 0$. Choose a compact neighborhood $K' \subseteq A$ of $K$, and let $C,\bar r > 0$ such that
\[\mu(B(x,r)) \leq C r^{d-\epsilon} \qquad\text{for $x \in K'$ and $0 < r \leq \bar r$.}\]
Let moreover $\bar r' = \min\{\bar r, d(A \setminus \mathring{K}', K)\}$, $C' = \max\{C, \mu(K)/(\bar r')^{d-\epsilon}\}$.

The identity
\[\mu_K(E) = \mu(E \cap K)\]
defines a positive regular Borel measure $\mu_K$ on $\R^n$, which coincides with $\mu$ on $K$, and with $\supp \mu_K \subseteq K$. Moreover
\[\mu_K(B(x,r)) \leq C' r^{d-\varepsilon} \qquad\text{for every $r > 0$ and $x \in \R^n$,}\]
by construction. By Theorem~\ref{thm:triebeltrace}, we then have
\[\|f\|_{L^2(\mu)} = \|f\|_{L^2(\mu_K)} \leq C_{K,s} \|f\|_{H^s(\R^n)}\]
for $s > (n-d)/2 + \varepsilon/2$ and $f \in \D(\R^n)$ with $\supp f \subseteq K$. Since $\varepsilon > 0$ was arbitrary, the conclusion follows.
\end{proof}

In the following, we will be particularly interested in measures which are homogeneous with respect to some family of dilations $\delta_t$ on $\R^n$. This property yields a sort of polar-coordinate decomposition of the measure.

\begin{prp}\label{prp:radialcoordinates}
Let $\sigma$ be a regular Borel measure on $\R^{n}$ such that
\[\sigma(\delta_{t}(A)) = t^Q \sigma(A),\]
for some $Q > 0$. Let $|\cdot|_\delta$ be a $\delta_t$-homogeneous norm and set
\[S = \{ \lambda \in \R^{n} \tc |\lambda|_\delta = 1\}.\]
Then $S$ is compact and there exists a regular Borel measure $\tau$ on $S$ such that
\[\int_{\R^{n}} f \,d\sigma = \int_S \int_{\left]0,+\infty\right[} f(\delta_{t}(\omega)) \,t^{Q-1} \,dt \,d\tau(\omega)\]
for all measurable $f : \R^{n} \to \C$ which are nonnegative or $L^1(\sigma)$.
\end{prp}
\begin{proof}
By the hypothesis on $\sigma$, the measure $\sigma'$ on $\R^{n}$ given by 
\[d\sigma'(\lambda) = |\lambda|_\delta^{-Q} \,d\sigma(\lambda)\]
is dilation-invariant.

Since $|\cdot|_\delta$ is a homogeneous norm, $S$ is a compact subset of $\R^{n} \setminus \{0\}$ and the map
\begin{equation}\label{eq:polarmap}
\R^{n} \setminus \{0\} \ni \lambda \mapsto (\delta_{|\lambda|_\delta^{-1}}(\lambda), \log_2 |\lambda|_\delta) \in S \times \R.
\end{equation}
is a homeomorphism, whose inverse is
\[S \times \R \ni (\omega,t) \mapsto \delta_{2^t}(\omega) \in \R^{n} \setminus \{0\}.\]
Let $\sigma''$ be the push-forward of $\sigma'$ on $S \times \R$ via this homeomorphism; then $\sigma''$ is a regular Borel measure on $S \times \R$, which is invariant by translations in the second coordinate.

Let now $B \subseteq S$ a Borel set and consider the measure $\sigma_B(C) = \sigma(B \times C)$ on $\R$ (it is the push-forward on $\R$ of the restriction of $\sigma''$ to $B \times \R$ via the canonical projection). Since $\sigma_B$ is invariant by translations and finite on bounded sets (being $S$ compact), it is a multiple of Lebesgue measure, i.e.\ there exists $c_B \geq 0$ such that $d\sigma_B(t) = c_B \,dt$.

In particular, $\sigma''(B \times [0,1]) = \sigma_B([0,1]) = c_B$. This means that $B \mapsto c_B$ is in fact a measure on $S$ (it is the push-forward on $S$ of the restriction of $\sigma''$ to $S \times [0,1]$ via the canonical projection), let us call it $\tau$. We have thus obtained that
\[\sigma''(B \times C) = \sigma_B(C) = c_B |C| = \tau(B) |C|,\]
which means that $\sigma''$ is the product measure of the regular Borel measure $\tau$ on $S$ and of Lebesgue measure on $\R$.

Putting all together, we easily obtain the conclusion.
\end{proof}

\begin{cor}\label{cor:radialcoordinates}
Under the hypotheses of Proposition~\ref{prp:radialcoordinates}, the measure $\sigma$ is locally $1$-bounded on $\R^n \setminus \{0\}$.
\end{cor}

\subsection{Spaces with dominating mixed smoothness}
We now introduce the Besov spaces $S^{\vec s}_{p,q}B(\R^{\vec n})$ with dominating mixed smoothness. Notice that the order $\vec s$ of differentiability is a vector, whereas on the contrary the parameters $p,q$ are still scalar.

Let $(\phi_0^{(l)},\phi_1^{(l)})$ be an admissible pair on $\R^{n_l}$, for $l=1,\dots,\ell$, and extend it as before to a sequence $(\phi_k^{(l)})_k$; moreover, set
\[\phi_{\vec{m}} = \phi^{(1)}_{m_1} \otimes \dots \otimes \phi^{(\ell)}_{m_\ell}.\]
Then $S^{\vec s}_{p,q}B(\R^{\vec n})$ is defined as the set of the $f \in \Sz'(\R^{\vec n})$ such that the norm
\[\|f\|_{S^{\vec{s}}_{p,q}B} = ( ( 2^{\vec{m} \cdot \vec{s}} \|(\Four^{-1} \phi_{\vec{m}}) * f\|_p )_{\vec m} )_{l^q}\]
is finite; if moreover
\[\lim_{|\vec m| \to \infty} 2^{\vec{m} \cdot \vec{s}} \|(\Four^{-1} \phi_{\vec{m}}) * f\|_p = 0,\]
then we say that $f \in S^{\vec{s}}_{p,q}b(\R^{\vec{n}})$.

Most of the arguments used and the properties found for Besov spaces extend in a multi-variate fashion to these new spaces. In particular, $S^{\vec{s}}_{p,q}B(\R^{\vec{n}})$ is a Banach space, with continuous inclusions
\[\Sz(\R^{\vec{n}}) \subseteq S^{\vec{s}}_{p,q}B(\R^{\vec{n}}) \subseteq \Sz'(\R^{\vec{n}}),\]
and it is a retract of the multi-parameter $L^p$-valued weighted sequence space
\begin{equation}\label{eq:multiparametersequencespace}
l^q_{\vec{s}}(L^p(\R^{\vec{n}})) = l^q_{\vec{s}}(\N^\ell; L^p(\R^{\vec{n}})) = l^q_{s_1}(l^q_{s_2}(\dots(l^q_{s_\ell}(L^p(\R^{\vec{n}}))))).
\end{equation}

Here we find an important difference between the one-variate and the multi-variate cases: Besov spaces with dominating mixed smoothness have a less rich interpolation theory. By applying iteratively Theorem~\ref{thm:sequenceinterpolation}, and then Theorem~\ref{thm:lebesgueinterpolation}, to the space \eqref{eq:multiparametersequencespace}, we obtain the following

\begin{prp}\label{prp:multibesovinterpolation}
Let $1 \leq p,p_0,p_1 \leq \infty$, $1 \leq q,q_0,q_1 \leq \infty$, $0 < \theta < 1$, $\vec{s}_0,\vec{s}_1 \in \R^\ell$.
\begin{itemize}
\item[(i)] If $q_0,q_1 < \infty$ and
\[\frac{1}{q} = \frac{1-\theta}{q_0} + \frac{\theta}{q_1},\]
then
\[(S^{\vec{s}_0}_{p,q_0}B(\R^{\vec{n}}),S^{\vec{s}_1}_{p,q_1}B(\R^{\vec{n}}))_{\theta,q} = S^{(1-\theta)\vec{s}_0 + \theta \vec{s}_1}_{p,q}B(\R^{\vec{n}}).\]
\item[(ii)] If $q_0,q_1 < \infty$ and
\[\frac{1-\theta}{p_0} + \frac{\theta}{p_1} = \frac{1}{p} = \frac{1-\theta}{q_0} + \frac{\theta}{q_1},\]
then
\[(S_{p_0,q_0}^{\vec{s}_0}B(\R^{\vec{n}}),S_{p_1,q_1}^{\vec{s}_1}B(\R^{\vec{n}})_{\theta,p} = S_{p,p}^{(1-\theta)\vec{s}_0 + \theta \vec{s}_1}B(\R^{\vec{n}}).\]
\item[(iii)] If $q_0 <\infty$ and
\[\frac{1}{p} = \frac{1-\theta}{p_0} + \frac{\theta}{p_1}, \qquad \frac{1}{q} = \frac{1-\theta}{q_0} + \frac{\theta}{q_1},\]
then
\[(S_{p_0,q_0}^{\vec{s}_0}B(\R^{\vec{n}}),S_{p_1,q_1}^{\vec{s}_1}B(\R^{\vec{n}})_{[\theta]} = S_{p,q}^{(1-\theta)\vec{s}_0 + \theta \vec{s}_1}B(\R^{\vec{n}}).\]
\end{itemize}
\end{prp}
This result is certainly less versatile than Proposition~\ref{prp:besovinterpolation}, but will be sufficient to our aims. See \cite{schmeisser_spaces_2004}, \S4.1 (and particularly Lemma~2), for an argument against the possibility of obtaining an interpolation result as flexible as in the one-variate case.

The continuous embeddings between multi-variate Besov spaces, instead, are perfectly analogous to the one-variate case.
\begin{prp}\label{prp:multibesovembeddings}
Let $1 \leq p,p_1,p_2,q,q_1,q_2 \leq \infty$, $\vec s_1,\vec s_2 \in \R^\ell$.
\begin{itemize}
\item[(i)] If $\vec s_1 > \vec s_2$, or otherwise if $\vec s_1 \geq \vec s_2$ and $q_1 > q_2$, then we have the continuous inclusion
\[S_{p,q_1}^{s_1}B(\R^n) \subseteq S_{p,q_2}^{s_2}b(\R^n).\]
\item[(ii)] If $p_1 \leq p_2$, $\vec s_1 - \vec n/p_1 \geq \vec s_2 - \vec n/p_2$ and $q_1 \geq q_2$, then we have the continuous inclusions
\[S_{p_1,q}^{s_1}B(\R^n) \subseteq S_{p_2,q}^{s_2}B(\R^n), \qquad S_{p_1,q}^{s_1}b(\R^n) \subseteq S_{p_2,q}^{s_2}b(\R^n).\]
\end{itemize}
\end{prp}

The Fourier multiplier results obtained in the one-variate case can be certainly extended to this context, at least for multipliers in factorized form (i.e., when the multiplier is the tensor product of symbols belonging to suitable classes on the factors $\R^{n_l}$ of $\R^{\vec{n}}$). In particular, we obtain easily this multi-variate version of the lifting properties:

\begin{prp}\label{prp:multibesovlifting}
Let $1 \leq p,q \leq \infty$, $s \in \R$.
\begin{itemize}
\item[(i)] For every $\vec{\sigma} \in \R^\ell$, the \emph{multi-parameter Bessel potential} operator
\[J_{\vec{\sigma}} : f \mapsto \Four^{-1} ( \langle \cdot \rangle^{\vec{\sigma}} (\Four f))\]
is an isomorphism $S_{p,q}^{\vec{s}}B(\R^{\vec{n}}) \to S_{p,q}^{\vec{s}-\vec{\sigma}}B(\R^{\vec{n}})$, with inverse $J_{-\vec{\sigma}}$.
\item[(ii)] For every $\vec{k} \in \N$, the three conditions
\begin{itemize}
\item[(a)] $f \in S_{p,q}^{\vec{s}}B(\R^{\vec{n}})$,
\item[(b)] $\partial^{\vec{\alpha}} f \in S_{p,q}^{\vec{s}-\vec{k}}B(\R^{\vec{n}})$ for $|\alpha_1| \leq k_1,\dots,|\alpha_\ell| \leq k_\ell$,
\end{itemize}
are equivalent; moreover, the quantity
\[\sum_{|\alpha_1| \leq k_1} \cdots \sum_{|\alpha_\ell| \leq k_\ell} \|\partial^{\vec{\alpha}} f\|_{S_{p,q}^{\vec{s}-\vec{k}}B}\]
is an equivalent norm of $f \in S_{p,q}^{\vec{s}}B(\R^{\vec{n}})$.
\end{itemize}
The same results hold also for little Besov spaces.
\end{prp}

Consequently, it is not difficult to obtain the inclusions corresponding to Corollary~\ref{cor:besovinclusions}:

\begin{prp}\label{prp:multibesovinclusions}
For every $\vec{k} \in \N^\ell$, the following are continuous inclusions:
\begin{itemize}
\item[(i)] $S_{p,1}^{\vec{k}} B(\R^{\vec{n}}) \subseteq S^{\vec{k},p} W(\R^{\vec{n}}) \subseteq S_{p,\infty}^{\vec{k}} B(\R^n)$ for $1 \leq p \leq \infty$,
\item[(ii)] $S_{\infty,1}^{\vec{k}} B(\R^{\vec{n}}) \subseteq S^{\vec{k}} C_{ub}(\R^{\vec{n}}) \subseteq S_{\infty,\infty}^{\vec{k}} B(\R^{\vec{n}})$.
\end{itemize}
Moreover, we have $S_{2,2}^{\vec{s}}B(\R^{\vec{n}}) = S^{\vec{s}} H(\R^{\vec{n}})$ for all $\vec{s} \in \R$.
\end{prp}

The result on pointwise multiplication can also be generalized to this setting:

\begin{prp}\label{prp:multibesovproduct}
Let $1 \leq p, p_1, p_2, q \leq \infty$, $\vec{s} \in \R^\ell$, with
\[\frac{1}{p} = \frac{1}{p_1} + \frac{1}{p_2}.\]
\begin{itemize}
\item[(i)] For every $\vec{\sigma} \in \R^\ell$ with $\sigma_l > |s_l|$ for $l=1,\dots,\ell$, we have
\[\| f g \|_{S^{\vec{s}}_{p,q}B} \leq C_{p_1,p_2,q,\vec{\sigma},\vec{s}} \|f\|_{S^{\vec{\sigma}}_{p_1,\infty}B} \|g\|_{S^{\vec{s}}_{p_2,q}B}.\]
\item[(ii)] If $\vec{s} > 0$, then we also have the more precise inequality
\[\| f g \|_{S^{\vec{s}}_{p,q}B} \leq C_{p_1,p_2,q,\vec{s}} \|f\|_{S^{\vec{s}}_{p_1,q}B} \|g\|_{S^{\vec{s}}_{p_2,q}B}.\]
\end{itemize}
\end{prp}
\begin{proof}
Analogously as in the proof of Proposition~\ref{prp:besovproduct}, the problem is reduced to verifying some inequalities involving sequences $(a_{\vec{h}})_{\vec{h} \in \N^\ell}$, $(b_{\vec{k}})_{\vec{k} \in \N^\ell}$ of non-negative reals:
\[ \left\| \left(\sum_{(h_1,k_1) \in I_m} \cdots \sum_{(h_\ell,k_\ell) \in I_m} a_{\vec{h}} b_{\vec{k}}\right)_{\vec{m}} \right\|_{l^q_{\vec{s}}} \leq C_{p_1,p_2,q,\vec{\sigma},\vec{s}} \| (a_{\vec{h}})_{\vec{h}} \|_{l^\infty_{\vec{\sigma}}} \| (b_{\vec{k}})_{\vec{k}} \|_{l^q_{\vec{s}}}\]
for (i), and
\[ \left\| \left(\sum_{(h_1,k_1) \in I_m} \cdots \sum_{(h_\ell,k_\ell) \in I_m} a_{\vec{h}} b_{\vec{k}} \right)_{\vec{m}} \right\|_{l^q_{\vec{s}}} \leq C_{p_1,p_2,q,\vec{s}} \| (a_{\vec{h}})_{\vec{h}} \|_{l^q_{\vec{s}}} \| (b_{\vec{k}})_{\vec{k}} \|_{l^q_{\vec{s}}}\]
for (ii).

For $\ell = 1$, the inequalities are true by the proof of Proposition~\ref{prp:besovproduct}. It is then sufficient to plug these inequalities into themselves, by exploiting the fact that
\[l^q_{\vec{s}} = l^q_{s_1}(l^q_{s_2}(\dots(l^q_{s_\ell}) \dots)), \qquad l^\infty_{\vec{\sigma}} = l^\infty_{\sigma_1}(l^\infty_{\sigma_2}(\dots(l^\infty_{\sigma_\ell}) \dots)),\]
to obtain recursively the inequalities for a general $\ell$.
\end{proof}

\begin{cor}\label{cor:multibesovcompactsupportinclusion}
For every compact $K \subseteq \R^{\vec{n}}$, $\vec{s} \in \R$, $1 \leq p_1,p_2,q \leq \infty$ with $p_1 \leq p_2$, we have
\[\|f\|_{S_{p_1,q}^{\vec{s}}B} \leq C_{K,p_1,p_2,q,{\vec{s}}} \|f\|_{S_{p_2,q}^{\vec{s}}B} \qquad\text{if $\supp f \subseteq K$.}\]
\end{cor}

Proposition~\ref{prp:besovapproximation} about approximation by smooth functions is immediately extended to the multi-parameter context:

\begin{prp}\label{prp:multibesovapproximation}
Let $1 \leq p,q \leq \infty$, $\vec{s} \in \R^\ell$, $f \in S_{p,q}^{\vec{s}}B(\R^{\vec{n}})$. If $u \in \Sz(\R^{\vec{n}})$, then $f * u \in W^{\infty,p}(\R^{\vec{n}})$. Moreover, if $f \in S_{p,q}^{\vec{s}}b(\R^{\vec{n}})$, and if $u_h \in \Sz(\R^{\vec{n}})$ is an approximate identity for $h \to \infty$, then $f * u_h \to f$ in $S_{p,q}^{\vec{s}}B(\R^{\vec{n}})$.
\end{prp}

A comparison between Besov spaces with dominating mixed smoothness and classical Besov spaces on $\R^{\vec{n}}$ is given by the following

\begin{prp}\label{prp:besovmultibesov}
Let $\vec{s} > 0$, $1 \leq p,q \leq \infty$. Then
\[B^{s_1 + \dots + s_\ell}_{p,q}(\R^{\vec{n}}) \subseteq S^{\vec{s}}_{p,q}B(\R^{\vec{n}}) \subseteq B_{p,q}^{\min \vec{s}}(\R^{\vec{n}})\]
with continuous inclusions.
\end{prp}
\begin{proof}
For $l=1,\dots,\ell$, choose a standard admissible pair $(\phi_0^{(l)},\phi_1^{(l)})$ on $\R^{n_l}$ with respect to the norm $|\cdot|_\infty$,
let $(\phi^{(l)}_{k_l})_{k_l \in \N}$ be the corresponding generated system, and set
\[\phi_{\vec{k}} = \phi_{k_1}^{(1)} \otimes \dots \otimes \phi_{k_\ell}^{(\ell)}.\]
Moreover take an admissible pair $(\eta_0,\eta_1)$ on $\R^{\vec{n}}$ such that
\[\{ \xi \in \R^{\vec{n}} \tc |\xi|_\infty \leq 2 \} \subseteq \{\eta_0 = 1\}, \quad \{ \xi \in \R^{\vec{n}} \tc 1 \leq |\xi|_\infty \leq 4 \} \subseteq \{\eta_1 = 1\},\]
and again extend it to a sequence $(\eta_m)_{m \in \N}$. Then, for $m = \max\vec{k}$, we have
\[\phi_{\vec{k}} \, \eta_{m} = \phi_{\vec{k}},\]
thus also, by Young's inequality,
\[\|(\Four^{-1} \phi_{\vec{k}}) * f\|_p \leq C \| (\Four^{-1} \eta_{m}) * f \|_p .\]
Therefore, for $q < \infty$,
\[ \sum_{\vec{k} \in \N^\ell} 2^{\vec{k} \cdot \vec{s} q} \| (\Four^{-1}\phi_{\vec{k}}) * f \|_p^q \leq C^q \sum_{m \in \N} \left(\sum_{\vec{k} \tc m = \max{\vec{k}}} 2^{\vec{k} \cdot \vec{s} q} \right) \| (\Four^{-1} \eta_m) * f\|_p^q\]
and, since
\[ \sum_{\vec{k} \tc m = \max{\vec{k}}} 2^{\vec{k} \cdot \vec{s} q} \leq \sum_{l=1}^\ell 2^{m s_l q} \prod_{l' \neq l} \sum_{k=0}^m 2^{k s_{l'} q} \leq C_{\vec{s},q} 2^{m(s_1+\dots+s_\ell) q},\]
the inclusion $B^{s_1 + \dots + s_\ell}_{p,q}(\R^{\vec{n}}) \subseteq S^{\vec{s}}_{p,q}B(\R^{\vec{n}})$ follows; the case $q=\infty$ is analogous.

For the second inclusion, without loss of generality (see Proposition~\ref{prp:multibesovinclusions}) we may suppose $s_* = s_1 = \dots = s_\ell = \min\vec{s}$. Notice now that, if we set
\[\psi_m = \sum_{\vec{k} \tc m = \max \vec{k}} \phi_{\vec{k}},\]
then $(\psi_0,\psi_1)$ is an admissible pair on $\R^{\vec{n}}$, whose generated system is $(\psi_m)_{m \in \N}$ (the scaling properties may be seen by considering the sums $\psi_0+ \dots+ \psi_m$, which can be written as tensor products of corresponding sums on the factors of $\R^{\vec{n}}$). For $1 < q < \infty$, since
\begin{multline*}
\| (\Four^{-1} \psi_m) * f\|_p \leq \sum_{\vec{k} \tc m = \max \vec{k}} \| (\Four^{-1} \phi_{\vec{k}}) * f\|_p \\
\leq \left( \sum_{\vec{k} \tc m = \max \vec{k}} 2^{-\vec{k} \cdot \vec{s} q'} \right)^{1/q'} \left( \sum_{\vec{k} \tc m = \max \vec{k}} 2^{\vec{k} \cdot \vec{s} q} \|(\Four^{-1} \phi_{\vec{k}}) * f\|_p^q \right)^{1/q},
\end{multline*}
and
\[ \sum_{\vec{k} \tc m = \max{\vec{k}}} 2^{- \vec{k} \cdot \vec{s} q'} \leq \sum_{l=1}^\ell 2^{- m s_* q'} \prod_{l' \neq l} \sum_{k=0}^m 2^{- k s_* q'} \leq C_{\vec{s},q} 2^{-m s_*  q'},\]
we get
\[ (2^{m s_*} \| (\Four^{-1} \psi_m) * f\|_p)^q \leq C_{\vec{s},q} \sum_{\vec{k} \tc m = \max \vec{k}} 2^{\vec{k} \cdot \vec{s} q} \|(\Four^{-1} \phi_{\vec{k}}) * f\|_p^q,\]
and the conclusion follows; the cases $q=1$ and $q=\infty$ are analogous.
\end{proof}

\subsection{Characterization in terms of differences}

We need some notation. For $l=1,\dots,\ell$, $y_l \in \R^{n_l}$, $m \in \N$, set
\[\Delta_{l,y_l}^{m} f(x_1,\dots,x_\ell) = \sum_{k=0}^m (-1)^k \binom{m}{k} f(x_1,\dots,x_{l-1},x_l + ky_l, x_{l+1}, \dots, x_\ell);\]
moreover, for $J = \{l_1,\dots,l_v\} \subseteq \{1,\dots,\ell\}$, $y \in \prod_{l \in J} \R^{n_l}$, set
\[\Delta_{J,y}^{\vec m} f = \Delta_{l_1,y_{l_1}}^{m_{l_1}} \cdots \Delta_{l_v,y_{l_v}}^{m_{l_v}} f,\]
and, for $t \in \left]0,+\infty\right[^J$,
\[\omega_{J,p}^{\vec m}(t,f) = \sup_{|y_{l_1}|_\infty < t_{\ell_1}} \cdots \sup_{|y_{l_v}|_\infty < t_{l_v}} \|\Delta_{J,y}^{\vec m} f\|_p.\]
Then we have (cf. \cite{schmeisser_topics_1987}, \S2.3.4)

\begin{prp}\label{prp:multibesovdifferences}
Suppose that $\vec{m} \in \N^\ell$, $0 < \vec{s} < \vec{m}$, $p,q \in [1,\infty]$. Then
\[\|f\|_p + \sum_{\emptyset \neq J \subseteq \{1,\dots,\ell\}} \left( \int_{\left]0,+\infty\right[^{J}} \left( \frac{\omega_{J,p}^{\vec m}(t,f)}{\prod_{l \in J} t_l^{s_l}} \right)^q \,\prod_{l \in J} \frac{dt_l}{t_l} \right)^{1/q}\]
is an equivalent norm on $S^{\vec{s}}_{p,q}B(\R^{\vec{n}})$. Moreover, a function $f \in L^p(\R^{\vec{n}})$ such that the above quantity is finite belongs to $S^{\vec{s}}_{p,q}B(\R^{\vec{n}})$.
\end{prp}
\begin{proof}
The proof of the equivalence of the norms is a (notationally involved) multi-variate analogue of the proof of Theorem~6.2.5 of \cite{bergh_interpolation_1976}. The second assertion follows easily by approximation via the decomposition
\[f = \sum_{\vec m \in \Z^\ell} (\Four^{-1} \phi_{\vec{m}}) * f\]
and Fatou's property.
\end{proof}

About finite differences, we also recall an elementary result which will be useful in the following.

\begin{lem}\label{lem:differencesobolev}
For $\vec{m} \in \N^\ell$, $J \subseteq \{1,\dots,\ell\}$, $y \in \prod_{l \in J} \R^{n_l}$, $1 \leq p \leq \infty$, we have
\[\|\Delta_{J,y}^{\vec{m}} f\|_p \leq C_{J,\vec{m},p} \|f\|_{S^{\vec{m},p}W} \prod_{l\in J} |y_l|_\infty^{m_l} \]
\end{lem}
\begin{proof}
Let $\Delta_h f(x) = f(x) - f(x+h)$. If $\mu_h$ is the measure defined by
\[\int_{\R^{\vec{n}}} \phi \,d\mu_h = \int_0^1 f \phi(-th) \,dt,\]
then the fundamental theorem of calculus gives
\[\Delta_h f = \partial_{-h} \mu_h * f = |h|_\infty \, \mu_h * \partial_{-h/|h|_\infty} f,\]
where $\partial_v$ denotes the derivative corresponding to the vector $v$. The conclusion follows by repeated applications of the previous identity (with different choices of the increment $h$) and by Young's inequality (since the $\mu_h$ have mass $1$).
\end{proof}

\section{Mihlin-H\"ormander conditions}

\subsection{Pointwise conditions}
Fix a system of dilations $\{\delta_t\}_{t > 0}$ on $\R^n$. Let $Q_\delta$ be the associated homogeneous dimension, and $|\cdot|_\delta$ be a homogeneous norm.

Choose a basis $\partial_1,\dots,\partial_n$ of (translation-invariant vector fields on) $\R^n$ diagonalizing the dilations $\delta_t$. If $\lambda_j$ is the homogeneity degree of $\partial_j$, then the basis $\partial_1,\dots,\partial_n$ can be thought of as an adapted basis (with weights $\lambda_1,\dots,\lambda_n$) of the homogeneous Lie algebra $\R^n$ with dilations $\delta_t$; correspondingly, for a multi-index $\alpha = (\alpha_1,\dots,\alpha_n) \in \N^n$, we have a notion of homogeneous degree
\[\|\alpha\| = \alpha_1 \lambda_1+ \dots + \alpha_n \lambda_n\]
analogous to the one introduced in \S\ref{subsection:noncommutativemultiindex} for non-commutative multi-indices.

We then say that a function $m : \R^n \to \C$ satisfies a \emph{pointwise Mihlin-H\"ormander condition} of order $s \in \N$ (adapted to the fixed system of dilations) if $m \in C^s(\R^n \setminus \{0\})$ and, for all $\alpha \in J(n)$ with $|\alpha| \leq s$,
\[|\partial^\alpha m(x)| \leq C_\alpha |x|_\delta^{-\|\alpha\|} \qquad\text{for $x \in \R^n \setminus \{0\}$,}\]
for some $C_\alpha > 0$. In fact, if we put, for $m \in C^s(\R^n \setminus \{0\})$,
\[\|m\|_{M_\delta C^s} = \max_{\alpha \tc |\alpha| \leq s} \sup_{x \neq 0} |x|_\delta^{\|\alpha\|} |\partial^\alpha m(x)|,\]
then $m$ satisfies a pointwise Mihlin-H\"ormander condition of order $s$ if and only if $\|m\|_{M_\delta C^s} < \infty$. We will also say that $m$ satisfies a pointwise Mihlin-H\"ormander condition of infinite order if $m \in C^\infty(\R^n \setminus \{0\})$ and $\|m\|_{M_\delta C^s} < \infty$ for all $s \in \N$.

It is not difficult to check that the previously defined conditions do not depend on the choice of the homogeneous norm, and also on the choice of the adapted basis of $\R^n$, but only on the fixed system of dilations. We also have, for every $s \in \N$,
\[\|m \circ \delta_t\|_{M_\delta C^s} = \|m\|_{M_\delta C^s} \qquad\text{for all $t > 0$},\]
so that the pointwise Mihlin-H\"ormander conditions are dilation-invariant.

In the case $n=1$, there is a unique choice (up to parameter-rescaling) for a system of dilations, which therefore needs not be specified; in this case, we will use the notation $\| \cdot \|_{M_* C^k}$.

Clearly
\[\| \cdot \|_{M_\delta C^{s_1}} \leq \| \cdot \|_{M_\delta C^{s_2}} \qquad\text{for $s_1 \leq s_2$,}\]
so that a pointwise Mihlin-H\"ormander condition of order $s$ implies conditions of all the orders smaller than $s$. We also have that $\| \cdot \|_{M_\delta C_0}$ coincides with the norm of the supremum over $\R^n \setminus \{0\}$, so that a function $m$ satisfying a pointwise Mihlin-H\"ormander condition of any order is necessarily a bounded continuous function on $\R^n \setminus \{0\}$; on the other hand, its derivatives are allowed to diverge at $0$, with a precise order depending on the order of the derivative, and correspondingly they have to vanish at infinity. In any case, the quantity $\|\cdot\|_{M_\delta C^s}$ majorizes the $C^s$-norm on every fixed compact subset of $\R^n \setminus \{0\}$.

\subsection{Besov conditions}
We introduce now a (possibly weaker) version of Mihlin-H\"ormander conditions, in which derivatives are controlled in terms of Besov norms; in fact, this allows us to be more precise, since we admit also conditions of fractional order. The Mihlin-H\"ormander conditions that we are going to introduce are local Besov conditions on $\R^n \setminus \{0\}$, with an added uniformity with respect to the fixed system of dilations.

For $1 \leq p,q \leq \infty$, $s \in \R$, $\phi \in \D(\R^n \setminus \{0\})$, we set
\[\|m\|_{M_\delta^\phi B_{p,q}^s} = \sup_{t > 0} \|(m \circ \delta_t) \phi\|_{B^s_{p,q}}.\]
We will be interested in this quantity when $\phi$ satisfies the following conditions:
\begin{equation}\label{eq:mhadm}
\phi \geq 0, \qquad \bigcup_{t > 0} \delta_{t^{-1}}(\{\phi \neq 0\}) = \R^n \setminus \{0\}.
\end{equation}
There is also a ``discrete'' version: for $a > 0$, we set
\[\|m\|_{M_\delta^{\phi,a} B_{p,q}^s} = \sup_{j \in \Z} \|(m \circ \delta_{2^{a j}}) \phi\|_{B_{p,q}^s}.\]
In this case, we will require a more restrictive condition on $\phi$:
\begin{equation}\label{eq:mhdiscadm}
\phi \geq 0, \qquad \bigcup_{j \in \Z} \delta_{2^{-a j}}(\{\phi \neq 0\}) = \R^n \setminus \{0\}.
\end{equation}

\begin{prp}\label{prp:mhequivalence}
\begin{itemize}
\item[(i)] If $\phi \in \D(\R^n \setminus \{0\})$ satisfies \eqref{eq:mhadm}, then it satisfies \eqref{eq:mhdiscadm} for sufficiently small $a > 0$.

\item[(ii)] Suppose that both $\phi_1,\phi_2 \in \D(\R^n \setminus \{0\})$ satisfy \eqref{eq:mhadm}. Then, for every $s \geq 0$, the norms
\[\|\cdot\|_{M_\delta^{\phi_1} B_{p,q}^s} \qquad\text{and}\qquad \|\cdot\|_{M_\delta^{\phi_2} B_{p,q}^s}\]
are equivalent.

\item[(iii)] Suppose that $\phi \in \D(\R^n \setminus \{0\})$ and $a > 0$ satisfy \eqref{eq:mhdiscadm}. Then, for every $s \geq 0$, the norms
\[\|\cdot\|_{M_\delta^{\phi} B_{p,q}^s} \qquad\text{and}\qquad \|\cdot\|_{M_\delta^{\phi,a} B_{p,q}^s}\]
are equivalent.
\end{itemize}
\end{prp}

The proof will be given after some preliminary results. Let $\lambda_1,\dots,\lambda_n$ be the weights of the elements of the adapted basis $\partial_1,\dots,\partial_n$, so that $Q_\delta = \sum_j \lambda_j$. In the following, we will use coordinates $(x_1,\dots,x_n)$ on $\R^n$ associated to the basis $\partial_1,\dots,\partial_n$, so that
\[\delta_t(x_1,\dots,x_n) = (t^{\lambda_1} x_1,\dots,t^{\lambda_n} x_n).\]
In analogy with the isotropic case, it is immediate to prove that
\begin{equation}\label{eq:fourierhomogeneous}
\Four (f \circ \delta_t) = t^{-Q_\delta} (\Four f) \circ \delta_{t^{-1}}.
\end{equation}
for all $f \in \Sz'(\R^n)$ and $t > 0$.

\begin{lem}\label{lem:continuousdilations}
Let $s \in \R$, $1 \leq p,q \leq \infty$. Then, for every $T > 1$,
\[\sup_{T^{-1} \leq t \leq T} \|f \circ \delta_t\|_{B_{p,q}^s} \leq C_{T,p,q,s} \|f\|_{B_{p,q}^s}.\]
\end{lem}
\begin{proof}
Let $(\phi_0,\phi_1)$ be an admissible pair as in \S\ref{subsection:besovbasic}. Then it is not difficult to find another admissible pair $(\tilde\phi_0,\tilde\phi_1)$ such that
\[\bigcup_{T^{-1} \leq t \leq T} \supp (\phi_j \circ \delta_{t^{-1}}) \subseteq \{\tilde\phi_j = 1\}\]
for $j=0,1$. Then, by \eqref{eq:fourierhomogeneous}, Young's inequality and the invariance of $\|\Four^{-1} \eta\|_1$ by dilations of $\eta$, it is not difficult to show that the Besov norm of $f \circ \delta_t$ with respect to $(\phi_0,\phi_1)$ is controlled (uniformly in $t \in [T^{-1},T]$) by the Besov norm of $f$ with respect to $(\tilde \phi_0,\tilde \phi_1)$.
\end{proof}

\begin{lem}\label{lem:coveringcondition}
Let $\eta \in \D(\R^n \setminus \{0\})$, $\eta \geq 0$, such that
\begin{equation}\label{eq:coveringcondition}
\bigcup_{t > 0} \delta_{t^{-1}}(\{\eta \neq 0\}) = \R^n \setminus \{0\}.
\end{equation}
Then there exists $a > 0$ such that
\[\bigcup_{j \in \Z} \delta_{2^{-a j}}(\{\eta \neq 0\}) = \R^n \setminus \{0\}.\]
\end{lem}
\begin{proof}
Let $|\cdot|_\delta$ be a $\delta_t$-homogeneous norm, set $S = \{ x \tc |x|_\delta = 1 \}$, and let $\Phi : \R^n \setminus \{0\} \to S \times \R$ be the homeomorphism defined by \eqref{eq:polarmap}. From \eqref{eq:coveringcondition} it follows that, for all $\omega \in S$, there exists $t \in \R$ such that $\eta(\Phi^{-1}(\omega,t)) \neq 0$. By continuity of $\eta$ and compactness of $S$, we can find a finite open cover $U_1,\dots,U_k$ of $S$ and open intervals $I_1,\dots,I_k$ of $\R$ such that
\[\Phi^{-1}\left(\bigcup_l (U_l \times I_l) \right) \subseteq \{\eta \neq 0\}.\]
In order to conclude, it is then sufficient to take $a > 0$ less than the minimum of the lengths of the intervals $I_1,\dots,I_k$.
\end{proof}

\begin{lem}\label{lem:mhequivalence}
Let $\eta \in \D(\R^n \setminus \{0\})$ and $a > 0$ be as in Lemma~\ref{lem:coveringcondition}. Let moreover $\psi \in \D(\R^n \setminus \{0\})$, and $s \in \R$. Then there is $C > 0$ such that, for every measurable function $m : \R^n \to \C$,
\begin{equation}\label{eq:discretization}
\sup_{t > 0} \|(m \circ \delta_t) \psi\|_{B_{p,q}^s} \leq C \sup_{j \in \Z} \| (m \circ \delta_{2^{a j}}) \eta\|_{B^s_{p,q}}.
\end{equation}
Moreover, if the right-hand side in \eqref{eq:discretization} is finite, then $m \in B^s_{p,q,\loc}(\R^n \setminus \{0\})$.
\end{lem}
\begin{proof}
If the right-hand side of \eqref{eq:discretization} is infinite, then there is nothing to prove.

Let us suppose therefore that the right-hand side is finite. In particular
\[\|m (\eta \circ \delta_{2^{-a j}})\|_{B_{p,q}^s} < \infty \qquad\text{for all $j \in \Z$,}\]
so that, by Proposition~\ref{prp:besovloc}, $m \in B_{p,q,\loc}^s(\R^n \setminus \{0\})$.

Let $\chi \in \D(\R^n \setminus \{0\})$ such that
\[\bigcup_{t \in [2^{-a},2^a]} \supp (\psi \circ \delta_t) \subseteq \{ \chi = 1 \}.\]
Then, for $t \in [2^{-a},2^a]$, we have $(\chi \circ \delta_t) \psi = \psi$, so that
\[\sup_{t \in [2^{-a},2^a]} \|(m \circ \delta_t) \psi \|_{B_{p,q}^s} = \sup_{t \in [2^{-a},2^a]} \|((m \chi) \circ \delta_t) \psi\|_{B_{p,q}^s} \leq C_{p,q,s,\psi,a} \|m \chi\|_{B_{p,q}^s}\]
by Lemma~\ref{lem:continuousdilations} and Proposition~\ref{prp:besovproduct}.

On the other hand, since $\supp \chi$ is compact, there is a finite $I \subseteq \Z$ such that
\[\supp \chi \subseteq \bigcup_{j \in I} \{\eta \circ \delta_{2^{-a j}} \neq 0\}.\]
Let
\[\phi = \frac{\chi}{\sum_{j \in I} \eta \circ \delta_{2^{-a j}}}.\]
Then $\phi \in \D(\R^n \setminus \{0\})$ and $\chi = \psi \sum_{j \in I} \eta \circ \delta_{2^{-a j}}$, so that, by Proposition~\ref{prp:besovproduct},
\[\|m \chi\|_{B_{p,q}^s} \leq C_\phi \left\|m \sum_{j \in I} \eta \circ \delta_{2^{-a j}}\right\|_{B_{p,q}^s} \leq C_{\phi,I} \sum_{j \in I} \| (m \circ \delta_{2^{a j}}) \eta \|_{B_{p,q}^s}.\]

Putting all together, we get
\[\sup_{t \in [2^{-a},2^a]} \| (m \circ \delta_t) \psi \|_{B_{p,q}^s} \leq \tilde C_\psi \sup_{j \in Z} \|(m \circ \delta_{2^{a j}}) \eta\|_{B_{p,q}^s}.\]
By replacing $m$ with $m \circ \delta_{2^k}$ for $k \in \Z$ in the last inequality, the right-hand side is not changed, and the conclusion follows.
\end{proof}

\begin{proof}[Proof of Proposition~\ref{prp:mhequivalence}]
(i) is simply Lemma~\ref{lem:coveringcondition}.

About (ii), if $a > 0$ is such that both $\phi_1,\phi_2$ satisfy \eqref{eq:mhdiscadm} (the existence of this $a$ follows from (i)),
then we get immediately from Lemma~\ref{lem:mhequivalence} that
\[\|\cdot\|_{M_\delta^{\phi_1} B_{p,q}^s} \leq C \|\cdot\|_{M_\delta^{\phi_2,a} B_{p,q}^s}, \qquad \|\cdot\|_{M_\delta^{\phi_2} B_{p,q}^s} \leq C \|\cdot\|_{M_\delta^{\phi_1,a} B_{p,q}^s},\]
but obviously we have also
\[\|\cdot\|_{M_\delta^{\phi_1,a} B_{p,q}^s} \leq \|\cdot\|_{M_\delta^{\phi_1} B_{p,q}^s}, \qquad \|\cdot\|_{M_\delta^{\phi_2,a} B_{p,q}^s} \leq \|\cdot\|_{M_\delta^{\phi_2} B_{p,q}^s},\]
and these inequalities together give the conclusion.

Finally, for (iii), the inequality
\[\|\cdot\|_{M_\delta^\phi B_{p,q}^s} \leq C \|\cdot\|_{M_\delta^{\phi,a} B_{p,q}^s}\]
follows from Lemma~\ref{lem:mhequivalence}, whereas the opposite inequality is trivial.
\end{proof}

In the following, we will denote by $\|\cdot\|_{M_\delta B_{p,q}^s}$ one of the equivalent quantities listed in Proposition~\ref{prp:mhequivalence}. We then say that a function $m : \R^n \to \C$ satisfies a \emph{$B_{p,q}^s$ Mihlin-H\"ormander condition} (adapted to the system of dilations $\delta_t$) if $\|m\|_{M_\delta B_{p,q}^s} < \infty$. Again, for $n=1$, we will use the notation $\|\cdot\|_{M_* B_{p,q}^s}$.

From the definition, it is clear that $\|\cdot\|_{M_\delta B_{p,q}^s}$ is dilation-invariant. Moreover, from Lemma~\ref{lem:mhequivalence} it follows that, if $m$ satisfies a $B_{p,q}^s$ Mihlin-H\"ormander condition, then $m \in B_{p,q,\loc}^s(\R^n \setminus \{0\})$. In the following, we will see that, for $s$ sufficiently large, a function $m \in B_{p,q}^s(\R^n)$ with compact support satisfies a $B_{p,q}^s$ Mihlin-H\"ormander condition; the threshold on $s$ is given in terms of the \emph{normalized homogeneous dimension}
\[\tilde Q_\delta = Q_\delta/\min\{\lambda_1,\dots,\lambda_n\}.\]
Notice that $\tilde Q_\delta \geq n$.

\begin{prp}\label{prp:besovlargedilations}
Let $1 \leq p,q \leq \infty$, $s > \tilde Q_\delta/p$, and $\eta \in \D(\R^n)$. Then
\[\sup_{0 < t \leq 1} \| (f \circ \delta_t) \eta \|_{B_{p,q}^s} \leq C_{\eta,p,q,s} \|f\|_{B_{p,q}^s}.\]
\end{prp}
\begin{proof}
Without loss of generality, we may suppose that $s$ is not an integer (the missing values can be recovered a posteriori by interpolation, see Proposition~\ref{prp:besovinterpolation}), thus there exists $m \in \N$ such that $m-1 < s < m$.

By Proposition~\ref{prp:multibesovdifferences}, we have to estimate
\begin{equation}\label{eq:differencesnorm}
\| (f \circ \delta_t) \eta \|_p + \left( \int_0^1 \left(\frac{\omega_p^m(r, (f \circ \delta_t) \eta)}{r^s}\right)^q \frac{dr}{r} \right)^{1/q}
\end{equation}
(in fact, the ``missing'' integral on $\left]1,+\infty\right[$ in the latter summand is easily controlled by the former).

The former summand in \eqref{eq:differencesnorm} is immediately majorized by H\"older's inequality and embeddings (see Proposition~\ref{prp:besovembeddings}), since $\eta$ is compactly supported and $s > \tilde Q_\delta/p \geq n/p$:
\[ \| (f \circ \delta_t) \eta \|_p \leq C_{\eta,p} \| f \circ \delta_t \|_\infty = C_{\eta,p} \|f\|_\infty \leq C_{\eta,p,q,s} \|f\|_{B_{p,q}^s}.\]

For the latter summand, notice first that, by the Leibniz rule for finite differences, H\"older's inequality, Lemma~\ref{lem:differencesobolev} and the fact that $\eta \in \D(\R^n)$, we have
\[ \| \Delta_y^m ((f \circ \delta_t) \eta) \|_p \leq C_{\eta,m,p,p_0,\dots,p_m} \sum_{k=0}^m |y|^{m-k} \| \Delta_y^k (f \circ \delta_t) \|_{p_k}\]
for any choice of $p_0,\dots,p_k \geq p$; since
\[ \| \Delta_y^k (f \circ \delta_t) \|_{p_k} = t^{-Q_\delta/p_k}  \| \Delta_{\delta_t(y)}^k f \|_{p_k} \qquad\text{and}\qquad |\delta_t(y)|_\infty \leq t^{\lambda_*} |y|_\infty\]
for $t \leq 1$, where $\lambda_* = \min\{\lambda_1,\dots,\lambda_n\}$, we then get also
\[ \omega_p^m(r,(f \circ \delta_t) \eta) \leq C_{\eta,m,p,p_0,\dots,p_m} \sum_{k=0}^m r^{m-k} t^{-Q_\delta/p_k} \omega_{p_k}^k (t^{\lambda_*} r, f).\]
Choose now $p_m = p$, and $p_k = ps/k$ for $k < m$. Then, for $k < p$, we have
\[p_k > p, \quad k - \frac{n}{p_k} = \frac{k}{s} \left( s - \frac{n}{p} \right) < s - \frac{n}{p}, \quad  \lambda_* k - \frac{Q_\delta}{p_k} = \frac{\lambda_* k}{s} \left( s - \frac{\tilde Q_\delta}{p} \right) > 0,\]
so that, by Lemma~\ref{lem:differencesobolev} and the embeddings $B_{p,q}^s \subseteq B_{p_k,1}^k \subseteq W^{k,p_k}$ (see Proposition~\ref{prp:besovembeddings} and Corollary~\ref{cor:besovinclusions}),
\[r^{m-k} t^{-Q_\delta/p_k} \omega_{p_k}^k (t^{\lambda_*} r, f) \leq C_{k,p_k} r^m t^{\lambda_* k -Q_\delta/p_k} \|f\|_{W^{k,p_k}} \leq C_{p,q,s} r^m \|f\|_{B_{p,q}^s},\]
thus
\[\frac{r^{m-k} t^{-Q_\delta/p_k} \omega_{p_k}^k (t^{\lambda_*} r, f)}{r^s} \leq C_{p,q,s} r^{m-s} \|f\|_{B_{p,q}^s}.\]
For $k=m$, instead,
\[\frac{t^{-Q_\delta/p} \omega_{p}^m (t^{\lambda_*} r, f)}{r^s} = t^{\lambda_* s -Q_\delta/p} \frac{\omega_{p}^m (t^{\lambda_*} r, f)}{(t^{\lambda_*}r)^s} \leq \frac{\omega_{p}^m (t^{\lambda_*} r, f)}{(t^{\lambda_*}r)^s}.\]
Putting all together, the latter summand in \eqref{eq:differencesnorm} is majorized by
\[ C_{\eta,p,q,s} \left( \|f\|_{B_{p,q}^s} \left(\int_0^1 (r^{m-s})^q \frac{dr}{r}\right)^{1/q} + \left(\int_0^{\lambda_*} \left(\frac{\omega_{p}^m (r, f)}{r^s}\right)^q \frac{dr}{r} \right)^{1/q} \right),\]
and the conclusion follows from Proposition~\ref{prp:multibesovdifferences}.
\end{proof}

\begin{cor}
Let $K \subseteq \R^n$ be compact; then, for $1 \leq p,q \leq \infty$, $s > \tilde Q_\delta$,
\[\| f \|_{M_\delta B_{p,q}^s} \leq C_{p,q,s,K} \| f \|_{B_{p,q}^s}\]
for all $f \in B_{p,q}^s(\R^n)$ with $\supp f \subseteq K$.
\end{cor}
\begin{proof}
Choose $\phi \in \D(\R^n \setminus \{0\})$ satisfying \eqref{eq:mhadm} and such that
\[\supp \phi \cap \bigcup_{t > 1} \delta_{t^{-1}}(K) = \emptyset.\]
Then, if $\supp f \subseteq K$, for $t > 1$ we have $(f \circ \delta_t) \phi = 0$, thus
\[\|f\|_{M_\delta^\phi B_{p,q}^s} = \sup_{0 < t \leq 1} \| (f \circ \delta_t) \phi \|_{B_{p,q}^s},\]
and the conclusion follows by Proposition~\ref{prp:besovlargedilations}.
\end{proof}

\subsection{Comparison}
We now show some implications among the Mihlin-H\"ormander conditions previously introduced. First of all, from Proposition~\ref{prp:besovembeddings} and Corollary~\ref{cor:besovcompactsupportinclusion}, we get immediately

\begin{prp}
Let $1 \leq p_1,p_2,q_1,q_2 \leq \infty$, $s_1,s_2 \in \R$. Then we have an inequality
\[\|\cdot\|_{M_\delta B_{p_2,q_2}^{s_2}} \lesssim \|\cdot\|_{M_\delta B_{p_1,q_1}^{s_1}}\]
in each of the following cases:
\begin{itemize}
\item $p_1 \leq p_2$, $s_1 - n/p_1 > s_2 - n/p_2$;
\item $p_1 \leq p_2$, $s_1 - n/p_1 = s_2 - n/p_2$, $q_1 \geq q_2$;
\item $p_1 \geq p_2$, $s_1 = s_2$, $q_1 \geq q_2$.
\end{itemize}
\end{prp}

Moreover, a comparison between pointwise and Besov conditions is given by the following

\begin{prp}\label{prp:pointwisebesovlmihlin}
Let $s \in \N$. Then
\[\|\cdot\|_{M_\delta B_{\infty,\infty}^{s}} \lesssim \|\cdot\|_{M_\delta C^s} \lesssim \|\cdot\|_{M_\delta B_{\infty,1}^s}.\]
\end{prp}
\begin{proof}
For the first inequality, since $\|\cdot\|_{M_\delta C^s}$ is $\delta_t$-invariant, it is sufficient to prove that, for $\eta \in \D(\R^n \setminus \{0\})$,
\[\|m \eta\|_{B_{\infty,\infty}^s} \leq C_{s,\eta} \|m\|_{M_\delta C^s},\]
but this follows easily from Corollary~\ref{cor:besovinclusions}(ii).

For the second inequality, choose $\phi \in \D(\R^n \setminus \{0\})$ such that
\[\phi \geq 0 \qquad\text{and}\qquad \bigcup_{t > 0} \delta_{t^{-1}}(\{\phi = 1\}\mathring{\,}) = \R^n \setminus \{0\}.\]
Let $\alpha \in \N^n$ such that $|\alpha| \leq s$, and let $x \in \R^n \setminus \{0\}$. Then there is $t > 0$ such that $\delta_t(x) \in \{\phi = 1\}\mathring{}$, hence
\[
|x|_\delta^{\|\alpha\|} |\partial^\alpha m(x)| = |\delta_t(x)|_\delta^{\|\alpha\|} |\partial^\alpha (m \circ \delta_{t^{-1}}) (\delta_t(x))| \leq C_{s,\phi} \|\partial^\alpha ((m \circ \delta_{t^{-1}}) \phi)\|_\infty,
\]
since $\phi(\delta_t(x)) = 1$, whereas $\partial^\beta \phi(\delta_t(x)) = 0$ for $|\beta| > 0$. Therefore, again by Corollary~\ref{cor:besovinclusions}(ii),
\[|x|_\delta^{\|\alpha\|} |\partial^\alpha m(x)| \leq C_{\phi,s} \|(m \circ \delta_{t^{-1}}) \phi\|_{B_{\infty,1}^s} \leq C_{\phi,s} \|m\|_{M_\delta B_{\infty,1}^s},\]
and the conclusion follows since $x \in \R^n \setminus \{0\}$ and $\alpha$ with $|\alpha| \leq s$ were arbitrary.
\end{proof}

\subsection{Change of variables}
We consider now two spaces $\R^{n_1}$ and $\R^{n_2}$, with systems of dilations $\{\delta_{1,t}\}_{t > 0}$ and $\{\delta_{2,t}\}_{t > 0}$ respectively. A map $\Phi : \R^{n_1} \to \R^{n_2}$ will be said \emph{homogeneous} (with respect to these systems of dilations) if
\begin{equation}\label{eq:maphomogeneous}
\Phi \circ \delta_{1,t} = \delta_{2,t} \circ \Phi \qquad\text{for all $t > 0$.}
\end{equation}
Notice that, if $\Phi : \R^{n_1} \setminus \{0\} \to \R^{n_2}$ is continuous and satisfies \eqref{eq:maphomogeneous}, then it can be extended to a continuous homogeneous map $\R^{n_1} \to \R^{n_2}$ by setting $\Phi(0) = 0$.

\begin{lem}
A continuous homogeneous map $\Phi : \R^{n_1} \to \R^{n_2}$ is proper (i.e., $\Phi^{-1}(K)$ is compact for every compact $K \subseteq \R^{n_2}$) if and only if
\[\Phi(x) = 0 \qquad\Longrightarrow\qquad x = 0\]
for every $x \in \R^{n_1}$. In this case, moreover, for every $\delta_{2,t}$-homogeneous norm $|\cdot|_{\delta_2}$ on $\R^{n_2}$, we have that $|\Phi(\cdot)|_{\delta_2}$ is a $\delta_{1,t}$-homogeneous norm on $\R^{n_1}$.
\end{lem}
\begin{proof}
Notice that, since $\Phi$ is homogeneous, the preimage of $\{0\}$ via $\Phi$ must be dilation-invariant. This means that, if $\Phi^{-1}(\{0\})$ contains elements other than $0$, then it is an unbounded subset of $\R^{n_1}$. In particular, if $\Phi$ is proper, then necessarily $\Phi^{-1}(\{0\}) \subseteq \{0\}$.

Vice versa, if $\Phi^{-1}(\{0\}) \subseteq \{0\}$, we have in fact that $\Phi^{-1}(\{0\}) = \{0\}$, since $0$ is the only dilation-invariant element in $\R^{n_2}$. Therefore, if $|\cdot|_{\delta_2}$ is a $\delta_{2,t}$-homogeneous norm on $\R^{n_2}$, by continuity of $\Phi$ we have immediately that $|\Phi(\cdot)|_{\delta_2}$ is a $\delta_{1,t}$-homogeneous norm on $\R^{n_1}$, and the fact that $\Phi$ is proper follows easily from the equivalence of homogeneous norms on $\R^{n_1}$.
\end{proof}

Let $\partial_{j,1},\dots,\partial_{j,n_j}$ be an adapted basis of $\R^{n_j}$ with dilations $\delta_{j,t}$ for $j=1,2$.

\begin{prp}\label{prp:uniformhomogeneouschange}
Suppose that $\Phi : \R^{n_1} \to \R^{n_2}$ is a continuous proper homogeneous map which is smooth off the origin. Then, for all $s \in \N$,
\[\|m \circ \Phi\|_{M_{\delta_1} C^s} \leq C_{s,\Phi} \|m\|_{M_{\delta_2} C^s}\]
for all Borel $m : \R^{n_2} \to \C$.
\end{prp}
\begin{proof}
As in the proof of Lemma~\ref{lem:elementarychange}, we get
\begin{equation}\label{eq:derivativescomposition}
\partial_1^\alpha(m \circ \Phi) = \sum_{\beta \tc |\beta| \leq |\alpha|} ((\partial_2^\beta m) \circ \Phi) \cdot \Psi_{\alpha,\beta},
\end{equation}
where $\Psi_{\alpha,\beta}$ are linear combinations of products of derivatives of the components of $\Phi$. Since $\Phi$ is continuous, proper and homogeneous, if $|\cdot|_{\delta_2}$ is a homogeneous norm on $\R^{n_2}$, then $|\cdot|_{\delta_1} = |\Phi(\cdot)|_{\delta_2}$ is a homogeneous norm on $\R^{n_1}$; moreover, the set $K = \{ x \in \R^{n_1} \tc |\Phi(x)|_{\delta_2} = 1\}$ is compact. We then have, for every $\alpha$ with $|\alpha| \leq s$ and every $x \in K$,
\[|\partial_1^\alpha(m \circ \Phi)(x)| \leq C_{\alpha,\Phi} \max_{\beta \tc |\beta| \leq |\alpha|} |\Phi(x)|_{\delta_2}^{\|\beta\|} |(\partial_2^\beta m)(\Phi(x))|  \leq C_{s,\Phi} \|m\|_{M_{\delta_2} C^s},\]
by the equivalence of homogeneous norms. Thus, for every $x \in \R^{n_1} \setminus \{0\}$, since $\delta_{1,t^{-1}}(x) \in K$ for $t = |\Phi(x)|_{\delta_2}$, we have
\begin{multline*}
|x|_{\delta_1}^{\|\alpha\|} |\partial_1^\alpha(m \circ \Phi)(x)| = |\partial_1^\alpha(m \circ \Phi \circ \delta_{1,t})(\delta_{t^{-1}}(x))| \\
= |\partial_1^\alpha(m \circ \delta_{2,t} \circ \Phi)(\delta_{t^{-1}}(x))| \leq C_{s,\Phi} \|m \circ \delta_{2,t}\|_{M_{\delta_2} C^s} = C_{s,\Phi} \|m\|_{M_{\delta_2} C^s},
\end{multline*}
by the previous inequality applied to $m \circ \delta_{2,t}$. Since $x \in \R^{n_1} \setminus \{0\}$ and $\alpha$ with $|\alpha| \leq s$ were arbitrary, we get the conclusion.
\end{proof}

We look now for a similar result about the Besov conditions.

\begin{prp}\label{prp:homogeneouschange}
Suppose that $\Phi : \R^{n_1} \to \R^{n_2}$ is a continuous proper homogeneous map which is smooth off the origin, $s > 0$, $p,q \in [1,\infty]$. If either $p=\infty$, or $d\Phi$ has constant rank $r \in \N$ on $\R^{n_1} \setminus \{0\}$, then
\[\|m \circ \Phi\|_{M_{\delta_1} B_{p,q}^s} \leq C_{p,q,s,\Phi} \|m\|_{M_{\delta_2} B_{p,q}^{s+(n_2-r)/p}}\]
for all Borel $m : \R^{n_2} \to \C$.
\end{prp}
\begin{proof}
Let $\eta \in \D(\R^{n_2} \setminus \{0\})$, $\eta \geq 0$, be such that
\[\bigcup_{t > 0} \delta_{2,t^{-1}}(\{\eta \neq 0\}) = \R^{n_2} \setminus \{0\}.\]
Since $\Phi$ is a continuous proper homogeneous map which is smooth off the origin, we have $\eta \circ \Phi \in \D(\R^{n_1} \setminus \{0\})$, $\eta \circ \Phi \geq 0$ and
\[\bigcup_{t > 0} \delta_{1,t^{-1}}(\{\eta \circ \Phi \neq 0\}) = \R^{n_1} \setminus \{0\}.\]
If we choose $\psi \in \D(\R^{n_1})$ such that $(\eta \circ \Phi) \psi = \eta \circ \Phi$, then we have
\[ ((m \circ \Phi) \circ \delta_{1,t}) (\eta \circ \Phi) = (((m \circ \delta_{2,t}) \eta) \circ \Phi) \psi \]
and the conclusion follows easily from Proposition~\ref{prp:besovchange}.
\end{proof}

\begin{rem}
If $\Phi: \R^{n_1} \to \R^{n_2}$ is a homogeneous map which is smooth off the origin, then the rank of $d\Phi$ on $\R^{n_1} \setminus \{0\}$ is at least $1$. Therefore, when $n_2 = 1$, there is no loss of regularity in the comparison given by Proposition~\ref{prp:homogeneouschange}.
\end{rem}

We extend the definition of homogeneous map in order to include maps which are defined only on subsets of $\R^{n_1}$; namely, if $\Sigma \subseteq \R^{n_1}$, we say that a map $\Phi : \Sigma \to \R^{n_2}$ is homogeneous if $\Sigma$ is $\delta_{1,t}$-invariant and $\Phi$ satisfies \eqref{eq:maphomogeneous}. Moreover, we say that a map $\Phi : \Sigma \subseteq \R^{n_1} \to \R^{n_2}$ is smooth if there exists some open $\Omega \subseteq \R^{n_1}$ containing $\Sigma$ such that $\Phi$ extends to a smooth map $\Omega \to \R^{n_2}$ (cf.\ Lemma~2.27 of \cite{lee_smooth_2003}).

\begin{prp}\label{prp:homogeneousextension}
Let $\Sigma \subseteq \R^{n_1}$ be closed, and let $\Phi : \Sigma \to \R^{n_2}$ be continuous, homogeneous and such that $\Phi^{-1}(0) \subseteq \{0\}$. Suppose moreover that either $\Phi$ is not surjective, or $n_2 \geq n_1$. Then there exists an extension $\tilde \Phi : \R^{n_1} \to \R^{n_2}$ of $\Phi$ which is continuous, homogeneous and proper. If moreover $\Phi|_{\Sigma \setminus \{0\}}$ is smooth, then the extension $\tilde\Phi$ can be taken to be smooth off the origin.
\end{prp}
\begin{proof}
By the implicit function theorem, for $j=1,2$ we can find a smooth homogeneous norm $|\cdot|_{\delta_j}$ on $\R^{n_j}$ such that
\[\{ x \in \R^{n_j} : |x|_{\delta_j} = 1\} = S^{n_j-1} = \{ x \in \R^{n_j} : |x|_2 = 1\},\]
and set $\nu_j(x) = \delta_{j,|x|_{\delta_2}^{-1}}(x)$. Let moreover, for $x \in \Sigma \cap S^{n_1-1}$,
\[\Phi_r(x) = |\Phi(x)|_{\delta_2}, \qquad \Phi_\omega(x) = \nu_2(\Phi(x)).\]

By the Tietze-Urysohn extension theorem (see \cite{bourbaki_topology2}, \S IX.4.2), the map $\Phi_r : \Sigma \cap S^{n_1-1} \to \R^+$ can be extended to a continuous map $\Psi_r : S^{n_1-1} \to \R^+$. On the other hand, a continuous extension $\Psi_\omega : S^{n_1-1} \to S^{n_2-1}$ of the map $\Phi_\omega : \Sigma \cap S^{n_1-1} \to S^{n_2-1}$ exists too: if $\Phi$ is not surjective, then by homogeneity also $\Phi_\omega$ is not, so that $\Phi_\omega$ can be thought of as valued in $\R^{n_2-1}$ modulo homeomorphisms, and the Tietze-Urysohn extension theorem can be applied componentwise; if $n_2 \geq n_1$, then the extension is given by Theorem~6-45 of \cite{hocking_topology_1988}. Moreover, if $\Phi$ is smooth off the origin, then the extensions $\Psi_r$ and $\Psi_\omega$ can be taken to be smooth by the Whitney approximation theorem (see \cite{lee_smooth_2003}, Theorem~10.21).

Finally, if $\tilde\Phi : \R^{n_1} \to \R^{n_2}$ is defined by
\[\tilde\Phi(x) = \begin{cases}
0 &\text{if $x = 0$,}\\
\delta_{2,|x|_{\delta_1} \Psi_r(\nu_1(x))}(\Psi_\omega(\nu_1(x)))  &\text{otherwise.}
\end{cases},\]
then it is easily checked that $\tilde\Phi$ is the looked-for extension of $\Phi$.
\end{proof}

\section{Marcinkiewicz conditions}

Marcinkiewicz conditions are the multi-variate analogue of Mihlin-H\"ormander conditions.

We will need some notation. Fix $\ell \in \N$, $\vec{n} \in \N^\ell$. For $l = 1,\dots,\ell$, let $\{\delta_{l,t}\}_{t > 0}$ be a system of dilations on $\R^{n_l}$, with homogeneous norm $|\cdot|_{\delta_l}$ and homogeneous dimension $Q_l$. We introduce a multi-parameter system of dilations on $\R^{\vec{n}}$ as follows:
\[\daleth_{\vec{t}}(x_1,\dots,x_\ell) = (\delta_{1,t_1}(x_1),\dots,\delta_{\ell,t_\ell}(x_\ell)),\]
and, as a particular case, also a one-parameter system of dilations on $\R^{\vec n}$:
\[\delta_t(x) = \daleth_{(t,\dots,t)}(x).\]
The latter system of dilations defines a homogeneous structure on $\R^{\vec{n}}$, with homogeneous dimension $Q = \sum_l Q_l$ and homogeneous norm $|\cdot|_\delta$ defined by
\[|(x_1,\dots,x_\ell)|_{\delta} = \max \{|x_1|_{\delta_1}, \dots, |x_\ell|_{\delta_\ell}\}.\]
Moreover, if $\partial_{l,1},\dots,\partial_{l,n_l}$ is an adapted basis of $\R^{n_l}$ for $l=1,\dots,\ell$, then the concatenation
\[\partial_{1,1}, \dots, \partial_{1,n_1}, \dots, \partial_{\ell,1}, \dots \partial_{\ell,n_\ell}\]
is an adapted basis of $\R^{\vec{n}}$.

The \emph{pointwise Marcinkiewicz condition} of order $\vec k \in \N^\ell$ (adapted to the multi-parameter dilations $\daleth_{\vec t}$) is defined by the finiteness of the ``norm''
\[\|m\|_{M_\daleth S^{\vec{k}}C} = \max_{|\alpha_1| \leq k_1,\dots,|\alpha_\ell| \leq k_\ell} \sup_{x_1 \neq 0, \dots, x_\ell \neq 0} |x_1|_{\delta_1}^{\|\alpha_1\|} \cdots |x_\ell|_{\delta_\ell}^{\|\alpha_\ell\|} |\partial_1^{\alpha_1} \dots \partial_\ell^{\alpha_l} m(x)|.\]
It is not difficult to see that
\[\|m \circ \daleth_{\vec t}\|_{M_\daleth S^{\vec{k}}C} = \|m\|_{M_\daleth S^{\vec{k}}C} \qquad\text{for all $t_1,\dots,t_n > 0$.}\]
Moreover, from the definition it is clear that the derivatives of a function satisfying a uniform Marcinkiewicz condition are allowed to diverge not only at the origin of $\R^{\vec{n}}$, but also on the set
\[\{(x_1,\dots,x_\ell) \in \R^{\vec{n}} \tc |x_1|_{\delta_1} \cdots |x_\ell|_{\delta_l} = 0\}.\]

In order to introduce the Besov version of the Marcinkiewicz condition, we need test functions $\phi_l \in \D(\R^{n_l} \setminus \{0\})$ satisfying \eqref{eq:mhadm} for $l = 1,\dots,\ell$. We then set
\[\phi = \phi_1 \otimes \dots \otimes \phi_\ell\]
and define, for $\vec{s} \in \R^\ell$, $1 \leq p,q \leq \infty$,
\[\|m\|_{M_\daleth^\phi S_{p,q}^{\vec{s}}B} = \sup_{t_1,\dots,t_n > 0} \|(m \circ \daleth_{\vec{t}}) \phi\|_{S_{p,q}^{\vec{s}}B}.\]
If moreover $a_l > 0$ is such that $\phi_l$ and $a_l$ satisfy \eqref{eq:mhdiscadm} for $l = 1,\dots,\ell$, then we introduce also the discrete version
\[\|m\|_{M_\daleth^{\phi,\vec{a}} S_{p,q}^{\vec{s}}B} = \sup_{j_1,\dots,j_n \in \Z} \|(m \circ \daleth_{(2^{a_1 j_1},\dots,2^{a_\ell j_\ell})}) \phi\|_{S_{p,q}^{\vec{s}}B}.\]
As in the one-parameter case, for every $\vec{s} \in \R^\ell$, $1\leq p,q\leq \infty$, all the previously defined ``norms'' are equivalent and their finiteness define the \emph{$S_{p,q}^{\vec{s}}B$ Marcinkiewicz condition} adapted to the multi-parameter dilations $\daleth_{\vec t}$.

When $n_1 = \dots = n_\ell = 1$, there is (essentially) one system of $\ell$-parameter dilations, which therefore needs not be specified; in this case, we will use the notation $\| \cdot \|_{M_* S^{\vec{k}}C}$ or $\|\cdot\|_{M_* S^{\vec{s}}_{p,q}B}$.

For the comparison of pointwise and Besov Marcinkiewicz conditions, the following multi-variate versions of the comparison results for Mihlin-H\"ormander conditions are easy to prove.

\begin{prp}
Let $1 \leq p_1,p_2,q_1,q_2 \leq \infty$, $\vec{s}_1,\vec{s}_2 \in \R^\ell$. Then we have an inequality
\[\|\cdot\|_{M_\daleth S_{p_2,q_2}^{\vec{s}_2}B} \lesssim \|\cdot\|_{M_\daleth S_{p_1,q_1}^{\vec{s}_1}B}\]
in each of the following cases:
\begin{itemize}
\item $p_1 \leq p_2$, $\vec{s}_1 - \vec{n}/p_1 > \vec{s}_2 - \vec{n}/p_2$;
\item $p_1 \leq p_2$, $\vec{s}_1 - \vec{n}/p_1 = \vec{s}_2 - \vec{n}/p_2$, $q_1 \geq q_2$;
\item $p_1 \geq p_2$, $\vec{s}_1 = \vec{s}_2$, $q_1 \geq q_2$.
\end{itemize}
\end{prp}

\begin{prp}
Let $\vec{s} \in \N^\ell$. Then
\[\|\cdot\|_{M_\daleth S_{\infty,\infty}^{\vec{s}}B} \lesssim \|\cdot\|_{M_\daleth S^{\vec{s}}C} \lesssim \|\cdot\|_{M_\daleth S_{\infty,1}^{\vec{s}}B}.\]
\end{prp}

What is more interesting is a comparison between Mihlin-H\"ormander conditions adapted to the dilations $\delta_t$ and Marcinkiewicz conditions adapted to the multi-parameter dilations $\daleth_{\vec{t}}$. For the pointwise conditions, there is the following simple result.

\begin{prp}
For $\vec{s} \in \N^\ell$, we have
\[\|m\|_{M_{\daleth} S^{\vec{s}}C} \leq \|m\|_{M_\delta C^{s_1+\dots+s_\ell}}.\]
\end{prp}
\begin{proof}
If $|\alpha_l| \leq s_l$ for $l=1,\dots,\ell$, then $|\vec{\alpha}| \leq s_1 + \dots + s_\ell$, and moreover
\[|x_1|_{\delta_1}^{\|\alpha_1\|} \cdots |x_\ell|_{\delta_\ell}^{\|\alpha_\ell\|} \leq |x|_{\delta}^{\|\vec{\alpha}\|}\]
for $x \in \R^{\vec{n}}$. From this, the conclusion follows immediately.
\end{proof}

For the comparison of integral conditions, we need an extra hypothesis, involving the normalized homogeneous dimensions $\tilde Q_l$ of the factors $\R^{n_l}$.

\begin{prp}\label{prp:mihlinmarcinkiewicz}
If $2 \leq p \leq \infty$ and $s_l > \tilde Q_l/p$ for $l=1,\dots,\ell$, then
\[\|m\|_{M_\daleth S_{p,p}^{\vec{s}}B} \leq C_{p,\vec{s}} \|m\|_{M_\delta B_{p,p}^{s_1 + \dots + s_\ell}}.\]
\end{prp}

The proof will be given after some preliminary lemmata. The first is a result analogous to Proposition~\ref{prp:besovlargedilations}.

\begin{lem}\label{lem:dilationtensorproduct}
Let $\phi_l \in \D(\R^{n_l} \setminus \{0\})$ for $l=1,\dots,\ell$, and let $\phi = \phi_1 \otimes \dots \otimes \phi_\ell$. Then, for $p \in [2,\infty]$ and $\vec{s} \in \R^\ell$ such that $s_l > \tilde Q_l/p$ for $l=1,\dots,\ell$,
\[\sup_{\vec{t} > 0, \, |\vec{t}|_\infty \leq 1} \|(m \circ \daleth_{\vec{t}}) \phi\|_{S_{p,p}^{\vec{s}}B} \leq C_{\phi,\vec{s},p} \|m\|_{S_{p,p}^{\vec{s}}B}.\]
\end{lem}
\begin{proof}
Let $T_{\vec{t}}$ be the linear operator
\[m \mapsto (m \circ \daleth_{\vec{t}}) \phi;\]
our aim is to obtain a uniform bound on the $S^{\vec{s}}_{p,p}B \to S^{\vec{s}}_{p,p}B$ norms of the operators $T_{\vec{t}}$ for $|\vec t|_\infty \leq 1$. Since the complex interpolation functors are exact, we need to consider only the cases $p=2$ and $p=\infty$, because the intermediate cases $2 < p < \infty$ follow by interpolation (see Proposition~\ref{prp:multibesovinterpolation}(iii)).

For $p = 2$, by \eqref{eq:sobolevtensorproduct} we have
\[S^{\vec{s}}_{2,2}B(\R^{\vec{n}}) = B_{2,2}^{s_1}(\R^{n_1}) \htimes \cdots \htimes B_{2,2}^{s_\ell}(\R^{n_\ell}),\]
and moreover the map $T_{\vec{t}}$ is the tensor product of the maps
\[T_{l,t_l} : m \mapsto (m \circ \delta_{l,t_l}) \phi_l\]
for $l=1,\dots,\ell$, so that the norm of $T_{\vec{t}}$ is the product of the norms of the $T_{l,t_l}$, which are uniformly bounded for $|\vec{t}|_\infty \leq 1$ by Proposition~\ref{prp:besovlargedilations}. 

For $p=\infty$, instead, we immediately get rid of $\phi$,
\[\|(f \circ \daleth_{\vec{t}}) \phi\|_{S_{\infty,\infty}^{\vec{s}}B} \leq C_{\vec{s},\phi} \| f \circ \daleth_{\vec{t}} \|_{S_{\infty,\infty}^{\vec{s}}B}\]
by Proposition~\ref{prp:multibesovproduct}, and then we use the characterization of the $S_{\infty,\infty}^{\vec{s}}B$ norm given by Proposition~\ref{prp:multibesovdifferences}:
\[\| f \circ \daleth_{\vec{t}} \|_{S_{\infty,\infty}^{\vec{s}}B} \sim \|f \circ \daleth_{\vec{t}} \|_\infty + \sum_{\emptyset \neq J \subseteq \{1,\dots,\ell\}} \sup_{r \in \left]0,+\infty\right[^{J}} \frac{\omega_{J,\infty}^{\vec m}(r,f \circ \daleth_{\vec{t}} )}{\prod_{l \in J} r_l^{s_l}}.\]
The first summand is simply $\|f\|_\infty$. For the second one, since $|\daleth_{\vec{t}}(y)|_\infty \leq |y|_\infty$ for $|\vec{t}|_\infty \leq 1$, we easily get $\omega_{J,\infty}^{\vec{m}}(r, f \circ \daleth_{\vec{t}}) \leq \omega_{J,\infty}^{\vec{m}}(r,f)$. The conclusion then follows again by Proposition~\ref{prp:multibesovdifferences}.
\end{proof}

\begin{rem}
Through a characterization of $S_{1,1}^{\vec{s}}B(\R^{\vec{n}})$ as a tensor product (cf.\ \cite{sickel_tensor_2009}), one could extend the previous result (and consequently Proposition~\ref{prp:mihlinmarcinkiewicz}) to the whole range $1 \leq p \leq \infty$.
\end{rem}

\begin{lem}\label{lem:mihlinmarcinkiewicz}
Let $b > a > 0$. For $l=1,\dots,\ell$, let $\phi_l \in \D(\R^{n_l} \setminus \{0\})$ be such that $\supp \phi_l \subseteq \{x_l \in \R^{n_l} \tc a \leq |x_l|_\infty \leq b\}$, and set $\phi = \phi_1 \otimes \dots \otimes \phi_\ell$. Let moreover $\eta \in \D(\R^{\vec n} \setminus \{0\})$ be such that $\eta|_{\{x \tc a \leq |x|_\infty \leq b\}} \equiv 1$. If $p \in [2,\infty]$ and $s_l > \tilde Q_l/p$ for $l=1,\dots,\ell$, then
\[\sup_{\vec t > 0, \, |\vec t|_\infty = 1} \|(m \circ \daleth_{\vec t}) \phi\|_{S_{p,p}^{\vec{s}}B} \leq C_{\phi,\eta,\vec{s},p} \|m \eta\|_{S_{p,p}^{\vec{s}}B}.\]
\end{lem}
\begin{proof}
Notice that, for $\vec t > 0$,
\[\supp (\phi \circ \daleth_{\vec t}^{-1}) \subseteq \prod_{l=1}^{\ell} \{x_l \tc t_l a \leq |x_l| \leq t_l b\}.\]
If $|\vec t|_\infty = 1$, then $t_l \leq 1$ for all $l$ and $t_l = 1$ for at least one $l$, so that
\[\supp (\phi \circ \daleth_{\vec t}^{-1}) \subseteq \{x \tc a \leq |x| \leq b\}.\]
Therefore
\[\eta (\phi \circ \daleth_{\vec t}^{-1}) = \phi \circ \daleth_{\vec t}^{-1},\]
that is,
\[(\eta \circ \daleth_{\vec t}) \phi = \phi,\]
and the conclusion follows immediately by Lemma~\ref{lem:dilationtensorproduct}.
\end{proof}

\begin{proof}[Proof of Proposition~\ref{prp:mihlinmarcinkiewicz}]
Choose $\phi_l \in \D(\R^{n_l} \setminus \{0\})$ and $\eta \in \D(\R^{\vec n} \setminus \{0\})$ satisfying \eqref{eq:mhadm} and the hypotheses of Lemma~\ref{lem:mihlinmarcinkiewicz}.

For $\vec{t} > 0$, let $r = |\vec{t}|_\infty$, so that $|r^{-1} \vec{t}|_\infty = 1$. Since
\[m \circ \daleth_{\vec{t}} = (m \circ \delta_r) \circ \daleth_{r^{-1} \vec{t}},\]
then
\[\|(m \circ \daleth_{\vec{t}}) \phi\|_{S_{p,p}^{\vec{s}}B} \leq C_{\phi,\eta,\vec{s},p} \|(m \circ \delta_r) \eta\|_{S_{p,p}^{\vec{s}}B} \leq C_{\phi,\eta,\vec{s},p} \|m\|_{M_\delta B_{p,p}^{s_1+\dots+s_n}}\]
by Lemma~\ref{lem:mihlinmarcinkiewicz} and Proposition~\ref{prp:besovmultibesov}. Since $\vec t > 0$ is arbitrary, the conclusion follows.
\end{proof}

Let $\epsilon_t$ be a system of dilations on $\R^d$. Proposition~\ref{prp:mihlinmarcinkiewicz}, combined with Proposition~\ref{prp:homogeneouschange}, gives in particular the following implication between Mihlin-H\"ormander conditions on $\R^d$ and Marcinkiewicz conditions on $\R^{\vec{n}}$.

\begin{cor}\label{cor:mihlinmarcinkiewiczsum}
Let $\Phi : \R^{\vec{n}} \to \R^d$ be a continuous proper map which is smooth off the origin and such that
\[\Phi \circ \delta_t = \epsilon_t \circ \Phi \qquad\text{for $t > 0$.}\]
If $p \in [2,\infty]$, $s_l > \tilde Q_l/p$ for $l=1,\dots,\ell$, and either $p=\infty$ or $d\Phi$ has constant rank $r \in \N$ off the origin, then
\[\|m \circ \Phi\|_{M_{\daleth} S^{\vec{s}}_{p,p}B} \leq C_{\Phi,p,\vec{s}} \|m\|_{M_\epsilon B^{s_1+\dots+s_\ell+(d-r)/p}_{p,p}}.\]
\end{cor}

A similar implication, but with a different homogeneity condition on the map $\Phi$, is considered in the following

\begin{prp}\label{prp:mihlinmarcinkiewiczproduct}
Set
\[X = \{x \in \R^{\vec{n}} \tc |x_1|_{\delta_1} \cdots |x_\ell|_{\delta_\ell} = 0\},\]
and let $\Phi : \R^{\vec{n}} \to \R^d$ be a continuous map which is smooth off $X$ and such that $\Phi^{-1}(0) = X$,
\[\Phi \circ \daleth_{\vec{t}} = \epsilon_{t_1 \cdots t_\ell} \circ \Phi \qquad\text{for $\vec{t} > 0$.}\]
If $p,q \in [1,\infty]$, $\vec{s} > 0$, and either $p = \infty$ or $d\Phi$ has constant rank $r \in \N$ outside $X$, then
\[\|m \circ \Phi\|_{M_{\daleth} S^{\vec{s}}_{p,q}B} \leq C_{\Phi,p,q,\vec{s}} \|m\|_{M_\epsilon B^{s_1+\dots+s_\ell+(d-r)/p}_{p,q}}.\]
\end{prp}
\begin{proof}
For $l=1,\dots,\ell$, let $\phi_l \in \D(\R^{n_l} \setminus \{0\})$ satisfying \eqref{eq:mhadm}, and set
\[\phi = \phi_1 \otimes \dots \otimes \phi_\ell.\]
Since $\supp \phi$ is compact and disjoint from $\{x \in \R^{\vec{n}} \tc |x_1|_{\delta_1} \cdots |x_\ell|_{\delta_\ell} = 0\}$, we can find a nonnegative $\psi \in \D(\R^d \setminus \{0\})$ such that
\[\Phi(\supp \phi) \subseteq \{\psi = 1\};\]
in particular $\psi$ satisfies \eqref{eq:mhadm} too (with respect to the dilations $\epsilon_t$), and moreover
\[(\psi \circ \Phi) \phi = \phi.\]
Therefore, for $\vec{t} > 0$,
\begin{multline*}
\| ((m \circ \Phi) \circ \daleth_{\vec{t}}) \phi \|_{S^{\vec{s}}_{p,q}B} \leq C_{\vec{s},p,q} \| (((m \circ \epsilon_{t_1 \cdots t_\ell}) \psi) \circ \Phi) \phi \|_{B_{p,q}^{s_1+\dots+s_\ell}} \\
\leq \| (m \circ \epsilon_{t_1 \cdots t_\ell}) \psi \|_{B_{p,q}^{s_1+\dots+s_\ell + (d-r)/p}}
\end{multline*}
by Propositions~\ref{prp:besovmultibesov} and \ref{prp:besovchange}, and the conclusion follows by taking the suprema.
\end{proof}

For instance, if we take $n_1=\dots=n_\ell=d=1$, then the map
\[(x_1,\dots,x_\ell) \mapsto x_1^2 + \dots + x_\ell^2\]
satisfies the hypotheses of Corollary~\ref{cor:mihlinmarcinkiewiczsum}, whereas
\[(x_1,\dots,x_\ell) \mapsto x_1\cdots x_\ell\]
satisfies the hypotheses of Proposition~\ref{prp:mihlinmarcinkiewicz}. Notice that, by Proposition~\ref{prp:homogeneousextension}, the map
\[(x_1,\dots,x_\ell) \mapsto x_1 + \dots + x_\ell\]
can be considered too, if one is only interested in its restriction to $\left[0,+\infty\right[^\ell$.

%% file: plancherel.tex
\chapter{Commutative algebras of differential operators}\label{chapter:plancherel}

Here we come to the heart of our work, i.e., the study of the properties of a joint functional calculus for a finite system of pairwise commuting, left-invariant differential operators on a connected Lie group. This will be the subject also of the following two chapters.

The present chapter is foundational and algebraic in character: on the one hand, we give --- following Nelson and Stinespring \cite{nelson_representation_1959} --- sufficient conditions for the existence of the functional calculus (in terms of a joint spectral resolution), and we introduce the language which is systematically used in the following; on the other hand, we study the relationships between the joint functional calculus and the algebraic structure of the group (representations, automorphisms...). It turns out that the choice of a particular family of pairwise commuting left-invariant differential operators extracts, from the general operator-valued non-commutative Fourier analysis of the Lie group, some scalar-valued commutative portion, which shares several features with the classical Euclidean Fourier theory.

Some of the results of this chapter are multi-variate analogues of results for a single operator, especially a Laplacian or a sublaplacian, which can be found in the literature; we refer mainly to the works of Hulanicki \cite{hulanicki_subalgebra_1974}, \cite{hulanicki_commutative_1975}, \cite{hulanicki_functional_1984}, Christ \cite{christ_multipliers_1991}, Ludwig and M\"uller \cite{ludwig_sub-laplacians_2000}.

Another source of inspiration for this chapter is the theory of Gelfand pairs on Lie groups, where the commutative algebra of differential operators to be considered is determined by the action of a compact group of automorphisms. In fact, at the end of the chapter, we show how Gelfand pairs (at least, those in semidirect-product form) fit into our wider framework.

Although the next chapters will focus on groups with polynomial growth, and especially homogeneous groups, here we try and not use hypotheses on the groups which are not really necessary, thus obtaining results which have quite a general character.

\section{Joint spectral resolution}
In the following, $G$ will be a connected Lie group.

\begin{lem}\label{lem:quasirockland}
Let $D,L \in \Diff(G)$ and suppose that $L$ is weighted subcoercive and formally self-adjoint. Then, for some $\bar r \in \N$, we have that, for all $r \geq \bar r$, $L^r + D$ is weighted subcoercive.
\end{lem}
\begin{proof}
Fix a reduced weighted algebraic basis of $\lie{g}$ with respect to which the operator $L$ is weighted subcoercive. Then there exists a weighted subcoercive form $C$ such that $d\RA_G(C) = L$, and also a form $B$ such that $d\RA_G(B) = D$. In fact, since $L^+ = L$, we can suppose that $C^+ = C$.

Let then $P$ be the principal part of $C$, so that, by Theorem~\ref{thm:robinsonterelst}, $d\RA_{G_*}(P)$ is Rockland. By definition, this implies that, for every $r \in \N \setminus \{0\}$, $P^r$ is Rockland too. Notice now that, if $r$ is sufficiently large so that $P^r$ has degree greater than that of $B$, then the principal part of $C^r + B$ is $P^r$ and this implies, by Theorem~\ref{thm:robinsonterelst} again, that $L^r + D = d\RA_G(C^r + B)$ is weighted subcoercive.
\end{proof}

\begin{prp}\label{prp:commutativealgebra}
Let $\Alg$ be a commutative unital subalgebra of $\Diff(G)$ closed by formal adjunction and containing a weighted subcoercive operator. Then, for every unitary representation $\pi$ of $G$, we have
\begin{equation}\label{eq:determinazioneaggiunto}
\overline{d\pi(D)} = d\pi(D^+)^* \qquad\text{for all $D \in \Alg$;}
\end{equation}
moreover, the operators $\overline{d\pi(D)}$ for $D \in \Alg$ are normal and commute strongly pairwise.
\end{prp}
\begin{proof}
Let $L \in \Alg$ be weighted subcoercive. Since $\Alg$ is closed by formal adjunction, by replacing $L$ with $(L + L^+)/2$, we can suppose that $L$ is formally self-adjoint (see Theorem~\ref{thm:robinsonterelst}).

Let $D \in \Alg$. By \eqref{eq:adjointrepresentation} and Lemma~2.3 of \cite{nelson_representation_1959}, in order to prove \eqref{eq:determinazioneaggiunto} it is sufficient to show that $d\pi(D^+ D)$ is essentially self-adjoint. However, by Lemma~\ref{lem:quasirockland}, it is possible to find $r \in \N$ sufficiently large so that both $A = L^{2r}$ and $C = L^{2r} + D^+ D$ are weighted subcoercive, which implies by Theorem~\ref{thm:robinsonterelst}(c) that $d\pi(A)$ and $d\pi(C)$ are essentially self-adjoint. The conclusion that $d\pi(D^+ D) = d\pi(C) - d\pi(A)$ is essentially self-adjoint then follows as in the proof of Corollary~2.4 of \cite{nelson_representation_1959}.

From \eqref{eq:determinazioneaggiunto} it follows that, for every formally self-adjoint $D \in \Alg$, $d\pi(D)$ is essentially self-adjoint. Let now
\[\mathcal{Q} = \{D^2 \tc D = D^+ \in \Alg\}.\]
For all $A,B \in \mathcal{Q}$, we have that $A, B, (1+A)(1+B)$ are formally self-adjoint elements of $\Alg$, so that $d\pi(A), d\pi(B), d\pi((1+A)(1+B))$ are essentially self-adjoint, and moreover $d\pi(A + B + AB)$ is positive (notice that $AB \in \mathcal{Q}$); this implies, as in the proof of Corollary~2.4 of \cite{nelson_representation_1959}, that $\overline{d\pi(A)}$ and $\overline{d\pi(B)}$ commute strongly.

In order to conclude, it will be sufficient to show that every operator of the form $\overline{d\pi(D)}$ for some $D \in \Alg$ is the joint function of some of the operators $\overline{d\pi(A)}$ for $A \in \mathcal{Q}$. In fact, let $D = D_1 + i D_2$, where
\[D_1 = (D+D^+)/2,\qquad D_2 = (D-D^+)/{2i}\]
are both formally self-adjoint elements of $\Alg$. Then
\[D_1^2, (D_1+1/2)^2, D_2^2, (D_2+1/2)^2\]
are all elements of $\mathcal{Q}$, and we can consider the joint spectral resolution $E$ on $\R^4$ of the corresponding operators in the representation $\pi$ (see \S\ref{section:spectraltheory}). We then have, for $j=1,2$,
\[d\pi(D_j) = d\pi((D_j+1/2)^2 - D_j^2 - 1/4) \subseteq \int_{\R^4} f_j \,dE,
\]
where $f_j(\lambda_{1,1},\lambda_{1,2},\lambda_{2,1},\lambda_{2,2}) = \lambda_{j,2}- \lambda_{j,1} - 1/4$, so that also
\[d\pi(D) \subseteq \int_{\R^4} (f_1 + if_2) \,dE, \qquad d\pi(D^+) \subseteq \int_{\R^4} (f_1 - if_2) \,dE;\]
by passing to the adjoints in the second inclusion and using \eqref{eq:determinazioneaggiunto}, we then get
\[\overline{d\pi(D)} = \int_{\R^4} (f_1 + if_2) \,dE,\]
and we are done.
\end{proof}

A system $L_1,\dots,L_n \in \Diff(G)$ will be called a \emph{weighted subcoercive system}\index{system of differential operators!weighted subcoercive} if $L_1,\dots,L_n$ are formally self-adjoint and pairwise commuting, and if moreover the unital subalgebra of $\Diff(G)$ generated by $L_1,\dots,L_n$ contains a weighted subcoercive operator (with respect to some reduced basis of $\lie{g}$). From the previous proposition and the spectral theorem (see \S\ref{section:spectraltheory}) we then have immediately

\begin{cor}\label{cor:commutativealgebra}
Let $L_1,\dots,L_n \in \Diff(G)$ be a weighted subcoercive system. For every unitary representation $\pi$ of $G$, the operators $\overline{d\pi(L_1)},\dots,\overline{d\pi(L_n)}$ admit a joint spectral resolution $E$ on $\R^n$ and, for every polynomial $p \in \C[X_1,\dots,X_n]$,
\begin{equation}\label{eq:polyweighted}
\overline{d\pi(p(L_1,\dots,L_n))} = \int_{\R^n} p \,dE.
\end{equation}
\end{cor}

The sign of closure for operators of the form \eqref{eq:polyweighted} will be generally omitted in the following.

\section{Kernel transform and Plancherel measure}\label{section:plancherel}

Let $L_1,\dots,L_n$ be a weighted subcoercive system on $G$. By applying Corollary~\ref{cor:commutativealgebra} to the (right) regular representation on $L^2(G)$, we obtain a joint spectral resolution $E$ of $L_1,\dots,L_n$. In particular, for every $f \in L^\infty(\R^n,E)$, we can consider the operator
\[f(L) = f(L_1,\dots,L_n) = \int_{\R^n} f\, dE,\]
which is a bounded left-invariant linear operator on $L^2(G)$, so that by Theorem~\ref{thm:schwartzkerneld} it admits a kernel $\breve f \in \Cv^2(G)$:
\[f(L)u = u * \breve f \qquad\text{for all $u \in \D(G)$.}\]
In place of $\breve f$, we use also the notation $\Kern_L f$. The correspondence
\[\Kern_L : f \mapsto \Kern_L f\]
will be called the \emph{kernel transform}\index{transform!kernel} associated to the weighted subcoercive system $L_1,\dots,L_n$.

\begin{lem}\label{lem:composition}
\begin{itemize}
\item[(a)] $\Kern_L$ is an isometric embedding of $L^\infty(\R^n,E)$ into $\Cv^2(G)$; in particular, for all $f \in L^\infty(\R^n,E)$,
\[\|\breve f\|_{\Cv^2} = \|f\|_{L^\infty(\R^n,E)}, \qquad \breve {\overline{f}} = (\breve f)^*.\]
\item[(b)] If $f, g \in L^\infty(\R^n,E)$ and $\breve g \in L^2(G)$, then
\[(fg)\breve{} = f(L) \breve g;\]
in particular, if $\breve g \in \D(G)$, then
\[(fg)\breve{} = \breve g * \breve f.\]
\item[(c)] If $f, g \in L^\infty(\R^n,E)$ and $g(\lambda) = \lambda_j f(\lambda)$ for some $j \in \{1,\dots,n\}$, then
\[\breve g = L_j \breve f\]
in the sense of distributions.
\end{itemize}
\end{lem}
\begin{proof}
(a) The conclusion follows immediately from the properties of the spectral integral (see \S\ref{section:spectralintegral}).

(b) Let $u \in \D(G)$; then
\[u * (fg)\breve{} = (fg)(L) u = f(L) g(L) u = f(L) (u * \breve g) = u * (f(L) \breve g)\]
since $f(L)$ is left-invariant and $\breve g \in L^2(G)$.

(c) Let $u \in \D(G)$; then
\[u * \breve g = g(L) u = L_j f(L) u = L_j (u * \breve f) = u * (L_j \breve f)\]
by the properties of convolution.
\end{proof}

The fact that the algebra generated by $L_1,\dots,L_n$ contains a weighted subcoercive operator implies that the kernel transform $\Kern_L$ satisfies more refined properties, which resemble those of an (inverse) Fourier transform.

In fact, by definition and formal self-adjointness, we can find a polynomial $p_*$ with real coefficients such that $p_*(L)$ is a weighted subcoercive operator. By replacing $p_*$ with $p_*^{2r}$ for some large $r \in \N$, we may suppose that $p_* \geq 0$ on $\R^n$ and that moreover, if we set
\[p_0(\lambda) = p_*(\lambda) + \sum_{j=1}^n \lambda_j^2 + 1,\]
\[p_k(\lambda) = p_0(\lambda) + \lambda_k \qquad\text{for $k=1,\dots,n$,}\]
then $p_0(L),p_1(L),\dots,p_n(L)$ are all weighted subcoercive (see Lemma~\ref{lem:quasirockland}). Notice that the polynomials $p_0,p_1,\dots,p_n$ are all strictly positive on $\R^n$ and
\[\lim_{\lambda \to \infty} p_k(\lambda) = +\infty \qquad\text{for $k=0,\dots,n$.}\]

\begin{lem}\label{lem:Jdensity}
The subalgebra of $C_0(\R^n)$ generated by the functions
\[e^{-p_0}, e^{-p_1}, \dots, e^{-p_n}.\]
is a dense $*$-subalgebra of $C_0(\R^n)$.
\end{lem}
\begin{proof}
Since the functions $e^{-p_0}, e^{-p_1}, \dots, e^{-p_n}$ are real valued, the algebra generated by them is a $*$-subalgebra of $C_0(\R^n)$.

Notice that $e^{-p_0}$ is nowhere null. Moreover, if $\lambda,\lambda' \in \R^n$ and $\lambda \neq \lambda'$, then we have two cases: either
\[e^{-p_0(\lambda)} \neq e^{-p_0(\lambda')},\]
or $e^{-p_0(\lambda)} = e^{-p_0(\lambda')}$, but in this case, if $k \in \{1,\dots,n\}$ is such that $\lambda_k \neq \lambda'_k$,
\[e^{-p_k(\lambda)} = e^{-p_0(\lambda)} e^{-\lambda_k} \neq e^{-p_0(\lambda')} e^{-\lambda'_k} = e^{-p_k(\lambda')}.\]
The conclusion then follows immediately by the Stone-Weierstrass theorem.
\end{proof}

Let now $\JJ_L$ be the subalgebra of $C_0(\R^n)$ generated by the functions of the form $e^{-q}$, where $q$ is a non-negative polynomial on $\R^n$ such that $q(L)$ is a weighted subcoercive operator on $G$ and $\lim_{\lambda \to \infty} q(\lambda) = +\infty$. Set moreover
\[C_0(L) = C_0(L_1,\dots,L_n) = \{\breve f \tc f \in C_0(\R^n)\}.\]
Finally, let $\Sigma$ be the joint spectrum\index{spectrum!joint spectrum} of $L_1,\dots,L_n$, i.e., the support of $E$.

\begin{prp}\label{prp:Jdensity}
$C_0(L)$ is a sub-$C^*$-algebra of $\Cv^2(G)$, which is isometrically isomorphic to $C_0(\Sigma)$ via the kernel transform. Moreover
\[\Kern_L(\JJ_L) = \{\breve f \tc f \in \JJ_L\}\]
is a dense $*$-subalgebra of $C_0(L)$.
\end{prp}
\begin{proof}
For the first part, see \S\ref{subsection:integralgelfand}. The second part follows immediately from Lemma~\ref{lem:Jdensity}.
\end{proof}

The results on weighted subcoercive operators and their heat kernels then imply that the elements of $\Kern_L(\JJ_L)$ are particularly well-behaved. The next proposition, which shows a sort of commutativity between joint functional calculus of $L_1,\dots,L_n$ and unitary representations of $G$, is a multi-variate analogue of Proposition~2.1 of \cite{ludwig_sub-laplacians_2000}.

\begin{prp}\label{prp:Jl1}
For every $f \in \JJ_L$, we have $\breve f \in L^{1;\infty}(G) \cap C^\infty_0(G)$ and moreover, for every unitary representation $\pi$ of $G$,
\[\pi(\breve f) = f(d\pi(L_1),\dots,d\pi(L_n)).\]
If $G$ is amenable, the last identity holds for every $f \in C_0(\R^n)$ with $\breve f \in L^1(G)$.
\end{prp}
\begin{proof}
Suppose first that $f$ is one of the generators $e^{-q}$ of $\JJ_L$. Then, by Corollary~\ref{cor:commutativealgebra} and the properties of the spectral integral (see \S\ref{subsection:pushforwardspectral}),
\[e^{-q}(d\pi(L_1),\dots,d\pi(L_n)) = e^{-d\pi(q(L))},\]
Since $q(L)$ is weighted subcoercive, we obtain from Theorem~\ref{thm:robinsonterelst}(d) that $\Kern_L(e^{-q}) \in L^{1;\infty} \cap C^\infty_0(G)$ and $e^{-q}(d\pi(L_1,\dots,L_n)) = \pi(\Kern_L(e^{-q}))$. The result is easily extended to every $f \in \JJ_L$ by Lemma~\ref{lem:composition}, the properties of convolution and those of the spectral integral.

Suppose now that $G$ is amenable, $f \in C_0(\R^n)$ and $\breve f \in L^1(G)$. By Proposition~\ref{prp:Jdensity}, we can find a sequence $f_j \in \JJ_L$ which converges uniformly to $f$ on $\R^n$. This implies in particular, by the properties of the spectral integral, that
\[f_j(d\pi(L_1),\dots,d\pi(L_n)) \to f(d\pi(L_1),\dots,d\pi(L_n))\]
in the operator norm, but also that $\breve f_j \to \breve f$ in $\Cv^2(G)$. Since $G$ is amenable, the representation $\pi$ is weakly contained in the regular representation (see \S\ref{subsection:amenability}), so that also
\[\pi(\breve f_j) \to \pi(\breve f)\]
in the operator norm. But then the conclusion follows immediately from the first part of the proof.
\end{proof}

We are now going to exploit the good properties of the kernels in $\Kern_L(\JJ_L)$ to obtain a Plancherel formula for the kernel transform $\Kern_L$. It should be noticed that, in the context of commutative Banach $*$-algebras, a general abstract argument yielding this kind of results is available (see \cite{loomis_introduction_1953}, \S26J). However, we believe that additional insight is provided by the explicit construction presented below, which follows essentially \cite{christ_multipliers_1991}, with some modifications due to our multi-variate and possibly non-unimodular setting.

\begin{prp}\label{prp:compactlysupported}
If $f \in L^\infty(\R^n,E)$ is compactly supported, then
\[\breve f \in L^{2;\infty} \cap C^\infty_0(G).\]
\end{prp}
\begin{proof}
Let $\xi_t = e^{-tp_*}$ for $t > 0$, so that $\breve \xi_t \in L^{1;\infty}(G) \cap C_0^\infty(G)$.

Since $f$ is compactly supported, $f = g \,\xi_1$ with $g = f/\xi_1 \in L^\infty(\R^n,E)$, so that
\[\breve f = g(L) \breve \xi_{1} \in L^2(G)\]
by Lemma~\ref{lem:composition}. Analogously, being $g$ compactly supported, also $\breve g \in L^2(G)$, but then
\[\breve f = \xi_1(L) \breve g = \breve g * \breve \xi_1 \in L^{2;\infty} \cap C^\infty_0(G),\]
by Lemma~\ref{lem:composition} and properties of convolution.
\end{proof}

Thus we have plenty of kernels $\breve f$ which are in $L^2(G)$; as we are going see, the $L^2$-norm can be interpreted as a certain operator norm of a convolution operator. Recall that $\|\cdot\|_{\hat 2}$ is the norm of the Lebesgue space $L^2(G,\Delta_G \mu_G)$ with respect to the left Haar measure; correspondingly, we denote by $\|\cdot\|_{\hat 2 \to \infty}$ the operator norm from $L^2(G,\Delta_G\mu_G)$ to $L^\infty(G)$.

\begin{lem}\label{lem:kernelnorm}
For all $f \in L^\infty(E)$, we have $\breve f \in L^2(G)$ if and only if
\[\|f(L)\|_{\hat 2 \to \infty} < \infty,\]
and in this case $\|\breve f\|_2 = \|f(L)\|_{\hat 2 \to \infty}$.
\end{lem}
\begin{proof}
If $\breve f \in L^2(G)$, then, by Young's inequality, $\|f(L)\|_{\hat 2 \to \infty} \leq \|\breve f\|_2 < \infty$.

Vice versa, suppose that $\|f(L)\|_{\hat 2 \to \infty} < \infty$. For $\phi \in \D(G)$, if
\[\check\phi(x) = \phi(x^{-1}),\]
then also $\check\phi \in \D(G)$, so that
\[f(L)\check\phi = \check\phi * \breve f\]
is continuous on $G$, therefore
\[|f(L)\check\phi(e)| \leq \|f(L)\check\phi\|_\infty \leq \|f(L)\|_{\hat 2 \to \infty} \|\phi\|_2.\]
This means that the map
\[\phi \mapsto f(L)\check\phi(e)\]
extends to a bounded linear functional on $L^2(G)$, thus there exists $k \in L^2(G)$ such that
\[f(L)\check\phi(e) = \int_G \phi(x) k(x) \,dx \qquad\text{for all $\phi \in \D(G)$,}\]
i.e.,
\[\phi * \breve f(e) = \int_G \phi(x^{-1}) k(x) \,dx \qquad\text{for all $\phi \in \D(G)$,}\]
but this, by the definition of convolution, means that the distribution $\breve f$ coincides with $k \in L^2(G)$. Moreover, since the norm of the linear functional is $\|k\|_2$, we have
\[\|\breve f\|_2 = \sup_{0 \neq \phi \in \D(G)} \frac{|f(L)\check\phi(e)|}{\|\phi\|_2} \leq \|f(L)\|_{\hat 2 \to \infty},\]
and we are done.
\end{proof}

We are now able to obtain a Plancherel formula for the kernel transform.

\begin{thm}\label{thm:plancherel}
The identity
\[\sigma(A) = \|E(A)\|_{\hat 2 \to \infty}^2 \qquad\text{for all Borel $A \subseteq \R^n$}\]
defines a regular Borel measure on $\R^n$ with support $\Sigma$, whose negligible sets coincide with those of $E$ and such that, for all $f \in L^\infty(E)$,
\[\int_{\R^n} |f|^2 \,d\sigma = \|f(L)\|_{\hat 2 \to \infty}^2 = \|\breve f\|_2^2.\]
\end{thm}
\begin{proof}
Clearly $\sigma(\emptyset) = 0$. Moreover, $\sigma$ is monotone: if $A \subseteq A'$ are Borel subsets of $\R^n$ and $\sigma(A') < \infty$, then, by Lemma~\ref{lem:kernelnorm}, $\breve\chr_{A'} \in L^2(G)$, so that, by Lemma~\ref{lem:composition}, also
\[\breve\chr_A = E(A) \breve\chr_{A'} \in L^2(G) \qquad\text{and}\qquad \|\breve \chr_A\|_2 \leq \|\breve\chr_{A'}\|_2,\]
i.e., $\sigma(A) \leq \sigma(A')$.

We now prove that $\sigma$ is finitely additive. Let then $A, B \subseteq \R^n$ be disjoint Borel sets. If $\sigma(A) = \infty$ or $\sigma(B) = \infty$, then by monotonicity we also have $\sigma(A \cup B) = \infty$ and we are done. Suppose instead that $\sigma(A) < \infty$ and $\sigma(B) < \infty$. Then, by Lemma~\ref{lem:kernelnorm}, both $\breve\chr_{A},\breve\chr_{B} \in L^2(G)$, but
\[E(A \cup B) = E(A) + E(B),\]
so that clearly
\[\breve\chr_{A \cup B} = \breve\chr_A + \breve\chr_B,\]
which implies that $\breve\chr_{A\cup B} \in L^2(G)$, and moreover, by Lemma~\ref{lem:kernelnorm},
\[\sigma(A \cup B) = \|\breve\chr_{A \cup B}\|_2^2 = \|\breve\chr_A\|_2^2 + \|\breve\chr_B\|_2^2 = \sigma(A) + \sigma(B),\]
since $\breve\chr_A = E(A)\breve\chr_A \perp E(B)\breve\chr_B = \breve\chr_B$ in $L^2(G)$ by Lemma~\ref{lem:composition}.

Finite additivity implies that, if $A_j$ ($j \in \N$) are pairwise disjoint Borel subsets of $\R^n$ and $A = \bigcup_j A_j$, then
\[\sum_j \sigma(A_j) \leq \sigma(A).\]
In particular, if the sum on the left-hand side diverges, then we have an equality. Suppose instead that the left-hand side sum converges. Then, by Lemmata~\ref{lem:kernelnorm} and \ref{lem:composition}, we have that the $\breve\chr_{A_j}$ are pairwise orthogonal elements of $L^2(G)$, and that their sum converges in $L^2(G)$ to some $k \in L^2(G)$ such that
\[\|k\|_2^2 = \sum_j \sigma(A_j).\]
But then, if $u \in \D(G)$, we have that, on one hand, by Lemma~\ref{lem:kernelnorm},
\[\sum_j u * \breve\chr_{A_j} = u * k \qquad\text{in $C_b(G)$,}\]
and, on the other hand,
\[\sum_j u * \breve\chr_{A_j} = \sum_j E(A_j) u = E(A)u \qquad\text{in $L^2(G)$,}\]
which gives, by uniqueness of limits, $E(A) u = u * k$. Thus $\breve\chr_A = k \in L^2(G)$ and
\[\sigma(A) = \|k\|_2^2 = \sum_j \sigma(A_j).\]

It is immediate from the definition that a Borel subset of $\R^n$ is $\sigma$-negligible if and only if it is $E$-negligible; in particular $\supp \sigma = \supp E = \Sigma$.

By Proposition~\ref{prp:compactlysupported}, $\sigma(A) = \|\chr_A(L)\|_{\hat 2 \to \infty}^2 = \|\breve\chr_A\|_2^2$ is finite if $A \subseteq \R^n$ is relatively compact. We can then conclude, by Theorem~2.18 of \cite{rudin_real_1974}, that $\sigma$ is regular.

Notice now that, for all Borel $A \subseteq \R^n$ with $\sigma(A) < \infty$, $\sigma$ coincides with the measure $\langle E(\cdot) \breve \chr_A, \breve \chr_A\rangle$ on the subsets of $A$: in fact, for all Borel $B \subseteq \R^n$,
\[\langle E(B) \breve\chr_A , \breve\chr_A \rangle = \|\breve\chr_{A \cap B}\|_2^2 = \sigma(A \cap B)\]
by Lemmata~\ref{lem:kernelnorm} and \ref{lem:composition}. In particular, for all $f \in L^\infty(E)$ with $\supp f \subseteq A$,
\[\int_{\R^n} |f|^2 \,d\sigma = \int_{\R^n} |f(\lambda)|^2 \,\langle E(d\lambda) \breve\chr_A, \breve\chr_A\rangle = \|f(L) \breve\chr_A\|_2^2 = \|\breve f\|_2^2 = \|f(L)\|_{\hat 2 \to \infty}^2\]
by the properties of the spectral integrals and Lemmata~\ref{lem:kernelnorm} and \ref{lem:composition}.

Take now a countable partition of $\R^n$ made of relatively compact Borel subsets $A_j$ ($j \in \N$). Then, for every $f \in L^\infty(\R^n,E)$, analogously as before we obtain\
\[\|f(L)\|_{\hat 2 \to \infty}^2 = \sum_j \|E(A_j) f(L)\|_{\hat 2 \to \infty}^2 = \sum_j \|\Kern_L(f \chr_{A_j})\|_2^2,\]
and putting all together we get the conclusion.
\end{proof}

The measure $\sigma$ of the previous proposition is called the \emph{Plancherel measure}\index{Plancherel measure!for a weighted subcoercive system} associated to the system $L_1,\dots,L_n$. Notice that
\[L^\infty(\R^n,E) = L^\infty(\sigma).\]

We show now that the estimates (for small times) on the heat kernel of weighted subcoercive operators give information on the behaviour at infinity of the Plancherel measure. Recall that $|\cdot|_2$ denotes the Euclidean norm.

\begin{prp}\label{prp:plancherelpolynomialgrowth}
The Plancherel measure $\sigma$ on $\R^n$ associated to a weighted subcoercive system $L_1,\dots,L_n$ has (at most) polynomial growth at infinity.
\end{prp}
\begin{proof}
If $\xi_t(\lambda) = e^{-t p_*(\lambda)}$, then, for every $r > 0$,
\[\sigma(\{p \leq r\}) = \|\chr_{\{p \leq r\}}\|_{L^2(\sigma)}^2 \leq e^2 \|\xi_{1/r}\|_{L^2(\sigma)}^2 = e^2 \|\breve \xi_{1/r}\|_{L^2(G)}^2.\]
Since $\breve \xi_t$ is the heat kernel of the operator $p(L_1,\dots,L_n)$, Theorem~\ref{thm:robinsonterelst}(e,f) gives, for large $r$,
\[\sigma(\{p \leq r\}) \leq C r^{Q_*/m},\]
where $m$ is the degree of $p(L_1,\dots,L_n)$ with respect to a suitable reduced weighted algebraic basis of $\lie{g}$, and $Q_*$ is the homogeneous dimension of the corresponding contraction $\lie{g}_*$. In particular, if $d$ is the degree of the polynomial $p$, we get, for large $a > 0$,
\[\sigma(\{\lambda \tc |\lambda|_2 \leq a\}) \leq \sigma(\{p \leq C(1 + a)^d\}) \leq C (1+a)^{Q_* d/m},\]
which is the conclusion.
\end{proof}

The proof of Proposition~\ref{prp:plancherelpolynomialgrowth} shows that the degree of growth at infinity of the Plancherel measure $\sigma$ is somehow related to the ``local dimension'' $Q_*$ of the group with respect to the control distance associated to the chosen weighted subcoercive operator (see \S\ref{subsection:controldistance}). In \S\ref{section:homogeneity} we will obtain more precise information on the behaviour of $\sigma$ under the hypothesis of homogeneity.

By Theorem~\ref{thm:plancherel}, $\Kern_L|_{L^2 \cap L^\infty(\sigma)}$ extends to an isometry from $L^2(\sigma)$ onto a closed subspace of $L^2(G)$. We give now an alternative characterization of this subspace.

\begin{lem}\label{lem:Jlq}
$\JJ_L$ is dense in $L^q(\sigma)$ for $1 \leq q < \infty$.
\end{lem}
\begin{proof}
Since $\sigma$ has polynomial growth at infinity (see Proposition~\ref{prp:plancherelpolynomialgrowth}), it is easily seen that the generators of $\JJ_L$ belong to $L^1 \cap L^\infty(\sigma)$, and therefore also $\JJ_L$ is contained (modulo restriction to $\Sigma$) in $L^1 \cap L^\infty(\sigma)$. Since $\sigma$ is a positive regular Borel measure on $\R^n$, in order to prove that the closure of $\JJ_L$ in $L^q(\sigma)$ is the whole $L^q(\sigma)$, it is sufficient to show that $C_c(\R^n)$ is contained in this closure (see \cite{rudin_real_1974}, Theorem~3.14).

Let then $m \in C_c(\R^n)$. By Lemma~\ref{lem:Jdensity}, we can find a sequence $m_k \in \JJ_L$ converging uniformly to $m$, so that $\sup_k \|m_k\|_\infty = C < \infty$. Thus, for every $t > 0$, $m_k e^{-tp_0}$ converges uniformly to $m e^{-tp_0}$, dominated by $ Ce^{-tp_0} \in L^q(\sigma)$, and consequently $m_k e^{-tp_0} \to m e^{-tp_0}$ also in $L^q(\sigma)$; we then have that $m e^{-tp_0}$ is in the closure of $\JJ_L$ in $L^q(\sigma)$ for all $t > 0$, and by monotone convergence also $m$ is in this closure.
\end{proof}

Let $\IS_L$ be the closure of $\Kern_L(\JJ_L)$ in $L^2(G)$.

\begin{prp}\label{prp:Jl2}
$\Kern_L|_{L^2 \cap L^\infty(\sigma)}$ extends to an isometric isomorphism
\[L^2(\sigma) \to \IS_L.\]
\end{prp}
\begin{proof}
By Theorem~\ref{thm:plancherel} and Lemma~\ref{lem:Jlq}, we have
\[\Kern_L(\JJ_L) \subseteq \Kern_L(L^2 \cap L^\infty(\sigma)) \subseteq \overline{\Kern_L(\JJ_L)}^{L^2(G)} = \IS_L,\]
and the conclusion follows.
\end{proof}

We now prove a sort of Riemann-Lebesgue lemma for $\Kern_L^{-1}$.

\begin{prp}\label{prp:riemannlebesgue1}
For every bounded Borel $f : \R^n \to \C$ with $\breve f \in L^1(G)$, we have
\[\|f\|_{L^\infty(\sigma)} \leq \|\breve f\|_1,\]
and moreover
\[\lim_{r \to +\infty} \| f \, \chr_{\{\lambda \tc |\lambda|_2 \geq r\}} \|_{L^\infty(\sigma)} = 0.\]
\end{prp}
\begin{proof}
The inequality follows immediately from Lemma~\ref{lem:composition}(a) and \eqref{eq:cvl1}.

Let $\xi_t = e^{-t p_0}$. Then, by Corollary~\ref{cor:kernelapproximateidentity}, $\breve\xi_t$ is an approximate identity for $t \to 0^+$. In particular, by Proposition~\ref{prp:approximateidentity}, if $\breve f \in L^1(G)$, then
\[\Kern_L(f \xi_t) = \breve f * \breve \xi_t \to \breve f \qquad\text{in $L^1(G)$}\]
for $t \to 0^+$, which implies, by the first inequality, that
\[\lim_{t \to 0^+} \|f (1 - \xi_t) \|_{L^\infty(\sigma)} = 0.\]
Therefore, for every $\varepsilon > 0$, there exists $t > 0$ such that $\|f (1 - \xi_t) \|_{L^\infty(\sigma)} \leq \varepsilon$; since $p_0(\lambda) \to +\infty$ for $\lambda \to \infty$, we may find $r > 0$  such that
\[\|\xi_t \, \chr_{\{\lambda \tc |\lambda|_2 \geq r\}}\|_\infty \leq 1/2,\]
but then necessarily $\|f \, \chr_{\{\lambda \tc |\lambda|_2 \geq r\}}\|_\infty \leq 2\varepsilon$.
\end{proof}

An analogous (and neater) result for $\Kern_L$ is obtained under the additional hypothesis of unimodularity.

\begin{prp}\label{prp:riemannlebesgue2}
Suppose that $G$ is unimodular. If $f \in L^1 \cap L^\infty(\sigma)$, then $\breve f \in C_0(G)$ and
\[\|\breve f\|_\infty \leq \|f\|_{L^1(\sigma)}.\]
\end{prp}
\begin{proof}
Since $f \in L^1 \cap L^\infty(\sigma)$, then $f = g_1 g_2$ for some Borel $g_1,g_2 : \R^n \to \C$ such that
\[|g_1|^2 = |g_2|^2 = |f|;\]
in particular, $g_1,g_2 \in L^2 \cap L^\infty(\sigma)$. Therefore $\breve g_1, \breve g_2 \in L^2(G)$ by Theorem \ref{thm:plancherel}, but then
\[\breve f = \breve g_1 * \breve g_2\]
by Lemma \ref{lem:composition}, which implies that $\breve f \in C_0(G)$ and
\[\|\breve f\|_\infty \leq \|\breve g_1\|_2 \|\breve g_2\|_2 = \|f\|_{L^1(\sigma)},\]
which is the conclusion.
\end{proof}

\section{Spectrum and eigenfunctions}\label{section:eigenfunctions}

We keep the notation of the previous section. Proposition~\ref{prp:Jl1} shows that there exists some relationship between the $L^2$ spectral theory of $L_1,\dots,L_n$ on $G$ and the spectral theory of the corresponding operators in a unitary representation $\pi$ of $G$. On the other hand, a joint eigenvector of $d\pi(L_1),\dots,d\pi(L_n)$ in a unitary representation $\pi$ yields (as we shall see) a joint eigenfunction of $L_1,\dots,L_n$ on $G$, which in general is smooth and bounded but (unless $G$ is compact) not in $L^2(G)$. We are now going to study such generalized \index{function!joint eigenfunction}joint eigenfunctions of $L_1,\dots,L_n$ and the corresponding generalized joint point spectrum.

\begin{prp}\label{prp:weakstrongeigenfunctions}
Let $\phi \in \D'(G)$ be such that, for some $(\lambda_1,\dots,\lambda_n) \in \C^n$,
\[L_j \phi = \lambda_j \phi \qquad\text{for $j=1,\dots,n$}\]
in the sense of distributions. Then $\phi \in \E(G)$, and the previous equalities hold in the strong sense.
\end{prp}
\begin{proof}
If $\lambda = (\lambda_1,\dots,\lambda_n)$, from the hypothesis we get immediately
\[p_*(L) \phi = p_*(\lambda) \phi.\]
Since $p_*(L) - p_*(\lambda)$ is hypoelliptic by Corollary~\ref{cor:hypoelliptic}, this implies that $\phi \in \E(G)$ and we are done.
\end{proof}

\begin{lem}\label{lem:eigenvectors}
Let $\pi$ be a unitary representation of $G$ on $\HH$. The following are equivalent for $v \in \HH \setminus \{0\}$:
\begin{itemize}
\item[(i)] $v \in \HH^\infty$ and $v$ is a joint eigenvector of $d\pi(L_1),\dots,d\pi(L_n)$;
\item[(ii)] $v$ is a joint eigenvector of the operators $\pi(\breve m)$ for $m \in \JJ_L$.
\end{itemize}
\end{lem}
\begin{proof}
Suppose that $v \in \HH^\infty \setminus \{0\}$ is a joint eigenvector of the operators $d\pi(L_1),\dots,d\pi(L_n)$, i.e.,
\[d\pi(L_j) v = c_j v\]
for some $c = (c_1,\dots,c_n) \in \R^n$, $j = 1,\dots,n$. By Proposition~\ref{prp:Jl1} and the properties of the spectral integral (see \S\ref{subsection:spectralunbounded}), for every $m \in \JJ_L$ we have
\[\pi(\breve m) v = m(d\pi(L_1),\dots,d\pi(L_n)) v = m(c) v,\]
which means that $v$ is an eigenvector of $\pi(\breve m)$.

Vice versa, suppose that $v \in \HH \setminus \{0\}$ is an eigenvector of $\pi(\breve m)$ for all $m \in \JJ_L$. Take $m = e^{-p_j}$ for $j=0,\dots,n$; by Proposition~\ref{prp:Jl1}, we have
\[\pi(\breve m) = e^{-p_j(d\pi(L))},\]
so that, by the properties of the spectral integral (see \S\ref{subsection:pushforwardspectral}), $\ker \pi(\breve m) = \{0\}$, therefore
\[\pi(\breve m) v = c v\]
for some $c > 0$. This implies that
\[v = c^{-1} \pi(\breve m) v \in \HH^\infty,\]
by Theorem~\ref{thm:robinsonterelst}(b), and moreover, again by the properties of the spectral integral,
\[p_j(d\pi(L)) v = (\log c) v,\]
that is, $v$ is an eigenvector of $p_j(d\pi(L))$ for $j=0,\dots,n$. Since
\[\lambda_j = p_j(\lambda) - p_0(\lambda) \qquad\text{for $j=1,\dots,n$,}\]
it follows that $v$ is a joint eigenvector of $d\pi(L_1),\dots,d\pi(L_n)$.
\end{proof}

\begin{prp}\label{prp:eigenfunctions}
For a function of positive type $\phi$ on $G$, the following are equivalent:
\begin{itemize}
\item[(i)] $\phi$ is a joint eigenfunction of $L_1,\dots,L_n$ and $\phi(e) = 1$;
\item[(ii)] $\phi$ has the form
\[\phi(x) = \langle \pi(x) v, v \rangle\]
for some unitary representation $\pi$ of $G$ on $\HH$ and cyclic vector $v$ of norm $1$, where $v \in \HH^\infty$ is a joint eigenvector of $d\pi(L_1),\dots,d\pi(L_n)$;
\item[(iii)] $\phi \neq 0$ and, for all $m \in \JJ_L$ and $f \in L^1(G)$,
\[\langle \breve m * f, \phi \rangle = \langle f * \breve m, \phi \rangle = \langle f, \phi \rangle \langle \breve m, \phi \rangle;\]
\item[(iv)] $\phi \neq 0$ and, for all $m \in \JJ_L$,
\[\langle \breve m * \breve m^*, \phi \rangle = |\langle \breve m, \phi \rangle|^2.\]
\end{itemize}
In this case, moreover, the eigenvalue of $L_j$ corresponding to $\phi$ is a real number and coincides with the eigenvalue of $d\pi(L_j)$ corresponding to $v$.
\end{prp}
\begin{proof}
(i) $\Rightarrow$ (ii). Since $\phi$ is of positive type and $\phi(e) = 1$, then $\phi$ is of the form
\[\phi(x) = \langle \pi(x) v, v \rangle\]
for some unitary representation $\pi$ of $G$ on $\HH$ and cyclic vector $v$ of norm $1$. From (i) we have
\[L_j \phi = \lambda_j \phi\]
for some $\lambda = (\lambda_1,\dots,\lambda_n) \in \C^n$. Being $L_1,\dots,L_n$ left-invariant, if
\[\phi_y(x) = \langle \pi(y^{-1} x) v, v \rangle = \langle \pi(x) v, \pi(y) v \rangle,\]
then also
\[L_j \phi_y = \lambda_j \phi_y.\]
Since $v$ is cyclic, for all $w \in \HH$ we can find a sequence $(w_n)_n$ in
\[\Span\{\pi(y) v \tc y \in G\}\]
such that $w_n \to w$ in $\HH$; if
\[\psi_n(x) = \langle \pi(x) v, w_n \rangle, \qquad \psi(x) = \langle \pi(x) v, w \rangle,\]
then the $\psi_n$ are linear combinations of the $\phi_y$, so that
\[L_j \psi_n = \lambda_j \psi_n\]
and, passing to the limit, we also have
\[L_j \psi = \lambda_j \psi\]
in the sense of distributions. But then $\psi \in \E(G)$ by Proposition~\ref{prp:weakstrongeigenfunctions}. Since $w \in \HH$ was arbitrary, we conclude that $v \in \HH^\infty$; moreover
\[\langle \lambda_j v, w \rangle = \lambda_j \psi(e) = L_j \psi(e) = \langle d\pi(L_j) v, w \rangle,\]
and again, from the arbitrariness of $w$, we get $d\pi(L_j) v = \lambda_j v$ for $j=1,\dots,n$. Finally, since $d\pi(L_j)$ is self-adjoint, we deduce that $\lambda_j \in \R$.

(ii) $\Rightarrow$ (i). Trivial.

(ii) $\Rightarrow$ (iii). If $m \in \JJ_L$, by Lemma~\ref{lem:eigenvectors}, $\pi(\breve m)^* v = \pi(\breve{\overline{m}}) v = c v$ for some $c \in \C$. Since $\|v\| = 1$, we have
\begin{multline*}
\langle f * \breve m, \phi \rangle = \langle \pi(f * \breve m) v, v \rangle = \langle \pi(\breve m) \pi(f) v, v \rangle = \overline{c} \langle \pi(f) v,v \rangle \\
= \langle \pi(f) v,v \rangle \langle \pi(\breve m) v, v \rangle = \langle f, \phi \rangle \langle \breve m, \phi \rangle.
\end{multline*}
The other identity is proved analogously.

(iii) $\Rightarrow$ (iv). Trivial.

(iv) $\Rightarrow$ (ii). Being of positive type, $\phi$ has the form
\[\phi(x) = \langle \pi(x) v, v \rangle\]
for some unitary representation $\pi$ of $G$ on $\HH$ and cyclic vector $v$. Then (iv) can be equivalently rewritten as
\begin{equation}\label{eq:multiplicativenorm}
\| \pi(\breve m) v \| = |\langle \pi(\breve m) v, v \rangle|
\end{equation}
for all $m \in \JJ_L$. In particular, by taking $m = e^{-tp_*}$, and passing to the limit for $t \to 0^+$, we obtain
\[\|v\| = \|v\|^2\]
(see Corollary~\ref{cor:kernelapproximateidentity} and Proposition~\ref{prp:approximateidentity}), so that $\|v\| = 1$ (since $\phi \neq 0$). Now, for an arbitrary $m \in \JJ_L$, \eqref{eq:multiplicativenorm} implies that $\pi(\breve m) v$ cannot have a component orthogonal to $v$, thus $v$ is an eigenvector of $\pi(\breve m)$, and (ii) follows from Lemma~\ref{lem:eigenvectors}.
\end{proof}

Let $\PP_L$ be the set of the joint eigenfunctions $\phi$ of $L_1,\dots,L_n$ of positive type with $\phi(e) = 1$. For every $\phi \in \PP_L$, by Proposition~\ref{prp:eigenfunctions} there exists a unique $\lambda = (\lambda_1,\dots,\lambda_n) \in \R^n$ such that
\[L_j \phi = \lambda_j \phi;\]
we then define $\evmap_L : \PP_L \to \R^n$ by setting $\evmap_L(\phi) = \lambda$.

\begin{prp}\label{prp:eigenvalues}
For every $m \in \JJ_L$ and $\phi \in \PP_L$, we have
\begin{equation}\label{eq:eigenvalues}
\phi * \breve m = m(\evmap_L(\phi)) \phi \qquad\text{and}\qquad \langle \breve m, \phi \rangle = m(\evmap_L(\phi)).
\end{equation}
If $G$ is amenable, then the previous identities hold also for every $m \in C_0(\R^n)$ such that $\breve m \in L^1(G)$.
\end{prp}
\begin{proof}
Let $\phi(x) = \langle \pi(x) v, v \rangle$ for some unitary representation $\pi$ of $G$ on $\HH$ and cyclic vector $v$ of norm $1$. If $\evmap_L(\phi) = (\lambda_1,\dots,\lambda_n)$, by Proposition~\ref{prp:eigenfunctions} we have
\[d\pi(L_j) v = \lambda_j v,\]
so that, by Proposition~\ref{prp:Jl1} and the properties of the spectral integral,
\[\pi(\breve m) v = m(d\pi(L_1),\dots,d\pi(L_n)) v = m(\evmap_L(\phi)) v.\]
We then have
\[\phi * \breve m(x) = \int_G \langle \pi(x) \pi(y^{-1}) v, v \rangle \breve m(x) \,dx = \langle \pi(x) \pi(\breve m) v, v \rangle = m(\evmap_L(\phi)) \phi(x),\]
which is the former identity of \eqref{eq:eigenvalues}; by evaluating in $x = e$ we also get the latter.
\end{proof}

\begin{cor}\label{cor:continuouseigenvalues}
If $\PP_L$ is endowed with the topology induced by the weak-$*$ topology of $L^\infty(G)$, then the map $\evmap_L : \PP_L \to \R^n$ is continuous.
\end{cor}
\begin{proof}
By Proposition~\ref{prp:eigenvalues}, for $j=0,\dots,n$, we have that
\[e^{-p_j(\evmap_L(\phi))} = \langle \Kern_L (e^{-p_j}), \phi \rangle,\]
which is continuous in $\phi$ with respect to the weak-$*$ topology of $L^\infty(G)$. If $\evmap_{L,1},\dots,\evmap_{L,n} : \PP_L \to \R$ are the components of $\evmap_L$, then
\[e^{-\evmap_{L,j}(\phi)} = e^{-p_j(\evmap_L(\phi))}/e^{-p_0(\evmap_L(\phi))}.\]
This shows that the components of $\evmap_L$ are continuous $\PP_L \to \R$, so that also $\evmap_L$ is continuous.
\end{proof}

\begin{prp}\label{prp:eigentopologies}
The topologies on $\PP_L$ induced by the weak-$*$ topology of $L^\infty(G)$, the compact-open topology of $C(G)$ and the topology of $\E(G)$ coincide. Moreover, the map $\evmap_L : \PP_L \to \R^n$ is a continuous, proper and closed map. In particular, the image $\evmap_L(\PP_L)$ is a closed subset of $\R^n$ and its topology as a subspace of $\R^n$ coincides with the quotient topology induced by $\evmap_L$.
\end{prp}
\begin{proof}
Since $G$ is second-countable, the three aforementioned topologies on $\PP_L$ are all metrizable (cf.\ \cite{megginson_introduction_1998}, Corollary~2.6.20). In particular, in order to prove that they coincide, it is sufficient to show that they induce the same notion of convergence of sequences.

Let $(\phi_k)_k$ be a sequence in $\PP_L$. If $(\phi_k)_k$ converges in $\E(G)$, then a fortiori it converges in $C(G)$. Moreover, since $\|\phi_k\|_\infty = 1$ for all $k$, convergence in $C(G)$ implies weak-$*$ convergence in $L^\infty(G)$ by dominated convergence.

Suppose now that $\phi_k \to \phi \in \PP_L$ with respect to the weak-$*$ topology of $L^\infty(G)$. Take $m = e^{-p_*} \in \JJ_L$, so that $m > 0$. By Proposition~\ref{prp:eigenvalues}, for all $D \in \Diff(G)$, we then have
\[D\phi_k = \frac{\phi_k * D \breve m}{m(\evmap_L(\phi_k))}, \qquad D\phi = \frac{\phi * D \breve m}{m(\evmap_L(\phi))};\]
in particular, for every $x \in G$, since $\RA_x D \breve m \in L^1(G)$,
\[D\phi_k(x) = \frac{\langle \RA_x D \breve m, \phi_k \rangle}{m(\evmap_L(\phi_k))} \to \frac{\langle \RA_x D \breve m, \phi \rangle}{m(\evmap_L(\phi))} = D\phi(x)\]
by Corollary~\ref{cor:continuouseigenvalues}. Moreover, again by Corollary~\ref{cor:continuouseigenvalues},
\[m(\evmap_L(\phi_k)) \geq c > 0\]
for some $c$ and all $k$, so that
\[\|D\phi_k\|_\infty \leq c^{-1} \|D\breve m\|_1.\]
This means that, for all $D \in \Diff(G)$, the family $\{D\phi_k\}_k$ is equibounded; but then also, for all $D \in \Diff(G)$, the family $\{D\phi_k\}_k$ is equicontinuous, so that the previously proved pointwise convergence $D\phi_k \to D\phi$ is in fact uniform on compacta. By arbitrariness of $D \in \Diff(G)$, we have then proved that $\phi_k \to \phi$ in $\E(G)$.

Let now $K \subseteq \R^n$ be compact, and take a sequence $(\phi_k)_k$ in $\PP_L$ such that $\evmap_L(\phi_k) \in K$ for all $k$. As before, the sequence $(\phi_k)_k$ is equibounded and equicontinuous, so that, by the Ascoli-Arzel\`a theorem (see \cite{bourbaki_topology2}, \S X.2.5), we can find a subsequence $\phi_{k_h}$ which converges uniformly on compacta to a function $\phi \in C(G)$, and such that moreover $\evmap_L(\phi_{k_h})$ converges to some $\lambda \in K$. It is now easy to show that $\phi$ is of positive type and $\phi(e) = 1$; moreover, for all $\eta \in \D(G)$,
\[\langle L_j \phi, \eta \rangle = \lim_h \langle L_j \phi_{k_h}, \eta \rangle = \lim_h \evmap_{L,j}(\phi_{k_h}) \langle \phi_{k_h}, \eta \rangle = \lambda_j \langle \phi, \eta \rangle,\]
so that
\[L_j \phi = \lambda_j \phi\]
in the sense of distributions, but then, by Proposition~\ref{prp:weakstrongeigenfunctions}, $\phi$ is smooth and is a joint eigenfunction of $L_1,\dots,L_n$, so that $\phi \in \PP_L$. Since $\PP_L$ is metrizable, this shows that $\evmap_L^{-1}(K)$ is compact in $\PP_L$. By the arbitrariness of the compact $K \subseteq \R^n$, we conclude that $\evmap_L$ is proper and closed (see \cite{bourbaki_topology1}, Propositions I.10.1 and I.10.7).
\end{proof}

\begin{prp}\label{prp:eigenextremals}
For $\lambda \in \R^n$, the set $\evmap_L^{-1}(\lambda)$ is a weakly-$*$ compact and convex subset of $L^\infty(G)$, whose extreme points are the ones associated with irreducible representations.
\end{prp}
\begin{proof}
It is immediate to show that $\evmap_L^{-1}(\lambda)$ is convex, whereas compactness follows from Proposition~\ref{prp:eigentopologies}.

Let $\PP_1$ be the set of functions $\phi$ of positive type on $G$ such that $\phi(e) = 1$. Since $\evmap_L^{-1}(\lambda) \subseteq \PP_1$, in order to conclude
it will be sufficient to show that the extreme points of $\evmap_L^{-1}(\lambda)$ are also extreme points of $\PP_1$ (see \S\ref{subsection:positivetype}).

Suppose then that $\phi \in \evmap_L^{-1}(\lambda)$ is not extreme in $\PP_1$, so that
\[\phi = \theta_0^2 \phi_0 + \theta_1^2 \phi_1\]
for some $\phi_0,\phi_1 \in \PP_1$ different from $\phi$ and some $\theta_0,\theta_1 > 0$ such that $\theta_0^2 + \theta_1^2 = 1$. For $k=0,1$,
\[\phi_k(x) = \langle \pi_k(x) v_k, v_k \rangle \qquad\text{for $k=0,1$,}\]
where $\pi_k$ is a unitary representation of $G$ on $\HH_k$ and $v_k$ is a cyclic vector of norm $1$. If
\[v = (\theta_0 v_0, \theta_1 v_1) \in \HH_0 \oplus \HH_1,\]
\[\HH = \overline{\Span\{(\pi_0 \oplus \pi_1)(x) v \tc x \in G\}},\]
and $\pi$ is the restriction of $\pi_0 \oplus \pi_1$ to $\HH$, then it is easy to see that $v$ is a cyclic vector for $\pi$ and that
\[\phi(x) = \langle \pi(x) v, v \rangle,\]
therefore by Proposition~\ref{prp:eigenfunctions} it follows that $v \in \HH^\infty$ and that $d\pi(L_j) v = \lambda_j v$ for $j=1,\dots,n$.

If $P_k : \HH \to \HH_k$ is the restriction of the canonical projection $\HH_0 \oplus \HH_1 \to \HH_k$, it is immediate to check that $P_k$ intertwines $\pi$ and $\pi_k$, and that $P_k v = \theta_k v_k$; hence, for all $w \in \HH_k$ and $x \in G$,
\[\langle \pi_k(x) v_k, w \rangle = \theta_k^{-1} \langle \pi_k(x) P_k v , w \rangle = \theta_k^{-1} \langle \pi(x) v, P_k^* w \rangle.\]
This identity, together with the arbitrariness of $w \in \HH_k$, shows that $v_k \in \HH_k^\infty$. Moreover, since $P_k$ intertwines $\pi(x)$ and $\pi_k(x)$ for all $x \in G$, it is easy to check that it intertwines also $d\pi(D)$ and $d\pi_k(D)$ for all $D \in \Diff(G)$, therefore
\[d\pi(L_j) v_k = \theta_k^{-1} P_k d\pi(L_j) v = \lambda_j v_k\]
for $j=1,\dots,n$. By Proposition~\ref{prp:eigenfunctions}, this shows that $\phi_0,\phi_1 \in \evmap_L^{-1}(\lambda)$, thus $\phi$ is not even extreme in $\evmap_L^{-1}(\lambda)$.
\end{proof}

The main algebraic and topological properties of the joint eigenfunctions of positive type of $L_1,\dots,L_n$ have been established. Now we are interested in determining the relationships between the generalized joint point spectrum $\evmap_L(\PP_L)$ --- which is, by Propositions~\ref{prp:eigenfunctions} and \ref{prp:eigenextremals}, the union of the joint point spectra of $d\pi(L_1),\dots,d\pi(L_n)$ for all (irreducible) unitary representations $\pi$ of $G$ --- and the joint $L^2(G)$ spectrum $\Sigma$ of $L_1,\dots,L_n$ --- which is canonically identified to the Gelfand spectrum $\GS(C_0(L))$ of the C$^*$-algebra $C_0(L)$ (see \S\ref{subsection:integralgelfand}). In fact, we will see that $\PP_L$ contains information about the joint spectrum of $d\pi(L_1),\dots,d\pi(L_n)$ for every unitary representation $\pi$ of $G$.

\begin{lem}\label{lem:weakcontainment}
Let $\varpi$ be a unitary representation of $G$. Let moreover $\phi$ be a function of positive type, of the form
\[\phi(x) = \langle \pi(x) v, v \rangle\]
for some unitary representation $\pi$ of $G$ on the Hilbert space $\HH$ and cyclic vector $v$ of unit norm. Then the following are equivalent:
\begin{itemize}
\item[(i)] $\pi$ is weakly contained in $\varpi$;
\item[(ii)] $|\langle f, \phi \rangle| \leq \|\varpi(f)\|$ for all $f \in L^1(G)$;
\item[(iii)] $|\langle f, \phi \rangle| \leq C \|\varpi(f)\|$ for some $C > 0$ and all $f \in L^1(G)$.
\end{itemize}
\end{lem}
\begin{proof}
(i) $\Rightarrow$ (ii). Since
\[\langle f, \phi \rangle = \langle \pi(f) v, v \rangle\]
and $v$ has unit norm, the conclusion follows immediately from the definition of weak containment.

(ii) $\Rightarrow$ (iii). Trivial.

(iii) $\Rightarrow$ (i). Let $\widetilde\HH$ be the Hilbert space on which $\varpi$ acts. The hypothesis (iii) implies that $\phi$ defines a (positive) continuous functional on the sub-C$^*$-algebra of $\Bdd(\widetilde\HH)$ which is the closure of $\varpi(L^1(G))$. By applying Proposition~2.1.5(ii) of \cite{dixmier_algebras_1982} to this functional, one obtains, for $f,g \in L^1(G)$,
\[\|\pi(f) \pi(g) v\|^2 = \langle g * f* f^* * g^* , \phi \rangle \leq \|\varpi(f*f^*)\| \langle g * g^*, \phi \rangle = \|\varpi(f)\|^2 \|\pi(g)v\|^2.\]
Since $v$ is cyclic and $L^1(G)$ contains an approximate identity, the set
\[\{\pi(g) v \tc g \in L^1(G)\}\]
is a dense subspace of $\HH$, therefore the previously proved inequality gives
\[\|\pi(f)\| \leq \|\varpi(f)\|,\]
and we are done.
\end{proof}

For a unitary representation $\varpi$ of $G$, we denote by $\PP_{L,\varpi}$ the set of the functions $\phi \in \PP_L$ which satisfy the equivalent conditions of Lemma~\ref{lem:weakcontainment}.

\begin{prp}
Let $\varpi$ be a unitary representation of $G$. Then $\PP_{L,\varpi}$ is a closed subset of $\PP_L$. Moreover, for every $\lambda \in \R^n$, $\PP_{L,\varpi} \cap \evmap_L^{-1}(\lambda)$ is compact and convex, and its extreme points are the ones corresponding to irreducible representations.
\end{prp}
\begin{proof}
Condition (ii) of Lemma~\ref{lem:weakcontainment} is a convex and closed condition (with respect to the weak-$*$ topology of $L^\infty(G)$) for every $f \in L^1(G)$. Therefore $\PP_{L,\varpi}$ is closed in $\PP_L$, and moreover, for $\lambda \in \R^n$, since $\evmap_L^{-1}(\lambda)$ is compact and convex (see Proposition~\ref{prp:eigenextremals}), $\PP_{L,\varpi} \cap \evmap_L^{-1}(\lambda)$ is compact and convex too.

In order to conclude, again by Proposition~\ref{prp:eigenextremals}, it is sufficient to show that an extreme point $\phi$ of $\PP_{L,\varpi} \cap \evmap_L^{-1}(\lambda)$ is also extreme in $\evmap_L^{-1}(\lambda)$. Suppose then that $\phi = (1-\theta)\phi_0 + \theta \phi_1$ for some $\phi_0,\phi_1 \in \evmap_L^{-1}(\lambda)$ and $0 < \theta < 1$. For $f \in L^1(G)$, we have
\[(1-\theta) |\langle f, \phi_0\rangle|^2 + \theta |\langle f, \phi_1\rangle|^2 = \langle f * f^* , \phi \rangle \leq \|\varpi(f)\|^2\]
by Lemma~\ref{lem:weakcontainment} and positivity, therefore
\[|\langle f, \phi_0 \rangle| \leq (1-\theta)^{-1/2} \|\varpi(f)\|, \qquad |\langle f, \phi_1 \rangle| \leq \theta^{-1/2} \|\varpi(f)\|,\]
and again by Lemma~\ref{lem:weakcontainment} we obtain $\phi_0,\phi_1 \in \PP_{L,\varpi} \cap \evmap_L^{-1}(\lambda)$.
\end{proof}

\begin{thm}\label{thm:spectrumeigenfunctions}
Let $\pi$ be a unitary representation of $G$ on a Hilbert space $\HH$. Then $\evmap_L(\PP_{L,\pi})$ is the joint spectrum of $d\pi(L_1),\dots,d\pi(L_n)$ on $\HH$.
\end{thm}
\begin{proof}
Let $E_\pi$ be the joint spectral resolution of $d\pi(L_1),\dots,d\pi(L_n)$. The joint spectrum of $d\pi(L_1),\dots,d\pi(L_n)$, i.e., the support of $E_\pi$, can be identified with the Gelfand spectrum of the C$^*$-algebra $E_\pi[C_0(\R^n)]$ (see \S\ref{subsection:integralgelfand}), i.e., with the closure in $\Bdd(\HH)$ of
\[\{\pi(\breve m) \tc m \in \JJ_L\}\]
(see Lemma~\ref{lem:Jdensity} and Proposition~\ref{prp:Jl1}).

In particular, if $\phi \in \PP_{L,\pi}$, then, by Lemma~\ref{lem:weakcontainment},
\[|\langle \breve m, \phi \rangle| \leq \|\pi(\breve m)\|\]
for all $m \in \JJ_L$, therefore $\phi$ defines a continuous functional on the C$^*$-algebra $E_\pi[C_0(\R^n)]$, which is multiplicative by Proposition~\ref{prp:eigenfunctions}, and thus belongs to the Gelfand spectrum of $E_\pi[C_0(\R^n)]$. Since
\[\langle \breve m, \phi \rangle = m(\evmap_L(\phi))\]
for all $m \in \JJ_L$ (see Proposition~\ref{prp:eigenvalues}), the element of $\supp E_\pi$ corresponding to this functional is $\evmap_L(\phi)$.

Conversely, if $\lambda \in \supp E_\pi$, then we can extend the corresponding character of $E_\pi[C_0(\R^n)]$ to a positive functional $\omega$ of norm $1$ on the whole $\Bdd(\HH)$ (see \cite{dixmier_algebras_1982}, \S2.10). Since $\omega \circ \pi : L^1(G) \to \C$ is linear and continuous, there exists $\phi \in L^\infty(G)$ such that
\[\langle f, \phi \rangle = \omega(\pi(f))\]
for all $f \in L^1(G)$; in fact, since $\omega$ is positive, $\phi$ must be a function of positive type on $G$ (see \cite{folland_course_1995}, \S3.3). Moreover, since $\omega$ extends a character of $E_\pi[C_0(\R^n)]$, it must be
\[\langle \breve m_1 * \breve m_2, \phi \rangle = \langle \breve m_1, \phi \rangle \langle \breve m_2, \phi \rangle\]
for all $m_1,m_2 \in \JJ_L$. Therefore, by Proposition~\ref{prp:eigenfunctions}, $\phi \in \PP_L$, and in fact $\phi \in \PP_{L,\pi}$ since $|\langle f, \phi\rangle| \leq \|\pi(f)\|$ (see Lemma~\ref{lem:weakcontainment}). Finally
\[m(\evmap_L(\phi)) = \langle \breve m, \phi \rangle = \omega(\pi(\breve m)) = m(\lambda)\]
for all $m \in \JJ_L$, by Proposition~\ref{prp:eigenvalues}, since $\omega$ extends the character corresponding to $\lambda$, and consequently $\evmap_L(\phi) = \lambda$.
\end{proof}

\begin{cor}\label{cor:spectrumeigenfunctions}
We have
\[\Sigma \subseteq \evmap_L(\PP_L),\]
with equality when $G$ is amenable.
\end{cor}
\begin{proof}
From Theorem~\ref{thm:spectrumeigenfunctions}, we have $\Sigma = \evmap_L(\PP_{L,\RA})$, where $\RA$ denotes the regular representation of $G$ on $L^2(G)$, from which we get the inclusion $\Sigma \subseteq \evmap_L(\PP_L)$. When $G$ is amenable, every representation is weakly contained in the regular representation (see \S\ref{subsection:amenability}), so that $\PP_L = \PP_{L,\RA}$, and we are done.
\end{proof}

Let $\II_L$ be the closure of $\Kern_L(\JJ_L)$ in $L^1(G)$. $\II_L$ is a commutative Banach $*$-subalgebra of $L^1(G)$ --- and also, by Proposition~\ref{prp:Jdensity}, a dense $*$-subalgebra of the $C^*$-algebra $C_0(L)$ --- which is somehow related to the joint $L^1(G)$ spectral theory of the operators $L_1,\dots,L_n$ (cf.\ \cite{hulanicki_subalgebra_1974}, \cite{hulanicki_commutative_1975} for the case of a single operator). We are then interested in comparing the Gelfand spectrum $\GS(\II_L)$ of $\II_L$ with the other two spectra $\evmap_L(\PP_L)$ and $\Sigma$ previously considered.

\begin{prp}
For $\phi \in \PP_L$, denote by $\Lambda(\phi)$ the linear functional
\[\II_L \ni f \mapsto \langle f, \phi \rangle \in \C.\]
Then $\Lambda$ is a continuous map $\PP_L \to \GS(\II_L)$ and it induces on $\PP_L$ the same equivalence relation as $\evmap_L$. In particular, $\Lambda$ induces a continuous injective map $\evmap_L(\PP_L) \to \GS(\II_L)$.
\end{prp}
\begin{proof}
From Proposition~\ref{prp:eigenfunctions}, it follows immediately that, if $\phi \in \PP_L$, then the functional $\Lambda(\phi)$ on $\II_L$ is multiplicative, and from Proposition~\ref{prp:eigenvalues} it is easy to see that $\Lambda(\phi)$ is non-null; therefore $\Lambda \in \GS(\II_L)$.

Since the topology of $\PP_L$ is the weak-$*$ topology of $L^\infty(G)$, whereas the topology of $\GS(\II_L)$ is the weak-$*$ topology of the dual of $\II_L$, and the map $\Lambda$ is simply the restriction map from the dual of $L^1(G)$ to the dual of $\II_L$, its continuity is trivial.

Finally, from Proposition~\ref{prp:eigenvalues}, it is clear that, at least on the dense set $\{ \breve m \tc m \in \JJ_L \}$ of $\II_L$, the (continuous) functional $\Lambda(\phi)$ is determined by $\evmap_L(\phi)$, so that
\[\evmap_L(\phi) = \evmap_L(\phi') \qquad\Longrightarrow\qquad \Lambda(\phi) = \Lambda(\phi').\]
On the other hand, if $\Lambda(\phi) = \Lambda(\phi')$, then, again by Proposition~\ref{prp:eigenvalues},
\[m(\evmap_L(\phi)) = m(\evmap_L(\phi')) \qquad\text{for all $m \in \JJ_L$,}\]
and $\evmap_L(\phi) = \evmap_L(\phi')$ follows since the elements of $\JJ_L$ separate the points of $\R^n$ (see Proposition~\ref{prp:Jdensity}).
\end{proof}

\begin{prp}\label{prp:hermitiancharacters}
Suppose that $G$ is hermitian. Then every character of $\II_L$ extends to a character of $C_0(L)$, so that the Gelfand spectra of the two Banach $*$-algebras coincide (also as topological spaces).
\end{prp}
\begin{proof}
Since $G$ is connected and hermitian, then it is also amenable, so that
\[\|f\|_{\Cv^2} = \sqrt{\rho(f^* f)} \qquad\text{for all $f \in L^1(G)$},\]
where $\rho(f)$ denotes the spectral radius of $f$ in $L^1(G)$. Notice that, since $\II_L$ is a closed subalgebra of $L^1(G)$, for every $f \in \II_L$, the spectral radius of $f$ in $\II_L$ coincides with its spectral radius in $L^1(G)$ (see \S\ref{subsection:spectrum}).
Moreover, since $L^1(G)$ is symmetric by hypothesis, also $\II_L$ is symmetric, so that, for every character $\psi \in \GS(\II_L)$,
\[\psi(f^*) = \overline{\psi(f)} \qquad\text{for all $f \in \II_L$;}\]
since $\psi(f)$ belongs to the spectrum of $f$ for every $f \in \II_L$, we have
\[|\psi(f)|^2 = \psi(f^* f) \leq \rho(f^* f) = \|f\|_{\Cv^2}^2.\]
This shows that every character $\psi \in \GS(\II_L)$ is continuous with respect to the norm of $C_0(L)$, so that it extends by density to a unique character of $C_0(L)$.

Notice that, since $\II_L$ is dense in $C_0(L)$ and the elements of $\GS(C_0(L))$, as functionals on $C_0(L)$, have norms bounded by $1$, it is easy to check that the topologies of $\GS(C_0(L))$ and $\GS(\II_L)$ coincide.
\end{proof}

\begin{cor}
If $G$ is hermitian, then the map
\[\Lambda : \PP_L \to \GS(\II_L)\]
is surjective. In particular, every multiplicative linear functional on $\II_L$ extends to a bounded linear functional $\eta$ on $L^1(G)$ such that
\[\eta(f * g) = \eta(f) \eta(g) \qquad\text{for all $f \in L^1(G)$ and $g \in \II_L$.}\]
\end{cor}
\begin{proof}
Let $\psi \in \GS(\II_L)$. By Proposition~\ref{prp:hermitiancharacters}, $\psi$ extends to a character of $C_0(L)$, which corresponds to some $\lambda \in \Sigma$. Now, by Corollary~\ref{cor:spectrumeigenfunctions}, there exists $\phi \in \PP_L$ such that $\evmap_L(\phi) = \lambda$, therefore, for every $m \in \JJ_L$, by Proposition~\ref{prp:eigenvalues},
\[\Lambda(\phi)(\breve m) = \langle \breve m, \phi \rangle = m(\evmap_L(\phi)) = m(\lambda) = \psi(\breve m),\]
from which by density we deduce $\Lambda(\phi) = \psi$.

In particular, if $\eta$ denotes the linear functional $f \mapsto \langle f, \phi \rangle$ on $L^1(G)$, then $\eta$ extends $\psi$ and, by Proposition~\ref{prp:eigenfunctions}, for all $m \in \JJ_L$,
\[\eta(f * \breve m) = \eta(f) \, \eta(\breve m) \qquad\text{for all $f \in L^1(G)$,}\]
from which we get the conclusion.
\end{proof}

The results obtained so far about the comparison of spectra are summarized in the following diagram, where dashed arrows denote inclusions subject to the condition indicated in the label:
\[
\xymatrix@C=40pt@M=5pt
{
\evmap_L(\PP_L) \ar@{_{(}-->}@<1ex>[r]^(.55){\text{amenable}} \ar@{_{(}->}@<1ex>[d] & \Sigma \ar@{=}[d] \ar@{_{(}->}@<1ex>[l]\\
\GS(\II_L) \ar@{_{(}-->}@<1ex>[r]^(.45){\text{hermitian}} \ar@{_{(}-->}@<1ex>[u]^{\text{hermitian}} & \GS(C_0(L)) \ar@{_{(}->}@<1ex>[l]
}
\]
The three objects do coincide when $G$ is hermitian, and in particular when $G$ has polynomial growth.

It has been proved that, to every element $\lambda$ of the joint $L^2$ spectrum $\Sigma$ of $L_1,\dots,L_n$, there corresponds a joint eigenvector of $d\pi(L_1),\dots,d\pi(L_n)$ of eigenvalue $\lambda$ in some irreducible unitary representation $\pi$ of $G$. We see now that, for a certain class of groups, in (almost) every irreducible representation $\pi$ an orthonormal basis of joint eigenvectors can be found.

\begin{lem}\label{lem:compactdecomposition}
Let $\pi$ be a unitary representation of $G$ and let $q$ be a non-negative polynomial on $\R^n$ such that the operator $q(L)$ is weighted subcoercive. Set $h = \Kern_L(e^{-q})$ and suppose that $\pi(h)$ is a compact operator on $\HH$. Then there exists a complete orthonormal system of $\HH$ made of joint eigenvectors of $d\pi(L_1),\dots,d\pi(L_n)$, and moreover each joint eigenvalue $\lambda \in \R^n$ of $d\pi(L_1),\dots,d\pi(L_n)$ has finite multiplicity.
\end{lem}
\begin{proof}
If $\HH$ is finite-dimensional, then the conclusion follows immediately, since $d\pi(L_1),\dots,d\pi(L_n)$ are pairwise commuting and self-adjoint.

Suppose instead that $\dim \HH$ is infinite. By the properties of the spectral integral, the operator
\[\pi(h) = e^{-\overline{d\pi(q(L))}}\]
is injective, positive and self-adjoint, and moreover it is compact by hypothesis. Therefore, from the spectral theorem for compact operators (see \S\ref{subsection:spectralcompact}), it follows that there exists a complete orthonormal system $(v_k)_{k \in \N}$ of $\HH$ made of eigenvectors of $\pi(h)$, and moreover the corresponding eigenvalues $c_k$ are positive and tend to $0$ for $k \to +\infty$.

In particular, if
\[V_c = \{v \in \HH \tc \pi(h)v = c v\} = \Span \{v_k \tc c_k = c \},\]
then $\HH = \bigoplus_c V_c$ and the $V_c$ are finite-dimensional. Moreover,
\[v_k = \lambda_k^{-1} \pi(h) v_k \in \HH^\infty\]
by Theorem~\ref{thm:robinsonterelst}(b), so that $V_c \subseteq \HH^\infty$. Since the $d\pi(L_j)$ commute with $\pi(h)$, the eigenspaces $V_c$ of $\pi(h)$ are invariant subspaces for $d\pi(L_1),\dots,d\pi(L_n)$, so that, as before, we can find an orthonormal basis of the finite-dimensional space $V_c$ made of joint eigenvectors of the $d\pi(L_j)$. By putting all these bases together, we get a complete orthonormal system for $\HH$ made of joint eigenvectors of $d\pi(L_1),\dots,d\pi(L_n)$.

Let now $\lambda \in \R^n$ such that
\[d\pi(L_j) v = \lambda_j v\]
for some $v \in \HH \setminus \{0\}$. By the properties of the spectral integral, we then have
\[\pi(h) v = e^{-q(\lambda)} v,\]
therefore $v \in V_c$ with $c = e^{-q(\lambda)}$. This shows that the joint eigenspace of $d\pi(L_1),\dots,d\pi(L_n)$ associated to $\lambda$ is contained in $V_c$ and consequently is finite-dimensional.
\end{proof}

\begin{prp}\label{prp:eigenvectordecomposition}
Suppose that $G$ is unimodular and type I, and let $\widehat G$ be the set of (equivalence classes of) irreducible representations of $G$. Then we can find a generic subset $\widehat G_{\mathrm{gen}}$ of $\widehat G$ (with respect to the group Plancherel measure) such that, for every $\pi \in \widehat G_{\mathrm{gen}}$, having denoted by $\HH$ the Hilbert space on which $\pi$ acts, there exists an orthonormal basis of $\HH$ made of joint eigenvectors of $d\pi(L_1),\dots,d\pi(L_n)$, and moreover each joint eigenvalue of $d\pi(L_1),\dots,d\pi(L_n)$ has finite multiplicity. If $G$ is CCR, then one can take $\widehat G_{\mathrm{gen}} = \widehat G$.
\end{prp}
\begin{proof}
Let $h$ be as in Lemma~\ref{lem:compactdecomposition}. Then we know that $h \in L^1 \cap L^2(G)$ by Theorem~\ref{thm:robinsonterelst}(e,f). Therefore, if $G$ is CCR, then $\pi(f)$ is compact for every irreducible representation $\pi$ of $G$, and the conclusion follows from Lemma~\ref{lem:compactdecomposition}. If $G$ is unimodular and type I, then by the group Plancherel formula (see \S\ref{subsection:typeI}) we have
\[\int_{\widehat G} \|\pi(h)\|_{\HS}^2 \,d\pi = \|h\|_2^2 < \infty;\]
in particular $\|\pi(h)\|_{\HS}$ is finite for almost every irreducible unitary representation $\pi$ of $G$, so that $\pi(h)$ is Hilbert-Schmidt (and consequently compact) for $\pi$ in a generic subset of $\widehat G$, and the conclusion follows again by Lemma~\ref{lem:compactdecomposition}.
\end{proof}

\section{Direct products}\label{section:directproducts}

For $l = 1,\dots,\ell$, let $G_l$ be a connected Lie group, and set
\[G^\times = G_1 \times \dots \times G_\ell.\]
We then have the identification
\[\lie{g}^\times = \lie{g}_1 \oplus \dots \oplus \lie{g}_\ell.\]

\begin{lem}\label{lem:concatweighted}
For $l=1,\dots,\ell$, suppose that $A_{l,1},\dots,A_{l,d_l}$ is a reduced basis of $\lie{g}_l$, with weights $w_{l,1},\dots,w_{l,d_l}$. Then
\begin{equation}\label{eq:concatweighted}
A_{1,1},\dots,A_{1,d_1},\dots,A_{\ell,1},\dots,A_{\ell,d_\ell}
\end{equation}
is a reduced basis of $\lie{g}^\times$, with weights
\[w_{1,1},\dots,w_{1,d_1},\dots,w_{\ell,1},\dots,w_{\ell,d_\ell}.\]
Moreover, if $(V_{l,\lambda})_\lambda$ is the filtration on $\lie{g}_l$ corresponding to the chosen reduced basis for $l=1,\dots,\ell$, then
\[V^\times_\lambda = V_{1,\lambda} \oplus \dots \oplus V_{\ell,\lambda}\]
gives the filtration on $\lie{g}^\times$ corresponding to the algebraic basis \eqref{eq:concatweighted}; therefore, by passing to the quotients, we obtain for the contractions
\[(\lie{g}^\times)_* = (\lie{g}_1)_* \oplus \dots \oplus (\lie{g}_\ell)_*.\]
\end{lem}
\begin{proof}
An iterated commutator $A_{[\alpha]}$ of the elements of \eqref{eq:concatweighted} is not null only if it coincides with an iterated commutator $(A_l)_{[\alpha']}$ of $A_{l,1},\dots,A_{l,n_l}$ for some $l \in \{1,\dots,\ell\}$. This can be easily checked by induction on the length $|\alpha|$ of the commutator. The identities involving the filtrations then follow immediately, from which we get easily the conclusion.
\end{proof}

\begin{prp}\label{prp:productwsub}
Suppose that $D_l \in \Diff(G_l)$ is a self-adjoint weighted subcoercive operator on $G_l$, for $l = 1,\dots,\ell$, and let $D_l^\times \in \Diff(G^\times)$ be the differential operator on $G^\times$ along the $l$-th factor corresponding to $D_l$. Then
\[D = (D_1^\times)^2 + \dots + (D_\ell^\times)^2\]
is a positive weighted subcoercive operator on $G^\times$.
\end{prp}
\begin{proof}
For $l=1,\dots,\ell$, let $A_{l,1},\dots,A_{l,d_l}$ be a reduced basis of $\lie{g}_l$, with weights $w_{l,1},\dots,w_{l,d_l}$, such that, for some self-adjoint weighted subcoercive form $C_l$ of degree $m_l$, we have $D_l = dR_{G_l}(C_l)$; let moreover $P_l$ be the principal part of $C_l$.

Clearly, we can find real numbers $\zeta_1,\dots,\zeta_\ell \geq 1$ such that
\[\zeta_1 m_1 = \dots = \zeta_\ell m_\ell.\]
Therefore, by rescaling the weights of the algebraic basis of $\lie{g}_l$ by $\zeta_l$, we may suppose that the forms $C_1,\dots,C_\ell$ have the same degree $m$.

By Lemma~\ref{lem:concatweighted}, the concatenation of the bases of $\lie{g}_1,\dots,\lie{g}_l$ gives a reduced weighted algebraic basis of $\lie{g}^\times$. We can then consider, for $l=1,\dots,\ell$, the forms $C_l^\times$, $P_l^\times$ corresponding to $C_l$,$P_l$ but re-indexed on the basis \eqref{eq:concatweighted}. In particular, if
\[C = (C_1^\times)^2 + \dots + (C_\ell^\times)^2, \qquad P = (P_1^\times)^2 + \dots + (P_\ell^\times)^2,\]
then $P = P^+$ is the principal part of $C$, and moreover
\[d\RA_{G^\times}(C) = (d\RA_{G_1}(C_1)^\times)^2 + \dots + (d\RA_{G_\ell}(C_\ell)^\times)^2 = D.\]

On the other hand, again by Lemma~\ref{lem:concatweighted}, we have the identification
\[(G^\times)_* = (G_1)_* \times \dots \times (G_\ell)_*,\]
so that
\[d\RA_{(G^\times)_*}(P) = (d\RA_{(G_1)_*}(P_1)^\times)^2 + \dots + (d\RA_{(G_\ell)_*}(P_\ell)^\times)^2.\]
By Theorem~\ref{thm:robinsonterelst}, we have that $d\RA_{(G_l)_*}(P_l)$ is Rockland on $(G_l)_*$ for $l=1,\dots,\ell$. We now prove that $dR_{(G^\times)_*}(P)$ is Rockland on $(G^\times)_*$.

If $\pi$ is a non-trivial irreducible unitary representation of $G^\times$ on a Hilbert space $\HH$, then (see \cite{folland_course_1995}, Theorem~7.25) we may suppose modulo equivalence that $\pi = \pi_1 \otimes \dots \otimes \pi_\ell$, where $\pi_l$ is an irreducible unitary representation of $G_l$ on a Hilbert space $\HH_l$ for $l=1,\dots,\ell$, so that
\[\HH = \HH_1 \mathop{\hat\otimes} \cdots \mathop{\hat\otimes} \HH_\ell,\]
and at least one of $\pi_1,\dots,\pi_\ell$ is non-trivial. Let $(w_{l,\nu_l})_{\nu_l}$ be a complete orthonormal system for $\HH_l$, for $l=1,\dots,\ell$, so that $(w_{1,\nu_1} \otimes \dots \otimes w_{\ell,\nu_{\ell}})_{\vec{\nu}}$ is a complete orthonormal system for $\HH$. Then, for every element
\[v = \sum_{\nu_1,\dots,\nu_\ell} a_{\nu_1,\dots,\nu_\ell} w_{1,\nu_1} \otimes \dots \otimes w_{\ell,\nu_\ell}\]
of $\HH$, we have
\begin{multline*}
\langle d\pi(d\RA_{(G^\times)_*}(P)) v, v \rangle_{\HH} \\
= \sum_{l=1}^\ell \sum_{\nu_1,\dots,\nu_{l-1},\nu_{l+1},\nu_\ell} \left\| d\pi_l(d\RA_{(G_l)_*}(P_l))\left( \sum_{\nu_l} a_{\nu_1,\dots,\nu_\ell} w_{l,\nu_l} \right)\right\|^2_{\HH_l};
\end{multline*}
since at least one of the $d\pi_l(d\RA_{(G_l)_*}(P_l))$ is injective (being $d\RA_{(G_l)_*}(P_l)$ Rockland and $\pi_l$ non-trivial), this formula gives easily that
\[v \neq 0 \qquad\Longrightarrow\qquad d\pi(d\RA_{(G^\times)_*}(P)) v \neq 0,\]
i.e., $d\pi(d\RA_{(G^\times)_*}(P))$ is injective. From the arbitrariness of $\pi$, we conclude that $d\RA_{(G^\times)_*}(P)$ is Rockland.

But then, again by Theorem~\ref{thm:robinsonterelst}, we get that $D = d\RA_{G^\times}(C)$ is weighted subcoercive.
\end{proof}

For $l=1,\dots,\ell$, let $L_{l,1},\dots,L_{l,n_l} \in \Diff(G_l)$ be a weighted subcoercive system. Let moreover $L_{l,j}^\times$ be the differential operator on $G^\times$ along the $l$-th factor corresponding to $L_{l,j}$. Then, by the previous proposition,
\begin{equation}\label{eq:concatsystem}
L_{1,1}^\times,\dots,L_{1,n_1}^\times,\dots,L_{\ell,1}^\times,\dots,L_{\ell,n_\ell}^\times
\end{equation}
is a weighted subcoercive system on $G^\times$.

\begin{prp}\label{prp:productfunctional}
Suppose, for $l=1,\dots,\ell$, that $m_l$ is a bounded Borel function on $\R^{n_l}$, and set
\[m = m_1 \otimes \dots \otimes m_\ell.\]
Then
\[m(L^\times) = m_1(L_1) \otimes \dots \otimes m_\ell(L_\ell),\]
and in particular
\[\breve m = \breve m_1 \otimes \dots \otimes \breve m_\ell.\]
\end{prp}
\begin{proof}
Since the operators \eqref{eq:concatsystem} commute strongly, by the properties of the spectral integral it is immediate to see that
\[m(L^\times) = m_1(L_1^\times) \cdots m_\ell(L_\ell^\times).\]
On the other hand, from the identity
\[L_{l,j}^\times(f_1 \otimes \dots \otimes f_\ell) = f_1 \otimes \dots \otimes f_{l-1} \otimes (L_{l,j} f_l) \otimes f_{l+1} \otimes \dots \otimes f_n\]
it follows that also
\[m(L_l^\times)(f_1 \otimes \dots \otimes f_\ell) = f_1 \otimes \dots \otimes f_{l-1} \otimes (m(L_l) f_l) \otimes f_{l+1} \otimes \dots \otimes f_n,\]
from which clearly
\[m_1(L_1^\times) \cdots m_\ell(L_\ell^\times) = m_1(L_1) \otimes \dots \otimes m_\ell(L_\ell).\]
In particular, for every $f_1 \in \D(G_1), \dots, f_\ell \in \D(G_\ell)$,
\begin{multline*}
m(L^\times) (f_1 \otimes \dots \otimes f_\ell) = (m_1(L_1) f_1) \otimes \dots \otimes (m_\ell(L_\ell) f_\ell)\\
= (f_1 * \breve m_1) \otimes \dots \otimes (f_\ell * \breve m_\ell) = (f_1 \otimes \dots \otimes f_\ell) * (\breve m_1 \otimes \dots \otimes \breve m_\ell),
\end{multline*}
from which $\breve m = \breve m_1 \otimes \dots \otimes \breve m_\ell$.
\end{proof}

Let now $\sigma_l$ be the Plancherel measure on $\R^{n_l}$ associated to the system $L_{l,1},\dots,L_{l,n_l}$, for $l=1,\dots,\ell$. Let moreover $\sigma^\times$ be the Plancherel measure on $\R^{\vec{n}}$ associated to the system \eqref{eq:concatsystem}.

\begin{prp}\label{prp:productplancherel}
$\sigma^\times = \sigma_1 \times \dots \times \sigma_\ell$.
\end{prp}
\begin{proof}
If $A_l \subseteq \R^{n_l}$ is a relatively compact Borel set for $l=1,\dots,\ell$, we have
\[\chr_{A_1 \times \dots \times A_\ell} = \chr_{A_1} \otimes \dots \otimes \chr_{A_\ell},\]
so that, by Theorem~\ref{thm:plancherel} and Proposition~\ref{prp:productfunctional},
\begin{multline*}
\sigma(A_1 \times \dots \times A_\ell) = \|\breve\chr_{A_1 \times \dots \times A_\ell}\|_{L^2(G^\times)} = \|\breve \chr_{A_1} \otimes \dots \otimes \breve \chr_{A_\ell}\|_{L^2(G^\times)} \\
= \|\breve \chr_{A_1}\|_{L^2(G_1)} \cdots \|\breve \chr_{A_\ell}\|_{L^2(G_\ell)} = \sigma_1(A_1) \cdots \sigma_\ell(A_\ell),
\end{multline*}
and we are done.
\end{proof}

\section{Change of generators}\label{section:automorphisms}

Let $L_1,\dots,L_n$ be a weighted subcoercive system on a connected Lie group $G$. Let $\sigma$ be the associated Plancherel measure on $\R^n$, and $\Sigma = \supp \sigma$. For given polynomials $p_1,\dots,p_{n'} : \R^n \to \R$, consider the operators
\[L'_1 = p_1(L_1,\dots,L_n), \quad\dots,\quad L'_{n'} = p_k(L_1,\dots,L_n),\]
and suppose that they still form a weighted subcoercive system. Let $\sigma'$ be the Plancherel measure on $\R^{n'}$ associated to the system $L'_1,\dots,L'_{n'}$, and $\Sigma'$ its support. We may ask if there is a relationship between the transforms $\Kern_{L}$ and $\Kern_{L'}$, and between the Plancherel measures $\sigma$ and $\sigma'$ associated to the two systems.

Let $p : \R^n \to \R^{n'}$ denote the polynomial map whose $j$-th component is the polynomial $p_j$.

\begin{lem}\label{lem:properwsub}
The map $p|_\Sigma : \Sigma \to \R^{n'}$ is a proper continuous map.
\end{lem}
\begin{proof}
Since $L_1',\dots,L_{n'}'$ is a weighted subcoercive system, we can find a non-negative polynomial $q : \R^{n'} \to \R$ such that $q(L') = q(p(L))$ is a weighted subcoercive operator. By Theorem~\ref{thm:robinsonterelst}(v), for sufficiently large $C > 0$ and $k \in \N$, we have that
\[\max_j \|L_j \phi\|_2 \leq C \|(1 + q(p(L))^k) \phi\|_2 \qquad\text{for $\phi \in \D(G)$,}\]
which means, by the spectral theorem, that
\[\max_j |\lambda_j| \leq C(1 + q(p(\lambda))^k) \qquad\text{for $\lambda \in \Sigma$,}\]
since $\Sigma$ is the joint $L^2(G)$-spectrum of $L_1,\dots,L_n$.

Now, if $K \subseteq \R^{n'}$ is compact, then by continuity there exists $M > 0$ such that $q|_K \leq M$, but then
\[\max_j |\lambda_j| \leq C(1 + M^k) \qquad\text{for $\lambda \in \Sigma \cap p^{-1}(K)$,}\]
thus $p^{-1}(K) \cap \Sigma$ is bounded in $\R^n$, and also closed (by continuity of $p$), therefore $p^{-1}(K)$ is compact.
\end{proof}

\begin{prp}\label{prp:pushforward}
For every bounded Borel $m : \R^{n'} \to \C$, we have:
\[m(L') = (m \circ p)(L), \qquad \Kern_{L'} m = \Kern_{L} (m \circ p).\]
Moreover
\[\sigma' = p(\sigma), \qquad \Sigma' = p(\Sigma).\]
\end{prp}
\begin{proof}
The first part of the conclusion follows immediately from the spectral theorem and uniqueness of the convolution kernel. From this, the identity $\sigma' = p(\sigma)$ is easily inferred by Theorem~\ref{thm:plancherel}. In particular,
\[\sigma(\R^n \setminus p^{-1}(\Sigma')) = \sigma'(\R^{n'} \setminus \Sigma') = 0,\]
i.e., by continuity of $p$, $p(\Sigma) \subseteq \Sigma'$.

In order to prove the opposite inclusion, we use the fact that $p|_\Sigma$ is proper (see Lemma~\ref{lem:properwsub}). Take $\lambda' \in \Sigma'$, and let $B_k$ be a decreasing sequence of compact neighborhoods of $p$ in $\R^{n'}$ such that $\bigcap_k B_k = \{p\}$. By definition of support, we then have $\sigma(p^{-1}(B_k)) = \sigma'(B_k) \neq 0$, therefore $p^{-1}(B_k) \cap \Sigma \neq \emptyset$ for all $k$. Since $p|_\Sigma$ is proper, we have a decreasing sequence $p^{-1}(B_k) \cap \Sigma$ of non-empty compacts of $\R^n$, which therefore has a non-empty intersection. If $\lambda$ belongs to this intersection, then clearly $\lambda \in \Sigma$ and moreover $p(\lambda) \in B_k$ for all $k$, that it, $p(\lambda) = \lambda'$.
\end{proof}

A particularly interesting case is when $L_1',\dots,L_{n'}'$ generate the same subalgebra of $\Diff(G)$ as $L_1,\dots,L_n$. In this case, there exists also a polynomial map $q = (q_1,\dots,q_n) : \R^{n'} \to \R^n$ such that
\[L_1 = q_1(L'), \quad\dots, \quad L_n = q_n(L').\]
Notice that in general $p$ and $q$ are not the inverse one of the other: from the spectral theorem, we only deduce that $(q \circ p)|_\Sigma = \id_\Sigma$, $(p \circ q)|_{\Sigma'} = \id_{\Sigma'}$ (in fact, these identities extend to the Zariski-closures of $\Sigma$ and $\Sigma'$). In particular,
\[p|_\Sigma : \Sigma \to \Sigma', \qquad q|_{\Sigma'} : \Sigma' \to \Sigma\]
are homeomorphisms.

Such polynomial changes of variables may be induced by particular automorphisms of the Lie group $G$. Namely, let $\mathcal{O}$ be the unital subalgebra of $\Diff(G)$ generated by $L_1,\dots,L_n$. If $k \in \Aut(G)$, then its derivative $k'$ is an automorphism of $\lie{g}$, therefore it extends to a unique filtered $*$-algebra automorphism of $\Diff(G) \cong \UEnA(\lie{g})$ (which will be still denoted by $k'$); we then say that $\mathcal{O}$ is $k$-invariant if $k(\mathcal{O}) \subseteq \mathcal{O}$, or equivalently, if $k(\mathcal{O}) = \mathcal{O}$ (the equivalence is due to the fact that $k'$ is an injective linear map and the filtration of $\Diff(G)$ is made of finitely dimensional subspaces).

Let $\Aut(G;\mathcal{O})$ denote the (closed) subgroup of $\Aut(G)$ made of the automorphisms $k$ such that $\mathcal{O}$ is $k$-invariant. If $k \in \Aut(G;\mathcal{O})$, then 
\[k'(L_1),\dots,k'(L_n)\]
must be a system of generators of $\mathcal{O}$; therefore, we can choose a polynomial map $p_k = (p_{k,1},\dots,p_{k,n}) : \R^n \to \R^n$ such that $k'(L_j) = p_{k,j}(L)$. 

Notice that, for every $k \in \Aut(G)$, the push-forward via $k$ of the right Haar measure $\mu$ on $G$ is a multiple of $\mu$, and in fact there is a Lie group homomorphism $c : \Aut(G) \to \R^+$ such that
\[k(\mu) = c(k) \mu.\]
In particular, if we set
\[T_k f = f \circ k^{-1}\]
for $k \in \Aut(G)$, then we have immediately

\begin{prp}\label{prp:automorphismskernel}
For $k \in \Aut(G)$, $T_k$ is a multiple of an isometry of $L^2(G)$; more precisely
\[\|T_k f\|_2^2 = c(k)^{-1} \|f\|_2^2.\]
Moreover, for all $D \in \Diff(G)$,
\[k'(D) = T_k D T_k^{-1}.\]
In particular, for every bounded Borel $m : \R^n \to \C$,
\[m(k'(L_1),\dots,k'(L_n)) = T_k m(L_1,\dots,L_n) T_k^{-1},\]
and consequently
\[\Kern_{k'(L)} m = c(k) T_k \Kern_L m.\]
\end{prp}
\begin{proof}
The first equality follows from the change-of-variable formula for a push-forward measure. The second one can be easily proved for vector fields $D \in \lie{g}$ and then extended to general $D \in \Diff(G)$. The third identity follows from the second one and uniqueness of the joint spectral resolution ($T_k$ is a multiple of an isometry, so that conjugation by $T_k$ preserves orthogonal projections). Consequently, we have, for $\phi \in \Diff(G)$,
\[m(k'(L)) \phi = T_k ( (T_k^{-1} \phi) * \Kern_L m) = c(k) \phi * T_k \Kern_L m\]
by the properties of convolution, and the fourth identity follows.
\end{proof}

\begin{cor}\label{cor:automorphismskernel}
If $k \in \Aut(G;\mathcal{O})$, then, for every bounded Borel function $m : \R^n \to \C$,
\[(m \circ p_k)(L_1,\dots,L_n) = T_k m(L_1,\dots,L_n) T_k^{-1},\]
and moreover
\[\Kern_L (m \circ p_k) = c(k) T_k \Kern_L m.\]
\end{cor}

The previous results, together with the characterization of Plancherel measure given in Theorem~\ref{thm:plancherel}, give finally

\begin{cor}\label{cor:automorphismsplancherel}
If $k \in \Aut(G;\mathcal{O})$, then
\[p_k(\sigma) = c(k) \sigma, \qquad p_k(\Sigma) = \Sigma.\]
\end{cor}

We then have that the restrictions $p_k|_\Sigma$ (which are uniquely determined by $k$) define an action of the group $\Aut(G;\mathcal{O})$ on the spectrum $\Sigma$ by homeomorphisms; more precisely

\begin{prp}
The map
\begin{equation}\label{eq:automorphismsaction}
\Aut(G;\mathcal{O}) \times \Sigma \ni (k,\lambda) \mapsto p_{k^{-1}}(\lambda) \in \Sigma
\end{equation}
is continuous, and defines a continuous (left) action of $\Aut(G;\mathcal{O})$ on $\Sigma$.
\end{prp}
\begin{proof}
Recall that $\Sigma$ may be identified, as a topological space, with the Gelfand spectrum of the sub-C$^*$-algebra $C_0(L)$ of $\Cv^2(G)$, where $\lambda \in \Sigma$ corresponds to the multiplicative linear functional $\psi_\lambda$ defined by
\[\psi_\lambda(\breve m) = m(\lambda).\]
By Corollary~\ref{cor:automorphismskernel} we then deduce
\[\psi_{p_k(\lambda)} = c(k) \psi_\lambda \circ T_k,\]
which clarifies that \eqref{eq:automorphismsaction} defines a left action on $\Sigma$.

Moreover, since $C_0(L) \cap L^1(G)$ is dense in $C_0(L)$ (see Proposition~\ref{prp:Jdensity}), and since $c(k)T_k$ is an isometry of $\Cv^2(G)$, we obtain easily that
\[k \mapsto c(k) T_k u\]
is continuous for every $u \in \Cv^2(G)$. Therefore, since the topology of the Gelfand spectrum is induced by the weak$^*$ topology, we immediately obtain that \eqref{eq:automorphismsaction} is separately continuous, and also jointly continuous since the $\psi_\lambda$ have uniformly bounded norms.
\end{proof}

The richer the group $\Aut(G;\mathcal{O})$ --- or rather, its quotient by the subgroup of automorphisms which fix each of the generators $L_1,\dots,L_n$ --- is, the more we may deduce about the structure of the spectrum $\Sigma$ and of the Plancherel measure $\sigma$.
An example of this fact is illustrated in the next section.

\section{Homogeneity}\label{section:homogeneity}

Let $G$ be a homogeneous Lie group, with automorphic dilations $\delta_t$ and homogeneous dimension $Q_\delta$. A weighted subcoercive system $L_1,\dots,L_n \in \Diff(G)$ will be called \emph{homogeneous}\index{system of differential operators!weighted subcoercive!homogeneous} if each $L_j$ is $\delta_t$-homogeneous.

In the following, $L_1,\dots,L_n$ will be a homogeneous weighted subcoercive system, with associated Plancherel measure $\sigma$, and $r_j$ will denote the degree of homogeneity of $L_j$, i.e.,
\[\delta_t(L_j) = t^{r_j} L_j.\]
The unital subalgebra of $\Diff(G)$ generated by $L_1,\dots,L_n$ is $\delta_t$-invariant for every $t > 0$. Therefore, if we set
\[D_t f = f \circ \delta_{t^{-1}},\]
and if we denote by $\epsilon_t$ the dilations on $\R^n$ given by
\begin{equation}\label{eq:spectraldilations}
\epsilon_t(\lambda) = (t^{r_1} \lambda_1,\dots,t^{r_n} \lambda_n).
\end{equation}
(which will be said \index{dilations!associated to a weighted subcoercive system}\emph{associated} to the system $L_1,\dots,L_n$), then from Corollaries~\ref{cor:automorphismskernel} and \ref{cor:automorphismsplancherel} we immediately deduce

\begin{prp}\label{prp:plancherelhomogeneous}
For every bounded Borel $m : \R^n \to \C$, we have
\[(m \circ \epsilon_t)(L) = D_t m(L) D_{t^{-1}}, \qquad (m \circ \epsilon_t)\breve{} = t^{-Q_\delta} \breve m \circ \delta_{t^{-1}}.\]
Moreover, the support $\Sigma$ of $\sigma$ is $\epsilon_t$-invariant, and
\[\sigma(\epsilon_t(A)) = t^{Q_\delta} \sigma(A)\]
for all Borel $A \subseteq \R^n$. In particular, $\sigma(\{0\}) = 0$.
\end{prp}

Since the Plancherel measure $\sigma$ is homogeneous, it admits the decomposition given by Proposition~\ref{prp:radialcoordinates}, and in particular from Corollary~\ref{cor:radialcoordinates} we get

\begin{cor}\label{cor:plancherelhomogeneous}
The Plancherel measure $\sigma$ associated to a homogeneous system $L_1,\dots,L_n$ is locally $1$-bounded on $\R^n \setminus \{0\}$.
\end{cor}

The homogeneous system $L_1,\dots,L_n \in \Diff(G)$ will be called a \emph{Rockland system}\index{system of differential operators!Rockland} if the unital subalgebra of $\Diff(G)$ generated by $L_1,\dots,L_n$ contains a Rockland operator. We now prove that every homogeneous weighted subcoercive system is also a Rockland system, with respect to a possibly different family of automorphic dilations of $G$; in doing so, we obtain a characterization of homogeneous weighted subcoercive systems which is more representation-theoretic in character.

\begin{prp}
Let $L_1,\dots,L_n \in \Diff(G)$ be pairwise commuting and formally self-adjoint.
\begin{itemize}
\item[(i)] If $L_1,\dots,L_n$ is a weighted subcoercive system, then, for every non-trivial irreducible unitary representation $\pi$ of $G$ on a Hilbert space $\HH$, the operators $d\pi(L_1), \dots, d\pi(L_n)$ are \emph{jointly injective} on $\HH^\infty$, i.e.,
\[d\pi(L_1) v = \dots = d\pi(L_n) v = 0 \qquad\Longrightarrow\qquad v = 0\]
for all $v \in \HH^\infty$.
\item[(ii)] If, conversely, for every non-trivial irreducible unitary representation $\pi$ of $G$ on a Hilbert space $\HH$, the operators $d\pi(L_1),\dots,d\pi(L_n)$ are jointly injective on $\HH^\infty$, then $L_1,\dots,L_n$ is a Rockland system with respect to every system of dilations on $G$ such that the homogeneity degrees of $L_1,\dots,L_n$ have a common multiple; in fact, such systems of dilations on $G$ do exist, and in particular $L_1,\dots,L_n$ is a homogeneous weighted subcoercive system with respect to the original dilations on $G$.
\end{itemize}
\end{prp}
\begin{proof}
(i) Let $p$ be a real polynomial such that $p(L) = p(L_1,\dots,L_n)$ is a weighted subcoercive operator. Choose moreover a system $X_1,\dots,X_d$ of generators of $\lie{g}$ made of $\delta_t$-homogeneous elements, so that $\delta_t(X_k) = t^{\nu_k} X_k$ for some $\nu_k > 0$. From Theorem~\ref{thm:robinsonterelst}(v) we deduce that, possibly by replacing $p$ with some power $p^m$, there exist a constant $C > 0$ such that, for every unitary representation $\pi$ of $G$ on a Hilbert space $\HH$,
\begin{equation}\label{eq:apriori}
\|d\pi(X_k) v\|^2 \leq C ( \|v\|^2 + \|d\pi(p(L)) v\|^2 )
\end{equation}
for $v \in \HH^\infty$, $k=1,\dots,d$.

Fix now a non-trivial irreducible unitary representation $\pi$ of $G$ on a Hilbert space $\HH$, and let $v \in \HH^\infty$ be such that
\[d\pi(L_1) v = \dots = d\pi(L_n) v = 0.\]
For $t > 0$, since $\delta_t \in \Aut(G)$, $\pi_t = \pi \circ \delta_t$ is also a unitary representation of $G$; moreover, it is easily checked that smooth vectors for $\pi_t$ coincide with smooth vectors for $\pi$, and that
\[d\pi_t(D) = d\pi(\delta_t(D)) \qquad\text{for every $D \in \Diff(G)$.}\]
In particular,
\[d\pi_t(p(L)) v = d\pi((p \circ \epsilon_t)(L)) v = p(0) v,\]
thus from \eqref{eq:apriori} applied to the representation $\pi_t$ we get
\[\|d\pi(X_k) v\|^2 \leq t^{-2\nu_k} C (1 + |p(0)|^2) \|v\|^2,\]
and, for $t \to +\infty$, we obtain
\[d\pi(X_1) v = \dots = d\pi(X_d) v = 0.\]
Since $X_1,\dots,X_d$ generate $\lie{g}$, this means that the function $x \mapsto \pi(x) v$ is constant, i.e.,
\[\pi(x) v = v \qquad\text{for all $x \in G$,}\]
but $\pi$ is irreducible and non-trivial, thus $v = 0$.

(ii) By the results of \cite{miller_parametrices_1980} (see in particular Proposition~1.1 and its proof), we can find a gradation on $G$ with respect to which the operators $L_1,\dots,L_n$ are still homogeneous; since in this case the degrees of homogeneity are integers, they must have a common multiple.

Suppose therefore that $\tilde\delta_t$ is a system of automorphic dilations of $G$ such that the degrees $r_1,\dots,r_n$ of $L_1,\dots,L_n$ have a common multiple $M$. Then
\[\Delta = L_1^{2M/r_1} + \dots + L_n^{2M/r_n}\]
is $\tilde\delta_t$-homogeneous of degree $2M$ and belongs to the unital subalgebra of $\Diff(G)$ generated by $L_1,\dots,L_n$. Moreover, for every irreducible unitary representation $\pi$ of $G$ on $\HH$, and for every $v \in \HH^\infty$, we have
\[\langle d\pi(\Delta) v, v \rangle = \|d\pi(L_1)^{M/r_1} v\|^2_\HH + \dots \|d\pi(L_n)^{M/r_n} v\|^2_\HH,\]
so that, if $d\pi(\Delta) v = 0$, then also $d\pi(L_j) v = 0$ for $j=1,\dots,n$, therefore $v = 0$. This proves that $\Delta$ is a Rockland operator.

Notice now that, by Proposition~\ref{prp:rocklandwsub}, $\Delta$ is a weighted subcoercive operator on $G$. Since $L_1,\dots,L_n$ are $\delta_t$-homogeneous, the conclusion follows.
\end{proof}

We now show that a Rockland operator in the unital algebra generated by $L_1,\dots,L_n$ plays the role of ``homogeneous norm'' on the spectral side.

\begin{prp}\label{prp:spectrum}
Let $|\cdot|_\epsilon$ be a $\epsilon_t$-homogeneous norm on $\R^n$. Suppose that, for some polynomial $p$, the operator $p(L)$ is Rockland of degree $r$. Then there exists a constant $C \geq 1$ such that
\[C^{-1} |\lambda|_\epsilon \leq |p(\lambda)|^{1/r} \leq C |\lambda|_\epsilon \qquad\text{for $\lambda \in \Sigma$.}\]
\end{prp}
\begin{proof}
If $M$ is a common multiple of
\[r,r_1,\dots,r_n,\]
it follows from Theorem~\ref{thm:robinsonterelst}(vi) applied to the positive Rockland operator $|p|^{2M/r}(L)$ that there exists $C > 0$ such that
\[\|L_j^{M/r_j} u\|_2 \leq C \|p(L)^{M/r} u\|_2 \]
for all $u \in \D(G)$, $j=1,\dots,n$; therefore, by the properties of the spectral integral, we deduce that
\[|\lambda_j|^{1/r_j} \leq C^{1/M} |p(\lambda)|^{1/r} \qquad\text{for $\lambda \in \Sigma$.}\]
Since
\[|\lambda|_\epsilon \sim \sum_{j=1}^n |\lambda_i|^{1/r_j},\]
the first inequality of the conclusion follows easily.

Finally, notice that
\[(p \circ \epsilon_t)(L) = D_t p(L) D_{t^{-1}} = t^r p(L)\]
since $p(L)$ has degree $r$, so that
\[p \circ \epsilon_t = t^r p \qquad\text{on $\Sigma$}\]
and the second inequality also follows.
\end{proof}

\section{Gelfand pairs}\label{section:gelfandpairs}

Let $G$ be a connected Lie group. In this paragraph, we describe a particular way of obtaining weighted subcoercive systems on $G$, which has been extensively studied in the literature.

Let $K$ be a compact subgroup of $\Aut(G)$. Notice that the homomorphism $c : \Aut(G) \to \R^+$ defined in \S\ref{section:automorphisms} must be identically $1$ on $K$ by compactness, so that the elements of $K$ preserve the Haar measure of $G$. A continuous function $f : G \to \C$ is said to be \emph{$K$-invariant} if
\[T_k f = f \qquad\text{for all $k \in K$;}\]
this definition is naturally extended to other spaces of functions on $G$, and also to distributions. We add a subscript $K$ to the symbol representing a particular space of functions or distributions in order to denote the corresponding (closed) subspace of $K$-invariant elements; for instance, $L^p_K(G)$ denotes the Banach space of $K$-invariant $L^p$ functions on $G$. Since
\[T_k (f * g) = (T_k f) * (T_k g), \qquad T_k(f^*) = (T_k f)^*,\]
it is immediately proved that $L^1_K(G)$ is a Banach $*$-subalgebra of $L^1(G)$.

We also define the projection
\[P_K : f \mapsto \int_K T_k f \,dk,\]
where the integration is with respect to the Haar measure on $K$ with mass $1$; by a suitable notion of integral, this projection is defined on various spaces of functions and distributions on $G$, and on most spaces it is a continuous operator (with unit norm), which maps the whole space onto the subspace of $K$-invariant elements. Moreover
\[P_K(f * (P_K g)) = P_K( (P_K f) * g) = (P_K f) * (P_K g), \qquad P_K(f^*) = (P_K f)^*.\]

Among the left-invariant differential operators on $G$, we can consider those which are $K$-invariant, i.e., which commute with $T_k$ for all $k \in K$. The set $\Diff_K(G)$ of left-invariant $K$-invariant differential operators on $G$ is a $*$-subalgebra of $\Diff(G)$, which is finitely generated since $K$ is compact (cf.\ \cite{helgason_differential_1962}, Corollary~X.2.8 and Theorem~X.5.6). Moreover, $\Diff_K(G)$ contains an elliptic operator (e.g., the Laplace-Beltrami operator associated to a left-invariant $K$-invariant metric on $G$, cf.\ \cite{helgason_groups_1984}, proof of Proposition~IV.2.2). Therefore, if one chooses a finite system of formally self-adjoint generators of $\Diff_K(G)$, the only property which is missing in order to have a weighted subcoercive system is commutativity of $\Diff_K(G)$.

In fact, under these hypotheses, the following properties are equivalent (cf.\ \cite{thomas_infinitesimal_1984}, or \cite{wolf_harmonic_2007}, \S8.3):
\begin{itemize}
\item $\Diff_K(G)$ is a commutative $*$-subalgebra of $\Diff(G)$;
\item $L^1_K(G)$ is a commutative Banach $*$-subalgebra of $L^1(G)$.
\end{itemize}
The latter condition corresponds to the fact that $(G \rtimes K, K)$ is a \emph{Gelfand pair}\index{Gelfand pair}\footnote{The general notion of Gelfand pair is the following: if $S$ is a locally compact group, and $K$ a compact subgroup of $S$, then $(S,K)$ is said to be a Gelfand pair if the (convolution) algebra $L^1(K;S;K)$ of bi-$K$-invariant integrable functions on $S$ is commutative. The study of a Gelfand pair $(S,K)$ involves the $K$-homogeneous space $S/K$: for instance, bi-$K$-invariant functions on $S$ correspond to $K$-invariant functions on $S/K$. In the case $S = G \rtimes K$, the homogeneous space $S/K$ can be identified with $G$; moreover, the convolution in $L^1(K;S;K)$ corresponds to the convolution in $L^1_K(G)$, and most of the notions and results about Gelfand pairs can be rephrased, in this particular case, in terms of the algebraic structure of $G$ (see, e.g., \cite{carcano_commutativity_1987} or \cite{benson_gelfand_1990}). This has to be kept in mind when comparing the results mentioned here with the ones presented in the literature.}. We now summarize in our context some of the main notions and results from the general theory of Gelfand pairs, for which we refer mainly to the expositions of Faraut \cite{faraut_analyse_1982} and Wolf \cite{wolf_harmonic_2007}, but also to the books of Helgason \cite{helgason_differential_1962}, \cite{helgason_differential_1978}, \cite{helgason_groups_1984}. In the following, we always suppose that $L^1_K(G)$ is commutative; consequently, $G$ must be unimodular (cf.\ \cite{helgason_groups_1984}, Theorem~IV.3.1).

The $K$-invariant joint eigenfunctions $\phi$ of the operators in $\Diff_K(G)$ with $\phi(e) = 1$ are called \index{function!spherical}\emph{$K$-spherical functions}. These functions can be equivalently characterized as the continuous and non-null functions $\phi$ on $G$ such that
\[\phi(x) \phi(y) = \int_K \phi(x \, k(y)) \,dk \qquad\text{for all $x,y \in K$.}\]
The set $\GS_K$ of bounded $K$-spherical functions, with the topology induced by the weak-$*$ topology of $L^\infty(G)$, is identified with the Gelfand spectrum $\GS(L^1_K(G))$ of the commutative Banach $*$-algebra $L^1_K(G)$, via the correspondence which associates to a bounded $K$-spherical function $\phi$ the (multiplicative) linear functional
\[f \mapsto \langle f, \phi \rangle\]
on $L^1_K(G)$. According to this identification, the \index{transform!Gelfand}Gelfand transform --- which is also called the \index{transform!Fourier!spherical}\emph{$K$-spherical Fourier transform} --- of an element $f \in L^1_K(G)$ is the function
\[\Gelf_K f : \GS_K \ni \phi \mapsto \langle f, \phi \rangle \in \C.\]
The $K$-spherical Fourier transform $\Gelf_K$ is a continuous homomorphism of Banach algebras $L^1_K(G) \to C_0(\GS_K)$, with unit norm.

If $\phi$ is a function of positive type on $G$, then its projection $P_K \phi$ is of positive type and $K$-invariant. Let $\PP_K$ denote the set of $K$-invariant functions $\phi$ of positive type on $G$ with $\phi(e) = 1$. Then $\PP_K$ is a closed and convex subset of $\PP_1$ (see \S\ref{subsection:positivetype}), whose extreme points are the elements of $\GS_K^+ = \GS_K \cap \PP_K$, i.e., the $K$-spherical functions of positive type; in particular, by the Krein-Milman theorem, the convex hull of $\GS_K^+$ is weakly-$*$ dense in $\PP_K$.

By restricting $K$-spherical transforms to $\GS_K^+$, one obtains that
\[(\Gelf_K (f^*))|_{\GS_K^+} = \overline{(\Gelf_K f)|_{\GS_K^+}},\]
therefore the map $f \mapsto (\Gelf_K f)|_{\GS_K^+}$ is a $*$-homomorphism $L^1_K(G) \to C_0(\GS_K^+)$ with unit norm and dense image. Moreover, there exists a unique positive regular Borel measure $\sigma_K$ on $\GS_K^+$, which is called the \index{Plancherel measure!for a Gelfand pair}\emph{Plancherel measure} of the Gelfand pair $(G \rtimes K, K)$, such that
\[\int_G |f(x)|^2 \,dx = \int_{\GS_K^+} |\Gelf_K f(\phi)|^2 \,d\sigma_K(\phi)\]
for all $f \in L^1_K \cap L^2_K(G)$; further, the map $f \mapsto (\Gelf_K f)|_{\GS_K^+}$ extends to an isomorphism $L^2_K(G) \to L^2(\GS_K^+,\sigma_K)$.

Choose now a finite system $L_1,\dots,L_n$ of formally self-adjoint generators of $\Diff_K(G)$. As we have seen before, the system $L_1,\dots,L_n$ is a weighted subcoercive system on $G$. If the map $\evmap_L$ of \S\ref{section:eigenfunctions} is extended to all the joint eigenfunctions of $L_1,\dots,L_n$, then it is known (see \cite{ferrari_ruffino_topology_2007}) that
\[\evmap_L|_{\GS_K} : \GS_K \to \C^n\]
is a homeomorphism with its image $\evmap_L(\GS_K)$, which is a closed subset of $\C^n$. Notice that
\[\GS_K^+ \subseteq \PP_L\]
(i.e., the $K$-spherical functions of positive type are joint eigenfunctions of positive type of $L_1,\dots,L_n$) and that
\[\evmap_L(\GS_K^+) = \evmap_L(\PP_L)\]
(since the $K$-invariant projection of a joint eigenfunction is still a joint eigenfunction); consequently, for every $\lambda \in \evmap_L(\PP_L)$, there exists a unique element of $\evmap_L^{-1}(\lambda) \cap \PP_L$ which is a $K$-spherical function (cf.\ \cite{helgason_groups_1984}, Proposition~IV.2.4).

The embedding $\evmap_L$ allows us to compare the notions of $K$-spherical transform $\Gelf_K$ and Plancherel measure $\sigma_K$ of the Gelfand pair $(G \rtimes K, K)$ with the notions of kernel transform $\Kern_L$ and Plancherel measure $\sigma$ associated to the weighted subcoercive system $L_1,\dots,L_n$. As a preliminary remark, notice that from Proposition~\ref{prp:automorphismskernel} it follows that, for every bounded Borel $m : \R^n \to \C$, the corresponding kernel $\Kern_L m$ is $K$-invariant.

\begin{prp}\label{prp:gelfandkernel}
Let $f \in L^1_K(G)$. Then there exists $m \in C_0(\R^n)$ such that
\[\Gelf_K f(\phi) = m(\evmap_L(\phi)) \qquad\text{for $\phi \in \GS_K^+$.}\]
For any of such $m$, and for every unitary representation $\pi$ of $G$, we have
\[\pi(f) = m(d\pi(L_1),\dots,d\pi(L_n)),\]
and in particular
\[f = \Kern_L m.\]
\end{prp}
\begin{proof}
Since $\Gelf_K f|_{\GS_K^+} \in C_0(\GS_K^+)$, and since $\evmap_L|_{\GS_K^+}$ is a homeomorphism with its image, which is a closed subset of $\R^n$, then by the Tietze-Urysohn extension theorem we can find $m \in C_0(\R^n)$ extending $(\Gelf_K f) \circ (\evmap_L|_{\GS_K^+})^{-1}$.

By Proposition~\ref{prp:Jl1}, for every $u \in \JJ_L$ and every unitary representation $\pi$ of $G$, we have
\[\pi(\breve u) = u(d\pi(L_1),\dots,d\pi(L_n));\]
therefore the map
\[\JJ_L \ni u \mapsto \breve u \in L^1(G)\]
extends by density (see Proposition~\ref{prp:Jdensity}) to a $*$-homomorphism
\[\Phi : C_0(\R^n) \to C^*(G),\]
and we have
\[\pi(\Phi(u)) = u(d\pi(L_1,\dots,L_n))\]
for all $u \in C_0(\R^n)$ and all unitary representations $\pi$ of $G$. The conclusion will then follow if we prove that
\[f = \Phi(m)\]
as elements of $C^*(G)$.

Recall that every $\phi \in \PP_1$ defines a positive continuous functional $\omega_\phi$ on $C^*(G)$ with unit norm, which extends
\[L^1(G) \ni h \mapsto \langle h, \phi \rangle \in \C;\]
in fact, the norm of any $g \in C^*(G)$ is given by
\[\|g\|_* = \sup_{\phi \in \PP_1} \omega_\phi(g * g^*)\]
(cf.\ \S\ref{subsection:positivetype}). Therefore, in order to conclude, it will be sufficient to show that the set $A$ of the $\phi \in \PP_1$ such that
\[\omega_\phi( (f - \Phi(m)) * (f - \Phi(m))^* ) = 0\]
coincides with the whole $\PP_1$.

Notice that both $f$ and $\Phi(m)$ belong to the closure $C^*_K(G)$ of $L^1_K(G)$ in $C^*(G)$, and it is easily checked that, for $\phi \in \PP_1$ and $g \in C^*_K(G)$,
\[\omega_\phi(g) = \omega_{P_K \phi}(g);\]
consequently, we are reduced to prove that $\PP_K \subseteq A$. In fact, since $A$ is a closed convex subset of $\PP_1$, it is sufficient to prove the inclusion $\GS_K^+ \subseteq A$.

On the other hand, the functionals $\omega_\phi$ for $\phi \in \GS_K^+$ are multiplicative on $L^1_K(G)$, thus they are also multiplicative on $C^*_K(G)$ by continuity, therefore
\[\omega_\phi( (f - \Phi(m)) * (f - \Phi(m))^* ) = |\omega_\phi(f - \Phi(m))|^2 = |\Gelf_K f(\phi) - m(\evmap_L(\phi))|^2 = 0\]
for every $\phi \in \GS_K^+$ (cf.\ Propositions~\ref{prp:eigenfunctions} and \ref{prp:eigenvalues}), and we are done.
\end{proof}

Thus, by applying first $\Gelf_K$ and then $\Kern_L$, we are back at the beginning. The composition of the transforms in reverse order is considered in the following statement, which gives also an improvement of Proposition~\ref{prp:riemannlebesgue1} in this particular context.

\begin{cor}
Let $m : \R^n \to \C$ be a bounded Borel function such that $\breve m \in L^1(G)$. Then $\breve m \in L^1_K(G)$ and
\[\Gelf_K (\Kern_L m) (\phi) = m(\evmap_L(\phi)) \qquad\text{for all $\phi \in \GS_K^+$ with $\evmap_L(\phi) \in \Sigma$.}\]
In particular $m|_\Sigma \in C_0(\Sigma)$.
\end{cor}
\begin{proof}
We already know that $\breve m$ is $K$-invariant, so that $\breve m \in L^1_K(G)$. Therefore, by Proposition~\ref{prp:gelfandkernel}, we can find $u \in C_0(\R^n)$ such that
\[\Gelf_K \breve m (\phi) = u(\evmap_L(\phi))\]
for all $\phi \in \GS_K^+$, and we have
\[\breve m = \breve u,\]
i.e.,
\[m(L_1,\dots,L_n) = u(L_1,\dots,L_n),\]
which means, by the properties of the spectral integral, that $m$ and $u$ must coincide on the joint spectrum $\Sigma$ of $L_1,\dots,L_n$, and we are done.
\end{proof}

Finally, we compare the Plancherel measures $\sigma$ and $\sigma_K$.

\begin{cor}
We have
\[\sigma = \evmap_L|_{\GS_K^+}(\sigma_K), \qquad \sigma_K = (\evmap_L|_{\GS_K^+})^{-1}(\sigma).\]
\end{cor}
\begin{proof}
Recall that $\evmap_L|_{\GS_K^+}$ is a homeomorphism with its image, which is a closed subset of $\R^n$ containing the support $\Sigma$ of $\sigma$, thus the two equalities to be proved are equivalent. 

Set $\tilde\sigma = (\evmap_L|_{\GS_K^+})^{-1}(\sigma)$. Then $\tilde\sigma$ is a positive regular Borel measure on $\GS_K^+$. Moreover, if $f \in L^1_K \cap L^2_K(G)$, then by Proposition~\ref{prp:gelfandkernel} there is $m \in C_0(\R^n)$ such that
\[\Gelf_K f(\phi) = m(\evmap_L(\phi)) \qquad\text{for all $\phi \in \GS_K^+$}\]
and
\[f = \breve m.\]
Since $f \in L^2(G)$, by Theorem~\ref{thm:plancherel} we also have $m \in L^2(\sigma)$, and
\[\int_G |f(x)|^2 \,dx = \int_{\R^n} |m|^2 \,d\sigma = \int_{\GS_K^+} |\Gelf_K f|^2 \,d\tilde\sigma\]
by the change-of-variable formula for push-forward measures. By arbitrariness of $f \in L^1_K \cap L^2_K(G)$ and uniqueness of the Plancherel measure of a Gelfand pair, we obtain that $\sigma_K = \tilde \sigma$, and we are done.
\end{proof}

We have thus shown that the study of the algebra $\Diff_K(G)$ of differential operators associated to a Gelfand pair $(G \rtimes K, K)$ fits into the more general setting of weighted subcoercive systems, where in general there is no compact group $K$ of automorphisms which determines the algebra of operators.

In fact, in the context of Gelfand pairs, there are additional tools (e.g., the projection $P_K$) which permit clearer formulations of results, and simplify several proofs. On the other hand, it should be noticed that the hypothesis of Gelfand pair is quite restrictive. Namely, if $L^1_K(G)$ is commutative, then $G$ must be unimodular. Moreover, the algebra $\Diff_K(G)$ contains an elliptic operator, so that its eigenfunctions are analytic (cf.\ \S X.2 of \cite{helgason_groups_1984}), and this is an essential tool in proving some properties of $K$-spherical functions; however in general a weighted subcoercive operator is not analytic hypoelliptic (see, e.g., \cite{helffer_conditions_1982}). Further, if $G$ is solvable, then $G$ must have polynomial growth, and, if $G$ is nilpotent, then $G$ is at most $2$-step (see \cite{benson_gelfand_1990}).

In this last case, which we will refer to as of \emph{nilpotent Gelfand pairs}\index{Gelfand pair!nilpotent}, notice that it is always possible to find a family of automorphic dilations $\delta_t$ on $G$ which commute with the elements of $K$ (in fact, one obtains a $K$-invariant stratification $\lie{g} = V_1 \oplus V_2$ by taking $V_2 = [\lie{g},\lie{g}]$ and $V_1 = V_2^\perp$ with respect to some $K$-invariant inner product on $\lie{g}$). With respect to such dilations, the algebra $\Diff_K(G)$ is a homogeneous $*$-subalgebra of $\Diff(G)$, so that it is possible to choose a system $L_1,\dots,L_n$ of formally self-adjoint generators of $\Diff_K(G)$ which are also $\delta_t$-homogeneous, and we fall into the case of homogeneous weighted subcoercive systems which are considered in the multiplier theorems of Chapter~\ref{chapter:multipliers}.

On the other hand, our results can be applied also to homogeneous groups which are $3$-step or more, and which therefore do not belong to the realm of Gelfand pairs. For instance, in \S\ref{subsection:example3step} we show that, on the free $3$-step nilpotent group $N_{2,3}$ with $2$ generators, endowed with an action of $SO_2$ by automorphisms, although the whole algebra of $SO_2$-invariant operators on $N_{2,3}$ cannot be commutative, yet there exists a non-trivial commutative subalgebra to which our results can be applied.

%% file: weighted.tex
\chapter{Weighted inequalities for kernels}\label{chapter:weighted}

Having established, in the previous chapter, the main properties of the kernel transform $\Kern_L$ associated to a weighted subcoercive system $L_1,\dots,L_n$ on a connected Lie group, here we investigate the problem of integrability of the kernel $\Kern_L m$ in terms of smoothness conditions on the multiplier $m$. The results of this chapter can thus be thought of as weak multiplier theorems, which are preliminary to the more refined results of Chapter~\ref{chapter:multipliers}, where singular integral operators are considered.

In the context of Euclidean Fourier multipliers, the above problem is reduced, by H\"older's inequality, to controlling a weighted $L^2$ norm of the kernel, which corresponds, by the properties of the Fourier transform, to an $L^2$ Sobolev norm of the multiplier. A similar route is followed here: a multi-variate analogue of a long-used Fourier-series-decomposition technique (see, e.g., \S6.B of \cite{folland_hardy_1982}), combined with heat kernel estimates, allows to control a weighted $L^2$ norm of the kernel in terms of a Sobolev (or Besov) norm of the multiplier, provided that the latter is supported in a fixed compactum; in the case of a group with polynomial growth, this yields also control on the integrability of the kernel. Interpolation plays here a fundamental role, as in \cite{mauceri_vectorvalued_1990}, in order to reduce the smoothness requirements on the multiplier.

For non-compactly supported multipliers, we describe a general result, holding in any group with polynomial growth, which is due to Hulanicki \cite{hulanicki_functional_1984} in the case of a single operator: if the multiplier $m$ is a Schwartz function, then the corresponding kernel $\Kern_L m$ is also Schwartz (and in particular is integrable).

Going back to the compactly supported case, further improvements on the smoothness requirements are achieved through a technique due to Hebisch and Zienkiewicz \cite{hebisch_multiplier_1995}, which exploits algebraic properties of the group in a non-trivial way, highlighting a phenomenon which is specific to the non-commutative realm.
 
Finally, some examples of weighted subcoercive systems on specific groups are presented, with explicit computations of the Plancherel measure, and the possibility of giving a ``coordinate-free'' formulation of the so-far obtained results is discussed.

\section{Weighted estimates}\label{section:weightedestimates}

Let $L_1,\dots,L_n$ be a weighted subcoercive system on a connected Lie group $G$, with joint spectral resolution $E$, joint spectrum $\Sigma \subseteq \R^n$, Plancherel measure $\sigma$, and define the polynomials $p_*,p_0,p_1,\dots,p_n$ as in \S\ref{section:plancherel}.
%

Let $p : \R^n \to \R^{1+n}$ be the map whose components are the polynomials $p_0,\dots,p_n$. For $l \in \Z^{1+n}$, set
\[E_l(\lambda) = e^{i l \cdot e^{-p(\lambda)}} - 1 = e^{i (l_0 e^{-p_0(\lambda)} + l_1 e^{-p_1(\lambda)} + \dots + l_n e^{-p_n(\lambda)})} - 1.\]
Then $E_l \in C_0(\R^n)$, and in fact
\[E_l = \sum_{0 \neq k \in \N^{1+n}} \frac{(il_0)^{k_0} \cdots (il_n)^{k_n}}{k_0! \cdots k_n!} e^{-k_0 p_0} \cdots e^{-k_n p_n},\]
with uniform convergence on $\R^n$. This means that, if $h_{\nu,t}$ is the heat kernel of $p_\nu(L)$ for $\nu =0,\dots,n$ (with $h_{\nu,0}$ denoting the Dirac delta at the identity of $G$), then
\begin{equation}\label{eq:convolutionseries}
\breve E_l = \sum_{0 \neq k \in \N^{1+n}} \frac{(il_0)^{k_0} \cdots (il_n)^{k_n}}{k_0! \cdots k_n!} h_{0,k_0} * \cdots * h_{n,k_n}
\end{equation}
in $\Cv^2(G)$.

\begin{lem}\label{lem:El2}
There exists $C > 0$ such that
\[\|\breve E_l\|_2 \leq C |l| \qquad\text{for all $l \in \Z^{1+n}$.}\]
\end{lem}
\begin{proof}
We have
\[|E_l(\lambda)| \leq |l \cdot e^{-p(\lambda)}| \leq \sum_{\nu = 0}^n |l_j| e^{-p_j(\lambda)} \leq (1+n) |l| e^{-p_*(\lambda)},\]
so that in particular, if
\[f = e^{p_*} E_l \qquad\text{and}\qquad k=(e^{-p_*})\breve{\,\,},\]
then, by Lemma~\ref{lem:composition}, $\breve E_l = f(L)k$ and
\[\|\breve E_l\|_2 \leq \|f\|_\infty \|k\|_2 \leq (1+n) \|k\|_2 |l|,\]
which is the conclusion.
\end{proof}

Let $|\cdot|_G$ be a connected modulus on $G$. We introduce the notation
\[\langle x \rangle_G = 1 + |x|_G.\]
Notice that, since $|\cdot|_G$ is subadditive, $\langle \cdot \rangle_G$ is submultiplicative.

\begin{lem}\label{lem:El2exp}
There exist $c,\omega > 0$ such that
\[\|\breve E_l \|_{L^2(G, e^{2|x|_G}\,dx)} \leq c e^{\omega |l|} \qquad\text{for all $l \in \Z^{1+n}$.}\]
\end{lem}
\begin{proof}
Since all the connected right-invariant distances on $G$ are equivalent in the large, by interpolating the estimates (d) and (e) of Theorem~\ref{thm:robinsonterelst}, we have that there exist $c \geq 1$ and $\omega > 0$ such that
\[\|h_{\nu,t} \,e^{|\cdot|_G}\|_{q} \leq c e^{\omega t} \qquad\text{for $t \geq 1$, $\nu = 0,\dots,n$ and $q \in [1,\infty]$.}\]

Let $k \in \N^{1+n} \setminus \{0\}$ and notice that, in the convolution product
\[h_{0,k_0} * \cdots * h_{n,k_n},\]
the factors $h_{\nu,k_\nu}$ with $k_\nu = 0$ can be simply omitted; since the weight $e^{|\cdot|_G}$ is submultiplicative, by applying Young's inequality to the so reduced product we obtain
\[\|(h_{0,k_0} * \cdots * h_{n,k_n}) \,e^{|\cdot|_G}\|_{q} \leq c^{1+n} e^{\omega(k_0 + \dots + k_n)}\]
for $q \in [1,\infty]$.

This means in particular that the series in \eqref{eq:convolutionseries} converges absolutely in $L^2(G,e^{2|x|_G}\,dx)$, with
\[\sum_{0 \neq k \in \N^{1+n}} \left\| \frac{(il_0)^{k_0} \cdots (il_n)^{k_n}}{k_0! \cdots k_n!} h_{0,k_0} * \cdots * h_{n,k_n} \right\|_{L^2(G,e^{2|x|_G}\,dx)} \leq c^{1+n} e^{e^{\omega} |l|}.\]
Therefore, by uniqueness of limits, we also have $\breve E_l \in L^2(G, e^{2|x|_G}\,dx)$ and
\[\|\breve E_l\|_{L^2(G,e^{2|x|_G}\,dx)} \leq c^{1+n} e^{e^{\omega} |l|},\]
which is the conclusion.
\end{proof}

\begin{lem}\label{lem:El2poly}
For all $\alpha \geq 0$, we have
\[\|\breve E_l\|_{L^2(G,\langle x \rangle_G^{2\alpha} \,dx)} \leq C_\alpha |l|^{\alpha+1} \qquad\text{for $l \in \Z^{1+n}$.}\]
\end{lem}
\begin{proof}
First, notice that, by Lemma~\ref{lem:El2exp}, it is sufficient to check the estimate for $|l|$ large. But then, for $|l|$ large, we have
\begin{multline*}
\int_{G} |\breve E_l(x)|^2 \langle x \rangle_G^{2\alpha} \,dx \leq \int_{|x|_G \leq \omega|l|} + \int_{|x|_G > \omega|l|} \\
\leq (1+\omega|l|)^{2\alpha} \|\breve E_l\|_2^2 + \left(\sup_{r > \omega|l|} e^{-2r} (1+r)^{2\alpha} \right) \|\breve E_l\|_{L^2(G,e^{2|x|_G}\,dx)}^2 \\
\leq C_\alpha |l|^{2(\alpha+1)}
\end{multline*}
by Lemmata~\ref{lem:El2} and \ref{lem:El2exp}.
%
\end{proof}

\begin{lem}\label{lem:fourierdecomposition}
Let $K \subseteq \R^n$ be compact. For every $f \in \D(\R^n)$ supported in $K$, there exists $g \in \D(\T^{1+n})$ (depending linearly on $f$) such that
\[f(\lambda) = g\left(e^{i e^{-p(\lambda)}}\right) = g\left(e^{i e^{-p_0(\lambda)}}, \dots, e^{i e^{-p_n(\lambda)}}\right),\]
\[g(1,\dots,1) = 0,\]
\[\|g\|_{H^s(\T^{1+n})} \leq C_{K,s} \|f\|_{H^s(\R^n)} \qquad\text{for all $s \geq 0$.}\]
In particular, if $g(e^{it}) = \sum_{l \in \Z^{1+n}} \hat{g}(l) e^{i l \cdot t}$ is the Fourier series development of $g$, then we have
\[f = \sum_{0 \neq l \in \Z^{1+n}} \hat g(l) E_l,\]
with uniform convergence on $\R^n$.
\end{lem}
\begin{proof}
Since $K \subseteq \R^n$ is compact and the polynomials $p_0,\dots,p_n$ are strictly positive, $p(K)$ is a compact subset of $\Omega = \left]0,+\infty\right[^{1+n}$. Therefore we can choose $\psi_K \in \D(\Omega)$ such that $\psi_K|_{p(K)} \equiv 1$. If we put
\[\tilde f(y) = f(y_1 - y_0,\dots,y_n - y_0) \psi_K(y) \qquad\text{for $y \in \R^{1+n}$,}\]
we then have that $\tilde f \in \D(\Omega)$, $f = \tilde f \circ p$ and
\[\|\tilde f\|_{H^s(\R^{1+n})} \leq C_{K,s} \|f\|_{H^s(\R^n)} \qquad\text{for all $s \geq 0$,}\]
by Proposition~\ref{prp:besovchange}, since the change of variables has maximal rank.

Notice now that the map
\[\Phi : \Omega \ni y \mapsto e^{i e^{-y}} = (e^{i e^{-y_0}},\dots,e^{i e^{-y_n}}) \in \T^{1+n}\]
is a smooth diffeomorphism with its image, which is an open subset of $\T^{1+n}$ not containing $(1,\dots,1)$. The function $g = \tilde f \circ \Phi^{-1} \in \D(\Phi(\Omega))$ can be then extended by zero to a smooth function on $\T^{1+n}$, and we have clearly
\[\|g\|_{H^s(\T^{1+n})} \leq C_{K,s} \|\tilde f\|_{H^s(\R^{1+n})} \qquad\text{for all $s \geq 0$.}\]

The construction shows that $g$ depends linearly on $f$ and that
\[f = g \circ \Phi \circ p, \qquad g(1,\dots,1) = 0,\]
\[\|g\|_{H^s(\T^{1+n})} \leq C_{K,s} \|f\|_{H^s(\R^n)} \qquad\text{for all $s \geq 0$.}\]
In particular, we have
\[\sum_{l \in \Z^{1+n}} \hat g(l) = 0,\]
so that the Fourier decomposition of $g$ can be rewritten as
\[g(e^{it}) = \sum_{0 \neq l \in \Z^{1+n}} \hat{g}(l) (e^{i l \cdot t} - 1)\]
(with uniform convergence since $g$ is smooth), which implies that
\[f = g \circ \Phi \circ p = \sum_{0 \neq l \in \Z^{1+n}} \hat g(l) E_l,\]
with uniform convergence on $\R^n$.
\end{proof}

\begin{prp}\label{prp:sobolevl2estimates}
Let $K \subseteq \R^n$ be compact, $\alpha \geq 0$. For all $f \in \D(\R^n)$ with $\supp f \subseteq K$, we have
\[\|\breve f\|_{L^2(G, \langle x \rangle_G^{2\alpha} \,dx)} \leq C_{K,\alpha,\beta} \|f\|_{H^\beta(\R^n)}\]
for $\beta > \alpha + (n+3)/2$.
\end{prp}
\begin{proof}
Let $g \in \D(\T^{1+n})$ be given by Lemma~\ref{lem:fourierdecomposition}. Then
\[\breve f = \sum_{0 \neq l \in \Z^{1+n}} \hat g(l) \breve E_l\]
in $\Cv^2(G)$. However, the series in the right-hand side converges absolutely in $L^2(G,\langle x \rangle_G^{2\alpha} \,dx)$, since
\begin{multline*}
\sum_{0 \neq l \in \Z^{1+n}} |\hat g(l)| \, \|\breve E_l\|_{L^2(G,\langle x \rangle_G^{2\alpha}\,dx)} \leq C_\alpha \sum_{0 \neq l \in \Z^{1+n}} |\hat g(l)| \, |l|^{\alpha + 1} \\
\leq C_{\alpha,\beta} \left(\sum_{0 \neq l \in \Z^{1+n}} |l|^{-2(\beta-\alpha-1)} \right)^{1/2} \|g\|_{H^\beta(\T^{1+n})} \leq C_{K,\alpha,\beta} \|f\|_{H^\beta(\R^n)},
\end{multline*}
by Lemma~\ref{lem:El2poly}, H\"older's inequality, Proposition~\ref{prp:sobolevtorus} and the fact that
\[2(\beta - \alpha - 1) > 1+n.\]
Therefore, by uniqueness of limits, also $\breve f \in L^2(G,\langle x \rangle_G^{2\alpha} \,dx)$ and
\[\|\breve f\|_{L^2(\langle x \rangle_G^{2\alpha} \,dx)} \leq C_{K,\alpha,\beta} \|f\|_{H^\beta(\R^n)},\]
which is the conclusion.
\end{proof}

\begin{thm}\label{thm:l2estimates}
\begin{itemize}
\item[(i)] Let $K \subseteq \R^n$ be compact, $D \in \Diff(G)$, $\alpha \geq 0$ and
\[\beta > \alpha.\]
For all $f \in B^{\beta}_{\infty,\infty}(\R^n)$ with $\supp f \subseteq K$, we have
\[\|D \breve f\|_{L^2(G, \langle x \rangle_G^{2\alpha} \,dx)} \leq C_{K,D,\alpha,\beta} \|f\|_{B_{\infty,\infty}^\beta(\R^n)}.\]
\item[(ii)] Suppose that, for some open $A \subseteq \R^n$ and $0 < d \leq n$, the Plancherel measure $\sigma$ is locally $d$-bounded on $A$. Let $K \subseteq A$ be compact, $D \in \Diff(G)$, $1 \leq p,q \leq \infty$, $\alpha \geq 0$ and
\[\beta > \alpha + \frac{n}{p} - \frac{d}{\max\{2,p\}}.\]
For all $f \in B^{\beta}_{p,q}(\R^n)$ with $\supp f \subseteq K$, we have
\[\|D \breve f\|_{L^2(G, \langle x \rangle_G^{2\alpha} \,dx)} \leq C_{K,D,\alpha,\beta,p,q} \|f\|_{B_{p,q}^\beta(\R^n)}.\]
\end{itemize}
\end{thm}
\begin{proof}
(i) Consider first the case $D = 1$. Let $\xi \in \D(\R^n)$ be such that $\xi|_K \equiv 1$, and let $K' \subseteq \R^n$ be compact with $\supp \xi \subseteq \mathring K'$. By Proposition~\ref{prp:sobolevl2estimates} and Corollary~\ref{cor:besovcompactsupportinclusion} we then have clearly, for $f \in \D(\R^n)$ with $\supp f \subseteq K'$,
\[\|\breve f\|_{L^2(G,\langle x \rangle_G^{2\alpha} \,dx)} \leq C_{K,\alpha,\beta} \|f\|_{H^\beta(\R^n)} \leq C_{K,\alpha,\beta} \|f\|_{B^\beta_{\infty,2}(\R^n)}\]
for $\beta > \alpha + (n+3)/2$. By Proposition~\ref{prp:besovapproximation}, every $f \in B^\beta_{\infty,2}(\R^n)$ with support contained in $\supp \xi$ can be approximated in $B^\beta_{\infty,2}(\R^n)$ by smooth functions with supports contained in $K$, so that the inequality
\[\|\breve f\|_{L^2(G,\langle x \rangle_G^{2\alpha} \,dx)} \leq C_{K,\alpha,\beta} \|f\|_{B^\beta_{\infty,2}(\R^n)},\]
for $\beta > \alpha + (n+3)/2$, extends to all $f \in B^\beta_{\infty,2}(\R^n)$ with $\supp f \subseteq \supp \xi$. In particular, if we consider the linear map
\[M : f \mapsto (f\xi)\breve{\,\,},\]
then, by Proposition~\ref{prp:besovproduct}, we have that
\[\text{$M$ is bounded } B^\beta_{\infty,2}(\R^n) \to L^2(G,\langle x \rangle_G^{2\alpha} \,dx) \text{ for $\beta > \alpha + (n+3)/2$}.\]
On the other hand, in the case $\alpha = 0$, by Theorem~\ref{thm:plancherel} and Propositions~\ref{prp:besovzeroinclusions}, \ref{prp:besovproduct}, we have, for $f \in B^0_{\infty,1}(\R^n)$,
\[\|(f\xi)\breve{\,\,}\|_2 = \|f\xi\|_{L^2(\sigma)} \leq C_K \|f\xi\|_\infty \leq C_K \|f\xi\|_{B^0_{\infty,1}(\R^n)} \leq C_K \|f\|_{B^0_{\infty,1}(\R^n)},\]
that is,
\[\text{$M$ is bounded } B^0_{\infty,1}(\R^n) \to L^2(G).\]
By interpolation (see Proposition~\ref{prp:besovinterpolation} and Theorem~\ref{thm:steinweiss}), we then get that
\[\text{$M$ is bounded } B^{\theta\beta}_{\infty,2}(\R^n) \to L^2(G,\langle x \rangle_G^{2\theta\alpha} \,dx) \text{ for $\theta\beta > \theta\alpha + \theta(n+3)/2$}\]
for all $\theta \in \left]0,1\right[$, that is,
\[\text{$M$ is bounded } B^\beta_{\infty,2}(\R^n) \to L^2(G,\langle x \rangle_G^{2\alpha} \,dx) \text{ for $\beta > \alpha$}.\]
Finally, if $f \in B^\beta_{\infty,\infty}(\R^n)$ for some $\beta > \alpha$ and $\supp f \subseteq K$, then $f = f\xi$ and moreover $f \in B^{\beta'}_{\infty,2}(\R^n)$ for some $\beta' \in \left]\alpha,\beta\right[$, so that
\begin{multline*}
\|\breve f\|_{L^2(G,\langle x \rangle_G^{2\alpha} \,dx)}  = \|M f\|_{L^2(G,\langle x \rangle_G^{2\alpha} \,dx)} \leq C_{K,\alpha,\beta} \|f\|_{B^{\beta'}_{\infty,2}(\R^n)} \\
\leq C_{K,\alpha,\beta} \|f\|_{B^\beta_{\infty,\infty}(\R^n)}
\end{multline*}
by Proposition~\ref{prp:besovembeddings}.

Take now an arbitrary $D \in \Diff(G)$. For $f \in B^\beta_{\infty,\infty}(\R^n)$ with $\supp f \subseteq K$, set
\[f_0 = f e^{p_*}, \qquad \xi = e^{-p_*}.\]
Then, by Lemma~\ref{lem:composition},
\[\breve f = \breve f_0 * \breve\xi,\]
and moreover, by Proposition~\ref{prp:besovproduct}, $f_0 \in B^\beta_{\infty,\infty}(\R^n)$ and $\supp f_0 \subseteq K$, so that
\begin{multline*}
\|D \breve f\|_{L^2(G,\langle x \rangle_G^{2\alpha} \,dx)} \leq \|\breve f_0\|_{L^2(G,\langle x \rangle_G^{2\alpha} \,dx)} \|D\breve\xi\|_{L^1(G,\langle x \rangle_G^{\alpha} \,dx)} \\
\leq C_{K,D,\alpha,\beta} \|f_0\|_{B^\beta_{\infty,\infty}(\R^n)} \leq C_{K,D,\alpha,\beta} \|f\|_{B^\beta_{\infty,\infty}(\R^n)}
\end{multline*}
by Young's inequality, Theorem~\ref{thm:robinsonterelst}(f) and Proposition~\ref{prp:besovproduct}.

(ii) By proceeding as in part (i), but with $K' \subseteq A$, and noticing that, by Proposition~\ref{prp:besovembeddings} and Corollary~\ref{cor:besovcompactsupportinclusion},
\[\|f\|_{H^s(\R^n)} \leq C_{K,s,s'} \|f\|_{B^{s'}_{p,2}(\R^n)}\]
if $\supp f \subseteq K'$ and $s' \geq s + n/2$, we easily obtain that
\[\text{$M$ is bounded } B^\beta_{p,2}(\R^n) \to L^2(G,\langle x \rangle_G^{2\alpha} \,dx) \text{ for $\beta > \alpha + (2n+3)/2$}.\]
If we prove that, for some $q \in [1,\infty]$,
\begin{equation}\label{eq:limitatezzabase}
\text{$M$ is bounded } B^\beta_{p,q}(\R^n) \to L^2(G) \text{ for $\beta > \frac{n}{p} - \frac{d}{\max\{2,p\}}$,}
\end{equation}
then the conclusion follows analogously as in part (i). However, for $p=\infty$, \eqref{eq:limitatezzabase} is proved as in part (i), whereas, for $p=2$, it follows from Corollary~\ref{cor:triebeltrace} and Proposition~\ref{prp:besovembeddings}; therefore, for $2 < p < \infty$, we get \eqref{eq:limitatezzabase} by interpolating the cases $p=2$ and $p=\infty$ (see Proposition~\ref{prp:besovinterpolation}), whereas, for $1 \leq p < 2$, we use the embeddings from Proposition~\ref{prp:besovembeddings}.
\end{proof}

By Sobolev's embedding, it is not difficult to obtain also pointwise estimates:

\begin{cor}\label{cor:pointwiseestimates}
Let $K \subseteq \R^n$ be compact, $D \in \Diff(G)$, $\alpha \geq 0$ and
\[\beta > \alpha.\]
For all $f \in B^{\beta}_{\infty,\infty}(\R^n)$ with $\supp f \subseteq K$, we have
\[\sup_{x \in G} \langle x \rangle_G^\alpha |D \breve f(x)| \leq C_{K,D,\alpha,\beta} \|f\|_{B_{\infty,\infty}^\beta(\R^n)}.\]
If moreover the Plancherel measure $\sigma$ is locally $d$-bounded on some open $A \subseteq \R^n$, $1 \leq p,q \leq \infty$, and we have $K \subseteq A$,
\[\beta > \alpha + \frac{n}{p} - \frac{d}{\max\{2,p\}},\]
then, for all $f \in B^{\beta}_{p,q}(\R^n)$ with $\supp f \subseteq K$,
\[\sup_{x \in G} \langle x \rangle_G^\alpha |D \breve f(x)| \leq C_{K,D,\alpha,\beta,p,q} \|f\|_{B_{p,q}^\beta(\R^n)}.\]
\end{cor}
\begin{proof}
If $\zeta \in \D(\R^n)$ is nonnegative and $\zeta(e) > 0$, and if we set
\[w_\alpha = \langle \cdot \rangle_G^\alpha * \zeta\]
for $\alpha \geq 0$, then $w_\alpha$ is smooth and nonnegative,
\[C_\alpha^{-1} \langle x \rangle_G^\alpha \leq w_\alpha(x) \leq C_\alpha \langle x \rangle_G^\alpha\]
and moreover
\[D w_\alpha(x) \leq C_{D,\alpha} \langle x \rangle_G^\alpha\]
for all $D \in \Diff(G)$.

We then have, for $\beta > \alpha$ and $D \in \Diff(G)$,
\begin{multline*}
\sup_{x \in G} \langle x \rangle_G^\alpha |D \breve f(x)| \leq C_{\alpha} \sup_{x \in G} w_\alpha(x) |D \breve f(x)| \\
\leq C_{\alpha} \sum_{|\gamma| \leq k} \| A^\gamma (w_\alpha D \breve f) \|_2 \leq C_{\alpha} \sum_{|\gamma'| + |\gamma''| \leq k} \| (A^{\gamma'} w_\alpha) (A^{\gamma''} D \breve f) \|_2 \\
\leq C_{\alpha} \sum_{|\gamma| \leq k} \| A^{\gamma} D \breve f \|_{L^2(G,\langle x \rangle_G^{2\alpha} \,dx)} \leq C_{K,D,\alpha,\beta,p,q} \|f\|_{B_{p,q}^\beta}
\end{multline*}
by Sobolev's embedding and Theorem~\ref{thm:l2estimates}.
\end{proof}

Finally, in the case of polynomial growth, by H\"older's inequality we obtain immediately weighted $L^1$ estimates.

\begin{cor}\label{cor:polyl1estimates}
Suppose that $G$ has polynomial growth of degree $Q_G$, and let $K \subseteq \R^n$ be compact, $D \in \Diff(G)$, $\alpha \geq 0$ and
\[\beta > \alpha + \frac{Q_G}{2}.\]
For all $f \in B^{\beta}_{\infty,\infty}(\R^n)$ with $\supp f \subseteq K$, we have $\breve f \in L^{1;\infty}(G)$ and
\[\|D \breve f\|_{L^1(G, \langle x \rangle_G^{\alpha} \,dx)} \leq C_{K,D,\alpha,\beta} \|f\|_{B^\beta_{\infty,\infty}(\R^n)}.\]
If moreover the Plancherel measure $\sigma$ is locally $d$-bounded on some open $A \subseteq \R^n$, $1 \leq p,q \leq \infty$, and we have $K \subseteq A$,
\[\beta > \alpha + \frac{Q_G}{2} + \frac{n}{p} - \frac{d}{\max\{2,p\}},\]
then, for all $f \in B^{\beta}_{p,q}(\R^n)$ with $\supp f \subseteq K$,
\[\|D \breve f\|_{L^1(G, \langle x \rangle_G^{\alpha} \,dx)} \leq C_{K,D,\alpha,\beta,p,q} \|f\|_{B^\beta_{p,q}(\R^n)}.\]
\end{cor}

\section{Schwartz functions}

As before, let $L_1,\dots,L_n$ be a weighted subcoercive system on a connected Lie group $G$, with joint spectrum $\Sigma$.

By Lemma~\ref{lem:composition}(c) and Proposition~\ref{prp:plancherelpolynomialgrowth}, it follows easily that, if a bounded Borel function $f : \R^n \to \C$ decays at infinity faster than the reciprocal of any polynomial, then $\Kern_L f \in L^{2;\infty}(G)$. In the case $G$ has polynomial growth, this statement can be amplified in terms of the Schwartz class (see \S\ref{subsection:schwartz}).

\begin{prp}
Suppose that $G$ has polynomial growth. If $f \in \Sz(\R^n)$, then $\Kern_L f \in \Sz(G)$, and the map
\[\Sz(\R^n) \ni f \mapsto \Kern_L f \in \Sz(G)\]
is continuous.
\end{prp}
\begin{proof}
Consider first the case $n = 1$, $L = L_1$ positive and weighted subcoercive. By Theorem~\ref{thm:robinsonterelst}(e), the heat kernel $k_t = \Kern_L(e^{-t\cdot})$ belongs to $\Sz(G)$ for all $t > 0$. The conclusion then follows by a result of Hulanicki \cite{hulanicki_functional_1984}.

Suppose next that $n \in \N$ is arbitrary, but all the $L_j$ are positive and weighted subcoercive. Then we can proceed as in the proof of \cite{astengo_gelfand_2009}, Theorem~5.2. Namely, if $f = f_1 \otimes \cdots \otimes f_n$, then
\[\Kern_L (f_1 \otimes \cdots \otimes f_n) = \Kern_{L_1}(f_1) \cdots \Kern_{L_n}(f_n)\]
by definition of joint spectral resolution, and the conclusion follows by the case $n=1$ since
\[\Sz(\R^n) = \Sz(\R) \ptimes \cdots \ptimes \Sz(\R).\]

Finally, in the general case, let $p_0,p_1,\dots,p_n$ be the polynomials defined in \S\ref{section:plancherel}. Then the operators
\[\tilde L_0 = p_0(L), \quad \tilde L_1 = p_1(L), \quad \dots, \quad \tilde L_n = p_n(L)\]
are all positive and weighted subcoercive, so that in particular they are a weighted subcoercive system. Moreover
\[\Kern_L f = \Kern_{\tilde L}((f \circ \Phi) \eta),\]
where
\[\Phi(\lambda_0,\lambda_1,\dots,\lambda_n) = (\lambda_1 - \lambda_0, \dots, \lambda_n - \lambda_0),\]
\[\eta(\lambda_0,\lambda_1,\dots,\lambda_n) = \psi(\lambda_0 - p_0(\lambda_1-\lambda_0,\dots,\lambda_n-\lambda_0)),\]
and $\psi \in \D(\R)$ is such that $\psi(0) = 1$.
The conclusion then follows by the previous case and by the fact that the map
\[f \mapsto (f \circ \Phi)\eta\]
is continuous $\Sz(\R^n) \to \Sz(\R^{1+n})$.
\end{proof}

\begin{rem}
The proof of the mentioned result of Hulanicki \cite{hulanicki_functional_1984} uses (in the one-variable case) a Fourier series decomposition analogous to the one exploited in \S\ref{section:weightedestimates} in order to obtain weighted $L^2$ estimates. However, since Schwartz functions in general are not compactly supported, the exponential change of variable of Lemma~\ref{lem:fourierdecomposition} does no longer apply. Hulanicki uses instead a polynomial change of variable, and correspondingly he requires, in place of heat kernel estimates, some estimate on the resolvent kernel.
\end{rem}

\begin{rem}
The inverse problem --- i.e., determining if, for each $f : \Sigma \to \C$ such that $\Kern_L f \in \Sz(G)$, $f$ is the restriction to $\Sigma$ of some $\tilde f \in \Sz(\R^n)$ --- is much more difficult. In the context of nilpotent Gelfand pairs, some positive results in this direction are known (see \cite{astengo_gelfand_2007}, \cite{astengo_gelfand_2009}, \cite{fischer_gelfand_2009}, \cite{fischer_nilpotent_2010}).
\end{rem}

\section{Improved weighted estimates}\label{section:metivier}

In the weighted $L^1$ estimates obtained in Corollary~\ref{cor:polyl1estimates}, 
the regularity threshold is half of the degree $Q_G$ of polynomial growth of the Lie group $G$. In the case $G = \R^n$, the degree $Q_G$ coincides with the topological dimension $\dim G = n$. For more general nilpotent Lie groups, however, $Q_G > \dim G$. Nevertheless, for a particular class of $2$-step homogeneous groups --- namely, Heisenberg and related groups --- multiplier theorems have been proved with $(\dim G)/2$ as the regularity threshold. Here we present a technique, due to Hebisch and Zienkiewicz \cite{hebisch_multiplier_1995}, which allows (in some cases, including those just mentioned) to lower the threshold in the weighted $L^1$ estimates, and whose generality is particularly suited to our context.

Let $G$ be a nilpotent Lie group, with Lie algebra $\lie{g}$. Let $\lie{z}$ be the center of $\lie{g}$, and set
\begin{equation}\label{eq:centercenter}
\lie{y} = \{ v \in \lie{g} \tc [v,\lie{g}] \subseteq \lie{z}\};
\end{equation}
$\lie{y}$ is an ideal of $\lie{g}$, which contains $\lie{z}$ and is characteristic, i.e., $\lie{y}$ is invariant by automorphisms of $\lie{g}$. Let moreover $P : \lie{g} \to \lie{g}/\lie{z}$ be the canonical projection. The bilinear map
\[[\cdot,\cdot] : \lie{g} \times \lie{g} \to \lie{g}\]
induces, by restriction, passage to the quotient and transposition, another bilinear map
\[J : \lie{g}/\lie{z} \times \lie{z}^* \to \lie{y}^*,\]
which we will call the \emph{capacity map}\index{capacity map} of $\lie{g}$, and is uniquely determined by
\[J(P(x),\tau)(y) = \tau([x,y])\]
for $x \in \lie{g}$, $y \in \lie{y}$, $\tau \in \lie{z}^*$. The group $G$ is said to be a \emph{Heisenberg-type group} (or \emph{H-type group}\index{Lie group!H-type} for short) if there exists an inner product on $\lie{g}$ such that, for every $\tau \in \lie{z}^*$ of norm $1$, the map
\[J(\cdot,\tau) : \lie{g}/\lie{z} \to \lie{y}^*\]
is an isometric embedding (this implies that $\lie{g} = \lie{y}$, so that $G$ is $2$-step). If $G$ is a H-type group, then in particular we have
\[|J(\bar x,\tau)| \geq |\bar x| |\tau|\]
for suitable norms on $\lie{g}/\lie{z}$, $\lie{z}^*$ and $\lie{y}^*$; the validity of such an inequality defines the class of \emph{M\'etivier groups}\index{Lie group!M\'etivier}, which has been introduced in the study of analytic hypoellipticity of Rockland operators (see \cite{metivier_hypoellipticite_1980}, or \cite{helffer_conditions_1982}); this class is strictly larger than that of H-type groups (see \cite{mller_singular_2004} for an example), but is still contained in the class of $2$-step groups.

In the following, we consider a more general inequality of the form
\[|J(\bar x,\tau)| \geq w(\bar x) \zeta(\tau)\]
for some non-negative functions $w : \lie{g}/\lie{z} \to \R$, $\zeta : \lie{z}^* \to \R$, which can be also rewritten as
\[w(\bar x)^{\beta} \leq |J(\bar x,\tau)|^\beta \zeta(\tau)^{-\beta}\]
for some $\beta > 0$. The idea is to interpret this inequality via the spectral theorem, in order to control a multiplication operator (corresponding to $w(\bar x)^\beta$) with a function of the central derivatives (corresponding to $\zeta(\tau)^{-\beta})$; in this interpretation, 
it turns out that $|J(\bar x,\tau)|^2$ corresponds to a sum of products of left-invariant and right-invariant differential operators on $G$, 
therefore the term $|J(\bar x,\tau)|^\beta$ can be dominated by a clever use of an a priori estimate for a weighted subcoercive operator on the direct product $G \times G$.

In order to fill in the details, it is convenient to introduce some notation.

For every smooth differential operator $D$ on $G$, the identity
\[(Df)^* = D^\dstar f^*\]
defines another differential operator $D^\dstar$ on $G$. It is immediate to check that
\[(\lambda_1 D_1 + \lambda_2 D_2)^\dstar = \overline{\lambda_1} D_1^\dstar + \overline{\lambda_2} D_2^\dstar, \qquad (D_1 D_2)^\dstar = D_1^\dstar D_2^\dstar, \qquad D^{\dstar\dstar} = D,\]
i.e., the map $D \mapsto D^\dstar$ is a conjugate-linear involutive automorphism of the unital algebra of all smooth differential operators on $G$; moreover, since
\[(L_x f)^* = R_x (f^*) \qquad\text{for all $x \in G$,}\]
this correspondence maps left-invariant operators to right-invariant ones, and vice versa.

The Lie algebra $\lie{\tilde g}$ of the direct product $\tilde G = G \times G$ is canonically isomorphic to $\lie{g} \oplus \lie{g}$, where the action of $(X,Y) \in \lie{g} \oplus \lie{g}$ on the smooth functions on $\tilde G$ is determined by
\[(X,Y) (f \otimes g) = (X f) \otimes (Y g).\]
We define the correspondence $D \mapsto D^\dtwist$ on $\Diff(\tilde G)$ as the unique conjugate-linear automorphism of the unital algebra $\Diff(\tilde G) \cong \UEnA(\lie{g} \oplus \lie{g})$ which extends the Lie algebra automorphism
\[(X,Y) \mapsto (Y,X)\]
of $\lie{g} \oplus \lie{g}$.

Let $\xi$ be the unitary representation of $\tilde G$ on $L^2(G)$ given by
\[\xi(x,y) f = R_x L_y f.\]
Then, for every $D \in \Diff(\tilde G)$, $d\xi(D)$ is a smooth differential operator on $G$, which in general is neither left- nor right-invariant. In fact, if $D \in \Diff(\tilde G)$ is an operator along the first factor of $G \times G$, then $d\xi(D)$ is left-invariant, whereas, if $D$ is along the second factor, then $d\xi(D)$ is right-invariant.

\begin{lem}
For every $D \in \Diff(\tilde G)$,
\[d\xi(D^\dtwist) = d\xi(D)^\dstar.\]
\end{lem}
\begin{proof}
Since both sides of the equality to be proved, as functions of $D \in \Diff(\tilde G)$, are conjugate-linear homomorphisms of unital algebras, it is sufficient to check the identity for vector fields. Let then $D = (X,Y) \in \lie{g} \oplus \lie{g}$. We have
\begin{multline*}
d\xi(D^\dtwist) f^* = \left.\frac{d}{dt}\right|_{t=0} \xi(\exp(tY),\exp(tX)) f^* \\
= \left.\frac{d}{dt}\right|_{t=0} (\xi(\exp(tX),\exp(tY)) f)^* = (d\xi(D) f)^*,
\end{multline*}
which gives the conclusion.
\end{proof}

For $D \in \Diff(G)$, let $\tilde D \in \Diff(\tilde G)$ be the differential operator along the first factor of $G \times G$ which corresponds to $D$, i.e.,
\[\tilde D(f \otimes g) = (Df) \otimes g,\]
so that in particular $d\xi(\tilde D) = D$.

\begin{lem}\label{lem:twistrockland}
Let $L = L^+ \in \Diff(G)$ be weighted subcoercive, and set $\Delta = L^2$. Then $\tilde \Delta + \tilde \Delta^\dtwist$ is positive weighted subcoercive on $\tilde G$.
\end{lem}
\begin{proof}
For $D \in \Diff(G)$, let $D^\dconj \in \Diff(G)$ be the differential operator uniquely determined by the identity
\[\overline{D f} = D^\dconj \overline{f}.\]
Clearly the correspondence $D \mapsto D^\dconj$ defines a conjugate-linear involutive automorphism of the unital algebra $\Diff(G)$. From this, it is not difficult to prove that, for every $D \in \Diff(G)$,
\[\tilde{D}^\dtwist(f \otimes g) = f \otimes (D^\dconj g).\]

In particular, we have that $\tilde \Delta + \tilde \Delta^\dtwist$ is the sum of two left-invariant differential operators on $\tilde G$, of which the first is along the first factor of $G \times G$, and corresponds to $L^2$, whereas the second is along the second factor, and corresponds to $(L^\dconj)^2$. We know from the hypothesis that $L$ is self-adjoint and weighted subcoercive; by Proposition~\ref{prp:productwsub}, in order to conclude it will be sufficient to show that $L^\dconj$ is weighted subcoercive too.

Fix a reduced weighted algebraic basis $A_1,\dots,A_d$ of $\lie{g}$ and a weighted subcoercive form $C$ such that $dR_G(C) = L$. If $\overline{C}$ is the form defined by $\overline{C}(\alpha) = \overline{C(\alpha)}$, then it is easy to see that
\[\overline{dR_G(C) f} = dR_G(\overline{C}) \overline{f},\]
so that, by definition, $L^\dconj = dR_G(\overline{C})$. On the other hand,
\[\Re \langle \phi, dR_G(\overline{C}) \phi \rangle = \Re \langle \overline{\phi}, dR_G(C) \overline{\phi} \rangle,\]
so that, from the definition of weighted subcoercive form, it follows immediately that $\overline{C}$ is weighted subcoercive.
\end{proof}

As in the previous sections, let $L_1,\dots,L_n \in \Diff(G)$ be a 
weighted subcoercive system on the 
nilpotent Lie group $G$, 
and let $\Delta = p(L)^2$, where $p$ is a real polynomial such that $p(L)$ is weighted subcoercive.

We define
\[\tilde A = (\tilde \Delta + \tilde \Delta^\dtwist)/2, \qquad A = d\xi(\tilde A) = (\Delta + \Delta^\dstar)/2.\]
By Lemma~\ref{lem:twistrockland}, $\tilde A$ is a (left-invariant) positive weighted subcoercive operator on $\tilde G$, whereas $A$ is a differential operator on $G$ which in general is neither left- nor right-invariant. Since $\tilde A, \tilde \Delta, \tilde \Delta^\dtwist$ commute pairwise, the corresponding operators $A,\Delta,\Delta^\dstar$ in the representation $\xi$ admit a joint spectral resolution by Corollary~\ref{cor:commutativealgebra}.

Let $h_t$ ($t > 0$) be the convolution kernel of $e^{-t\Delta}$.

\begin{lem}\label{lem:kerneloperations}
Suppose that $u \in L^2(G)$ commutes with all the $h_t$ ($t > 0$). For all Borel $m : \R \to \C$, $u$ is in the domain of $m(\Delta)$ if and only if it is in the domain of $m(A)$, and in this case
\[m(A) u = m(\Delta) u.\]
\end{lem}
\begin{proof}
Notice that, for all $f \in \D(G)$,
\[\Delta f^* = (\Delta^\dstar f)^*\]
Since $f \mapsto f^*$ is an isometric automorphism of $L^2(G)$, it intertwines also the spectral resolutions of $\Delta$ and $\Delta^\dstar$, so that in particular,
\[e^{-t\Delta^\dstar} f = (e^{-t \Delta} f^*)^* = (f^* * h_t)^* = h_t * f\]
and then
\[e^{-tA} f = e^{-t\Delta/2} e^{-t\Delta^\dstar/2} f = h_{t/2} * f * h_{t/2},\]
so that, since $u$ commutes with $h_{t/2}$,
\[e^{-tA} u = e^{-t\Delta} u.\]

Let $\xi_t(\lambda) = e^{-t \lambda}$, and set $\JJ_0 = \Span \{\xi_t \tc t > 0\}$. By the previous identity, it is immediate to show that, if $m \in \JJ_0$, then $m(A)u = m(\Delta) u$. 

If $m \in C_0(\R)$, then by the Stone-Weierstrass theorem we can find a sequence $m_k \in \JJ_0$ such that $m_k \to m$ uniformly on $\left[0,+\infty\right[$. Therefore, since $\Delta$ and $A$ are positive operators, by passing to the limit for $k \to +\infty$ in the identity $m_k(A) u = m_k(\Delta) u$, we obtain $m(A)u = m(\Delta)u$.

It is not difficult to extend the previous identity to all Borel $m : \R \to \C$ by the spectral theorem and dominated convergence.
\end{proof}

\begin{lem}\label{lem:leftrightdiff}
Let $X \in \lie{g}$. Then, for all $v \in \lie{g}$,
\[(X + X^\dstar)|_{\exp(v)} = d\exp_v([v,X]).\]
\end{lem}
\begin{proof}
Notice that the semigroup on $\tilde G$ associated to $\tilde X + \tilde X^\dtwist$ is
\[t \mapsto (\exp(tX),\exp(tX)).\]
This means that, for all $f \in \D(G)$, $v \in \lie{g}$,
\[(X+X^\dstar)|_{\exp(v)} f = \left.\frac{d}{d t}\right|_{t=0} f(\exp(-tX) \exp(v) \exp(tX)).\]
On the other hand,
\[\exp(-tX) \exp(v) \exp(tX) = \exp(\Ad(\exp(-tX))(v)),\]
so that
\[\left.\frac{d}{d t}\right|_{t=0} (\exp(-tX) \exp(v) \exp(tX)) = d\exp_v(\ad(-X)(v)) = d\exp_v([v,X]),\]
which is the conclusion.
\end{proof}

In the following, since $G$ is nilpotent, we will identify $G$ with $\lie{g}$ via the exponential map.
%
Choose a basis $\nu_1,\dots,\nu_r$ of $(\lie{g}/\lie{z})^*$ and a basis $T_1,\dots,T_d$ of $\lie{z}$, and set $P_j = \nu_j \circ P$. The functions $P_j : G \to \R$ can be thought of as multiplication operators on $L^2(G)$, and it is not difficult to show that the operators
\[P_1,\dots,P_r,-iT_1,\dots,-iT_d\]
are (essentially) self-adjoint on $L^2(G)$ and commute strongly pairwise
, so that they admit a joint spectral resolution.

Via the chosen bases, $J$ can be identified with a bilinear map $\R^r \times \R^d \to \lie{y}^*$. Therefore, for every $Y \in \lie{y}$, we have a bilinear form $J(\cdot,\cdot)(Y) : \R^r \times \R^d \to \R$, which in fact is a polynomial; we can then evaluate this particular polynomial in the operators $P_1,\dots,P_r,-iT_1,\dots,-iT_d$, and denote by
\[J(P,-iT)(Y)\]
the resulting operator on $L^2(G)$.

Finally, choose an inner product on $\lie{y}$ --- which induces an inner product on $\lie{y}^*$ --- and an orthonormal basis $\{Y_l\}_l$ of $\lie{y}$. Then also the map
\[|J(\cdot,\cdot)|^2 : \R^r \times \R^d \to \R\]
is a polynomial, thus as before we can consider the operator $|J(P,-iT)|^2$ on $L^2(G)$, and clearly
\[|J(P,-iT)|^2 = \sum_l (J(P,-iT)(Y_l))^2.\]

\begin{lem}\label{lem:jdiffop}
For all $Y \in \lie{y}$, $J(P,-iT)(Y)$ is a differential operator on $G$; more precisely,
\[J(P,-iT)(Y) = -i(Y + Y^\dstar).\]
In particular
\[|J(P,-iT)|^2 = -\sum_l (Y_l + Y_l^\dstar)^2.\]
\end{lem}
\begin{proof}
Let $\hat T_1,\dots,\hat T_d \in \lie{z}^*$ be the dual basis of $T_1,\dots,T_d$, and $\hat \nu_1,\dots,\hat \nu_r \in \lie{g}/\lie{z}$ be the dual basis of $\nu_1,\dots,\nu_r$. Then, by bilinearity, for every $Y \in \lie{y}$,
\[J(P,-iT)(Y) = -i \sum_{j,k} J(\hat \nu_j,\hat T_k)(Y) P_j T_k.\]
This shows that $J(P,-iT)(Y)$ is a differential operator on $G$. In fact, for all $x \in G = \lie{g}$, we have
\[\sum_j P_j(x) \hat\nu_j = P(x),\]
therefore
\[\begin{split}
J(P,-iT)(Y)|_x  &= -i \sum_k J(P(x),\hat T_k)(Y) T_k \\
&= -i \sum_k \hat T_k([x,Y]) T_k = -i [x,Y] = -i (Y + Y^\dstar)|_x
\end{split}\]
by Lemma~\ref{lem:leftrightdiff} (notice that, since $T_1,\dots,T_d$ are central, they are constant vector fields in exponential coordinates).
\end{proof}

Since $T_1,\dots,T_d$ are central, the left-invariant differential operators
\begin{equation}\label{eq:centralsystem}
L_1,\dots,L_n,-iT_1,\dots,-iT_d
\end{equation}
on $G$ are a weighted subcoercive system
. We can thus consider the Plancherel measure $\sigma'$ on $\R^n \times \lie{z}^*$ associated to this system, which can be shown not to depend on the choice of the basis of $\lie{z}$
.

The core of the technique under discussion is contained in the following

\begin{prp}\label{prp:partialweight}
Let us suppose that, for some nonnegative Borel functions $w : \lie{g}/\lie{z} \to \R$ and $\zeta : \lie{z}^* \to \R$, we have
\[|J(\bar x,\tau)| \geq w(\bar x) \, \zeta(\tau) \qquad \text{for all $\bar x \in \lie{g}/\lie{z}$, $\tau \in \lie{z}^*$.}\]
If $K \subseteq \R^n$ is compact, then for all $m \in \D(\R^n)$ with $\supp m \subseteq K$ and for all $\beta \geq 0$ we have
\[\||w \circ P|^\beta \breve m\|_2^2 \leq C_{K,\beta} \int_{\R^n \times \lie{z}^*} |m(\lambda)|^2 \,\zeta(\tau)^{-2\beta} \,d\sigma'(\lambda,\tau).\]
\end{prp}
\begin{proof}
From the hypothesis we deduce
\[w(\bar x)^\beta \leq |J(\bar x,\tau)|^\beta \zeta(\tau)^{-\beta}\]
and then also, by the spectral theorem,
\[\| |w \circ P|^\beta f  \|_2 \leq C_\beta \| |J(P,-iT)|^\beta \zeta(-iT)^{-\beta} f\|_2\]
for $f \in L^2(G)$.

By Lemma~\ref{lem:jdiffop}, we have
\[|J(P,-iT)|^2 = - \sum_l (Y_l + Y_l^\dstar)^2 = d\xi\left( - \sum_l (\tilde Y_l + \tilde Y_l^\dtwist)^2 \right).\]
Since $\tilde A$ is weighted subcoercive on $\tilde G$ and $-\sum_l (\tilde Y_l + \tilde Y_j^\dtwist)^2 \in \Diff(\tilde G)$, by Theorem~\ref{thm:robinsonterelst}(v), for some polynomial $q_\beta$ we have, in the representation $\xi$,
\[\| |J(P,-iT)|^\beta \psi \|_2 \leq C_\beta \| q_\beta(A) \psi\|_2,\]
therefore, by putting the two inequalities together, we get
\[\| |w \circ P|^\beta f \|_2 \leq C_\beta \|\zeta(-iT)^{-\beta} q_\beta(A) f \|_2\]
(since the $T_j$ commute strongly with $A$). In particular, if we take $f = \breve m$,
\[\| |w \circ P|^\beta \breve m \|_2 \leq C_\beta \|\zeta(-iT)^{-\beta} q_\beta(A) \breve m \|_2 = C_\beta \|\zeta(-iT)^{-\beta} q_\beta(\Delta) \breve m \|_2\]
by Lemma~\ref{lem:kerneloperations}, since $\breve m$ commutes with all the $h_t$ by Lemma~\ref{lem:composition}.

On the other hand, by Lemma~\ref{lem:composition} and Theorem~\ref{thm:plancherel},
\[\begin{split}
\|\zeta(-iT)^{-\beta} q_\beta(\Delta) \breve m \|_2^2 &= \int_{\R^n \times \lie{z}^*} |\breve m(\lambda)|^2 \, q_\beta(p(\lambda))^2 \, \zeta(\tau)^{-2\beta} \,d\sigma'(\lambda,\tau)\\
&\leq C_{K,\beta} \int_{\R^n \times \lie{z}^*} |\breve m(\lambda)|^2 \, \zeta(\tau)^{-2\beta} \,d\sigma'(\lambda,\tau),
\end{split}\]
where $C_{K,\beta}  = \sup_{\lambda \in K} q_\beta(p(\lambda))^2$.
\end{proof}

By simple manipulations, the previous estimate may take a slightly more general form.

\begin{cor}\label{cor:partialweight}
Let us suppose that, for some nonnegative Borel functions $w_j : \lie{g}/\lie{z} \to \R$ and $\zeta_j : \lie{z}^* \to \R$ ($j=1,\dots,h$), we have
\[|J(\bar x,\tau)| \geq w_j(\bar x) \, \zeta_j(\tau) \qquad \text{for all $\bar x \in \lie{g}/\lie{z}$, $\tau \in \lie{z}^*$,}\]
and set $\tilde w_j(x) = 1 + w_j(P(x))$. If $K \subseteq \R^n$ is compact, then for all $m \in \D(\R^n)$ with $\supp m \subseteq K$ and for all $\vec{\beta} = (\beta_1,\dots,\beta_h) \geq 0$ we have
\[\int_G |\breve m(x)|^2 \,\prod_{j=1}^h \tilde w_j(x)^{2\beta_j} \,dx 
\leq C_{K,\vec\beta} \int_{\R^n \times \lie{z}^*} |m(\lambda)|^2 \,\prod_{j=1}^h (1+\zeta_j(\tau)^{-2\beta_j}) \,d\sigma'(\lambda,\tau).\]
\end{cor}
\begin{proof}
If we set, for $I \subseteq \{1,\dots,h\}$,
\[\beta_I = \sum_{j\in I} \beta_j, \qquad w_{\vec{\beta},I}(\bar x) = \prod_{j \in I} w_j(\bar x)^{\beta_j/\beta_I}, \qquad \zeta_{\vec{\beta},I}(\tau) = \prod_{j \in I} \zeta_j(\tau)^{\beta_j/\beta_I},\]
then clearly
\[|J(\bar x,\tau)| \geq w_I(\bar x) \, \zeta_I(\tau)  \qquad \text{for all $\bar x \in \lie{g}/\lie{z}$, $\tau \in \lie{z}^*$,}\]
and moreover
\[\prod_{j=1}^h \tilde w_j(x)^{2\beta_j} \leq C_{\vec{\beta}} \sum_{I \subseteq \{1,\dots,h\}} w_{\vec{\beta},I}(P(x))^{2\beta_I},\]
\[\prod_{j=1}^h (1+\tau_j(\tau)^{-2\beta_j}) = \sum_{I \subseteq \{1,\dots,h\}} \zeta_{\vec{\beta},I}(\tau)^{-2\beta_I},\]
therefore the conclusion follows by repeated application of Proposition~\ref{prp:partialweight}.
\end{proof}

Therefore, under some particular hypotheses, we may control a weighted $L^2$ norm of $\breve m$ in terms of a weighted $L^2$ norm of $m$, where the weight on the spectral side is a function of the central coordinates. The following result gives some information about the latter norm.

\begin{prp}\label{prp:homogeneouspushforward}
Let $\zeta \in L^1_\loc(\lie{z}^*)$ be nonnegative. Then the push-forward $\sigma_\zeta$ of the measure
\begin{equation}\label{eq:measurenearlyproduct}
\zeta(\tau) \,d\sigma'(\lambda,\tau)
\end{equation}
on the first factor of $\R^n \times \lie{z}^*$ is a regular Borel measure on $\R^n$.

Suppose moreover that $G$ is a homogeneous group, with dilations $\delta_t$ and homogeneous dimension $Q_\delta$, and that $L_1,\dots,L_n$ is a homogeneous system, with associated dilations $\epsilon_t$. If the function $\zeta$ is homogeneous of degree $a$, i.e.,
\[\zeta(\tau \circ \delta_t) = t^a \zeta(\tau),\]
then $\sigma_\zeta$ is homogeneous of degree $Q_\delta+a$, i.e.,
\[\sigma_\zeta(\epsilon_t(A)) = t^{Q_\delta+a} \sigma_\zeta(A)\]
for all Borel $A \subseteq \R^n$.
\end{prp}
\begin{proof}
Let $K \subseteq \R^n$ be compact. By Lemma~\ref{lem:properwsub}, the canonical projection $\R^n \times \lie{z}^* \to \R^n$ is a proper continuous map when restricted to $\supp \sigma'$, therefore there is a compact $K' \subseteq \lie{z}^*$ such that
\[(K \times \lie{z}^*) \cap \supp \sigma' \subseteq K \times K',\]
and consequently
\[\begin{split}
\sigma_\zeta(K) &= \int_{K \times \lie{z}^*} \zeta(\tau) d\sigma'(\lambda,\tau) \\
&\leq C_K \int_{K \times K'} e^{-2p(\lambda)} \,\zeta(\tau) d\sigma'(\lambda,\tau)
= C_K \| (\zeta \chr_{K'})^{1/2}(-iT) h_1\|_2^2,
\end{split}\]
by Lemma~\ref{lem:composition} and Theorem~\ref{thm:plancherel}. On the other hand, since $h_1 \in \Sz(G)$, the last quantity is easily seen to be finite by using the Euclidean Fourier transform and the fact that $(\zeta \chr_{K'})^{1/2} \in L^2(\lie{z}^*)$.

We have thus proved that, if $K\subseteq \R^n$ is compact, then $\sigma_\zeta(K) < \infty$. By Theorem~2.18 of \cite{rudin_real_1974}, this means that $\sigma_\zeta$ is a regular Borel measure on $\R^n$.

Suppose now that $G$ is a homogeneous group and $L_1,\dots,L_n$ a homogeneous system. Without loss of generality, we may take the basis $T_1,\dots,T_d$ of $\lie{z}$ as composed by $\delta_t$-homogeneous elements, so that
\[L_1,\dots,L_n,-iT_1,\dots,-iT_d\]
is a homogeneous weighted subcoercive system, and the associated dilations $\epsilon'_t$ on $\R^n \times \lie{z}^*$ are given by
\[\epsilon'_t(\lambda,\tau) = (\epsilon_t(\lambda), \tau \circ \delta_{t}).\]
By Proposition~\ref{prp:plancherelhomogeneous}, $\sigma'$ is $\epsilon'_t$-homogeneous of degree $Q_\delta$. Therefore, if $\zeta$ is homogeneous of degree $a$, then clearly the measure \eqref{eq:measurenearlyproduct} is homogeneous of degree $Q_\delta+a$; since the canonical projection $\R^n \times \lie{z}^* \to \R^n$ intertwines the two system of dilations, we infer that also $\sigma_\zeta$ is homogeneous of degree $Q_\delta+a$.
\end{proof}

The previous results give, via interpolation, an improvement of the estimates of \S\ref{section:weightedestimates}.

\begin{prp}\label{prp:partialweight2}
Let us suppose that, for some nonnegative Borel functions $w_j : \lie{g}/\lie{z} \to \R$ and $\zeta_j : \lie{z}^* \to \R$ ($j=1,\dots,k$), we have
\[|J(\bar x,\tau)| \geq w_j(\bar x) \, \zeta_j(\tau) \qquad \text{for all $\bar x \in \lie{g}/\lie{z}$, $\tau \in \lie{z}^*$,}\]
and set $\tilde w_j(x) = 1 + w_j(P(x))$. Let us suppose moreover that, for some open $A \subseteq \R^n$, $\vec{\beta} \geq 0$, $\bar\gamma \geq 0$, $1 \leq p,q \leq \infty$, and all compact $K \subseteq A$,
\begin{equation}\label{eq:tracehypothesis}
\int_{\R^n \times \lie{z}^*} |f(\lambda)|^2 \,\prod_{j=1}^h (1+\zeta_j(\tau)^{-2\beta_j}) \, d\sigma'(\lambda,\tau) \leq C_{K,\vec{\beta},\gamma,p,q} \|f\|^2_{B^\gamma_{p,q}(\R^n)}
\end{equation}
for all $\gamma > \bar\gamma$ and $f \in \D(\R^n)$ with $\supp f \subseteq K$. Then
\[\left(\int_G |\breve m(x)|^2 \,\langle x\rangle_G^{2\alpha} \,\prod_{j=1}^h \tilde w_j(x)^{2\beta_j} \,dx\right)^{1/2} \leq C_{K,\alpha,\vec{\beta},\gamma} \|m\|_{B^\gamma_{p,q}(\R^n)}\]
for all $\alpha \geq 0$, $\gamma > \alpha+\bar\gamma$, $K \subseteq A$ compact, $m \in \D(\R^n)$ with $\supp m \subseteq K$.
\end{prp}
\begin{proof}
By \eqref{eq:tracehypothesis} it follows that the function $\zeta_j$ cannot be everywhere null, therefore the inequality
\[|J(\bar x,\tau)| \geq w_j(\bar x) \, \zeta_j(\tau),\]
together with the bilinearity of $J$ and Proposition~\ref{prp:nilpotentgrowth}, implies that, for some $C,\theta \geq 0$
\[w_j(P(x)) \leq C \langle x \rangle_G^\theta \qquad\text{for all $x \in G$;}\]
thus we also have
\[\prod_{j=1}^h \tilde w_j(x)^{2\beta_j} \leq C_{\vec{\beta}} \langle x \rangle_G^{2\theta(\beta_1 + \dots + \beta_h)}\]
for some $C_{\vec{\beta}} \geq 0$.

Let $\psi \in \D(\R^n)$ such that $\psi|_K = 1$, $K' = \supp \psi \subseteq A$. The operator
\[T : m \mapsto (m \psi)\breve{}\]
is then continuous
\[B^\gamma_{p,q}(\R^n) \to L^2(G, \langle x \rangle_G^{2\alpha} \textstyle\prod_{j=1}^h \tilde w_j(x)^{2\beta_j} \,dx)\]
for $\alpha \geq 0$, $\gamma > \alpha + \theta(\beta_1 + \dots \beta_h) + n$ by Theorem~\ref{thm:l2estimates}, whereas it is continuous
\[B^\gamma_{p,q}(\R^n) \to L^2(G, \textstyle\prod_{j=1}^h \tilde w_j(x)^{2\beta_j} \,dx)\]
for $\gamma > \bar\gamma$ by Corollary~\ref{cor:partialweight} and \eqref{eq:tracehypothesis}. The conclusion then follows by interpolation (see Proposition~\ref{prp:besovinterpolation}).
\end{proof}



\begin{cor}\label{cor:improvedl1}
Under the hypotheses of Proposition~\ref{prp:partialweight2}, suppose moreover that
\begin{equation}\label{eq:compensazione}
\int_{G} \langle x \rangle_G^{-2\alpha} \, \prod_{j=1}^h \tilde w_j(x)^{-2\beta_j} \,dx < \infty
\end{equation}
for $\alpha > \bar\alpha_{\vec{\beta}}$. Then we have
\[\int_G |\breve m(x)| \, \langle x \rangle_G^\alpha \,dx \leq C_{K,\alpha,\vec{\beta},\gamma} \|m\|_{B^\gamma_{p,q}}\]
for all $\alpha \geq 0$, $\gamma > \alpha+\bar\alpha_{\vec{\beta}}+\bar\gamma$, $K \subseteq A$ compact, $m \in \D(\R^n)$ with $\supp m \subseteq K$.
\end{cor}
\begin{proof}
Choose $\alpha'$ strictly between $\bar\alpha_{\vec{\beta}}$ and $\gamma - \alpha - \bar\gamma$. Then, by H\"older's inequality and Proposition~\ref{prp:partialweight2},
\[\int_G |\breve m(x)| \, \langle x \rangle_G^\alpha \,dx \leq C_{K,\alpha,\gamma} \left(\int_G \langle x \rangle_G^{-2\alpha'} \prod_{j=1}^h \tilde w_j(x)^{-2\beta_j} \,dx\right)^{1/2} \|m\|_{B^\gamma_{p,q}}\]
and the conclusion follows from \eqref{eq:compensazione}.
\end{proof}

As we mentioned at the beginning, the clearest example of application of this machinery is that of M\'etivier groups\index{Lie group!M\'etivier}, i.e., the connected, simply connected groups $G$ such that
\[|J(\bar x,\tau)| \geq |\bar x| |\tau| \qquad\text{for all $\bar x \in \lie{g}/\lie{z}$, $\tau \in \lie{z}^*$,}\]
for some norms on $\lie{g}/\lie{z}$ and $\lie{z}^*$.

\begin{lem}\label{lem:metivier}
If $G$ is a M\'etivier group, then:
\begin{itemize}
\item[(i)] $\lie{z} = [\lie{g},\lie{g}]$, and in particular $G$ is $2$-step nilpotent;
\item[(ii)] $\dim (\lie{g}/\lie{z})$ is even;
\item[(iii)] $\dim \lie{z} \leq \dim (\lie{g}/\lie{z})$.
\end{itemize}
\end{lem}
\begin{proof}
(i) From the definition, it follows immediately that, if $\tau \in \lie{z}^* \setminus \{0\}$, then the linear map
\[J(\cdot,\tau) : \lie{g}/\lie{z} \to \lie{y}^*\]
is injective, and moreover it takes its values in the subspace of $\lie{y}^*$ corresponding to $(\lie{y}/\lie{z})^*$. Thus we have $\dim (\lie{y}/\lie{z}) \geq \dim (\lie{g}/\lie{z})$, but $\lie{z} \subseteq \lie{y} \subseteq \lie{g}$, therefore $\lie{y} = \lie{g}$, which means that $[\lie{g},\lie{g}] \subseteq \lie{z}$. If the inclusion were strict, then we would find $\tau \in \lie{z}^* \setminus \{0\}$ such that $\tau|_{[\lie{g},\lie{g}]} = 0$, so that, by the definition of $J$, we would have $J(\bar x,\tau) = 0$ for all $\bar x \in \lie{g}/\lie{z}$, which contradicts the hypothesis.

(ii) Choose $\tau \in \lie{z}^* \setminus \{0\}$. Then, by the hypothesis, the skew-symmetric bilinear form on $\lie{g}/\lie{z}$ induced by
\[\lie{g} \times \lie{g} \ni (x,y) \mapsto \tau([x,y]) \in \R\]
is non-degenerate. In particular, $\dim(\lie{g}/\lie{z})$ must be even.

(iii) Fix $x \in \lie{g} \setminus \lie{z}$. The linear map $\lie{g}/\lie{z} \to \lie{z}$ induced by
\[\lie{g} \ni y \mapsto [x,y] \in \lie{z}\]
determines, by transposition, a linear map $\lie{z}^* \to (\lie{g}/\lie{z})^*$, which is injective by the hypothesis, and in particular $\dim \lie{z} \leq \dim(\lie{g}/\lie{z})$.
\end{proof}

\begin{prp}\label{prp:metivierl1}
Suppose that $G$ is a M\'etivier group. If $\alpha \geq 0$ and
\[\gamma > \alpha + \frac{\dim G}{2},\]
then, for every $K \subseteq \R^n$ compact,
\[\|\breve m\|_{L^1(G,\langle x \rangle_G^\alpha \,dx)} \leq C_{K,\alpha,\gamma} \|m\|_{B^\gamma_{\infty,\infty}}\]
for all $m \in \D(G)$ with $\supp m \subseteq K$.
\end{prp}
\begin{proof}
We take as $w_1$ and $\zeta_1$ the chosen norms on $\lie{g}/\lie{z}$ and $\lie{z}^*$ respectively. By Proposition~\ref{prp:homogeneouspushforward}, the hypotheses of Proposition~\ref{prp:partialweight2} are satisfied for $h=1$, $A = \R^n$, $\bar\gamma=0$, $p=q=\infty$, if $\zeta_1^{-2\beta_1} \in L^1_\loc(\lie{z}^*)$, i.e., if $2\beta_1 < \dim \lie{z}$. Having fixed such a $\beta_1$, by Proposition~\ref{prp:nilpotentgrowth} it follows easily that
\[\int_{G} (1+|x|_G)^{-2\alpha} \, (1+ w_1(P(x)))^{-2\beta_1} \,dx < \infty\]
if $2\alpha > Q_G - 2\beta_1$ (since $2\beta_1 < \dim(\lie{g}/\lie{z})$ by Lemma~\ref{lem:metivier}). We can then apply Corollary~\ref{cor:improvedl1} with $\bar\alpha_{\vec{\beta}} = Q_G/2 - \beta_1$, thus obtaining the inequality
\[\|\breve m\|_{L^1(G,\langle x \rangle_G^\alpha \,dx)} \leq C_{K,\alpha,\gamma} \|m\|_{B^\gamma_{\infty,\infty}}\]
for $\gamma > \alpha + Q_G/2 - \beta_1$. On the other hand, $Q_G = \dim G + \dim\lie{z}$; since $\beta_1$ can be chosen arbitrarily near $(\dim\lie{z})/2$, we get the conclusion.
\end{proof}

We have thus shown that, for M\'etivier groups, the regularity threshold in weighted $L^1$ estimates may be lowered to half of the topological dimension. The same result can be obtained, by a slight generalization of the previous argument, also for nilpotent groups which are direct products of several M\'etiver and/or abelian groups. In fact, as it is noticed in a remark at the end of \cite{hebisch_multiplier_1995}, there are further cases of groups for which this technique gives an improvement of weighted $L^1$ estimates.

In order to attempt a systematic treatment of these various cases, we introduce the following definition: for $h \in \N$, we say that a homogeneous Lie group $G$ is \emph{$h$-capacious}\index{Lie group!homogeneous!$h$-capacious} if there exist linearly independent homogeneous elements $\omega_1,\dots,\omega_h \in (\lie{g}/\lie{z})^*$ and linearly independent homogeneous elements $z_1,\dots,z_h \in \lie{z}$ such that,
for $j=1,\dots,h$,
\begin{equation}\label{eq:capacious}
|J(\bar x,\tau)| \geq |\omega_j(x)| |\tau(z_j)| \qquad\text{for all $\bar x \in \lie{g}/\lie{z}$, $\tau \in \lie{z}^*$.}
\end{equation}

Clearly, every homogeneous Lie group is $0$-capacious. Here are some criteria which may be of some use in showing that a certain homogeneous group is $h$-capacious. Recall that we denote by
\[\lie{g}_{[1]} = \lie{g}, \qquad \lie{g}_{[r+1]} = [\lie{g},\lie{g}_{[r]}]\]
the descending central series of a Lie algebra $\lie{g}$.

\begin{prp}\label{prp:capacitycriteria}
Let $G$ be a homogeneous Lie group, with dilations $\delta_t$.
\begin{itemize}
\item[(i)] If $G$ is a M\'etivier group (with any family of automorphic dilations), then $G$ is $(\dim \lie{z})$-capacious.
\item[(ii)] Suppose that, for some $r \geq 2$, $\dim \lie{g}_{[r]} = 1$. Then $G$ is $1$-capacious.
\item[(iii)] If $\lie{g}$ admits a $\C$-linear structure which is compatible with its homogeneous Lie algebra structure, and if moreover $\dim_\C \lie{g}_{[r]} = 1$ for some $r \geq 2$, then $\lie{g}$ is $2$-capacious.
\item[(iv)] Suppose that $G = G_1 \times G_2$, where $G_1$ and $G_2$ are homogeneous Lie groups with dilations $\delta_{1,t}$ and $\delta_{2,t}$ respectively, so that $\delta_t = \delta_{1,t} \times \delta_{2,t}$. If $G_1$ is $h_1$-capacious and $G_2$ is $h_2$-capacious, then $G$ is $(h_1+h_2)$-capacious.
\end{itemize}
\end{prp}
\begin{proof}
(i) Since the $\delta_t$ are automorphisms, $\lie{z}$ is a homogeneous ideal. Therefore, if $h = \dim \lie{z}$, by Lemma~\ref{lem:metivier} we can choose linearly independent elements $z_1,\dots,z_h$ of $\lie{z}$, and linearly independent homogeneous $\omega_1,\dots,\omega_h \in (\lie{g}/\lie{z})^*$. By a suitable renormalization, we then have
\[|\omega_j(\bar x)| |\tau(z_j)| \leq |\bar x| |\tau| \leq |J(\bar x,\tau)| \qquad\text{for all $\bar x \in \lie{g}/\lie{z}$, $\tau \in \lie{z}^*$,}\]
since $G$ is M\'etivier.
%

(ii) Since $G$ is nilpotent, it must be $r$-step, so that $\lie{g}_{[r]} \subseteq \lie{z}$. Notice that the ideal $\lie{g}_{[r-1]}$ is preserved by every automorphism of $\lie{g}$, therefore it is generated by $\delta_t$-homogeneous elements; since $[\lie{g},\lie{g}_{[r-1]}] = \lie{g}_{[r]} \neq 0$, then there must exist a $\delta_t$-homogeneous element $y \in \lie{g}_{[r-1]}$ such that, for some $x_0 \in \lie{g}$, $[x_0,y] = z \neq 0$. In particular $y \neq 0$ and moreover, since the ideal $\lie{g}_{[r]}$ is $\delta_t$-homogeneous and $1$-dimensional, necessarily $z$ is $\delta_t$-homogeneous.

Since $y \in \lie{g}_{[r-1]}$, the linear map $[\cdot,y] : \lie{g} \to \lie{g}$ takes its values in $\lie{g}_{[r]} = \R z$; therefore, there exists $\omega \in (\lie{g}/\lie{z})^*$ such that
\[[x,y] = \omega(P(x)) z \qquad\text{for all $x \in \lie{g}$.}\]
Notice that $\omega(P(x_0)) = 1$, thus $\omega \neq 0$; moreover, since both $y$ and $z$ are homogeneous, also $\omega$ is homogeneous.
Finally
\begin{equation}\label{eq:linearidentity}
J(\bar x,\tau)(y) = \omega(\bar x) \tau(z) \qquad\text{for all $\bar x \in \lie{g}/\lie{z}$, $\tau \in \lie{z}^*$,}
\end{equation}
which implies immediately that $G$ is $1$-capacious.


(iii) Arguing as in part (ii), but with a complex Lie algebra $\lie{g}$, one finds an identity analogous to \eqref{eq:linearidentity}, where now $\omega$ is a $\C$-linear functional on $\lie{g}/\lie{z}$, and $z \in \lie{z}$. The conclusion then follows by taking the $\R$-linearly independent $\R$-linear functionals $\Re \omega, \Im \omega$ on $\lie{g}/\lie{z}$, and the $\R$-linearly independent elements $z,iz \in \lie{z}$.


(iv) Via the canonical identification
\[\lie{g} = \lie{g}_1 \times \lie{g}_2,\]
we have (with the obvious meaning of the notation)
\[\lie{z} = \lie{z}_1 \times \lie{z}_2, \qquad \lie{y} = \lie{y}_1 \times \lie{y}_2,\]
thus also
\[\lie{z}^* = \lie{z}_1^* \times \lie{z}_2^*, \qquad \lie{y}^* = \lie{y}_1^* \times \lie{y}_2^*, \qquad \lie{g}/\lie{z} = (\lie{g}_1/\lie{z}_1) \times (\lie{g}_2 / \lie{z}_2).\]
Moreover it is easy to check that
\[J((\bar x_1,\bar x_2),(\tau_1,\tau_2)) = (J_1(\bar x_1,\tau_1),J_2(\bar x_2,\tau_2)),\]
therefore
\[|J((\bar x_1,\bar x_2),(\tau_1,\tau_2))| \geq \max \{|J_1(\bar x_1,\tau_1)|,|J_2(\bar x_2,\tau_2)|\}\]
and the conclusion follows immediately.
\end{proof}

Notice that the previous proposition is not sufficient to exhaust all the cases of $h$-capacious groups. For instance, one can check that all the nilpotent groups listed in \cite{nielsen_unitary_1983}, except for the free nilpotent groups, are (at least) $1$-capacious, for at least one choice of a homogeneous structure on them; see also the examples of \S\ref{section:examples}.

\begin{lem}\label{lem:homogeneousdualbasis}
Suppose that $G$ is $h$-capacious, and let $\omega_1,\dots,\omega_h \in (\lie{g}/\lie{z})^*$ be as in the definition. Then the functionals $\omega_j \circ P$ are null on $[\lie{g},\lie{g}]$. In particular
\[h \leq \min\{\dim \lie{z}, \dim \lie{g} - \dim(\lie{z} + [\lie{g},\lie{g}])\}.\]
Moreover, we can find a homogeneous basis of of $\lie{g}$ compatible with the descending central series such that the functionals $\omega_1 \circ P,\dots,\omega_h \circ P$ are part of the dual basis.
\end{lem}
\begin{proof}
Notice that
\[[[\lie{g},\lie{g}],\lie{y}] \subseteq [\lie{g},[\lie{g},\lie{y}]] \subseteq [\lie{g},\lie{z}] = 0.\]
Therefore, from the definition of $J$ it follows that, for every $x \in [\lie{g},\lie{g}]$,
\[J(P(x),\tau) = 0 \qquad\text{for all $\tau \in \lie{z}^*$.}\]
In particular, by choosing in \eqref{eq:capacious} a $\tau \in \lie{z}^*$ such that $\tau(z_j) \neq 0$, we obtain that the functional $\omega_j \circ P$ is null on $[\lie{g},\lie{g}]$. In particular, the $\omega_j \circ P$ correspond to linearly independent elements of 
$(\lie{g}/([\lie{g},\lie{g}] + \lie{z}))^*$, thus the inequality about $h$ follows.

Let now $W = \ker (\omega_1 \circ P) \cap \dots \cap \ker (\omega_h \circ P)$. Then $W$ is a homogeneous subspace of $\lie{g}$ containing $[\lie{g},\lie{g}]$. Moreover, if $\tilde\omega_j$ is the element of $(\lie{g}/W)^*$ corresponding to $\omega_j$, then $\tilde\omega_1,\dots,\tilde\omega_h$ are a homogeneous basis of $(\lie{g}/W)^*$. We can then choose homogeneous elements $v_1,\dots,v_h \in \lie{g}$ such that the corresponding elements in the quotient $\lie{g}/W$ are the dual basis of $\tilde\omega_1,\dots,\tilde\omega_h$. Finally, we append to $v_1,\dots,v_h$ a homogeneous basis of $W$ compatible with the descending central series (which, apart from $\lie{g}_{[1]}$, is contained in $W$, and is made of homogeneous ideals), and we are done.
\end{proof}

Here is finally the improvement of Corollary~\ref{cor:polyl1estimates} for $h$-capacious groups.

\begin{thm}\label{thm:improvedl1estimates}
For some $h \in \N$, suppose that the homogeneous group $G$ is $h$-capacious, and let $Q_G$ be its degree of polynomial growth. Let moreover
\[L_1,\dots,L_n\]
be a homogeneous weighted subcoercive system on $G$. If $p,q \in [1,\infty]$, $\alpha \geq 0$ and
\[\gamma > \alpha + \frac{Q_G - h}{2} + \frac{n}{p} - \frac{1}{\max\{2,p\}},\]
then, for every $K \subseteq \R^n \setminus \{0\}$ compact,
\[\|\breve m\|_{L^1(G,\langle x \rangle_G^\alpha \,dx)} \leq C_{K,\alpha,\gamma} \|m\|_{B^\gamma_{p,q}}\]
for all $m \in \D(G)$ with $\supp m \subseteq K$.
\end{thm}
\begin{proof}
Let $\omega_1,\dots,\omega_h \in (\lie{g}/\lie{z})^*$ and $z_1,\dots,z_h \in \lie{z}$ be given by the definition of $h$-capacious, and set
\[w_j(\bar x) = |\omega_j(\bar x)|, \qquad \zeta_j(\tau) = |\tau(z_j)|.\]
Notice now that, since the $z_j$ are linearly independent, for every choice of $\beta_1,\dots,\beta_h \in \left[0,1/2\right[$, the push-forward $\sigma_{\vec{\beta}}$ of the measure
\begin{equation}\label{eq:weightedpushforwardedmeasure}
\prod_{j=1}^h (1+\zeta_j(\tau)^{-2\beta_j}) \, d\sigma'(\lambda,\tau)
\end{equation}
via the canonical projection on the first factor of $\R^n \times \lie{z}^*$ is, by Proposition~\ref{prp:homogeneouspushforward}, a regular Borel measure on $\R^n$; in fact, since the $z_j$ are homogeneous, $\sigma_{\vec{\beta}}$ is the sum of $\epsilon_t$-homogeneous positive regular Borel measures on $\R^n$ (with possibly different degrees of homogeneity), hence $\sigma_{\vec{\beta}}$ is locally $1$-bounded on $\R^n \setminus \{0\}$. By Corollary~\ref{cor:triebeltrace}, embeddings and interpolation, this means that the hypotheses of Proposition~\ref{prp:partialweight2} are satisfied for $A = \R^n \setminus \{0\}$ and $\bar\gamma = n/p - 1/\max\{2,p\}$.

By Lemma~\ref{lem:homogeneousdualbasis}, we can find a homogeneous basis $v_1,\dots,v_k$ of $\lie{g}$, compatible with the descending central series, such that, if $\hat v_1,\dots,\hat v_h$ is the dual basis, then $\hat v_j = \omega_j \circ P$ for $j=1,\dots,h$; in particular we have $\tilde w_j(x) = 1 + |\hat v_j(x)|$. If we set
\[\kappa_j = \max \{r \tc v_j \in \lie{g}_{[r]}\},\]
then by Proposition~\ref{prp:nilpotentgrowth} we have
\[Q_G = \sum_{j=1}^k \kappa_j, \qquad \langle x \rangle_G \sim 1 + \sum_{j=1}^k |\hat v_j(x)|^{1/\kappa_j},\]
and in particular
\[\langle x \rangle_G^{-2\alpha_j} \leq C_{\alpha_j} (1 + |\hat v_j(x)|)^{-2\alpha_j/\kappa_j}\]
for $j=1,\dots,k$ and $\alpha_j \geq 0$. Moreover, since the $\omega_j \circ P$ are null on $[\lie{g},\lie{g}]$, then $\kappa_j = 1$ for $j=1,\dots,h$.

Notice now that, for fixed $\beta_1,\dots,\beta_h \in \left[0,1/2\right[$, if $\alpha \geq 0$ satisfies
\[2\alpha > 2\alpha_{\vec{\beta}} = \sum_{j=1}^h (1-2\beta_j) + \sum_{j=h+1}^k \kappa_j,\]
then we may choose $\alpha_1,\dots,\alpha_k \geq 0$ such that
\[\alpha = \sum_{j=1}^k \alpha_j, \qquad 2\alpha_j > \begin{cases}
1-2\beta_j &\text{for $j=1,\dots,h$,}\\
\kappa_j &\text{for $j=h+1,\dots,k$,}
\end{cases}\]
therefore
\begin{multline*}
\int_G \langle x \rangle_G^{-2\alpha} \, \prod_{j=1}^h \tilde w_j(x)^{-2\beta_j} \,dx \\
\leq \int_G \prod_{j=1}^h (1+|\hat v_j(x)|)^{-2(\alpha_j +\beta_j)} \prod_{j=h+1}^k (1+|\hat v_j(x)|)^{-2\alpha_j/\kappa_j} \,dx < \infty
\end{multline*}
We can thus apply Corollary~\ref{cor:improvedl1}, and the conclusion follows because, if the $\beta_j$ tend to $1/2$, then $\alpha_{\vec{\beta}}$ tends to $\sum_{j=h+1}^k \kappa_j = Q_G - h$.
\end{proof}

Notice that, as in Corollary~\ref{cor:polyl1estimates}, the estimates can be improved for $p<\infty$ by results of local $d$-boundedness; in this case, however, the measure to be examined is not the original Plancherel measure $\sigma$, but its ``modified version'' $\sigma_{\vec{\beta}}$ considered in the proof.

\input{examples}

\section{An intrinsic perspective}

Let $L_1,\dots,L_n$ be a weighted subcoercive system on a connected Lie group $G$. It has already been noticed that the joint spectrum $\Sigma$ of $L_1,\dots,L_n$ can be identified with the Gelfand spectrum of the sub-C$^*$-algebra $C_0(L)$ of $\Cv^2(G)$. By Proposition~\ref{prp:Jdensity} and the following remarks, we have that $C_0(L)$ is the closure of the algebra generated by
\[\{\Kern_L (e^{-r}) \tc \text{$r \geq 0$ is a polynomial and $r(L)$ is weighted subcoercive}\}.\]
By Proposition~\ref{prp:pushforward}, this set depends only on the subalgebra $\mathcal{O}$ of $\Diff(G)$ generated by $L_1,\dots,L_n$, and not on the choice of a particular system of generators of $\mathcal{O}$. Therefore, the same holds for $C_0(L)$, which can thus be denoted also by $C_0(\mathcal{O})$.

In particular, we can think of $\Sigma$ as a particular immersion of the ``abstract'' Gelfand spectrum of the C$^*$-algebra $C_0(\mathcal{O})$ in $\R^n$, due to a choice of generators $L_1,\dots,L_n$ of $\mathcal{O}$. Notice that the results of \S\ref{section:eigenfunctions} are consistent with this perspective, since the concept of ``joint eigenfunction'' depends only on $\mathcal{O}$ and not on the choice of the generators $L_1,\dots,L_n$.

In order to have a more ``intrinsic'' presentation, one may start with the notion of \emph{weighted subcoercive algebra}\index{algebra!weighted subcoercive}, i.e., a finitely generated commutative unital $*$-subalgebra of $\Diff(G)$ containing a weighted subcoercive operator. A weighted subcoercive system is then a (finite) system of (formally self-adjoint) generators of a weighted subcoercive algebra. If $\mathcal{O}$ is a weighted subcoercive algebra, then it is not difficult to define:
\begin{itemize}
\item the \emph{spectrum} $\Sigma$ of $\mathcal{O}$, i.e., the Gelfand spectrum of the C$^*$-algebra $C_0(\mathcal{O})$;
\item a resolution $E$ of the identity of $L^2(G)$ on $\Sigma$ and an isometric embedding $\Kern_\mathcal{O} : L^\infty(\Sigma,E) \to \Cv^2(G)$, extending the inverse of the Gelfand transform of $C_0(\mathcal{O})$, such that
\[\phi * \Kern_\mathcal{O} m  = \int_\Sigma m \,dE \qquad\text{for all $\phi \in \D(G)$;}\]
\item the \emph{Plancherel measure}, i.e., a regular Borel measure $\sigma$ on $\Sigma$ such that
\[\int_G | \Kern_{\mathcal{O}} m (x) |^2 \,dx = \int_\Sigma |m|^2 \,d\sigma\]
for every bounded Borel $m : \Sigma \to \C$;
\item for every choice of generators $L_1,\dots,L_n$ of $\mathcal{O}$, a topological embedding $\eta_L : \Sigma \to \R^n$, which maps homeomorphically $\Sigma$ onto the joint spectrum of $L_1,\dots,L_n$.
\end{itemize}

In order to give an intrinsic version of the results of \S\ref{section:weightedestimates}, one would like to have a notion of smoothness on the spectrum $\Sigma$ of a weighted subcoercive algebra $\mathcal{O}$. Unfortunately, the embeddings of $\Sigma$ in $\R^n$ in general are not particularly regular closed subsets of $\R^n$ (see the examples of \S\ref{section:examples}). However, if $L_1,\dots,L_n$ and $L'_1,\dots,L'_{n'}$ are two different systems of generators of $\mathcal{O}$, and if $p : \R^n \to \R^{n'}$ is a polynomial map such that $L' = p(L)$, then $\eta_{L'} = p \circ \eta_L$. Therefore, the properties of functions which are preserved by these particular (polynomial) changes of variables can be thought of as intrinsic.

In particular, for $k \in \N \cup \{\infty\}$, we may define $C^k(\Sigma)$ to be the set of (continuous) functions $f : \Sigma \to \C$ such that, for some (equivalently, for every) system of generators $L$ of $\mathcal{O}$, the function $f \circ \eta_L^{-1} : \eta_L(\Sigma) \to \C$ extends to a function of class $C^k$ on $\R^n$. Analogously, one can define the local Besov spaces $B_{\infty,\infty,\loc}^s(\Sigma)$ for $s > 0$ (see Proposition~\ref{prp:besovchange}).

From Theorem~\ref{thm:l2estimates}, Corollary~\ref{cor:polyl1estimates} and Theorem~\ref{thm:improvedl1estimates} then we immediately deduce the following

\begin{thm}
Let $\mathcal{O}$ be a weighted subcoercive algebra on a connected Lie group $G$, and let $\Sigma$ be its spectrum.
\begin{itemize}
\item If $\alpha \geq 0$, then, for every compactly supported $m \in B^s_{\infty,\infty,\loc}(\Sigma)$ with $s > \alpha$, $\Kern_{\mathcal{O}} m$ and all its left-invariant derivatives belong to $L^2(G,\langle x \rangle_G^{2\alpha} \,dx)$.
\item If $G$ has polynomial growth of degree $Q_G$, then, for every compactly supported $m \in B^s_{\infty,\infty\,loc}(\Sigma)$ with $s > Q_G/2$, it is $\Kern_{\mathcal{O}} m \in L^{1;\infty}(G)$.
\item If $G$ is a homogeneous group which is $h$-capacious for some $h \in \N$, then, for every compactly supported $m \in B^s_{\infty,\infty,\loc}(\Sigma)$ with $s > (Q_G-h)/2$, it is $\Kern_{\mathcal{O}} m \in L^{1;\infty}(G)$.
\end{itemize}
\end{thm}

Notice also that, if the group $G$ is homogeneous, and if the weighted subcoercive algebra is invariant by dilations (so that it admits a finite system of homogeneous generators), then we have the corresponding dilations on the abstract spectrum $\Sigma$. Therefore, also the Mihlin-H\"ormander and Marcinkiewicz conditions --- which has been considered in Chapter~\ref{chapter:conditions} and will be used as hypotheses for the multiplier theorems of Chapter~\ref{chapter:multipliers} --- can be given an intrinsic formulation (at least by restricting to minimal homogeneous systems of generators of $\mathcal{O}$, cf.\ Proposition~\ref{prp:homogeneousextension}).


%% file: examples.tex
\section{Examples}\label{section:examples}

We now see how the techniques and results previously shown can be applied in some particular cases, involving nilpotent Lie groups. Preliminarily, some general computations are performed, which concern, on the one side, the Plancherel measure associated to a weighted subcoercive system and, on the other side, the capacity map of a nilpotent Lie algebra. The results thus obtained and the introduced notation shall be understood in the subsequent specific examples.

\subsection{Computation of the Plancherel measure}\index{Plancherel measure!for a weighted subcoercive system}

Here we present a ``heuristic'' method which can be used in order to determine the Plancherel measure associated to a specific weighted subcoercive system, in connection with the group Plancherel measure.

If $G$ is a connected Lie group which is type I and of polynomial growth, and $L_1,\dots,L_n$ is a weighted subcoercive system on $G$, then we know (see Proposition~\ref{prp:eigenvectordecomposition}) that there exists a generic set $\widehat G_{\mathrm{gen}}$ of (equivalence classes of) irreducible unitary representations of $G$ such that, if $\pi \in \widehat G_{\mathrm{gen}}$, then the Hilbert space $\HH_\pi$ admits a complete orthonormal system of joint eigenvectors $\{v_{\pi,\alpha}\}_{\alpha}$ of $L_1,\dots,L_n$. If $\lambda_{\pi,\alpha} \in \R^n$ denotes the eigenvalues of $L_1,\dots,L_n$ corresponding to the eigenvector $v_{\pi,\alpha}$, then, for every $m \in \D(\R^n)$,
\begin{multline*}
\int_{\R^n} |m(\lambda)|^2 \,d\sigma(\lambda) = \int_G |\breve m(x)|^2 \,dx = \int_{\widehat G_{\mathrm{gen}}} \|\pi(\breve m)\|_{\HS}^2 \,d\pi \\
= \int_{\widehat G_{\mathrm{gen}}} \sum_{\alpha} \|\pi(\breve m) v_{\pi,\alpha} \|_{\HH_\pi}^2 \, d\pi = \int_{\widehat G_{\mathrm{gen}}} \sum_\alpha |m(\lambda_{\pi,\alpha})|^2 \,d\pi.
\end{multline*}
If one is able to determine both the group Plancherel measure and the eigenvectors $v_{\pi,\alpha}$ in such a way that the function $(\pi,\alpha) \mapsto \lambda_{\pi,\alpha}$ is sufficiently regular, then the measure $\sigma$ on $\R^n$ is determined by the previous identity as the push-forward of the product of the group Plancherel measure times a counting measure.

For nilpotent Lie groups, the Kirillov theory gives a fair amount of information about irreducible representations and the group Plancherel measure (see \cite{corwin_representations_1990}, and also \cite{nielsen_unitary_1983}, where irreducible unitary representations and the group Plancherel measure are computed for every nilpotent Lie group of dimension up to $6$). However, except for some particular cases, an ``explicit'' formula for the joint eigenvalues of $L_1,\dots,L_n$ in a generic irreducible representation is not easy to find.

\paragraph{Two-step nilpotent Lie groups.}

In the case of a two-step nilpotent Lie group, we follow the computation of the group Plancherel measure given in \cite{astengo_hardys_2000}.

Namely, let $\lie{g} = \lie{v} \oplus \lie{z}$, where $\lie{z}$ is the center of $\lie{g}$, and let $\langle \cdot,\cdot \rangle$ be an inner product on $\lie{g}$ such that $\lie{v} \perp \lie{z}$. To every $\tau \in \lie{z}^*$, we associate the skew-symmetric endomorphism $B(\tau)$ of $\lie{v}$ defined by
\[\langle B(\tau) v, w \rangle = \tau([v,w]),\]
and set $\lie{r}_\tau = \ker B(\tau)$, $\lie{m}_\tau = \lie{v} \cap \lie{r}_\tau^\perp$. Since $B(\tau)$ is skew-symmetric, $\lie{m}_\tau$ has even dimension. Denote by $\Lambda$ the Zariski-open subset of $\lie{z}^*$ such that $\dim \lie{m}_\tau$ is maximum, and let $m \in \N$ be such that $\dim \lie{m}_\tau = 2m$ for $\tau \in \Lambda$. We can then choose, for $\tau \in \Lambda$, an orthonormal basis
\[E_1(\tau),\dots,E_{m}(\tau),\bar E_1(\tau),\dots,\bar E_m(\tau)\]
of $\lie{m}_\tau$ such that
\[B(\tau) E_j(\tau) = b_j(\tau) \bar E_j(\tau), \qquad B(\tau) \bar E_j(\tau) = -b_j(\tau) E_j(\tau)\]
for some $b_j(\tau) \in \R \setminus \{0\}$ and for $j=1,\dots,m$. For $\tau \in \Lambda$, if
\[\lie{x}_\tau = \Span\{E_1(\tau),\dots,E_m(\tau)\}, \qquad \lie{y}_\tau = \Span\{\bar E_1(\tau),\dots,\bar E_m(\tau)\},\] and we use coordinates $(X,Y,R,Z) \in \lie{x}_\tau \oplus \lie{y}_\tau \oplus \lie{r}_\tau \oplus \lie{z}$ on $G = \lie{g}$, then, for $\mu \in \lie{r}_\tau^*$, an irreducible unitary representation $\pi_{\tau,\mu}$ of $G$ on $L^2(\lie{x}_\tau)$ is defined by
\[\pi_{\tau,\mu}(X,Y,R,Z) \phi(X') = e^{i \tau(Z + [X'+X/2,Y])} e^{i \mu(R)} \phi(X+X').\]
If moreover, for $\tau \in \Lambda$, we denote by $D(\tau) = b_1(\tau)^2 \cdots b_m(\tau)^2$ the determinant of the restriction of $B(\tau)$ to $\lie{m}_\tau$, then $D(\tau)$ is a polynomial function of $\tau$, and moreover, as in \cite{astengo_hardys_2000}, one can show that, for $f \in \Sz(G)$,
\[\int_G |f(x)|^2 \,dx = (2\pi)^{m - \dim G} \int_{\Lambda} \int_{\lie{r}_\tau^*} \|\pi_{\tau,\mu}(f)\|_{\HS}^2 \, D(\tau)^{1/2} \,d\mu \,d\tau.\]

Notice now that, for $Z \in \lie{z}$, $R \in \lie{r}_\tau$,
\[d\pi_{\tau,\mu}(-iZ) = \tau(Z), \qquad d\pi_{\tau,\mu}(-iR) = \mu(R)\]
(these differential operators are represented as constants), whereas
\[d\pi_{\tau,\mu}(-iE_j(\tau)) = -i \frac{\partial}{\partial t_j}, \qquad d\pi_{\tau,\mu}(-i\bar E_j(\tau)) = b_j(\tau) t_j,\]
where $(t_1,\dots,t_m)$ are the coordinates related to the basis $E_1(\tau),\dots,E_m(\tau)$ of $\lie{x}_\tau$. In particular, if $L_j(\tau) = -(E_j(\tau)^2 + \bar E_j(\tau)^2)$, then
\[d\pi_{\tau,\mu}(L_j(\tau)) = - \left(\frac{\partial}{\partial t_j}\right)^2 + b_j(\tau)^2 t_j^2.\]
It can be easily checked that an orthonormal basis of $L^2(\lie{x}_\tau)$ made of joint eigenfunctions of $d\pi_{\tau,\mu}(L_1(\tau)),\dots,d\pi_{\tau,\mu}(L_m(\tau))$ is given by
\[u_{\tau,\mu,\alpha}(t) = \prod_{j=1}^m |b_j(\tau)|^{1/4} \psi_{\alpha_j}\left(|b_j(\tau)|^{1/2} \,t_j\right)\]
for $\alpha = (\alpha_1,\dots,\alpha_m) \in \N^m$, where
\[\psi_k(x) = (k! \, 2^k \sqrt{\pi})^{-1/2} e^{-x^2/2} H_k(x)\]
and
\[H_k(x) = (-1)^k e^{x^2} \left(\frac{d}{dx}\right)^{k} e^{-x^2}\]
is the $k$-th Hermite polynomial; we have in fact
\[d\pi_{\tau,\mu}(L_j(\tau)) u_{\tau,\mu,\alpha} = |b_j(\tau)| (1 + 2\alpha_j) u_{\tau,\mu,\alpha}.\]

The differential operators $L_j(\tau) \in \Diff(G)$ depend on $\tau$, and moreover they do not commute in general. However, in the following we will find several examples of operators that, for every $\tau \in \Lambda$, can be written as polynomials in the $L_j(\tau)$, so that the previous decomposition will give us the eigenvalues of those operators in the representation $\pi_{\tau,\mu}$. For instance, if $\lie{z} = [\lie{g},\lie{g}]$ and $L$ is the orthonormal sublaplacian on $G$ associated to the inner product $\langle \cdot, \cdot \rangle$, then
\[L = L_1(\tau) + \dots + L_m(\tau) - (R_1(\tau)^2 + \dots + R_h(\tau)^2)\]
for every $\tau \in \Lambda$, where $R_1(\tau),\dots,R_h(\tau)$ is an orthonormal basis of $\lie{r}_\tau$, thus
\[d\pi_{\tau,\mu}(L) u_{\tau,\mu,\alpha} = \left( \sum_{j=1}^m |b_j(\tau)| (1 + 2\alpha_j) + |\mu|^2 \right) u_{\tau,\mu,\alpha}.\]

\subsection{Computation of the capacity map}\label{subsection:capacity}\index{capacity map}

Let $\lie{g}$ be a nilpotent Lie algebra, $\lie{z}$ its center and $\lie{y}$ as in \eqref{eq:centercenter}. We now describe a simple algorithm for computing the capacity map $J : \lie{g}/\lie{z} \times \lie{z}^* \to \lie{y}^*$. This is useful for checking if the techniques of \S\ref{section:metivier} can be applied to a given group.

Notice first of all that $J$ takes its values in $(\lie{y}/\lie{z})^*$ (which can be thought of as the subspace of $\lie{y}^*$ made of the functionals which vanish on $\lie{z}$). Choose a basis $T_1,\dots,T_d$ of $\lie{z}$, and complete it to a basis $Y_1,\dots,Y_h,T_1,\dots,T_d$ of $\lie{y}$, and also to a basis $X_1,\dots,X_k,T_1,\dots,T_d$ of $\lie{g}$ (one can always take $Y_1,\dots,Y_h$ to be a subset of $X_1,\dots,X_k$, but this is not necessary here). Correspondingly, we have a basis $\bar Y_1,\dots,\bar Y_h$ of $\lie{y}/\lie{z}$, and a basis $\bar X_1,\dots,\bar X_k$ of $\lie{g}/\lie{z}$. Finally, we take the dual bases $\bar Y_1^*,\dots,\bar Y_h^*$ of $(\lie{y}/\lie{z})^*$ and  $T_1^*,\dots,T_d^*$ of $\lie{z}^*$.

We construct now a matrix $M$, with $h$ rows and $k$ columns, whose entries are homogeneous polynomials of degree $1$ in the indeterminates $t_1,\dots,t_d$. Namely, the $(i,j)$ entry of this matrix is $\sum_{r=1}^d T^*_r([X_j,Y_i]) \, t_r$. Notice that this polynomial is simply the expression of $[X_j,Y_i]$ in terms of the basis $T_1,\dots,T_d$, where the elements $T_r$ have been replaced with the indeterminates $t_r$. Consequently, this matrix is nothing else than a part of the ``multiplication table'' of $\lie{g}$.

Finally, we take the product of $M$ with the column vector $V$ whose entries are the indeterminates $x_1,\dots,x_k$. The result $MV$ is a column vector with $h$ rows, whose entries are polynomials which are separately homogeneous of degree $1$ both in the $t_1,\dots,t_d$ and in the $x_1,\dots,x_k$. Specifically, its $i$-th entry is
\[\sum_{j=1}^k \sum_{r=1}^d T^*_r([X_j,Y_i]) t_r x_j = J\left(\sum_{j=1}^k x_j \bar X_j,\sum_{r=1}^d t_r T_r^*\right)(Y_i),\]
i.e., the $i$-th component of $J(\sum_{j=1}^k x_j \bar X_j,\sum_{r=1}^d t_r T_r^*)$ in the basis $\bar Y_1^*,\dots,\bar Y_h^*$. In particular, by taking the sum of the squares of the entries of $MV$, one obtains the square of the norm of $J(\sum_{j=1}^k x_j \bar X_j,\sum_{r=1}^d t_r T_r^*)$, with respect to the inner product which makes $\bar Y_1^*,\dots,\bar Y_h^*$ into an orthonormal basis; this sum of squares is a polynomial $F(x_1,\dots,x_k,t_1,\dots,t_d)$, and in order to apply the techniques of \S\ref{section:metivier} one looks for inequalities of the form
\[F(x_1,\dots,x_k,t_1,\dots,t_d) \geq w(x_1,\dots,x_k)^2 \zeta(t_1,\dots,t_d)^2\]
for some nonnegative functions $w$ and $\zeta$. For instance, if one finds that
\[F(x_1,\dots,x_k,t_1,\dots,t_d) = (x_1^2+\dots+x_k^2) (t_1^2+\dots+t_d^2),\]
then $\lie{g}$ is H-type.

\subsection{Some $2$-step groups}\label{subsection:example2step}

\paragraph{The Heisenberg group $H_n$} is the nilpotent Lie group determined by the relations
\[[X_1,Y_1] = T, \quad\dots, \quad [X_n,Y_n] = T,\]
where $X_1,Y_1,\dots,X_n,Y_n,T$ is a basis of its Lie algebra $\lie{h}_n$. The group $H_n$ is easily shown to be H-type (see \S\ref{subsection:capacity}).

Automorphic dilations $\delta_t$ on $H_n$ are defined by
\[\delta_t(X_j) = tX_j, \qquad \delta_t(Y_j) = tY_j, \qquad \delta_t(T) = t^2 T,\]
which makes $H_n$ into a stratified Lie group of homogeneous dimension $2n+2$.

If we set $L_j = -(X_j^2 + Y_j^2)$ (the $j$-th partial sublaplacian), then
\[L_1,\dots,L_n,-iT\]
is easily checked to be a system of pairwise commuting, formally self-adjoint, $\delta_t$-homogeneous left-invariant differential operators on $H_n$; in fact, they are a Rockland system, since the algebra generated by them contains the full sublaplacian $L = L_1 + \dots + L_n$.

\begin{figure}
\centering
\includegraphics{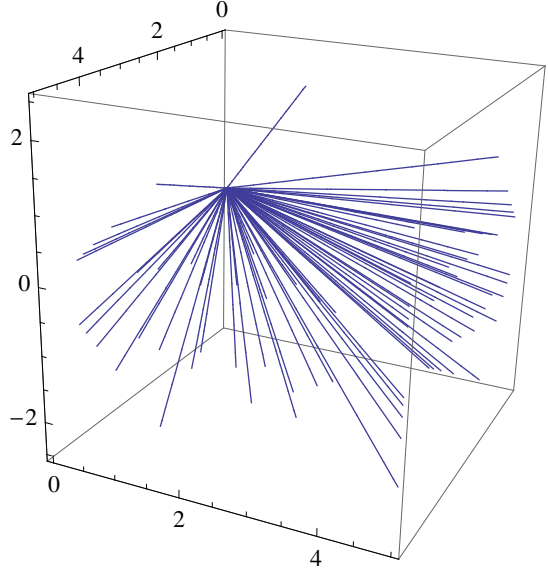}
\caption{The Heisenberg brush for $H_2$}
\end{figure}
The Plancherel measure $\sigma$ on $\R^{n+1}$ associated to that system can be easily determined by the aforementioned techniques: if we choose an inner product on $\lie{h}_n$ which makes $X_1,Y_1,\dots,X_n,Y_n,T$ into an orthonormal basis, and if we identify the dual of the center $\lie{z} = \R T$ with $\R$ via the basis $\{T\}$ of $\lie{z}$, we have
\[B(\tau) X_j = \tau Y_j, \qquad B(\tau) Y_j = - \tau X_j,\]
so that $\Lambda = \R \setminus \{0\}$, $\lie{r}_\tau = 0$, and we can choose $E_j(\tau) = X_j, \bar E_j(\tau) = Y_j$, therefore $b_j(\tau) = \tau$ and, if $I = (1,\dots,1) \in \N^n$, then
\begin{multline*}
\int_{\R^{n+1}} |m|^2 \,d\sigma = (2\pi)^{-(n+1)} \int_{\R \setminus \{0\}} \sum_{\alpha \in \N^n} |m(|\tau|(I+2\alpha),\tau)|^2 \,|\tau|^n \,d\tau \\
= \sum_{\substack{\alpha \in \N^n \\ \varepsilon \in \{-1,1\}}} (2\pi |I + 2\alpha|_1)^{-(n+1)} \int_0^\infty \left|m\left(\lambda \frac{I+2\alpha}{|I + 2\alpha|_1},\frac{\varepsilon \lambda}{|I + 2\alpha|_1}\right)\right|^2 \,\lambda^n \,d\lambda.
\end{multline*}
This shows that the Plancherel measure $\sigma$ is the sum of (suitably weighted) Lebesgue measures on half-lines of $\R^{n+1}$, which are contained in $\left[0,+\infty\right[^n \times \R$ and which accumulate on $\left[0,+\infty\right[^n \times \{0\}$; for this reason, one cannot hope to obtain that $\sigma$ is locally $d$-bounded for some $d$ greater than $1$. The support of $\sigma$ is called the \emph{Heisenberg brush}.

The Plancherel measure $\sigma'$ on $\R^2$ associated to the system
\[L,-iT\]
can be simply obtained as a push-forward of $\sigma$; we have in fact
\[\int_{\R^2} |m|^2 \,d\sigma' = \sum_{\substack{\beta \in \N \\ \varepsilon \in \{-1,1\}}} \frac{\binom{\beta + n - 1}{n - 1}}{(2\pi (n + 2\beta))^{(n+1)}} \int_0^\infty \left|m\left(\lambda,\frac{\varepsilon \lambda}{n+2\beta}\right)\right|^2 \,\lambda^n \,d\lambda.\]
The support of $\sigma'$ is called the \emph{Heisenberg fan}. 
\begin{figure}
\centering
\includegraphics{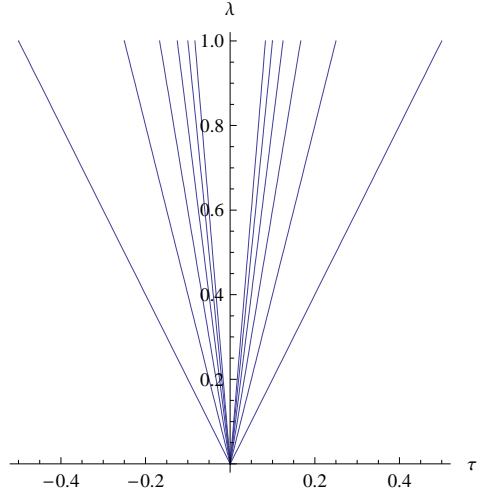}
\caption{The Heisenberg fan for $H_2$}
\end{figure}

\paragraph{The quaternionic Heisenberg group $\mathbb{H} H_n$} is defined by the relations
\[[X_j,Y_{j,1}] = T_1, \quad [X_j,Y_{j,2}] = T_2, \quad [X_j,Y_{j,3}] = T_3,\]
\[[Y_{j,1},Y_{j,2}] = T_3, \quad [Y_{j,2},Y_{j,3}] = T_1, \quad [Y_{j,3},Y_{j,1}] = T_2\]
for $j=1,\dots,n$, where $X_1,Y_{1,1},Y_{1,2},Y_{1,3},\dots,X_n,Y_{n,1},Y_{n,2},Y_{n,3},T_1,T_2,T_3$ is a basis of its $(4n+3)$-dimensional Lie algebra. The group $\mathbb{H} H_n$ is easily seen to be H-type (see \S\ref{subsection:capacity}).

A family of automorphic dilations is given by
\[\delta_t(X_j) = tX_j, \quad \delta_t(Y_{j,k}) = tY_{j,k}, \quad \delta_t(T_k) = t^2 T_k,\]
for $j=1,\dots,n$, $k=1,2,3$, which makes $\mathbb{H} H_n$ into a stratified group of homogeneous dimension $4n+6$.

The partial sublaplacians
\[L_j = -(X_j^2 + Y_{j,1}^2 + Y_{j,2}^2 + Y_{j,3}^2),\]
and also the full sublaplacian $L = L_1 + \dots + L_n$, are $\delta_t$-homogeneous. Moreover, the partial sublaplacians commute pairwise. Therefore, the system
\[L_1,\dots,L_n,-iT_1,-iT_2,-iT_3\]
is a Rockland system on $\mathbb{H} H_n$.

Let us identify the dual of the center with $\R^3$ via the basis $T_1,T_2,T_3$ of the center, and choose an inner product on the Lie algebra of $\mathbb{H} H_n$ which makes $X_1,Y_{1,1},Y_{1,2},Y_{1,3},\dots,X_n,Y_{n,1},Y_{n,2},Y_{n,3},T_1,T_2,T_3$ into an orthonormal basis. Then $\Lambda = \R^3 \setminus \{0\}$, and moreover, by a suitable choice of the $E_j(\tau), \bar E_j(\tau)$, we have
\[L_j = -(E_{2j-1}(\tau)^2 + \bar E_{2j-1}(\tau)^2 + E_{2j}(\tau)^2 + \bar E_{2j}(\tau)^2)\]
and $b_j(\tau) = |\tau|$, therefore the Plancherel measure $\sigma$ on $\R^{n+3}$ associated to the system $L_1,\dots,L_n,-iT_1,-iT_2,-iT_3$ is
\begin{multline*}
\int_{\R^{n+3}} |m|^2 \,d\sigma = \sum_{\alpha \in \N^n} \frac{\prod_{j=1}^n (\alpha_j + 1)}{(2\pi)^{2n+3}} \int_{\R^3 \setminus \{0\}} |m(|\tau|(2I+2\alpha),\tau)|^2 \,|\tau|^{2n} \,d\tau \\
= \sum_{\alpha \in \N^n} \frac{\prod_{j=1}^n (\alpha_j + 1)}{(4\pi |I + \alpha|_1)^{2n+3}} \int_{S^2} \int_0^\infty \left|m\left(\lambda \frac{I+\alpha}{|I + \alpha|_1},\frac{\varepsilon \lambda}{2|I + \alpha|_1}\right)\right|^2 \,\lambda^{2n+2} \,d\lambda \,d\varepsilon.
\end{multline*}
We thus obtain that the Plancherel measure $\sigma$ is the sum of suitably weighted Lebesgue measures on countably many $3$-cones contained in $\left[0,+\infty\right[^n \times \R^3$, whose vertex is the origin and which do not intersect elsewhere, whose axes are half-lines laying in $\left[0,+\infty\right[^n \times \{0\}$ and whose apertures tend to zero.

If we restrict now to the system
\begin{equation}\label{eq:hhcentral}
L,-iT_1,-iT_2,-iT_3,
\end{equation}
then the corresponding Plancherel measure $\sigma'$ on $\R \times \R^3$ is given by
\[
\int_{\R^4} |m|^2 \,d\sigma' 
= \sum_{\beta \in \N} \frac{\binom{\beta+2n-1}{2n-1}}{(4\pi (n+\beta))^{2n+3}} \int_{S^2} \int_0^\infty \left|m\left(\lambda, \frac{\varepsilon \lambda}{2(n+\beta)}\right)\right|^2 \,\lambda^{2n+2} \,d\lambda \,d\varepsilon.
\]
Notice that $\sigma'$ has a spherical symmetry in the central coordinates, due to the invariance by suitable automorphisms of $\mathbb{H} H_n$ of the algebra generated by $L,-iT_1,-iT_2,iT_3$ (see \S\ref{section:automorphisms}). Moreover, it is easily checked that $\sigma'$ is locally $3$-bounded on $\R^4 \setminus \{0\}$. However, since $\mathbb{H} H_n$ is H-type, the techniques of \S\ref{section:metivier} can be applied to get sharper weighted estimates, and correspondingly, in order to take advantage also of local $d$-boundedness properties, one has to look at the ``modified'' measure
\[d\sigma'_\gamma(\lambda,\tau) = (1+|\tau|^{-\gamma}) \, d\sigma'(\lambda,\tau)\]
for $\gamma \in \left[0,3\right[$ arbitrarily near to $3$ (cf.\ Corollary~\ref{cor:partialweight} and Proposition~\ref{prp:metivierl1}). This measure is locally $3$-bounded on $\R \times (\R^3 \setminus \{0\})$, but, for $\gamma$ near $3$, it is only locally $(4-\gamma)$-bounded on $\R^4 \setminus \{0\}$ (which is essentially nothing more than the local $1$-boundedness given by homogeneity). Consequently, we get sharper weighted estimates only for kernels corresponding to multipliers supported away from the line $\R \times \{0\}$, but we do not get a general improvement in the multiplier theorems of Chapter~\ref{chapter:multipliers}.

\paragraph{The free $2$-step nilpotent group $N_{3,2}$ on three generators} is defined by the relations
\[[X_1,X_2] = T_3, \quad [X_2,X_3] = T_1, \quad [X_3,X_1] = T_2,\]
where $X_1,X_2,X_3,T_1,T_2,T_3$ is a basis of its Lie algebra $\lie{n}_{3,2}$.

A family $\delta_t$ of automorphic dilations is given by
\[\delta_t(X_j) = t X_j, \quad \delta_t(T_j) = t^2 T_j\]
for $j=1,2,3$, which makes $N_{3,2}$ into a stratified group of homogeneous dimension $9$.

If $L = -(X_1^2 + X_2^2 + X_3^2)$ is the sublaplacian and $D = -(X_1 T_1 + X_2 T_2 + X_3 T_3)$, then it is easily checked that
\[L,D,-iT_1,-iT_2,-iT_3\]
is a Rockland system.

Let $\langle \cdot, \cdot \rangle$ be the inner product which makes $X_1,X_2,X_3,T_1,T_2,T_3$ into an orthonormal basis, and identify both $\lie{z}^*$ and $\lie{v}$ with $\R^3$ via the chosen bases. Then it is not difficult to see that
\[B(\tau) v = \tau \wedge v\]
for $\tau \in \lie{z}^*$, $v \in \lie{v}$. In particular, $\Lambda = \R^3 \setminus \{0\}$, $\lie{r}_\tau = \tau^\perp$; moreover, for $\tau \in \Lambda$, if $E(\tau)$, $\bar E(\tau)$ are chosen so that $\tau/|\tau|,E(\tau),\bar E(\tau)$ are a positive orthonormal basis of $\R^3$, then $b(\tau) = |\tau|$. Finally, if we identify $\lie{r}_\tau^*$ with $\R$ via the basis $\tau/|\tau|$ of $\lie{r}_\tau$, then we have, for $\tau \in \Lambda$, $\mu \in \lie{r}_\tau^*$,
\[d\pi_{\tau,\mu}(L) = - \left(\frac{d}{dt}\right)^2 + |\tau|^2 t^2 + \mu^2, \quad d\pi_{\tau,\mu}(D) = \mu|\tau|,\]
therefore
\[\begin{split}
\int_{\R^{5}} |m|^2 \,d\sigma 
&= (2\pi)^{-5} \sum_{\alpha \in 2\N+1}  \int_{\R^3 \setminus \{0\}} \int_{\R} |m(|\tau|\alpha + \mu^2, \mu |\tau| , \tau)|^2 \, |\tau| \,d\mu \,d\tau \\
&= (2\pi)^{-5} \sum_{\alpha \in 2\N+1} \int_0^\infty \int_{\R \times \R^3} | m(\lambda, \lambda^{3/2} \xi, \lambda \eta)|^2 \,d\sigma_\alpha(\xi,\eta) \,\lambda^{7/2} \,d\lambda,
\end{split}\]
where $\sigma_\alpha$ is the regular Borel measure on $\R^4 = \R \times \R^3$ given by
\[
\int_{\R^4} f \,d\sigma_\alpha 
= \alpha^{-4} \int_{S^2} \int_{-1}^1 f\left( \textstyle\frac{\theta (1 - \theta^2)}{\alpha}, \frac{(1-\theta^2)\omega}{\alpha} \right) \,(1-\theta^2)^3 \,d\theta \,d\omega.
\]
Each $\sigma_\alpha$ is supported on a compact hypersurface in $\R^4$, and it is not difficult to see that
\[\sigma_1(B(x,r)) \leq C r^3 \leq C_\epsilon r^{3-\varepsilon}\]
for every $x \in \R^4$, $r > 0$ and $\varepsilon > 0$, therefore
\[
\sum_{\alpha \in 2\N+1} \sigma_\alpha(B(x,r)) = \sum_{\alpha \in 2\N+1} \alpha^{-4} \sigma_1(B(\alpha x, \alpha r)) \leq C_\varepsilon r^{3-\varepsilon} \sum_{\alpha \in 2\N+1} \alpha^{-(1+\epsilon)},
\]
i.e., $\sum_{\alpha} \sigma_\alpha$ is locally $3$-bounded on $\R^4$ and consequently $\sigma$ is locally $4$-bounded on $\R^5 \setminus \{0\}$.

Let $\Delta = -(T_1^2 + T_2^2 + T_3^2)$ be the central Laplacian. Then
\[L,D,\Delta\]
is a Rockland system on $N_{3,2}$; in fact (see \cite{fischer_nilpotent_2010}, Theorem~7.5) this is a system of generators of the left-invariant $SO_3$-invariant differential operators relative to the Gelfand pair $(N_{3,2} \rtimes SO_3, SO_3)$. The associated Plancherel measure $\sigma'$ on $\R^3$ is easily obtained as the push-forward of the previously considered $\sigma$:
\[\int_{\R^{3}} |m|^2 \,d\sigma'
= (2\pi)^{-5} \sum_{\alpha \in 2\N+1} \int_0^\infty \int_{\R^2} | m(\lambda, \lambda^{3/2} \xi, \lambda^2 \zeta)|^2 \,d\sigma'_\alpha(\xi,\zeta) \,\lambda^{7/2} \,d\lambda,\]
where $\sigma'_\alpha$ is the regular Borel measure on $\R^2$ given by
\[
\int_{\R^2} f \,d\sigma'_\alpha 
= 4\pi \alpha^{-4} \int_{-1}^1 f\left( \textstyle\frac{\theta (1 - \theta^2)}{\alpha}, \frac{(1-\theta^2)^2}{\alpha^2} \right) \,(1-\theta^2)^3 \,d\theta.
\]
\begin{figure}
\centering
\includegraphics{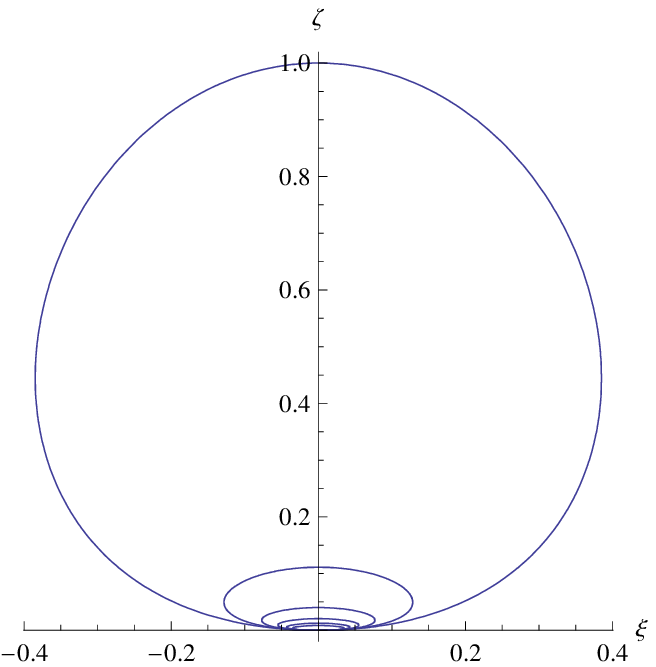}
\caption{Group $N_{3,2}$: supports of the measures $\sigma'_\alpha$}
\end{figure}
Each $\sigma'_\alpha$ is supported on a curve in $\R^2$, and it is not difficult to see that
\[\sigma'_1(B((\xi,\zeta),r)) \leq C r,\]
therefore
\[
\sum_{\alpha \in 2\N+1} \sigma'_\alpha(B((\xi,\zeta),r)) 
\leq \sum_{\alpha \in 2\N+1} \alpha^{-4} \sigma'_1(B((\alpha \xi, \alpha^2 \zeta),  \alpha^2 r)) 
\leq C r \sum_{\alpha \in 2\N+1} \alpha^{-2}\,
\]
which implies that $\sum_\alpha \sigma'_\alpha$ is locally $1$-bounded on $\R^2$, and consequently $\sigma'$ is locally $2$-bounded on $\R^3 \setminus \{0\}$.

Via projections, one can now obtain the Plancherel measures associated to the systems 
\[L, -iT_1,-iT_2,-iT_3, \qquad L,\Delta, \qquad L,D,\]
and estimate with analogous techniques the optimal exponent $d$ for the local $d$-boundedness, which in these cases is half-integer.

For instance, for the first of these systems, the Plancherel measure $\sigma''$ is determined by
\[\int_{\R^{4}} |m|^2 \,d\sigma'' 
= (2\pi)^{-5} \sum_{\alpha \in 2\N+1} \int_0^\infty \int_{\R^3} | m(\lambda, \lambda \eta)|^2 \,d\sigma''_\alpha(\eta) \,\lambda^{7/2} \,d\lambda,\]
where $\sigma''_\alpha$ is the measure on $\R^3$ such that
\[\int_{\R^3} f \,d\sigma''_\alpha 
= \alpha^{-4} \int_{S^2} \int_{0}^1 f\left(\frac{\nu \omega}{\alpha} \right) \,\nu^3 (1-\nu)^{-1/2} \,d\nu \,d\omega.\]
Each of the measures $\sigma''_\alpha$ is supported on a ball of $\R^3$ centered at the origin, and is absolutely continuous with respect to the Lebesgue measure, but the density is not bounded, and diverges along the boundary of the support. Because of this, the $\sigma''_\alpha$ are not locally $3$-bounded on $\R^3$, but we have
\[\sigma''_1(B(\eta,r)) \leq C r^{5/2};\]
thus, as before, it follows that $\sigma''$ is locally $\frac{7}{2}$-bounded on $\R^4 \setminus \{0\}$.

Similarly, one obtains that the Plancherel measure associated to the systems $L,\Delta$ and $L,D$ are both locally $\frac{3}{2}$-bounded on $\R^2 \setminus \{0\}$.

\paragraph{The group $G_{5,2}$ of \cite{nielsen_unitary_1983}} is defined by the relations
\begin{equation}\label{eq:g52relations}
[X_5,X_4] = X_2, \qquad [X_5,X_3] = X_1,
\end{equation}
where $X_5,X_4,X_3,X_2,X_1$ is a basis of its Lie algebra $\lie{g} = \lie{g}_{5,2}$. It is isomorphic to the quotient of the free nilpotent group $N_{3,2}$ by some central element; it is also isomorphic to a semidirect product $\R \ltimes \R^4$.

The following are characteristic ideals of $\lie{g}$:
\begin{alignat*}{4}
\lie{z} = [\lie{g},\lie{g}] & = \Span\{ & & X_2, & X_1 \}, \\
\lie{h} & = \Span\{ X_4, & X_3, & X_2, & X_1 \};
\end{alignat*}
in fact, $\lie{z}$ is the center of $\lie{g}$, whereas $\lie{h}$ is the unique $4$-dimensional abelian subalgebra of $\lie{g}$. The ideal $\lie{y}$ of $\lie{g}$ corresponding to the center of the quotient $\lie{g}/\lie{z}$ in this case coincides with the whole $\lie{g}$.

Let $\delta_t$ be automorphic dilations on $\lie{g}$. Then $\lie{z}$ and $\lie{h}$ are $\delta_t$-homogeneous. Consequently, we can choose a homogeneous element $\tilde X_5 \in X_5 + \lie{h}$ and homogeneous elements $\tilde X_4,\tilde X_3$ in $\lie{h}$ which are linearly independent modulo $\lie{z}$. If we set moreover $\tilde X_2 = [\tilde X_5,\tilde X_4]$ and $\tilde X_1= [\tilde X_5,\tilde X_3]$, then clearly in the (homogeneous) basis $\tilde X_5,\tilde X_4,\tilde X_3,\tilde X_2,\tilde X_1$ the relations of $\lie{g}$ have the same form as in \eqref{eq:g52relations}. Therefore, without loss of generality, we may suppose that the initial basis $X_5,X_4,X_3,X_2,X_1$ is homogeneous.

If $\lambda_j$ is the homogeneity degree of $X_j$ for $j=1,\dots,5$, then it must be
\[\lambda_2= \lambda_5 + \lambda_4, \qquad \lambda_1 = \lambda_5+\lambda_3;\]
conversely, for any choice of $\lambda_5,\lambda_4,\lambda_3 \geq 1$, if $\lambda_2,\lambda_1$ are determined by the previous equalities, then the dilations $\delta_t$ of $\lie{g}$ defined by $\delta_t(X_j) = t^{\lambda_j} X_j$ are automorphic.

With respect to the bases $\bar X_5,\bar X_4,\bar X_3$ of $\lie{g}/\lie{z}$ and $X_2^*,X_1^*$ of $\lie{z}^*$, and to a suitable norm on $\lie{g}^*$, we have (see \S\ref{subsection:capacity})
\[|J(x_5 \bar X_5 + x_4 \bar X_4 + x_3 \bar X_3, t_2 X_2^* + t_1 X_1^*)|^2 = (t_2 x_4 + t_1 x_3)^2 + (t_2^2 + t_1^2) x_5^2.\]
The homogeneous elements $\bar X_5^*$ of $(\lie{g}/\lie{z})^*$ and $X_2$ of $\lie{z}$ then attest that $G_{5,2}$ is $2$-capacious (despite the fact that Proposition~\ref{prp:capacitycriteria} does not apply in this case), with respect to any homogeneous structure on $\lie{g}$.

We fix now the homogeneous structure with $\lambda_5 = \lambda_4 = \lambda_3 = 1$, which makes $G_{5,2}$ into a stratified group of homogeneous dimension $7$. We consider moreover the Rockland system
\[L,-iX_2,-iX_1,\]
where $L = -(X_5^2+X_4^2+X_3^2)$ is the sublaplacian. The computation of the associated Plancherel measure $\sigma$ is performed as in the previous cases, and gives
\[\int_{\R^{3}} |m|^2 \,d\sigma
= (2\pi)^{-4} \sum_{\alpha \in 2\N+1} \int_0^\infty \int_{\R^2} | m(\lambda, \lambda \eta)|^2 \,d\sigma_\alpha(\eta) \,\lambda^{5/2} \,d\lambda,\]
where $\sigma_\alpha$ is the measure on $\R^2$ such that
\[\int_{\R^2} f \,d\sigma_\alpha 
= \alpha^{-3} \int_{S^1} \int_{0}^1 f\left(\frac{\nu \omega}{\alpha} \right) \,\nu^2 (1-\nu)^{-1/2} \,d\nu \,d\omega.\]

It is then not difficult to show, as in the case of the group $N_{3,2}$, that $\sigma_1$ is locally $\frac{3}{2}$-bounded on $\R^2$, and consequently $\sigma$ is locally $\frac{5}{2}$-bounded on $\R^3 \setminus \{0\}$. Here, however, we can also apply the techniques of \S\ref{section:metivier}, since 
\[|J(x_5 \bar X_5 + x_4 \bar X_4 + x_3 \bar X_3, t_2 X_2^* + t_1 X_1^*)| \geq |(t_2,t_1)|_2 \, |x_5|_2;\]
thus, as in the case of $\mathbb{H} H_n$, we consider the ``modified'' measure
\[d\tilde\sigma_\gamma(\lambda,\tau) = (1+|\tau|^{-\gamma}) \,d\sigma(\lambda,\tau)\]
for $\gamma \in \left[0,1\right[$ arbitrarily near to $1$. It is not difficult to show that $\tilde \sigma_\gamma$ is still locally $\frac{5}{2}$-bounded on $\R \times (\R^2 \setminus\{0\})$, but, for $\gamma$ near $1$, it is only locally $(3-\gamma)$-bounded on $\R^3 \setminus\{0\}$. Anyway, all the measures $\tilde\sigma_\gamma$ are locally $2$-bounded on $\R^3 \setminus \{0\}$, and this will yield --- differently from the case of the groups $\mathbb{H}H_n$ --- an improvement in the multiplier theorems of Chapter~\ref{chapter:multipliers}.

\paragraph{The free $2$-step nilpotent group $N_{4,2}$ on four generators} is defined by the relations
\[[X_j,X_k] = T_{jk}\]
for $1 \leq j < k \leq 4$, where $X_1,X_2,X_3,X_4,T_{12},T_{13},T_{14},T_{23},T_{24},T_{34}$ is a basis of its Lie algebra $\lie{n}_{4,2}$.

A family $\delta_t$ of automorphic dilations is given by
\[\delta_t(X_j) = t X_j, \quad \delta_t(T_{jk}) = t^2 T_{jk},\]
which makes $N_{4,2}$ into a stratified group of homogeneous dimension $16$.

If $L = -(X_1^2 + X_2^2 + X_3^2 + X_4^2)$ is the sublaplacian and
\begin{multline*}
D = - (X_2 T_{12} + X_3 T_{13} + X_4 T_{14})^2 - (-X_1 T_{12} + X_3 T_{23} + X_4 T_{24})^2 \\
    - (-X_1 T_{13} - X_2 T_{23} + X_4 T_{34})^2 - (-X_1 T_{14} - X_2 T_{24} - X_3 T_{34})^2,
\end{multline*}
then it is easily checked that
\[L,D,-iT_{12},-iT_{13},-iT_{14},-iT_{23},-iT_{24},-iT_{34}\]
is a Rockland system.

Let us fix on $\lie{n}_{4,2}$ the inner product which makes \[X_1,X_2,X_3,X_4,T_{12},T_{13},T_{14},T_{23},T_{24},T_{34}\]
into an orthonormal basis, and identify $\lie{z}$ and its dual with $\lie{so}_4$, the space of skew-symmetric $4 \times 4$ matrices, where the vector $T_{jk}$ corresponds to the matrix
\[(\delta_{jm} \delta_{kl} - \delta_{jl} \delta_{km})_{l,m}\]
(where $\delta_{rs}$ is the Kronecker delta). If $\lie{v}$ is identified with $\R^4$ via the basis $X_1,X_2,X_3,X_4$, then it is easily seen that, for $\tau \in \lie{so}_4$, the matrix representing the endomorphism $B(\tau)$ of $\lie{v}$ is $\tau$ itself. Moreover, $\Lambda = \{ \tau \tc \det \tau \neq 0\}$, and $\lie{r}_\tau = 0$ for $\tau \in \Lambda$.

For $b = (b_1,b_2) \in \R^2$, let
\[\tau_b = \left(\begin{smallmatrix}
0   & -b_1 & 0   & 0 \\
b_1 & 0    & 0   & 0 \\
0   & 0    & 0   & -b_2 \\
0   & 0    & b_2 & 0
\end{smallmatrix}\right) \in \lie{so}_4.\]
In order to write the Plancherel formula, we will use the following change of variables (cf.\ \cite{helgason_groups_1984}, Theorem~I.5.17):
\[\int_{\lie{so}_4} f(\tau) \,d\tau = \pi^2 \int_{SO_4} \int_{\R^2} f( \rho \tau_b \rho^{-1}) \, (b_1^2 - b_2^2)^2 \,db \,d\rho,\]
where the integration on $SO_4$ is made with respect to the Haar measure of total mass $1$; in fact, if $\tau = \rho \tau_b \rho^{-1}$, then we simply have $b_1(\tau) = b_1$, $b_2(\tau) = b_2$. Therefore, the Plancherel measure $\sigma$ associated to the aforementioned Rockland system on $N_{4,2}$ is given by
\begin{multline*}
\int_{\R^8} |m|^2 \,d\sigma \\
= \frac{1}{2^8 \pi^6} \sum_{\alpha \in (1+2\N)^2} \int_{SO_4} \int_{\R^2} |m(|b_1|\alpha_1 + |b_2|\alpha_2,|b_1|^3 \alpha_1 + |b_2|^3 \alpha_2, \rho \tau_b \rho^{-1})|^2  \\
\times \,|b_1 b_2| (b_1^2 - b_2^2)^2 \,db \,d\rho
\end{multline*}

\begin{figure}
\centering
\includegraphics{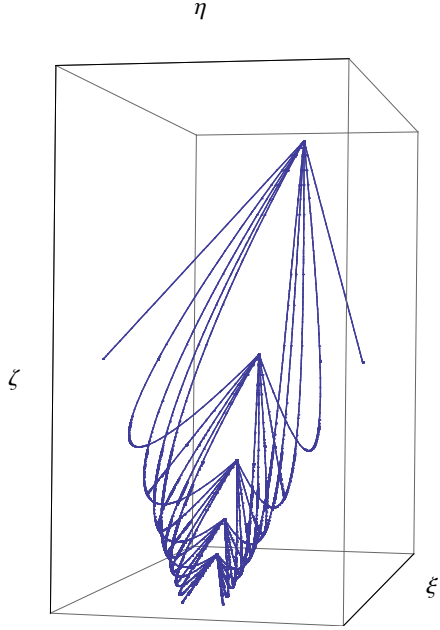}
\caption{Group $N_{4,2}$: supports of the measures $\sigma'_{\alpha,\varepsilon}$}
\end{figure}

In particular, if $\Delta = -(T_{12}^2 + T_{13}^2 + T_{14}^2 + T_{23}^2 + T_{24}^2 + T_{34}^2)$ is the central Laplacian, and $P = - T_{12} T_{34} + T_{13} T_{24} - T_{14} T_{23}$, then also
\[L,D,P,\Delta\]
is a Rockland system (in fact, it is a system of generators of the $SO_4$-invariant differential operators relative to the Gelfand pair $(N_{4,2} \rtimes SO_4,SO_4)$, see Theorem~7.5 of \cite{fischer_nilpotent_2010}), and the associated Plancherel measure $\sigma'$ is given by
\begin{multline*}
\int_{\R^4} |m|^2 \,d\sigma' \\
= \frac{1}{2^8 \pi^6} \sum_{\alpha \in (1+2\N)^2} \int_{\R^2} |m(|b_1|\alpha_1 + |b_2|\alpha_2,|b_1|^3 \alpha_1 + |b_2|^3 \alpha_2, b_1 b_2, b_1^2 + b_2^2)|^2 \\
\times \,|b_1 b_2| (b_1^2 - b_2^2)^2 \,db,
\end{multline*}
i.e.,
\[\int_{\R^4} |m|^2 \,d\sigma' = \frac{1}{2^7 \pi^6} \sum_{\substack{\alpha \in (1+2\N)^2 \\\varepsilon \in \{-1,1\}}} \int_0^\infty \int_{\R^3} |m(\lambda, \lambda^3 \xi, \lambda^2 \eta, \lambda^2 \zeta)|^2 \,d\sigma'_{\alpha,\varepsilon}(\xi,\eta,\zeta) \,\lambda^7 \,d\lambda,\]
where $\sigma'_{\alpha,\varepsilon}$ is the regular Borel measure on $\R^3$ given by
\begin{multline*}
\int_{\R^3} f \,d\sigma'_{\alpha,\varepsilon} = \int_0^1 f \left( \textstyle \frac{(1-\theta)^3}{\alpha_1^2} + \frac{\theta^3}{\alpha_2^2}, \frac{\varepsilon (1-\theta) \theta}{\alpha_1 \alpha_2}, \frac{(1-\theta)^2}{\alpha_1^2} + \frac{\theta^2}{\alpha_2^2} \right) \\
\times \,{\textstyle\frac{(1-\theta) \theta}{\alpha_1^2 \alpha_2^2}} \left(\textstyle \frac{(1-\theta)^2}{\alpha_1^2} - \frac{\theta^2}{\alpha_2^2}\right)^2  \,d\theta.
\end{multline*}
By projection, we also obtain the Plancherel measure $\sigma''$ associated to the Rockland system
\[L,P,\Delta,\]
namely
\[\int_{\R^3} |m|^2 \,d\sigma'' = \frac{1}{2^7 \pi^6} \sum_{\substack{\alpha \in (1+2\N)^2 \\\varepsilon \in \{-1,1\}}} \int_0^\infty \int_{\R^2} |m(\lambda, \lambda^2 \eta, \lambda^2 \zeta)|^2 \,d\sigma''_{\alpha,\varepsilon}(\eta,\zeta) \,\lambda^7 \,d\lambda,\]
where the $\sigma''_{\alpha,\varepsilon}$ are the corresponding projections of the $\sigma'_{\alpha,\varepsilon}$.

Notice now that, if $p_\varepsilon : \R^2 \to \R$ is the linear map defined by
\[p_\varepsilon(\eta,\zeta) = \zeta + 2\epsilon\eta,\]
then the measure $\nu_\alpha = p_1(\sigma_{\alpha,-1}) = p_{-1}(\sigma_{\alpha,1})$ on $\R$ satisfies
\[\begin{split}
\int_\R f \,d\nu_\alpha &= \int_0^1 f\left(\left(\textstyle \frac{1-\theta}{\alpha_1} - \frac{\theta}{\alpha_2} \right)^2 \right) \,{\textstyle\frac{(1-\theta) \theta}{\alpha_1^2 \alpha_2^2}} \left(\textstyle \frac{1-\theta}{\alpha_1} - \frac{\theta}{\alpha_2} \right)^2 \left(\textstyle \frac{1-\theta}{\alpha_1} + \frac{\theta}{\alpha_2} \right)^2 \,d\theta\\
&\leq {\textstyle \frac{1}{4\alpha_1^2 \alpha_2^2} \left(\frac{1}{\alpha_1} + \frac{1}{\alpha_2}\right)^2} \int_0^1 f(s) \,ds,
\end{split}\]
from which we get
\[\sigma_{\alpha,\varepsilon}(B((\eta,\zeta),r)) \leq \nu_\alpha(B(p_{-\varepsilon}(\zeta,\eta)),r\sqrt{5}) \leq \frac{\sqrt{5}}{4 \alpha_1^2 \alpha_2^2} \left(\frac{1}{\alpha_1} + \frac{1}{\alpha_2}\right)^2 r.\]
Consequently, the sum of the $\sigma''_{\alpha,\varepsilon}$ is locally $1$-bounded on $\R^2$, and the measure $\sigma'$ is locally $2$-bounded on $\R^3 \setminus \{0\}$. But then also the sum of the $\sigma'_{\alpha,\varepsilon}$ is locally $1$-bounded on $\R^3$, so that $\sigma'$ is locally $2$-bounded on $\R^4 \setminus \{0\}$.

Analogously, one deduces that the previously considered measure $\sigma$ and the Plancherel measure associated to the system
\[L,-iT_{12},-iT_{13},-iT_{14},-iT_{23},-iT_{24},-iT_{34}\]
are both locally $2$-bounded off the origin.

\subsection{Some $3$-step groups}\label{subsection:example3step}

\paragraph{The free $3$-step nilpotent group $N_{2,3}$ on two generators} is defined by the relations
\[[X_1,X_2] = Y, \quad [X_1,Y] = T_1, \quad [X_2,Y] = T_2,\]
where $X_1,X_2,Y,T_1,T_2$ is a basis of its Lie algebra $\lie{n}_{2,3}$.

A family $\delta_t$ of automorphic dilations is given by
\[\delta_t(X_j) = t X_j, \quad \delta_t(Y) = t^2 Y, \quad \delta_t(T_j) = t^3 T_j,\]
which makes $N_{2,3}$ into a stratified group of homogeneous dimension $10$.

If $L = -(X_1^2 + X_2^2)$ is the sublaplacian and $D = 2X_2 T_1 - 2 X_1 T_2 - Y^2$, then it is easily checked that
\[L,D,-iT_1,-iT_2\]
is a Rockland system. Notice that the group $SO_2$ acts on $N_{2,3}$ by automorphisms given by simultaneous rotations of $\R X_1+ \R X_2$ and $\R T_1 + \R T_2$, and that the algebra generated by $L,D,-iT_1,-iT_2$ is $SO_2$-invariant. As we are going to see, this invariance by rotations gives some information about the Plancherel measure, which yields improvements in the weighted estimates.

We consider first the subsystem
\[L,-iT_1,-iT_2;\]
by homogeneity (cf.\ Propositions~\ref{prp:radialcoordinates} and \ref{prp:spectrum}), we know that the associated Plancherel measure $\sigma_0$ is given by
\[\int_{\R^3} f \,d\sigma_0 = \int_0^\infty \int_{\R^2} f(\lambda,\lambda^{3/2} \xi) \,d\sigma_0'(\xi) \,\lambda^4 \,d\lambda,\]
where $\sigma_0'$ is a compactly supported regular Borel measure on $\R^2$ which is invariant by rotations, so that in turn it can be written in the form
\[\int_{\R^2} f \,d\sigma'_0 = \frac{1}{2\pi} \int_{S^1} \int_{\left[0,+\infty\right[} f(\rho \omega) \,d\sigma''_0(\rho) \,d\omega,\]
where $\sigma''_0$ is the push-forward of $\sigma'_0$ via the map $\xi \mapsto |\xi|_2$.

Notice now that, for $r \leq 1$, $\alpha > 0$,
\[\begin{split}
\sigma''_0(\left[0,r\right[) &= 5 \int_{\R^3} \chr_{\left]0,1\right[}(\lambda) \,\chr_{B(0,r)}(\lambda^{-3/2} \tau) \,d\sigma_0(\lambda,\tau)\\
&\leq 5 r^\alpha \int_{\R^3} \chr_{\left]0,1\right[}(\lambda) \,\chr_{B(0,1)}(\tau) \,|\tau|_2^{-\alpha} \,d\sigma_0(\lambda,\tau).
\end{split}\]
Since, by Proposition~\ref{prp:homogeneouspushforward}, the last integral is finite for $\alpha < 2$, we deduce that
\[\sigma_0'(B(0,r)) = \sigma''_0(\left[0,r\right[) \leq C_\alpha r^\alpha \qquad\text{for $0 < \alpha < 2$}\]
(notice that, for $r \geq 1$, one can use the fact that $\sigma''_0$ is compactly supported).

Consequently, $\sigma_0'$ is locally $1$-bounded on $\R^2$: in fact, for $\xi \in \R^2$ and $r > 0$, if $|\xi|_2 \leq 2 r$ then
\[\sigma_0'(B(\xi,r)) \leq \sigma_0'(B(0,3r)) \leq C r,\]
whereas, for $|\xi|_2 \geq 2r$, the ball $B(\xi,r)$ is seen from the origin in an angle of magnitude less than $\pi r/|\xi|_2$, so that
\[\sigma_0'(B(\xi,r)) \leq \frac{r}{2|\xi|_2} \sigma_0''\left(\left[0, {\textstyle\frac{3}{2}} |\xi|_2\right[\right) \leq C r.\]
Therefore, $\sigma_0$ is locally $2$-bounded on $\R^3 \setminus \{0\}$.

This result has consequences also for the larger system $L,D,-iT_1,-iT_2$. In fact, if $\sigma$ is the Plancherel measure associated to this system, then, as before,
\[\int_{\R^4} f \,d\sigma = \int_0^\infty \int_{\R^3} f(\lambda,\lambda^2 \eta, \lambda^{3/2} \xi) \,d\sigma'(\eta,\xi) \,\lambda^4 \,d\lambda,\]
where $\sigma'$ is a compactly supported regular Borel measure on $\R^3$; since the push-forward of $\sigma'$ via the map $\R^3 \ni (\eta,\xi) \mapsto \xi \in \R^2$ is the measure $\sigma'_0$, then also $\sigma'$ is locally $1$-bounded on $\R^3$, so that $\sigma$ is locally $2$-bounded on $\R^4 \setminus \{0\}$.

It is also possible to consider the systems
\[L,D, \qquad L,\Delta, \qquad L,D,\Delta,\]
where $\Delta = -(T_1^2 + T_2^2)$ is the central Laplacian. In this case, all the mentioned operators are $SO_2$-invariant, so that the action of $SO_2$ on the joint spectrum is trivial, and we do not recover any information on the Plancherel measure.

\paragraph{The group $G_{6,19}$ of \cite{nielsen_unitary_1983}} is defined by the relations
\begin{equation}\label{eq:g619relations}
[X_6,X_5] = X_4, \qquad [X_6,X_3]=X_1, \qquad [X_5,X_4]=X_2,
\end{equation}
where $X_6,X_5,X_4,X_3,X_2,X_1$ is a basis of its Lie algebra $\lie{g} = \lie{g}_{6,19}$. This group also appears, in a different form\footnote{In \cite{jenkins_dilations_1979}, a $6$-dimensional Lie algebra $\lie{g}$ with basis $Y_1,Y_2,Y_3,Y_4,Y_5,Y_6$ and relations
\[[Y_1,Y_2] = Y_4, \qquad [Y_2,Y_3]=Y_5, \qquad [Y_1,Y_3]=Y_6 = [Y_1,Y_4] = [Y_2,Y_4]\]
is introduced, and it is claimed that there exist no dilations on $\lie{g}$ such that the elements of degree $1$ generate the whole algebra (i.e., that $\lie{g}$ is not stratifiable); the argument supporting the claim, however, considers only dilations which make $Y_1,\dots,Y_6$ homogeneous. On the other hand, by setting
\[X_1 = Y_5 - Y_6, \quad X_2 = -Y_6, \quad X_3 = Y_3 - Y_4, \quad X_4= -Y_4, \quad X_5= Y_1, \quad X_6 = Y_2 - Y_1,\]
one can easily check that the relations in the new basis $X_1,\dots,X_6$ coincide with the above \eqref{eq:g619relations}, and in the following we show that $\lie{g}_{6,19}$ admits a stratification.}, in \cite{jenkins_dilations_1979}.

For any ideal $I$ of $\lie{g}$, let $C(I)$ denote the ideal of $\lie{g}$ corresponding to the center of the quotient $\lie{g}/I$, and $Z(I)$ the centralizer of $I$ in $\lie{g}$. We then find quite a rich collection of characteristic ideals of $\lie{g}$:
\begin{alignat*}{6}
[\lie{g},\lie{g}] &= \Span\{ & & X_4,&  & X_2, & X_1 \}, \\
[\lie{g},[\lie{g},\lie{g}]] &= \Span\{ & & &  & X_2 & \}, \\
\lie{z} = C(0) &= \Span\{  & & & & X_2, & X_1 \}, \\
\lie{y} = C(\lie{z}) &= \Span\{ & & X_4, & X_3, & X_2, & X_1\}, \\
\lie{h} = Z([\lie{g},\lie{g}]) &= \Span\{X_6, & & X_4, & X_3, & X_2, & X_1 \}, \\
[\lie{h},\lie{h}] &= \Span\{ & & & & & X_1 \}, \\
\lie{k} = C([\lie{h},\lie{h}]) &= \Span \{ & & & X_3, & X_2, & X_1 \}, \\
Z(\lie{k}) &= \Span \{ & X_5, &X_4, &X_3, &X_2, &X_1\}.
\end{alignat*}

If $\delta_t$ are automorphic dilations on $\lie{g}$, then the previous ideals must be all $\delta_t$-homogeneous. In particular, $\tilde X_2 = X_2$ and $\tilde X_1 = X_1$ are certainly homogeneous. Moreover, in $\lie{h}$ we can find a homogeneous element $\tilde X_6 \in X_6 + \lie{y}$, whereas in $Z(\lie{k})$ we find a homogeneous $\tilde X_5 \in X_5 + \lie{y}$; then $\tilde X_4 = [\tilde X_6,\tilde X_5] \in X_4 + \lie{z}$ is homogeneous. Finally, in $\lie{k}$ we find a homogeneous $\tilde X_3 \in X_3 + \lie{z}$. It is then not difficult to show that, in the basis $\tilde X_6,\tilde X_5,\tilde X_4,\tilde X_3,\tilde X_2,\tilde X_1$, the relations of $\lie{g}$ are the same as in \eqref{eq:g619relations}, with $X_j$ replaced by $\tilde X_j$; therefore, without loss of generality, we may suppose that the initial basis $X_6,X_5,X_4,X_3,X_2,X_1$ is $\delta_t$-homogeneous.

Notice that, if $\lambda_j$ is the homogeneity degree of $X_j$ for $j=1,\dots,6$, then we must have
\[\lambda_4 = \lambda_6 + \lambda_5, \qquad \lambda_2 = \lambda_6 + 2\lambda_5, \qquad \lambda_1= \lambda_6 + \lambda_3 ;\]
in fact, for any choice of $\lambda_6,\lambda_5,\lambda_3 \geq 1$, if $\lambda_4,\lambda_2,\lambda_1$ are determined by the previous equalities, we have a system of automorphic dilations $\delta_t$ on $\lie{g}$ such that $X_j$ has degree $\lambda_j$ for $j=1,\dots,6$. In particular, for $\lambda_6 = \lambda_5 = \lambda_3 = 1$ we have a stratification on $\lie{g}$.

By computing the capacity map $J$ as in \S\ref{subsection:capacity}, one easily obtains that
\[|J(x_6 \bar X_6 + x_5 \bar X_5 + x_4 \bar X_4 + x_3 \bar X_3, t_2 X_2^* + t_1 X_1^*)|^2 = x_6^2 t_1^2 + x_5^2 t_2^2.\]
In particular, the linearly independent elements $\bar X_6^*,\bar X_5^* \in (\lie{g}/\lie{z})^*$ and $X_2,X_1 \in \lie{z}$ attest that $G_{6,19}$ is $2$-capacious (although Proposition~\ref{prp:capacitycriteria} gives only that $G_{6,19}$ is $1$-capacious), for any choice of a homogeneous structure on $G_{6,19}$.

A Rockland system on $G_{6,19}$ is given, e.g., by
\[L,-iX_2,-iX_1,\]
where $L$ is any Rockland operator (with respect to some homogeneous structure) on $G_{6,19}$; for instance, one can take
\[L = (-iX_6)^{2k_6} + (-iX_5)^{2k_5} + (-iX_3)^{2k_3}\]
for any choice of $k_6,k_5,k_3 \in \N \setminus \{0\}$.

\paragraph{The group $G_{6,23}$ of \cite{nielsen_unitary_1983}} is defined by the relations
\begin{equation}\label{eq:g623relations}\begin{split}
&[X_6,X_5]= X_4, \qquad [X_6,X_4]=X_2, \qquad [X_6,X_3]=-X_1, \\
&[X_5,X_4]=X_1, \qquad [X_5,X_3]=X_2,
\end{split}\end{equation}
where $X_6,X_5,X_4,X_3,X_2,X_1$ is a basis of its Lie algebra $\lie{g} = \lie{g}_{6,23}$.

In this case, we have a less rich list of characteristic ideals:
\begin{alignat*}{4}
[\lie{g},\lie{g}] &= \Span\{ X_4,&  & X_2, & X_1 \}, \\
\lie{z} = C(0) = [\lie{g},[\lie{g},\lie{g}]] &= \Span\{ &  & X_2 & X_1\}, \\
\lie{y} = C(\lie{z}) &= \Span\{ X_4, & X_3, & X_2, & X_1\}.
\end{alignat*}
However, since $[\lie{y},\lie{y}] = 0$, the bracket $[\cdot,\cdot] : \lie{g}\times\lie{g} \to \R$ defines a bilinear map
\[B : \lie{g}/\lie{y} \times \lie{y}/\lie{z} \to \lie{z},\]
and it is easily checked that $B(\bar x, \bar y) = 0$ implies $\bar x = 0$ and $\bar y = 0$.

Let $\delta_t$ be automorphic dilations on $\lie{g}$. Then we can find a $\delta_t$-homogeneous basis $v_6,v_5$ of $\lie{g}/\lie{y}$, with degrees $\lambda_6,\lambda_5$ respectively, and a $\delta_t$-homogeneous basis $w_4,w_3$ of $\lie{y}/\lie{z}$, with degrees $\lambda_4,\lambda_3$ respectively. Notice that, by the aforementioned non-degeneracy, $B(v_6,w_4), B(v_5,w_4)$ must be a homogeneous basis of $\lie{z}$, with degrees $\lambda_2 = \lambda_6+\lambda_4,\lambda_1 = \lambda_5+\lambda_4$ respectively. On the other hand, also $B(v_6,w_4), B(v_6,w_3)$ must be a homogeneous basis of $\lie{z}$, thus $B(v_6,w_3)$ must have degree $\lambda_1$, so that $\lambda_1 = \lambda_6+\lambda_3$. Analogously, $B(v_6,w_3),B(w_5,w_3)$ is a basis of $\lie{z}$, thus $B(w_5,w_3)$ must have degree $\lambda_2$, thus $\lambda_2 = \lambda_5 + \lambda_3$. By putting all together, we get $\lambda_6 = \lambda_5$, $\lambda_4= \lambda_3$, $\lambda_2 = \lambda_1$.

Consequently, we can choose $\delta_t$-homogeneous elements $\tilde X_6 \in X_6 + \lie{y}$ and $\tilde X_5 \in X_5+ \lie{y}$, and set $\tilde X_4 = [\tilde X_6,\tilde X_5] \in X_4 + \lie{z}$; moreover, we can choose in $\lie{y}$ a homogeneous $\tilde X_3 \in X_3 + \lie{z}$. Finally, if we set $\tilde X_2 = X_2$ and $\tilde X_1=X_1$, then it is easily seen that the relations of $\lie{g}$ in the basis $\tilde X_6,\tilde X_5,\tilde X_4,\tilde X_3,\tilde X_2,\tilde X_1$ are the same, mutatis mutandis, as in \eqref{eq:g623relations}; therefore, without loss of generality, we may suppose that the initial basis $X_6,X_5,X_4,X_3,X_2,X_1$ is $\delta_t$-homogeneous. Moreover, since $[X_6,X_5] = X_4$, we must have $\lambda_4 = 2\lambda_6$ and $\lambda_2= 3\lambda_6$; this means that $\lie{g}_{6,23}$ admits an essentially unique homogeneous structure, which is not stratified, because $X_6,X_5$ do not generate the whole Lie algebra.

The computation of the capacity map $J$ (see \S\ref{subsection:capacity}) gives
\[|J(x_6 \bar X_6 + x_5 \bar X_5 + x_4 \bar X_4 + x_3 \bar X_3, t_2 X_2^* + t_1 X_1^*)|^2 = (x_6^2 + x_5^2)(t_2^2 + t_1^2).\]
In particular, the linearly independent elements $\bar X_6^*,\bar X_5^* \in (\lie{g}/\lie{z})^*$ and $X_2,X_1 \in \lie{z}$ attest that $G_{6,23}$ is $2$-capacious (despite the fact that Proposition~\ref{prp:capacitycriteria} does not apply to this group), for any choice of a homogeneous structure on $G_{6,23}$.

A Rockland system on $G_{6,23}$ is given, e.g., by
\[L,-iX_2,-iX_1,\]
where $L$ is any Rockland operator on $G_{6,23}$; for instance, one can take
\[L = (-iX_6)^{4k} + (-iX_5)^{4k} + (-iX_3)^{2k}\]
for any choice of $k \in \N \setminus \{0\}$.

%% file: mihlin.tex
\chapter{Multiplier theorems}\label{chapter:multipliers}

In this last part of the work, homogeneity (of the group and the operators under consideration) plays a crucial role.

By means of homogeneity, in fact, along with singular integral operator theory, the weighted estimates obtained in the previous chapter can be put together, thus obtaining a multiplier theorem of Mihlin-H\"ormander type for a homogeneous weighted subcoercive system.

Next, we develop a sort of product theory of the above: several homogeneous groups $G_j$ are considered, each with its own homogeneous system of operators. By a non-conventional use of transference techniques (together with Littlewood-Paley decompositions and maximal operators), we obtain a multiplier theorem of Marcinkiewicz type, not only on the direct product of the $G_j$, but also on fairly different groups, even non-nilpotent.

Finally, some applications are presented, and particularly it is shown how, by the use of the mentioned theorems, it is possible to improve known multiplier results for a single operator on some non-nilpotent groups.

\section{Singular integral operators}

Here we summarize the main notions and techniques related to singular integral operators which will be used in the following sections in order to prove the multiplier theorems. Thus we shall restrict to the case of left-invariant operators $T$ on a homogeneous Lie group $G$ (with automorphic dilations $\delta_t$ and homogeneous dimension $Q_\delta$), which are given by convolution:
\[T f = f * k\]
for some distribution $k \in \Sz'(G)$ (see Theorem~\ref{thm:schwartzkernels}). In our context, the underlying metric measure space is the homogeneous group $G$ with a Haar measure (i.e., essentially $\R^d$ with the Lebesgue measure, where $d = \dim G$) and a homogeneous left-invariant distance induced by a subadditive homogeneous norm $|\cdot|_\delta$. It should be noticed, however, that some of the results mentioned below admit an extension to more general settings, such as the spaces of homogeneous type developed by Coifman, de Guzman and Weiss (see \cite{coifman_singular_1970}, \cite{coifman_analyse_1971}), where the notions of distribution and convolution cannot be used.

Our first multiplier theorem will be proved in \S\ref{section:mihlin} via the Calder\'on-Zygmund singular integral theory, which allows, starting from an operator $T$ bounded on $L^2(G)$, to obtain boundedness of $T$ on $L^p(G)$ for some $p \neq 2$, under the hypothesis that the convolution kernel $k$ is a function off the origin satisfying suitable estimates.

\begin{thm}\label{thm:singularintegral}
Let $T$ be a bounded operator on $L^2(G)$ whose convolution kernel $k \in \Sz'(G)$ is such that $k|_{G \setminus \{e\}} \in L^1_\loc(G \setminus \{e\})$ and, for some $K \geq 0$, $c > 1$,
\[
\int_{|x|_\delta \geq c|h|_\delta} |k(xh) - k(x)| \,dx \leq K, \qquad \int_{|x|_\delta \geq c|h|_\delta} |k(hx) - k(x)| \,dx \leq K,
\]
for all $h \in G \setminus \{e\}$. Then $T$ is of weak type $(1,1)$ and bounded on $L^p(G)$ for $1 < p < \infty$, with $\|T\|_{p \to p} \leq C_p (\|T\|_{2 \to 2} + K)$.
\end{thm}
\begin{proof}
See e.g.\ \cite{stein_harmonic_1993}, \S I.5, Theorem~3 and \S I.7.4(iii).
\end{proof}

For the second multiplier theorem, proved in \S\ref{section:marcinkiewicz}, a multi-parameter structure is required, and Littlewood-Paley theory is exploited. An important tool will be the following result, which summarizes a well-known argument for proving properties of square functions.

\begin{prp}\label{prp:khinchin}
Let $(X,\mu)$ be a $\sigma$-finite measure space, $\ell \geq 1$, $1 \leq p < \infty$, $T_{\vec{k}}$ ($\vec{k} \in \N^\ell$) bounded linear operators on $L^p(X,\mu)$. Let $A > 0$ be such that, for all choices of $\varepsilon^i_k \in \{-1,1\}$ ($1 \leq i \leq \ell$, $k \in \N$) and of a finite subset $I \subseteq \N^\ell$, we have
\begin{equation}\label{eq:segni}
\left\|\textstyle\sum_{\vec{k} \in I} \varepsilon^1_{k_1} \cdots \varepsilon^\ell_{k_\ell} T_{\vec{k}}\right\|_{p \to p} \leq A.
\end{equation}
Then, for all $f \in L^p(X,\mu)$,
\begin{equation}\label{eq:khinchin1}
\left\|\left(\textstyle\sum_{\vec{k} \in \N^{\ell}} |T_{\vec{k}} f|^2\right)^{1/2}\right\|_p \leq C_{\ell,p} A \|f\|_p.
\end{equation}
Moreover, if $p > 1$, for all $\{f_{\vec{k}}\}_{\vec{k} \in \N^\ell} \subseteq L^p(X,\mu)$, if $\left(\sum_{\vec{k}} |f_{\vec{k}}|^2\right)^{1/2} \in L^p(X,\mu)$, then
\[\left\|\textstyle\sum_{\vec{k} \in \N^\ell} T_{\vec{k}} f_{\vec{k}}\right\|_p \leq C_{\ell,p'} A \left\|\left(\textstyle\sum_{\vec{k} \in \N^\ell} |f_{\vec{k}}|^2 \right)^{1/2} \right\|_p\]
where the series on the left-hand side converges unconditionally in $L^p$.
\end{prp}
\begin{proof}
For $n \in \N$, let $r_n : [0,1] \to \R$ be the $n$-th Rademacher function,
\[r_n(t) = (-1)^{\lfloor 2^n t \rfloor},\]
and set $r_{\vec{k}} = r_{k_1} \otimes \cdots \otimes r_{k_\ell}$ for $\vec{k} = (k_1,\dots,k_\ell) \in \N^\ell$. Then $(r_{\vec{k}})_{\vec{k}}$ is an (incomplete) orthonormal system in $L^2([0,1]^\ell)$, and the following Khinchin's inequalities hold: for $1 \leq p < \infty$, there exist constants $c_{\ell,p},C_{\ell,p}> 0$ such that
\[c_{\ell,p}^{-1} \|f\|_p \leq \|f\|_2 \leq C_{\ell,p} \|f\|_p \qquad\text{for all $f \in \Span \{ r_{\vec{k}} \tc \vec{k} \in \N^\ell\}$.}\]
(see \cite{stein_singular_1970}, Appendix~D, or \cite{grafakos_classical_2008}, Appendix~C).

Consequently, for all finite $I \subseteq \N^\ell$ and $f \in L^p(X,\mu)$, we have
\[\begin{split}
\left\|\left(\textstyle\sum_{\vec{k} \in I} |T_{\vec{k}} f|^2\right)^{1/2}\right\|_p^p &= \int_X \left(\textstyle\sum_{\vec{k} \in I} |T_{\vec{k}} f(x)|^2\right)^{p/2} \,d\mu(x)\\
&\leq C_{\ell,p}^p \int_X \int_{[0,1]^\ell} \left|\textstyle\sum_{\vec{k} \in I} T_{\vec{k}} f(x) r_{\vec{k}}(t) \right|^p \,dt \,d\mu(x) \\
&= C_{\ell,p}^p \int_0^1 \left\| \left(\textstyle\sum_{\vec{k} \in I} r_{\vec{k}}(t) T_{\vec{k}}\right) f \right\|^p \,dt \leq C_{\ell,p}^p A^p \|f\|_p^p.
\end{split}\]
Since $I \subseteq \N^\ell$ was arbitrary, \eqref{eq:khinchin1} follows by monotone convergence.

Notice now that the vector-valued Lebesgue space $V_p = L^p(X,\mu; l^2(\N^\ell))$ can be thought of as a space of sequences of $L^p(X,\mu)$-functions:
\[V_p = \left\{(f_{\vec{k}})_{\vec{k} \in \N^\ell} \in L^p(X,\mu)^{\N} \tc \textstyle \left(\sum_{\vec{k}} |f_{\vec{k}}|^2\right)^{1/2} \in L^p(X,\mu) \right\},\]
with norm
\[\|(f_{\vec{k}})_{\vec{k} \in \N^\ell}\|_{V_p} = \left\|\textstyle \left(\sum_{\vec{k}} |f_{\vec{k}}|^2\right)^{1/2}\right\|_p.\]
The inequality \eqref{eq:khinchin1} therefore means that the operator $f \mapsto (T_{\vec{k}} f)_{\vec{k} \in \N^\ell}$ is bounded $L^p(X,\mu) \to V_p$, with norm not greater than $C_{\ell,p} A$.

If $p > 1$, the hypothesis \eqref{eq:segni} is equivalent to
\[\left\|\textstyle\sum_{\vec{k} \in I} \varepsilon^1_{k_1} \cdots \varepsilon^n_{k_n} T_{\vec{k}}^*\right\|_{p' \to p'} \leq A;\]
consequently we also have that $S : f \mapsto (T^*_{\vec{k}} f)_{\vec{k} \in \N^\ell}$ is bounded $L^{p'}(X,\mu) \to V_{p'}$, with norm not greater than $C_{\ell,p'} A$. This means that the transpose operator $S^* : V_{p} \to L^{p}(X,\mu)$ is bounded too, with the same norm; since it is easily shown that
\[S^* \left( (f_{\vec{k}})_{\vec{k}} \right) = \textstyle\sum_{\vec{k}} T_{\vec{k}} f_{\vec{k}},\]
where the series on the right-hand side converges unconditionally in $L^p$, the remaining part of the conclusion follows.
\end{proof}

The other fundamental tool will be the boundedness of some maximal operators, given by the following general result of Christ \cite{christ_strong_1992}:

\begin{thm}\label{thm:strongmaximal}
Let $G$ be a nilpotent Lie group, and let $X_1,\dots,X_d$ a basis of its Lie algebra $\lie{g}$. Identify $G = \lie{g}$ via exponential coordinates, and let $\daleth_{\vec{t}}$ be the multi-parameter dilations on $G$ given by
\[\daleth_{(t_1,\dots,t_d)} X_j = t_j X_j.\]
Then the \emph{strong maximal operator} $M_{\strong}$ associated to the basis $X_1,\dots,X_d$, defined by
\begin{equation}\label{eq:strongmaximal}
M_{\strong} f(x) = \sup_{\vec{t} > 0} \int_K |f(x \cdot (\daleth_{\vec{t}}(y))^{-1})| \,dy
\end{equation}
(where $K$ is some compact neighborhood of the identity), is bounded on $L^p(G)$ for $1 < p \leq \infty$.
\end{thm}

\section{Mihlin-H\"ormander multipliers}\label{section:mihlin}

Let $G$ be a homogeneous Lie group, with automorphic dilations $\delta_t$ and homogeneous dimension $Q_\delta$. We denote by $Q_G$ the degree of polynomial growth of $G$. Moreover, we fix a connected modulus $|\cdot|_G$ on $G$ and we set
\[\langle x \rangle_G = 1 + |x|_G.\]
We denote also by $|\cdot|_\delta$ a subadditive homogeneous norm on $G$.

Let $L_1,\dots,L_n$ be a homogeneous weighted subcoercive system on $G$, and $\epsilon_t$ be the associated dilations on $\R^n$, as in \eqref{eq:spectraldilations}. Denote moreover by $|\cdot|_\epsilon$ a $\epsilon$-homogeneous norm on $\R^n$, smooth away from the origin.

Our starting point is, for some $p \in \left[1,\infty\right]$ and $s \in \R$, the following
\begin{quote} \noindent{\bf hypothesis \HP{p}{s}}: 
for some compact $K_0 \subseteq \R^n \setminus \{0\}$ such that
\[\bigcup_{t > 0} \epsilon_t(\mathring{K_0}) = \R^n \setminus \{0\},\]
for some $q_0 \in \left[1,\infty\right]$, for all $\beta > s$ and for all $m \in \D(\R^n)$ with $\supp m \subseteq K_0$, we have
\[\|\breve m\|_{L^1(G)} \leq C_{\beta} \|m\|_{B_{p,q_0}^\beta(\R^n)}.\]
\end{quote}

This hypothesis can be checked by exploiting the results of Chapter~\ref{chapter:weighted}; for instance, we have

\begin{prp}\label{prp:HPsatisfaction}
For every $p \in [1,\infty]$, the hypothesis \HP{p}{s} holds in each of the following cases:
\begin{itemize}
\item $s = Q_G/2 + n/p - 1/\max\{2,p\}$;
\item $\sigma$ is locally $d$-bounded on $\R^n \setminus \{0\}$ and $s = Q_G/2+ n/p-d/\max\{2,p\}$;
\item $G$ is $h$-capacious and $s = (Q_G-h)/2+n/p-1/\max\{2,p\}$.
\end{itemize}
\end{prp}
\begin{proof}
The conclusion follows easily from the weighted estimates of \S\ref{section:weightedestimates} and \S\ref{section:metivier}. We omit the details since an alternative proof is obtained by combining the following Proposition~\ref{prp:HPKsatisfaction} and Corollary~\ref{cor:hypothesescomparison}.
\end{proof}

\begin{prp}\label{prp:interpolatedweightedestimates}
Suppose that \HP{p}{s} holds for some $p \in [1,\infty]$ and $s \in \R$. Then $s \geq n/p$. Moreover, for every compact $K \subseteq \R^n \setminus \{0\}$, for every $q \in [1,\infty]$, for every $\alpha \geq 0$ and $\beta > \alpha + s$, for every $D \in \Diff(G)$, for every $m \in B_{p,q}^\beta(\R^n)$ with $\supp m \subseteq K$, we have
\[\|D \breve m\|_{L^1(G,\langle x \rangle^\alpha \,dx)} \leq C_{K,D,\alpha,\beta,q} \|m\|_{B_{p,q}^\beta}.\]
\end{prp}
\begin{proof}
Let $\lambda \in \mathring{K}_0$. For every $m \in \D(\R^n)$ with $\supp m \subseteq K_0$, we then have
\[|m(\lambda)| \leq \|m\|_\infty = \|\breve m\|_{\Cv^2} \leq \|\breve m\|_1 \leq C_\beta \|m\|_{B_{p,q_0}^\beta}\]
for all $\beta > s$; by Proposition~\ref{prp:besovoptimalembedding}, we conclude that $s \geq n/p$.

Let now $K \subseteq \R^n \setminus \{0\}$ be compact. Then we can find $t_1,\dots,t_k$ such that
\[K \subseteq \epsilon_{t_1}(\mathring K_0) \cup \dots \cup \epsilon_{t_k}(\mathring K_0)\]
by compactness. If
\[\phi_0,\phi_1,\dots,\phi_k\]
is a partition of unity on $\R^n$ subordinated to the open cover
\[\R^n \setminus K, \epsilon_{t_1}(\mathring K_0), \dots, \epsilon_{t_k}(\mathring K_0),\]
then $\phi_1 + \dots + \phi_k \equiv 1$ on $K$, so that, for every $m \in \D(\R^n)$ with $\supp m \subseteq K$,
\[m = m\phi_1 + \dots + m\phi_k = m_1 \circ \epsilon_{t_1^{-1}} + \dots + m_k \circ \epsilon_{t_k^{-1}},\]
where the $m_j = (m \phi_j) \circ \epsilon_{t_j}$ are supported in $K_0$; therefore, for all $\beta > s$, by Propositions~\ref{prp:plancherelhomogeneous} and \ref{prp:besovproduct},
\[\|\breve m\|_1 \leq \sum_{j=1}^k \|\breve m_j\|_1 \leq C_\beta \sum_{j=1}^k \|m_j\|_{B_{p,q_0}^\beta} \leq C_{\beta,K} \|m\|_{B_{p,q_0}^\beta}.\]

If now, given $K \subseteq \R^n \setminus \{0\}$ compact and $\beta > s$, we choose some $K' \subseteq \R^n \setminus \{0\}$ compact such that $K \subseteq \mathring K'$ and some $\beta' \in \left]s,\beta\right[$, by Proposition~\ref{prp:besovapproximation} it is easy to see that every $m \in B_{p,q_0}^\beta(\R^n)$ with $\supp m \subseteq K$ can be approximated in $B_{p,q_0}^{\beta'}(\R^n)$ by a sequence of smooth functions supported in $K'$, therefore, by passing to the limit in the previously obtained inequality and by Proposition~\ref{prp:besovembeddings},
\[\|\breve m\|_1 \leq C_{\beta',K'} \|m\|_{B_{p,q_0}^{\beta'}} \leq C_{\beta,K} \|m\|_{B_{p,q_0}^\beta}.\]

On the other hand, by Corollary~\ref{cor:polyl1estimates} and Proposition~\ref{prp:besovembeddings}, we also have, for $\alpha \geq 0$, $\beta > \alpha + Q_G/2 + n/p$, $K \subseteq \R^n \setminus \{0\}$ compact and $m \in B_{p,\infty}^\beta(\R^n)$ with $\supp m \subseteq K$, that
\[\|\breve m\|_{L^1(\langle x \rangle^\alpha \,dx)} \leq C_{K,\alpha,\beta} \|m\|_{B_{p,\infty}^\beta}.\]
Hence, analogously as in the proof of Theorem~\ref{thm:l2estimates}, we obtain first the conclusion for $D = 1$ by interpolation, and then also for an arbitrary $D \in \Diff(G)$.
\end{proof}

Notice that, by Proposition~\ref{prp:nilpotentgrowth}, there are constants $a,C > 0$ such that
\begin{equation}\label{eq:moduluscomparison}
1 + |x|_\delta \leq C \langle x \rangle_G^a.
\end{equation}

\begin{cor}\label{cor:usefulestimates}
Suppose that $\HP{p}{s}$ holds for some $p \in [1,\infty]$ and $s \in \R$. Let $K \subseteq \R^n \setminus \{0\}$ be compact, $\beta > s$ and $q \in [1,\infty]$. If $m \in B^\beta_{p,q}(\R^n)$ and $\supp m \subseteq K$, then $\breve m \in L^1(G)$, thus $m(L)$ is bounded on $L^p(G)$ for $1 \leq p \leq \infty$. Moreover, for $0 \leq \alpha < (\beta - s)/a$,
\begin{equation}\label{eq:weightedl1}
\int_G (1+|x|_\delta)^\alpha |\breve m(x)| \,dx \leq C_{K,\alpha,\beta,p,q} \|m\|_{B_{p,q}^\beta}
\end{equation}
and, for all $h \in G$,
\begin{gather}
\label{eq:rightlipschitz}
\int_G |\breve m(xh) - \breve m(x)| \,dx \leq C_{K,\beta,p,q} \|m\|_{B_{p,q}^\beta} |h|_\delta,\\
\label{eq:leftlipschitz}
\int_G |\breve m(hx) - \breve m(x)| \,dx \leq C_{K,\beta,p,q} \|m\|_{B_{p,q}^\beta} |h|_\delta.
\end{gather}
\end{cor}
\begin{proof}
Since $\beta > s + a\alpha$, by Proposition~\ref{prp:interpolatedweightedestimates} and \eqref{eq:moduluscomparison} we have
\[\int_G (1+|x|_\delta)^\alpha |\breve m(x)| \,dx \leq C_\alpha \int_G \langle x \rangle_G^{a\alpha} |\breve m(x)| \,dx \leq C_{K,\alpha,\beta,p,q} \|m\|_{B_{p,q}^\beta},\]
and in particular $\breve m \in L^1(G)$.

Starting from the inequality
\[\int_G |\breve m(x \exp(tX)) - \breve m(x)| \,dx \leq \|X \breve m\|_1 |t|,\]
true for all $X \in \lie{g}$ and $t \in \R$, chosen a basis $X_1,\dots,X_k$ of $\lie{g}$, with $X_j$ homogeneous of degree $d_j$, it is not difficult to see that
\[\int_G |\breve m(x h) - \breve m(x)| \,dx \leq C \sum_{j=1}^k \|X_j \breve m\|_1 |h|_\delta^{d_j}.\]
In particular, we also have
\[\int_G |\breve m(x h) - \breve m(x)| \,dx \leq C \left(\|\breve m\|_1 + \sum_{j=1}^k \|X_j \breve m\|_1\right) |h|_\delta\]
(in fact, for $|h|_\delta$ small we have $|h|_\delta^{d_j} \leq |h|_\delta$, whereas for $|h|_\delta$ large the left-hand side is not greater than $2\|\breve m\|_1$). However, by Proposition~\ref{prp:interpolatedweightedestimates},
\[\|\breve m\|_1 + \sum_{j=1}^k \|X_j \breve m\|_1 \leq C_{K,\beta,p,q} \|m\|_{B^\beta_{p,q}},\]
thus we get \eqref{eq:rightlipschitz}. The remaining inequality is simply obtained by replacing $m$ with $\overline{m}$.
\end{proof}

\begin{lem}\label{lem:spectraldecomposition}
Let $m$ be a bounded measurable function on $\R^n$. Then we can find bounded measurable functions $m_j$ on $\R^n$ (for $j \in \Z$) such that
\[\supp m_j \subseteq \{\lambda \tc 2^{-1} \leq |\lambda|_\epsilon \leq 2\},\]
\[\|m_j\|_{B_{p,q}^\beta} \leq C_{p,q,\beta} \|m\|_{M_\epsilon B_{p,q}^\beta}\]
for all $p,q \in [1,\infty]$ and $\beta \geq 0$, and moreover
\begin{equation}\label{eq:spectraldecomposition}
\breve m = \sum_{j \in \Z} 2^{-Q_\delta j} \breve m_j \circ \delta_{2^{-j}},
\end{equation}
in the sense of strong convergence of the corresponding convolution operators.
\end{lem}
\begin{proof}
Set $K = \{\lambda \tc 2^{-1} \leq |\lambda|_\epsilon \leq 2\}$. Chosen a nonnegative $\eta \in \D(\R^n)$ supported in $K$ and such that\footnote{Take a nonnegative $\gamma \in \D(\left]0,+\infty\right[)$ with $[3/4,3/2] \subseteq \{\gamma  \neq 0 \} \subseteq \left[1/2,2\right]$. Then, for all $t > 0$, there is at least one and at most two $j \in \Z$ such that $\gamma(2^j t) \neq 0$, so that the sum $s(t) = \sum_{j \in \Z} \gamma(2^j t)$ is in fact locally a finite sum, $s \in \E(\left]0,+\infty\right[)$ and $s > 0$; moreover, clearly by definition $s(2^j t) = s(t)$ for all $j \in \Z$, $t > 0$. Using $s$ we can renormalize $\gamma$ obtaining $\tilde\gamma = \gamma/s$, which is again a smooth nonnegative function supported in $[1/2,2]$, but in addition it satisfies $\sum_{j \in \Z} \tilde\gamma(2^j t) = 1$ for all $t > 0$. Since the homogeneous norm is smooth away from $0$, the function $\eta(x) = \tilde\gamma(|x|_\epsilon)$ has the required properties.}
\[\sum_{j \in \Z} \eta(\epsilon_{2^j}(\lambda)) = 1 \qquad\text{for all $\lambda \neq 0$},\]
we have (see Proposition~\ref{prp:mhequivalence})
\[\|m\|_{M_\epsilon B_{p,q}^\beta} = \sup_{t > 0} \|(m \circ \epsilon_t) \,\eta \|_{B_{p,q}^\beta(\R^n)}.\]
Let 
\[m_j =  (m \circ \epsilon_{2^{-j}}) \, \eta.\]
Then clearly $\|m_j\|_{B_{p,q}^\beta} \leq \|m\|_{M_\epsilon B_{p,q}^\beta}$ and $\supp m_j \subseteq K$. Moreover we have (away from $0$)
\[m = \sum_{j \in \Z} m_j \circ \epsilon_{2^j}.\]
In fact, this is locally a finite sum and the convergence holds pointwise (away from $0$), dominated by the constant $\|m\|_\infty$. Since $E(\{0\}) = 0$, by the spectral theorem and Proposition~\ref{prp:plancherelhomogeneous} we then have \eqref{eq:spectraldecomposition}, in the sense of strong convergence of the corresponding convolution operators.
\end{proof}

\begin{prp}\label{prp:singularkernel}
Suppose that $\HP{p}{s}$ holds for some $p \in [1,\infty]$ and $s \in \R$. Let $\beta > s$. If $m$ is a bounded Borel function on $\R^n$ such that $\|m\|_{M_\epsilon B_{p,q}^\beta} < \infty$, then $\breve m|_{G \setminus \{e\}} \in L^1_\loc(G \setminus \{e\})$, and moreover
\begin{gather}
\label{eq:rightcz} \int_{|x|_\delta \geq 2|h|_\delta} |\breve m(xh) - \breve m(x)| \,d\mu(x) \leq C_{p,q,\beta} \|m\|_{M_\epsilon B_{p,q}^\beta},\\
\label{eq:leftcz} \int_{|x|_\delta \geq 2|h|_\delta} |\breve m(hx) - \breve m(x)| \,d\mu(x) \leq C_{p,q,\beta} \|m\|_{M_\epsilon B_{p,q}^\beta}
\end{gather}
for all $h \in G \setminus \{e\}$.
\end{prp}
\begin{proof}
Let $m_j$ ($j \in \Z$) be given by Lemma~\ref{lem:spectraldecomposition} and set $u_j = 2^{-Q_\delta j} \breve m_j \circ \delta_{2^{-j}}$.

Firstly we prove that the convergence in \eqref{eq:spectraldecomposition} holds also in $L^1_\loc(G \setminus \{e\})$. In fact, let $B_k = \{ x \in G \tc 2^k \leq |x|_\delta \leq 2^{k+1}\}$; it is sufficient to prove the convergence in each $L^1(B_k)$. We have
\[\int_{B_k} |u_j| \,d\mu = \int_{B_{k-j}} |\breve m_j| \,d\mu\]
and, for $j \leq k$,
\[\int_{B_{k-j}} |\breve m_j(x)| \,dx \leq 2^{\alpha (j-k)} \int_{B_{k-j}} |\breve m_j(x)| |x|_\delta^{\alpha} \,dx \leq C 2^{\alpha(j-k)} \|m\|_{M_\epsilon B_{p,q}^\beta}\]
(where $\alpha > 0$ is as in \eqref{eq:weightedl1}), whereas, for $j \geq k$,
\[\int_{B_{k-j}} |\breve m_j(x)| \,dx \leq \|\breve m_j\|_2 \, \mu(B_{k-j})^{1/2} \leq \sigma(K)^{1/2} \|m\|_\infty \, \mu(B_0)^{1/2} 2^{Q_\delta (k-j)/2},\]
(here we use a uniform estimate on the $L^2$-norms of the $\breve m_j$) so that
\[\sum_j \int_{B_k} |u_j| \,d\mu \leq C' \sum_{j \leq k} 2^{\alpha(j-k)} + C'' \sum_{j \geq k} 2^{Q_\delta (k-j)/2} < \infty.\]
This shows (by uniqueness of limits) that the restriction of the distribution $\breve m$ to $G \setminus \{e\}$ coincides with a function in $L^1_\loc(G \setminus \{e\})$.

Notice that, if $|x|_\delta \geq 2|h|_\delta$, then $|xh|_\delta \geq |h|_\delta$ by subadditivity of $|\cdot|_\delta$, so that in the integrals in \eqref{eq:rightcz}, \eqref{eq:leftcz} the function $\breve m$ is always evaluated away from the origin.

Since $\breve m = \sum_{j \in \Z} u_j \in L^1_\loc(G \setminus \{e\})$, then $\RA_h \breve m - \breve m = \sum_{j \in \Z} (\RA_h u_j - u_j)$ is in $L^1_\loc(G \setminus \{e,h^{-1}\})$, so that in particular
\[\int_{R \geq |x|_\delta \geq 2|h|_\delta} |\breve m(xh) - \breve m(x)| \,dx \leq \sum_{j \in \Z} \int_{R \geq |x|_\delta \geq 2|h|_\delta} |u_j(xh) - u_j(x)| \,dx\]
for all $R > 0$ and then
\[\int_{|x|_\delta \geq 2|h|_\delta} |\breve m(xh) - \breve m(x)| \,dx \leq \sum_{j \in \Z} \int_{|x|_\delta \geq 2|h|_\delta} |u_j(xh) - u_j(x)| \,dx.\]

Let $k \in \Z$. Then for $j < k$
\begin{multline*}
\int_{|x|_\delta \geq 2|h|_\delta} |u_j(xh) - u_j(x)| \,dx \leq 2 \int_{|x|_\delta \geq |h|_\delta} |u_j(x)| \,dx \\
= 2 \int_{|y|_\delta \geq 2^{-j} |h|_\delta} |\breve m_j(y)| \,dy \leq C_{p,q,\beta} \frac{2^{\alpha j}}{|h|_\delta^\alpha} \|m\|_{M_\epsilon B_{p,q}^\beta}
\end{multline*}
by \eqref{eq:weightedl1}, whereas for $j \geq k$
\begin{multline*}
\int_{|x|_\delta \geq 2|h|_\delta} |u_j(xh) - u_j(x)| \,dx \leq \int_G |\breve m_j(y \delta_{2^{-j}}(h)) - \breve m_j(y)| \,dy \\
\leq C_{p,q,\beta} \frac{|h|_\delta}{2^j} \|m\|_{M_\epsilon B_{p,q}^\beta}
\end{multline*}
by \eqref{eq:rightlipschitz}. Putting all together,
\[\int_{|x|_\delta \geq 2|h|_\delta} |\breve m(xh) - \breve m(x)|\,dx \leq C_{p,q,\beta} \|m\|_{M_\epsilon B_{p,q}^\beta} \left(\frac{2^{k\alpha}}{|h|_\delta^\alpha} \sum_{j<0} 2^{j\alpha} + \frac{|h|_\delta}{2^{k}} \sum_{j \geq 0} 2^{-j}\right)\]
and, in order to obtain an estimate independent of $h$, it is sufficient to choose a $k$ such that $2^k \leq |h|_\delta < 2^{k+1}$.

Hence we have proved \eqref{eq:rightcz}; the inequality \eqref{eq:leftcz} is obtained in the same way, using \eqref{eq:leftlipschitz} instead of \eqref{eq:rightlipschitz}.
\end{proof}

Here is finally the multiplier theorem.

\begin{thm}\label{thm:mihlinhoermander}
Suppose that $\HP{p}{s}$ holds for some $p \in [1,\infty]$ and $s \in \R$. If $m$ is a bounded measurable function on $\R^n$ such that $\|m\|_{MHB_{p,q}^\beta} < \infty$ for some $\beta > s$ and $q \in [1,\infty]$, then the operator $m(L)$ is of weak type $(1,1)$ and bounded on $L^r(G)$ for $1 < r < \infty$, with
\[\|m(L)\|_{r \to r} \leq C_{r,p,q,\beta} \|m\|_{M_\epsilon B_{p,q}^\beta}.\]
\end{thm}
\begin{proof}
By Proposition~\ref{prp:interpolatedweightedestimates}, we have $\beta > n/p$, so that, by Proposition~\ref{prp:besovembeddings},
\[\|m\|_\infty \leq C_{p,q,\beta} \|m\|_{M_\epsilon B_{p,q}^\beta},\]
thus Theorem~\ref{thm:singularintegral} and Proposition~\ref{prp:singularkernel} yield immediately the conclusion.
\end{proof}

The ``smoothness'' conditions \eqref{eq:rightcz}, \eqref{eq:leftcz} proved in Proposition~\ref{prp:singularkernel} for the kernel $\breve m$ off the origin are sufficient, together with the $L^2$-boundedness of $m(L)$, to obtain the previous multiplier theorem. However, we show now that, under the same hypotheses, the kernel $\breve m$ satisfies also some ``size'' and ``cancellation'' conditions, analogous to the ones considered for convolution kernels in $\R^n$ (see, e.g., \S4.4 of \cite{grafakos_classical_2008}, or \S VI.4.5 of \cite{stein_harmonic_1993}).

Recall that a \emph{normalized bump function} (or \emph{n.b.f.}\ in short) of order $N$ is a smooth function on $G$ supported in a fixed compact neighborhood of the identity, which is bounded by a fixed constant together with its derivatives up to order $N$.

\begin{prp}\label{prp:cancellation}
Under the hypotheses of Proposition~\ref{prp:singularkernel}, we also have:
\begin{itemize}
\item[(a)] for $r > 0$,
\[\int_{r < |x|_\delta < 2r} |\breve m(x)| \,dx \leq C_{p,q,\beta} \|m\|_{M_\epsilon B_{p,q}^\beta};\]
\item[(b)] for every n.b.f.\ $\phi$ of order $1$ and every $t > 0$,
\[|\langle \breve m, \phi \circ \delta_t \rangle| \leq C_{p,q,\beta} \|m\|_{M_\epsilon B_{p,q}^\beta};\]
\item[(c)] for $0 < r_1 < r_2$,
\[\left| \int_{r_1 < |x|_\delta < r_2} \breve m(x) \,dx \right| \leq C_{p,q,\beta} \|m\|_{M_\epsilon B_{p,q}^\beta}.\]
\end{itemize}
\end{prp}
\begin{proof}
(a) From the proof of Proposition~\ref{prp:singularkernel}, we get that, if
\[B_k = \{ x \in G \tc 2^k \leq |x|_\delta \leq 2^{k+1}\},\]
then
\begin{equation}\label{eq:size}
\int_{B_k} |\breve m(x)| \,dx \leq C_{p,q,\beta} \|m\|_{M_\epsilon B_{p,q}^\beta},
\end{equation}
where the constant $C_{p,q,\beta}$ is independent of $k \in \Z$, and we are done.

(b) Set $N = \lfloor \frac{\dim G}{2} + 1 \rfloor$. The estimate for a n.b.f.\ $\phi$ of order $N$ follows simply from boundedness of $m$, by the use of Sobolev's embedding:
\[\begin{split}
|\langle \breve m, \phi \circ \delta_t \rangle | &= |\langle (m \circ \epsilon_t)\breve{}\, , \phi \rangle| \\
&= |(\phi^* * (m \circ \epsilon_t)\breve{}\,)(e)| \\
&\leq C \max_{|\alpha| \leq N} \| (X^\alpha)^\dstar (\phi^* * (m \circ \epsilon_t)\breve{}\,) \|_2 \\
&= C \max_{|\alpha| \leq N} \| (X^\alpha \phi)^* * (m \circ \epsilon_t)\breve{}\, \|_2 \\
&\leq C \max_{|\alpha| \leq N} \| (X^\alpha \phi)^* \|_2 \|(m \circ \epsilon_t)(L)\|_{2 \to 2} \leq C \|m\|_\infty.
\end{split}\]
Observe that, by Propositions~\ref{prp:interpolatedweightedestimates} and \ref{prp:besovembeddings}, $\|m\|_{M_\epsilon B_{p,q}^\beta}$ majorizes $\|m\|_\infty$.

In order to extend the result to n.b.f.\ functions of order $1$, we follow Remark 2.1.7 of \cite{nagel_singular_2001}. Fix exponential coordinates $x = (x_1,\dots,x_d)$ on $G$ such that
\[\delta_t(x_1,\dots,x_d) = (t^{\lambda_1} x_1,\dots, t^{\lambda_d} x_d),\]
and let $k_j \in \D'(G)$ bet the product of the distribution $\breve m$ by the smooth function $x \mapsto x_j$ (for $j=1,\dots,d$).
Let $\eta \in \D(G)$, and let $\eta_j$ be the product of $\eta$ by $x \mapsto x_j$. Then for every $f \in \D(G)$ and every $t > 0$ we have
\[\langle k_j, f \rangle = \langle k_j, (1 - \eta \circ \delta_{t}) f \rangle + t^{\lambda_j} \langle \breve m, f (\eta_j \circ \delta_{t}) \rangle.\]
Notice now that, for $t \geq 1$, the functions $(f \circ \delta_{t^{-1}}) \eta_j$ are n.b.f.\ of order $N$ (modulo a constant factor not depending on $t \geq 1$), thus, taking the limit for $t \to +\infty$, we get
\begin{equation}\label{eq:distributionlimit}
\langle k_j, f \rangle = \lim_{t \to +\infty} \langle k_j, (1 - \eta \circ \delta_{t}) f \rangle.
\end{equation}
By choosing an $\eta$ equal to $1$ in a neighborhood of the origin, we conclude that $k_j$ is determined by its restriction to $G \setminus \{e\}$.

Since the restriction of $\breve m$ off the origin is an $L^1_\loc$ function, this holds also for the $k_j$. Moreover, from \eqref{eq:size}, by a simple dyadic decomposition we get
\[\int_{0 < |x|_\delta \leq M} |x_j \, \breve m(x)| \,dx \leq C_{p,q,\beta} \, M^{\lambda_j} \|m\|_{M_\epsilon B_{p,q}^\beta} < \infty\]
for all $M > 0$. Therefore \eqref{eq:distributionlimit} gives
\[\langle k_j, f \rangle = \int_{G \setminus \{e\}} \overline{f(x)} \, x_j \, \breve m(x) \,dx\]
for all $f \in \D(G)$; in other words, $k_j$ is an $L^1_\loc$ function on the whole of $G$, which coincides with the pointwise product of $\breve m|_{G \setminus \{e\}}$ by $x \mapsto x_j$.

Fix a n.b.f.\ $\eta$ of order $N$ such that $\eta(0) = 1$. For every n.b.f.\ $\phi$ of order $1$, we may decompose
\[\phi(x) = \phi(0) \eta(x) + \sum_{j=1}^d x_j \phi_j(x)\]
for some $\phi_j \in \D(G)$ which are bounded by a fixed constant and supported in a fixed compact neighborhood of the origin. In particular
\[\langle \breve m, \phi \rangle = \phi(0) \langle \breve m, \eta \rangle + \sum_{j=1}^d \langle k_j , \phi_j \rangle,\]
thus from the previous estimates we deduce
\[|\langle \breve m, \phi \rangle| \leq C_{p,q,\beta} \|m\|_{M_\epsilon B_{p,q}^\beta},\]
but then, for an arbitrary $t > 0$,
\[|\langle \breve m, \phi \circ \delta_t \rangle| = |\langle (m \circ \epsilon_t)\breve{}\, , \phi \rangle| \leq C_{p,q,\beta} \|m \circ \epsilon_t\|_{M_\epsilon B_{p,q}^\beta} = C_{p,q,\beta} \|m\|_{M_\epsilon B_{p,q}^\beta}.\]

(c) For $0 < r_1 < r_2$, if we choose $k_1,k_2 \in \Z$ such that
\[2^{k_1} \leq r_1 < 2^{k_1+1}, \qquad 2^{k_2} \leq r_2 < 2^{k_2+1},\]
and if we fix a nonnegative $\phi \in \D(G)$ supported in $\{ x \tc |x|_\delta \leq 1\}$ and equal to $1$ on $\{ x \tc |x|_\delta \leq 1/2\}$, then
\[\left| \int_{r_1 < |x|_\delta < r_2} \breve m(x) \,dx \right| \leq C \sum_{j=1}^2 \left(\int_{B_{k_j}} |\breve m| \,d\mu + |\langle \breve m, \phi \circ \delta_{2^{-(k_j+1)}}\rangle|\right),\]
and the previously proved inequalities give the conclusion.
\end{proof}

If we require a greater regularity on the function $m$, we obtain a more precise characterization of the corresponding convolution kernel $\breve m$.

\begin{prp}\label{prp:kernelinfiniteorder}
Suppose that $m : \R^n \to \C$ satisfies a Mihlin-H\"ormander condition of infinite order, adapted to the dilations $\delta_t$. Then $\breve m|_{G \setminus \{e\}}$ is a smooth function and, for every $\delta_t$-homogeneous $D \in \Diff(G)$, we have
\begin{equation}\label{eq:pointwisecz}
|D \breve m(x)| \leq C_{m,D} |x|_\delta^{-(Q_\delta + w)},
\end{equation}
where $w$ is the homogeneity degree of $D$.
\end{prp}
\begin{proof}
As before, let $m_j$ ($j \in \Z$) be given by Lemma~\ref{lem:spectraldecomposition} and set
\[u_j = 2^{-Q_\delta j} \breve m_j \circ \delta_{2^{-j}}, \qquad B_k = \{ x \in G \tc 2^k \leq |x|_\delta \leq 2^{k+1}\}\]
for $j,k \in \Z$.

Let $D \in \Diff(G)$ be $\delta_t$-homogeneous of degree $w$. Let $\alpha \geq 0$, $\beta > a\alpha$; by Corollary~\ref{cor:pointwiseestimates} and \eqref{eq:moduluscomparison}, we have
\[\sup_{x \in G} (1+|x|_\delta)^\alpha |D \breve m_j(x)| \leq C_{D,\alpha,\beta} \|m\|_{M_\epsilon B_{\infty,\infty}^\beta},\]
which can be easily manipulated to obtain
\begin{equation}\label{eq:uestimate}
|D u_j(x)| \leq C_{D,\alpha,\beta} \|m\|_{M_\epsilon B_{\infty,\infty}^\beta} 2^{-j(Q_\delta + w - \alpha)} (2^{j} + |x|_\delta)^{-\alpha}.
\end{equation}
In particular, for $x \in B_k$, we have
\[|D u_j(x)| \leq C_{D,\alpha,\beta} \|m\|_{M_\epsilon B_{\infty,\infty}^\beta} |x|_\delta^{-(Q_\delta+w)} 2^{(k-j)(Q_\delta + w - \alpha)} (1 + 2^{j-k})^{-\alpha}.\]
For $j \geq k$, we choose $\alpha = Q_\delta + w$, so that
\[\sum_{j \geq k} |D u_j(x)| \leq C_{m,D} |x|_\delta^{-(Q_\delta + w)} \sum_{j \geq 0} 2^{-j(Q_\delta + w)},\]
whereas, for $j < k$, we choose $\alpha > Q_\delta + w$, so that
\[\sum_{j < k} |D u_j(x)| \leq C_{m,D} |x|_\delta^{-(Q_\delta + w)} \sum_{j > 0} 2^{j(Q_\delta + w-\alpha)}\]
(notice that $\beta$ can be taken arbitrarily large, since $m$ satisfies a Mihlin-H\"ormander condition of infinite order). In any case, the quantities on the right-hand side do not depend on $k$. Putting all together, we obtain that the series $\sum_{j \in \Z} D u_j$ converges uniformly on the compacta of $G \setminus \{e\}$, and that its sum, which coincides with $D \breve m$ by uniqueness of limits, satisfies \eqref{eq:pointwisecz}.
\end{proof}

Kernels satisfying this kind of estimates on a homogeneous Lie group are considered, e.g., in \S6.A of \cite{folland_hardy_1982}.

%% file: marcinkiewicz.tex
\section{Marcinkiewicz multipliers}\label{section:marcinkiewicz}

Let $G$ be a homogeneous Lie group, with automorphic dilations $\delta_t$ and homogeneous dimension $Q_\delta$. For $w \in L^1(G)$, we define the \emph{maximal operator} associated to $w$:
\[M_w \phi(x) = \sup_{t > 0} |\phi * (t^{-Q_\delta} w \circ \delta_{t^{-1}})(x)| = \sup_{t > 0} \left| \int_G \phi(x \,\delta_t(y)^{-1}) \, w(y) \,dy \right|.\]
We say that the function $w$ is \emph{M-admissible}\index{function!M-admissible} if the associated maximal operator $M_w$ is bounded on $L^p(G)$ for $1 < p < \infty$.

In terms of maximal operators, we formulate the following hypothesis about the homogeneous group $G$ and a chosen homogeneous weighted subcoercive system $L_1,\dots,L_n$ on it:
\begin{quote}
{\bf hypothesis \HPK{s}{d}:} for every $\beta > s$ there exist
\begin{itemize}
\item a Borel function $u_\beta$ on $G$ with $u_\beta = u_\beta^*$ and $u_\beta \geq c \langle \cdot \rangle_G^{-\theta}$ for some $c,\theta > 0$,
\item a positive regular Borel measure $\sigma_\beta$ on $\R^n$, which is locally $d$-bounded on $\R^n \setminus \{0\}$,
\item a non-negative real number $\gamma_\beta < 2\beta$,
\end{itemize}
such that
\begin{itemize}
\item the function $\langle \cdot \rangle_G^{-\gamma_\beta} u_\beta$ is M-admissible, and
\item for every compact $K \subseteq \R^n \setminus \{0\}$ and every $m \in \D(\R^n)$ with $\supp m \subseteq K$, we have
\[\|\breve m\|_{L^2(G, u_\beta^{-1}(x) \,dx)} \leq C_{K,\beta} \|m\|_{L^2(\sigma_\beta)}.\]
\end{itemize}
\end{quote}

The next proposition shows how this quite technical hypothesis can be verified by the techniques and results of the previous chapters.

\begin{prp}\label{prp:HPKsatisfaction}
Let $G$ be a homogeneous group, with degree of polynomial growth $Q_G$, and let $L_1,\dots,L_n$ be a homogeneous weighted subcoercive system on $G$.
\begin{itemize}
\item[(i)] The hypothesis \HPK{Q_G/2}{1} holds. More generally, if the Plancherel measure $\sigma$ is locally $d$-bounded on $\R^n \setminus \{0\}$, then \HPK{Q/2}{d} holds.
\item[(ii)] If $G$ is $h$-capacious, then \HPK{(Q_G-h)/2}{1} holds.
\end{itemize}
\end{prp}
\begin{proof}
(i) For $\beta > Q_G/2$, we choose $u_\beta \equiv 1$, $\sigma_\beta = \sigma$. By Proposition~\ref{prp:plancherelhomogeneous}, $\sigma$ is $\delta_t$-homogeneous, so that it is locally $1$-bounded on $\R^n \setminus \{0\}$. Therefore, by Theorem~\ref{thm:plancherel}, in order to conclude, it is sufficient to show that, for $\gamma_\beta \in \left]Q_G,2\beta\right[$, the function $w_\beta = \langle \cdot \rangle_G^{-\gamma_\beta}$ is M-admissible.

Set
\[\lie{g}_1 = \lie{g}, \qquad \lie{g}_{k+1} = [\lie{g},\lie{g}_k].\]
The ideals $\lie{g}_k$ are characteristic, i.e., they are invariant by any automorphism of $\lie{g}$. In particular, they are $\delta_t$-invariant, so that we can find $\delta_t$-invariant subspaces $V_k$ such that $\lie{g}_k = V_k \oplus \lie{g}_{k+1}$. Let $\tilde\delta_t : \lie{g} \to \lie{g}$ be the linear map defined by
\[\tilde\delta_t(x) = t^k x \qquad\text{for $x \in V_k$.}\]
The $\tilde\delta_t$ are dilations of the vector space $\lie{g}$, which commute with the $\delta_t$, but in general the $\tilde\delta_t$ are not automorphism of the Lie algebra $\lie{g}$. However, by Proposition~\ref{prp:nilpotentgrowth}, if $|\cdot|_{\tilde\delta}$ is a homogeneous norm with respect to the dilations $\tilde\delta_t$, then
\[c^{-1} (1+|\cdot|_{\tilde\delta}) \leq \langle \cdot \rangle_G \leq c (1+|\cdot|_{\tilde\delta})\]
for some $c > 0$. We then have
\[\begin{split}
\left| \int_G \phi(x \, \delta_t(y)^{-1} ) \, w_\beta(y) \,dy \right| &= \left|\int_{|y|_{\tilde\delta} \leq 1} + \sum_{h \geq 1} \int_{2^{h-1} < |y|_{\tilde\delta} \leq 2^h}\right|\\
&\leq C_{\gamma_\beta} \sum_{h \geq 0} 2^{-h\gamma_\beta} \int_{|y|_{\tilde\delta} \leq 2^h} |\phi(x \,\delta_t(y)^{-1})| \,dy \\
&\leq C_{\gamma_\beta} \sum_{h \geq 0} 2^{-h(\gamma_\beta-Q_G)} \int_{|y|_{\tilde\delta} \leq 1} |\phi(x \,\delta_t(\tilde\delta_{2^h}(y))^{-1})| \,dy.
\end{split}\]
Therefore, since $\gamma_\beta > Q_G$, if $M_\strong$ is the strong maximal function \eqref{eq:strongmaximal} on $G$ associated to a basis of simultaneous eigenvectors of the $\delta_t$ and the $\tilde\delta_t$, we conclude that
\[M_{w_\beta} \phi \leq C_{\gamma_\beta} M_{\strong} \phi,\]
which gives the conclusion by Theorem~\ref{thm:strongmaximal}.

(ii) Let $\lie{z}$ be the center of $\lie{z}$, and $P : \lie{g} \to \lie{g}/\lie{z}$ be the canonical projection. Let $\omega_1,\dots,\omega_h \in (\lie{g}/\lie{z})^*$ and $z_1,\dots,z_h \in \lie{z}$ be given by the definition of $h$-capacious. By Lemma~\ref{lem:homogeneousdualbasis}, there exists a homogeneous basis $v_1,\dots,v_k$ of $\lie{g}$ compatible with the descending central series such that, if $\hat v_1,\dots,\hat v_k$ is the dual basis, then $\omega_j \circ P = \hat v_j$ for $j=1,\dots,h$. Moreover, if we set
\[\kappa_j = \max \{r : v_j \in \lie{g}_{[r]}\},\]
then $\kappa_j = 1$ for $j=1,\dots,h$ and
\begin{equation}\label{eq:nilpotentgrowth}
Q_G = \sum_{j=1}^k \kappa_j, \qquad \langle x \rangle_G \sim 1 + \sum_{j=1}^k |\hat v_j(x)|^{1/\kappa_j}
\end{equation}
by Proposition~\ref{prp:nilpotentgrowth}.

Let $\daleth_{\vec{t}}$ be the $k$-parameter family of dilations on $\lie{g}$ given by
\[\daleth_{\vec{t}}(v_j) = t_j v_j.\]
Clearly the $\daleth_{\vec{t}}$ are in general not automorphisms, but the automorphic dilations $\delta_t$ can be obtained as a particular case:
\[\delta_t = \daleth_{(t^{b_1},\dots,t^{b_k})},\]
where $b_j$ is the $\delta_t$-homogeneous degree of $v_j$.

If $\beta > (Q_G-h)/2$, then
\[2\beta > Q_G - h = \sum_{j=h+1}^k \kappa_j,\]
so that we can find $\eta_{\beta,1},\dots,\eta_{\beta,h} \in \left[0,1\right[$ and $\gamma_{\beta,1},\dots,\gamma_{\beta,k} > 0$ such that
\[2\beta > \gamma_\beta = \sum_{j=1}^k \gamma_{\beta,j}, \qquad \gamma_{\beta_j} > \begin{cases}
1 - \eta_{\beta,j} &\text{for $j=1,\dots,h$,}\\
\kappa_j &\text{for $j=h+1,\dots,k$.}
\end{cases}\]

Let now $\sigma'$ be the Plancherel measure on $\R^n \times \lie{z}^*$ associated to the system $L_1,\dots,L_n$ extended with the central derivatives, as in \S\ref{section:metivier}, and let $\sigma_\beta$ be the push-forward of the measure
\[\prod_{j=1}^h (1+|\tau(z_j)|^{-\eta_{\beta,j}}) \,d\sigma'(\lambda,\tau)\]
via the canonical projection on the first factor of $\R^n \times \lie{z}^*$. By Proposition~\ref{prp:homogeneouspushforward}, since $\eta_{\beta,1},\dots,\eta_{\beta,h} < 1$, the measure $\sigma_{\beta}$ is a regular Borel measure on $\R^n$; moreover, since the $z_j$ are $\delta_t$-homogeneous, $\sigma_{\beta}$ is the sum of $\epsilon_t$-homogeneous regular Borel measures of different degrees (where $\epsilon_t$ are the dilations associated to the system $L_1,\dots,L_n$), and consequently $\sigma_\beta$ is locally $1$-bounded on $\R^n \setminus \{0\}$. Finally, if we set
\[u_\beta(x) = \prod_{j=1}^h (1+|\omega_j(P(x))|)^{-\eta_j} = \prod_{j=1}^h (1+|\hat v_j(x)|)^{-\eta_j},\]
then $u_{\beta} = u_{\beta}^*$ and, by \eqref{eq:nilpotentgrowth}, $u_\beta^{-1}$ is dominated by some power of $\langle \cdot \rangle_G$; moreover, by Corollary~\ref{cor:partialweight}, for every compact $K \subseteq \R^n \setminus \{0\}$ and every $m \in \D(\R^n)$ with $\supp m \subseteq K$, we have
\[\|\breve m\|_{L^2(G, u_\beta^{-1}(x) \,dx)} \leq C_{K,\beta} \|m\|_{L^2(\sigma_\beta)}.\]

In order to conclude, we have to show that $w_\beta = \langle \cdot \rangle_G^{-\gamma_\beta} u_\beta$ is M-admissible. In fact, again by \eqref{eq:nilpotentgrowth},
\[w_\beta(x) \leq C_\beta \prod_{j=1}^h (1+|\hat v_j(x)|)^{-(\gamma_{\beta,j} + \eta_{\beta,j})} \prod_{j=h+1}^k (1+|\hat v_j(x)|)^{-\gamma_{\beta,j}/\kappa_j},\]
and the exponents $\gamma_{\beta,j} + \eta_{\beta,j}$, $\gamma_{\beta,j}/\kappa_j$ are all greater than $1$ by construction. The conclusion then follows as in part (i), but with a multi-variate decomposition, by Theorem~\ref{thm:strongmaximal} applied to the multi-parameter dilations $\daleth_{\vec t}$.
\end{proof}

Suppose now that, for $l=1,\dots,\ell$, $G_l$ is a homogeneous Lie group, with dilations $(\delta_{l,t})_{t > 0}$, and that $L_{l,1},\dots,L_{l,n_l}$ is a homogeneous weighted subcoercive system on $G_l$. Let
\[G^\times = G_1 \times \dots \times G_\ell,\]
and let $L^\times_{l,j}$ be the left-invariant differential operator on $G^\times$ along the $l$-th factor, corresponding to $L_{l,j}$ on $G_l$, for $l=1,\dots,\ell$, $j=1,\dots,n_l$. By the results of \S\ref{section:directproducts}, we know than that
\[L_{1,1}^\times,\dots,L_{1,n_1}^\times,\dots,L_{\ell,1}^\times,\dots,L_{\ell,n_\ell}^\times\]
is a homogeneous weighted subcoercive system on $G^\times$.

We then show how the hypotheses on the factor groups $G_l$ can be put together in order to obtain weighted estimates on the product group $G^\times$.

\begin{prp}\label{prp:productestimates}
Suppose that, for $l=1,\dots,\ell$, the homogeneous group $G_l$, with the system $L_{l,1},\dots,L_{l,n_l}$, satisfies \HPK{s_l}{d_l}. For $p,q \in [1,\infty]$, if
\[\vec{\beta} > \vec{s} + \frac{\vec{n}}{p} - \frac{\vec{d}}{\max\{2,p\}},\]
where $\vec{s} = (s_1,\dots,s_\ell)$, $\vec{d} = (d_1,\dots,d_\ell)$, then there exists
\[w_{\vec{\beta}} = w_{\vec{\beta},1} \otimes \dots \otimes w_{\vec{\beta},\ell} \in L^1(G^\times),\]
with $w_{\vec{\beta}} > 0$, $w_{\vec{\beta}}^* = w_{\vec{\beta}}$, such that $w_{\vec{\beta},l}$ is M-admissible on the group $G_l$ for $l=1,\dots,\ell$, and moreover, for every compact
\[K = \prod_{l=1}^\ell K_l \subseteq \prod_{l=1}^\ell (\R^{n_l} \setminus \{0\}),\]
and for every $m \in S^{\vec{\beta}}_{p,q}B(\R^{\vec{n}})$ with $\supp m \subseteq K$, we have
\[\|\breve m\|_{L^2(G^\times, w_{\vec{\beta}}^{-1}(x) \,dx)} \leq C_{K,\vec{\beta},p,q} \|m\|_{S^{\vec{\beta}}_{p,q}B(\R^{\vec{n}})}.\]
\end{prp}
\begin{proof}
Take $\vec{\alpha}$ such that
\[\vec{\alpha} > \vec{s}, \qquad \vec{\beta} > \vec{\alpha} + \frac{\vec{n}}{p} - \frac{\vec{d}}{\max\{2,p\}}.\]
For $l=1,\dots,\ell$, since $\alpha_l > s_l$, by \HPK{s_l}{d_l} we can find a function $u_{\vec{\alpha},l} = u_{\vec{\alpha},l}^* > 0$ on $G_l$ such that $u_{\vec{\alpha},l} \geq c_l \langle \cdot \rangle_{G_l}^{-\theta_l}$ for some $c_l,\theta_l > 0$, a positive regular Borel measure $\sigma_{\vec{\alpha},l}$ on $\R^{n_l}$ locally $d_l$-bounded on $\R^{n_l} \setminus \{0\}$, and a positive real number $\gamma_{\vec{\alpha},l} < 2\alpha_l$ such that the function $w_{\vec{\beta},l} = \langle \cdot \rangle_{G_l}^{-\gamma_{\vec{\alpha},l}} u_{\vec{\alpha},l}$ is M-admissible on $G_l$ and
\begin{equation}\label{eq:l2weightmeasure}
\|\breve m_l\|_{L^2(G_l,u_{\vec{\alpha},l}^{-1}(x_l) \,dx_l)} \leq C_{K_l,\alpha_l} \|m\|_{L^2(\sigma_{\vec{\alpha},l})}
\end{equation}
for every compact $K_l \subseteq \R^{n_l} \setminus \{0\}$ and every $m_l \in \D(\R^{n_l})$ with $\supp m_l \subseteq K_l$.

Set
\[u_{\vec{\alpha}} = u_{\vec{\alpha},1} \otimes \dots \otimes u_{\vec{\alpha},\ell}, \qquad \sigma_{\vec{\alpha}} = \sigma_{\vec{\alpha},1} \times \dots \times \sigma_{\vec{\alpha},\ell}.\]
By ``taking the Hilbert tensor product'' (see \S\ref{subsection:directproducts}) of the inequalities \eqref{eq:l2weightmeasure}, from Proposition~\ref{prp:productfunctional} we deduce that
\[\|\breve m\|_{L^2(G^\times,u_{\vec{\alpha}}^{-1}(x) \,dx)} \leq C_{K,\vec{\alpha}} \|m\|_{L^2(\sigma_{\vec{\alpha}})}\]
for every compact $K = \prod_{l=1}^\ell K_l \subseteq \prod_{l=1}^\ell (\R^{n_\ell} \setminus \{0\})$ and every $m \in \D(\R^{\vec{n}})$ with $\supp m \subseteq K$.

Notice now that, again by taking Hilbert tensor products, Corollary~\ref{cor:triebeltrace} and \eqref{eq:sobolevtensorproduct} give
\[\|m\|_{L^2(\sigma_{\vec{\alpha}})} \leq C_{K,\vec{\alpha},\vec{\gamma}} \|m\|_{S^{\vec{\eta}}_{2,2}B(\R^{\vec{n}})}\]
for $\vec{\eta} > (\vec{n} - \vec{d})/2$, whereas, by Proposition~\ref{prp:multibesovinclusions},
\[\|m\|_{L^2(\sigma_{\vec{\alpha}})} \leq C_{K,\vec{\alpha}} \|m\|_{S^{0}_{\infty,1}B(\R^{\vec{n}})},\]
so that, by embeddings and interpolation (Propositions~\ref{prp:multibesovembeddings} and \ref{prp:multibesovinterpolation}), we deduce that
\[\|m\|_{L^2(\sigma_{\vec{\alpha}})} \leq C_{K,\vec{\alpha},\vec{\gamma},p} \|m\|_{S^{\vec{\eta}}_{p,p}B(\R^{\vec{n}})}\]
for $\vec{\eta} > \vec{n}/p - \vec{d}/\max\{2,p\}$.

In particular, we have
\[\|\breve m\|_{L^2(G^\times,u_{\vec{\alpha}}^{-1}(x) \,dx)} \leq C_{K,\vec{\alpha},\vec{\eta},p} \|m\|_{S^{\vec{\eta}}_{p,p}B(\R^{\vec{n}})}\]
for $\vec{\eta} > \vec{n}/p - \vec{d}/\max\{2,p\}$. On the other hand, for $l=1,\dots,\ell$, by Theorem~\ref{thm:l2estimates}, we have
\[\|\breve m_l\|_{L^2(G_l,\langle x_l \rangle_{G_l}^{\gamma_l} u_{\vec{\alpha},l}^{-1}(x_l) \,dx_l)} \leq C_{K_l,\vec{\alpha},\gamma_l,\eta_l} \|m_l\|_{B^{\eta_l}_{2,2}(\R^{n_l})}\]
for $\eta_l > \gamma_l/2 + \theta_l/2 + n_l/2$, so that, by Hilbert tensor products and embeddings,
\[\|\breve m\|_{L^2(G^\times,\langle x_1 \rangle_{G_1}^{\gamma_1} \cdots \langle x_{\ell} \rangle_{G_\ell}^{\gamma_\ell} u_{\vec{\alpha}}^{-1}(x) \,dx)} \leq C_{K,\vec{\alpha},\vec{\gamma},\vec{\eta},p} \|m\|_{S^{\vec{\eta}}_{p,p}B(\R^{\vec{n}})}\]
for $\vec{\eta} > \vec{\gamma}/2 + \vec{\theta}/2 + \vec{n}$. By interpolation, we obtain that
\[\|\breve m\|_{L^2(G^\times,\langle x_1 \rangle_{G_1}^{\gamma_1} \cdots \langle x_{\ell} \rangle_{G_\ell}^{\gamma_\ell} u_{\vec{\alpha}}^{-1}(x) \,dx)} \leq C_{K,\vec{\alpha},\vec{\gamma},\vec{\eta},p} \|m\|_{S^{\vec{\eta}}_{p,p}B(\R^{\vec{n}})}\]
for $\vec{\eta} > \vec{\gamma}/2 + \vec{n}/p - \vec{d}/\max\{2,p\}$.

In particular, if we take $\vec{\gamma} = (\gamma_{\vec{\alpha},1},\dots,\gamma_{\vec{\alpha},\ell})$, $\vec{\eta} = \vec{\alpha} + \vec{n}/p - \vec{d}/\max\{2,p\}$ and set $w_{\vec{\beta}} = w_{\vec{\beta},1} \otimes \dots \otimes w_{\vec{\beta},\ell}$, we get
\[\|\breve m\|_{L^2(G^\times,w_{\vec{\beta}}^{-1}(x) \,dx)} \leq C_{K,\vec{\beta},p} \|m\|_{S^{\vec{\eta}}_{p,p}B(\R^{\vec{n}})},\]
for every compact $K = \prod_{l=1}^\ell K_l \subseteq \prod_{l=1}^\ell (\R^{n_\ell} \setminus \{0\})$ and every $m \in \D(\R^{\vec{n}})$ with $\supp m \subseteq K$. The conclusion then follows by approximations and embeddings (Propositions~\ref{prp:multibesovapproximation} and \ref{prp:multibesovembeddings}).
\end{proof}

Notice that, in the particular case $\ell = 1$, the previous proposition, together with H\"older's inequality and Proposition~\ref{prp:interpolatedweightedestimates}, gives the following

\begin{cor}\label{cor:hypothesescomparison}
If a homogeneous weighted subcoercive system $L_1,\dots,L_n$ on a homogeneous Lie group $G$ satisfies the hypothesis \HPK{s}{d}, then, for $p \in [1,\infty]$, it satisfies also \HP{p}{s + n/p - d/\max\{2,p\}}. In particular, $s \geq d/2$.
\end{cor}

The weighted estimate on $G^\times$ given by Proposition~\ref{prp:productestimates} will be the starting point for the following multi-variate multiplier results. In fact, we are going to consider a setting which is more general than the product group $G^\times$.

Let $G$ be a connected Lie group, endowed with Lie group homomorphisms
\[\upsilon_l : G_l \to G \qquad\text{for $l=1,\dots,\ell$.}\]
Then, for $l=1,\dots,\ell$, the operators $L_{l,1},\dots,L_{l,n_l}$ correspond (via the derivative $\upsilon_l'$ of the homomorphism) to operators $L^\flat_{l,1}, \dots L^\flat_{l,n_l} \in \Diff(G)$, which are essentially self-adjoint by Corollary~\ref{cor:commutativealgebra}. Since we want to give a meaning to joint functions of these operators on $G$, we suppose in the following that
\[L^\flat_{1,1}, \dots L^\flat_{1,n_1},\dots,L^\flat_{\ell,1}, \dots, L^\flat_{\ell,n_\ell}\]
commute strongly, i.e., they admit a joint spectral resolution on $L^2(G)$.

In order to obtain multiplier results on $G$, we would like to ``transfer'' to $G$ the estimates obtained on the product group $G^\times$. However, we cannot apply directly the classical transference results, since the map
\[\upsilon^\times : G^\times \ni (x_1,\dots,x_n) \mapsto \upsilon_1(x_1) \cdots \upsilon_\ell(x_\ell) \in G\]
in general is not a group homomorphism --- because the elements of $\upsilon_l(G_l)$ are not supposed to commute in $G$ with the elements of $\upsilon_{l'}(G_{l'})$ for $l \neq l'$ --- and consequently it does not yield an action of $G^\times$ on $L^p(G)$ by translations. Nevertheless, under the sole assumption of (strong) commutativity of the differential operators $L^\flat_{l,j}$ on $G$, we are able to express the operator $m(L^\flat)$ on $G$ by a sort of convolution with the kernel $\Kern_{L^\times} m$ of the operator $m(L^\times)$ on $G^\times$.

\begin{prp}\label{prp:convolutionproduct}
For every $m \in \D(\R^{\vec{n}})$, we have
\[m(L^\flat) \phi(x) = \int_{G^\times} \phi(x \,\upsilon^\times(y)^{-1} ) \, \Kern_{L^\times} m(y) \,dy\]
for all $\phi \in L^2 \cap C_0(G)$.
\end{prp}
\begin{proof}
If $m \in \D(\R^{\vec{n}}) \cong \D(\R^{n_1}) \ptimes \cdots \ptimes \D(\R^{n_\ell})$, then we can decompose
\[m = \sum_{k \in \N} g_{k,1} \otimes \cdots \otimes g_{k,\ell},\]
where $g_{k,l} \in \D(\R^{n_l})$ for $k \in \N$, $l = 1,\dots,\ell$, and the convergence is in $\D(\R^{\vec{n}})$. In particular, by applying Corollary~\ref{cor:polyl1estimates} and Proposition~\ref{prp:productfunctional} to the group $G^\times$, we obtain that
\[\Kern_{L^\times} m = \sum_{k \in \N} \Kern_{L_1} g_{k,1} \otimes \cdots \otimes \Kern_{L_\ell} g_{k,\ell}\]
in $L^1(G^\times)$.

On the other hand, for all $\phi \in L^2 \cap C_0(G)$, by Proposition~\ref{prp:Jl1} we have
\[g_{k,l}(L^\flat_l) \phi(x) = \int_{G_l} \phi(x \,\upsilon_l(y_l)^{-1}) 	\,\Kern_{L_l} g_{k,l}(y_l) \,dy_l,\]
and in particular (being $\Kern_{L_l} g_{k,l} \in L^1(G_l)$) also $g_{k,l}(L^\flat_l)\phi \in L^2 \cap C_0(G)$, so that, by iterating,
\begin{multline*}
(g_{k,1} \otimes \dots \otimes g_{k,\ell})(L^\flat) \phi(x) = g_{k,\ell}(L^\flat_\ell) \cdots g_{k,1}(L^\flat_1) \phi(x) \\
= \int_{G^\times} \phi(x \,\upsilon_\ell(y_\ell)^{-1} \cdots \upsilon_1(y_1)^{-1}) 	\,\Kern_{L_1} g_{k,1}(y_1) \cdots \Kern_{L_\ell} g_{k,\ell}(y_\ell) \,dy.
\end{multline*}
Summing over $k \in \N$, the left-hand side converges in $L^2(G)$ to $m(L^\flat)\phi$, whereas (since $y \mapsto \phi(x \,\upsilon^\times(y)^{-1})$ is bounded) the right-hand side converges pointwise to
\[\int_{G^\times} \phi(x \,\upsilon^\times(y)^{-1}) \, \Kern_{L^\times} m(y) \,dy\]
and we get the conclusion.
\end{proof}

\begin{cor}\label{cor:convolutionproduct}
Under the hypotheses of Proposition~\ref{prp:productestimates}, if $p,q \in [1,\infty]$ and
\[\vec{\beta} > \vec{s} + \frac{\vec{n}}{p} - \frac{\vec{d}}{\max\{2,p\}},\]
then the conclusion of Proposition~\ref{prp:convolutionproduct} holds also for every $m \in S^{\vec{\beta}}_{p,q}B(\R^{\vec{n}})$ with compact support $\supp m \subseteq \prod_{l=1}^\ell (\R^{n_l} \setminus \{0\})$.
\end{cor}
\begin{proof}
Choose $\beta'$ such that
\[\vec{\beta} > \vec{\beta}' > \vec{s} + \frac{\vec{n}}{p} - \frac{\vec{d}}{\max\{2,p\}}.\]
By Proposition~\ref{prp:multibesovapproximation}, we can find a compact $K = \prod_{l=1}^\ell K_l \subseteq \prod_{l=1}^\ell (\R^{n_l} \setminus \{0\})$ and a sequence $m_k \in \D(\R^{\vec{n}})$ with $\supp m_k \subseteq K$ such that $m_k \to m$ in $S^{\vec{\beta}'}_{p,q}(\R^{\vec{n}})$. By Proposition~\ref{prp:productestimates} and H\"older's inequality, we then have $\Kern_{L^\times} m_k \to \Kern_{L^\times} m$ in $L^1(G^\times)$; moreover, by Corollary~\ref{cor:hypothesescomparison}, $\beta'_l > n_l/p$ for $l=1,\dots,\ell$, so that, by Propositions~\ref{prp:multibesovembeddings} and \ref{prp:multibesovinclusions}, $m_k \to m$ uniformly. Therefore, the conclusion follows by applying Proposition~\ref{prp:convolutionproduct} to the functions $m_k$ and passing to the limit.
\end{proof}

For $l=1,\dots,\ell$, let $\epsilon_{l,t}$ be the dilations on $\R^{n_l}$ associated to the weighted subcoercive system $L_{l,1},\dots,L_{l,n_l}$, and fix a $\epsilon_l$-homogeneous norm $|\cdot|_{\epsilon_l}$ on $\R^{n_l}$, smooth off the origin. Choose a non-negative $\xi \in \D(\R)$ with $\supp \xi \subseteq [1/2,2]$ and such that, if $\xi_k(t) = \xi(2^{-k} t)$, then
\begin{equation}\label{eq:sumofsquares}
\sum_{k \in \Z} \xi_k^2(t) = 1 \qquad\text{for $t > 0$,}
\end{equation}
and set, for $l=1,\dots,\ell$ and $k \in \Z$,
\[\chi_{l,k}(\lambda) = \xi(|\epsilon_{l,2^{-k}}(\lambda)|_{\epsilon_l}) = \xi_k(|\lambda|_{\epsilon_l}) \qquad\text{for $\lambda \in \R^{n_l}$.}\]
Finally, for $\vec{k} = (k_1,\dots,k_\ell) \in \Z^\ell$, let
\[\chi_{\vec{k}} = \chi_{1,k_1} \otimes \cdots \otimes \chi_{\ell,k_\ell}, \qquad T_{\vec{k}} = \chi_{\vec{k}}(L^\flat) = \chi_{1,k_1}(L^\flat_1) \cdots \chi_{\ell,k_\ell}(L^\flat_\ell).\]

\begin{lem}\label{lem:littlewoodpaley}
For $1 < p < \infty$ and for all $\phi \in L^2 \cap L^p(G)$, 
\[c_p \|\phi\|_p \leq \left\| \left( \sum_{\vec{k} \in \Z^\ell} |T_{\vec{k}} \phi|^2 \right)^{1/2} \right\|_p \leq C_p \|\phi\|_p.\]
\end{lem}
\begin{proof}
For every $s \in \N$, $(\varepsilon^l_k)_{k \in \Z} \in \{-1,0,1\}^\Z$ and $N \in \N$, it is easy to prove that
\[\left\|\sum_{|k| \leq N} \varepsilon^l_k \xi_k\right\|_{M_* C^s} \leq C_s,\]
where $C_s > 0$ does not depend on $(\varepsilon^l_k)_k$ or $N$; therefore, by Proposition~\ref{prp:uniformhomogeneouschange}, for $l=1,\dots,\ell$, also
\[\left\|\sum_{|k| \leq N} \varepsilon^l_k \chi_{l,k}\right\|_{M_{\epsilon_l} C^s} \leq C_{l,s},\]
where $C_{l,s} > 0$ does not depend on $(\varepsilon^l_k)_k$ or $N$.

By Theorem~\ref{thm:mihlinhoermander} applied to the group $G_l$, and by transference to the group $G$ (see Theorem~\ref{thm:transference} and Proposition~\ref{prp:Jl1}), we then have
\[\left\|\sum_{|k| \leq N} \varepsilon^l_k \chi_{l,k}(L^\flat_l) \right\|_{p \to p} \leq \left\|\sum_{|k| \leq N} \varepsilon^l_k \chi_k(L_l) \right\|_{p \to p} \leq C_{l,p}\]
for $1 < p < \infty$, $l = 1,\dots,\ell$, where $C_{l,p} > 0$ does not depend on $(\varepsilon^l_k)_k$ or $N$, and consequently also
\[\left\|\sum_{|k_1|,\dots,|k_\ell| \leq N} \varepsilon^1_{k_1} \cdots \varepsilon^\ell_{k_\ell} T_{\vec{k}} \right\|_{p \to p} \leq C_{1,p} \cdots C_{\ell,p}.\]
Moreover, by \eqref{eq:sumofsquares} and the properties of the spectral integral, $\sum_{\vec{k} \in \Z^\ell} T_{\vec{k}}^2$ converges strongly to the identity of $L^2(G)$. The conclusion follows then immediately by Proposition~\ref{prp:khinchin}.
\end{proof}

In the following, we will consider Marcinkiewicz conditions on $\R^{\vec{n}}$ adapted to the system
\[\gimel_{\vec{t}} = \epsilon_{1,t_1} \times \dots \times \epsilon_{\ell,t_\ell}\]
of multi-variate dilations.

\begin{thm}\label{thm:marcinkiewicz}
Suppose that, for $l=1,\dots,\ell$, the homogeneous group $G_l$, with the system $L_{l,1},\dots,L_{l,n_l}$, satisfies the hypothesis \HPK{s_l}{d_l}. If $p,q \in [1,\infty]$ and
\[\vec{\beta} > \vec{s} + \frac{\vec{n}}{p} - \frac{\vec{d}}{\max\{2,p\}},\]
then, for every bounded measurable function $m$ on $\R^{\vec{n}}$ with $\|m\|_{M_\gimel S^{\vec{\beta}}_{p,q}B} < \infty$, the operator $m(L^\flat)$ is bounded on $L^r(G)$ for $1 < r < \infty$ and
\[\|m(L^\flat)\|_{r \to r} \leq C_{\vec{\beta},p,q,r} \|m\|_{M_\gimel S^{\vec{\beta}}_{p,q}B}.\]
\end{thm}
\begin{proof}
Choose a non-negative $\zeta \in \D(\R)$ with $\supp \zeta \subseteq [1/4,4]$ and such that $\zeta \equiv 1$ on $[1/2,2]$. Set, for $l=1,\dots,\ell$,
\[\eta_l(\lambda) = \zeta(|\lambda|_{\epsilon_l}) \qquad\text{for $\lambda \in \R^{n_l}$,}\]
and let
\[\eta = \eta_1 \otimes \dots \otimes \eta_\ell.\]
If we set
\[m_{\vec{k}} = (m \circ \gimel_{(2^{k_1},\dots,2^{k_\ell})}) \eta, \qquad f_{\vec{k}} = m_{\vec{k}} \circ \gimel_{(2^{-k_1},\dots,2^{-k_\ell})}\]
for $\vec k \in \Z^{\ell}$, then we have $\chi_{\vec{k}} m = f_{\vec{k}} \chi_{\vec{k}}$, so that
\[T_{\vec{k}} m(L^\flat) = f_{\vec{k}}(L^\flat) T_{\vec{k}}.\]

Let $w_{\vec{\beta}} = w_{\vec{\beta},1} \otimes \dots \otimes w_{\vec{\beta},\ell} \in L^1(G^\times)$ be given by Proposition~\ref{prp:productestimates}. Set
\[w_{\vec{\beta},l,k} = 2^{-kQ_{\delta_l}} w_{\vec{\beta},l} \circ \delta_{l,2^{-k}}\]
for $k \in \Z$, $l =1,\dots,\ell$, and let
\[w_{\vec{\beta},\vec{k}} = w_{\vec{\beta},1,k_1} \otimes \dots \otimes w_{\vec{\beta},\ell,k_\ell}\]
for $\vec{k} \in \Z^\ell$. For $l=1,\dots,\ell$, if $\pi_l$ denotes the unitary representation of $G_l$ on $L^2(G)$ induced by the homomorphism $\upsilon_l$, since $w_{\vec{\beta},l}$ is M-admissible on $G_l$, then the maximal function $M_{\vec{\beta},l}$ on $G$ defined by
\[M_{\vec{\beta},l} \phi(x) = \sup_{k \in \Z} |\pi_l(w_{\vec{\beta},l,k}) \phi(x)|\]
is bounded on $L^r(G)$ for $1 < r < \infty$, by transference (see Theorem~\ref{thm:transferencemaximal}).

If $\phi \in L^2 \cap C_0(G)$, then we have, by Corollary~\ref{cor:convolutionproduct} and H\"older's inequality,
\begin{multline*}
|f_{\vec{k}}(L^\flat) T_{\vec{k}} \phi(x)|^2 \leq \left(\int_{G^\times} |T_{\vec{k}} \phi(x \,\upsilon^\times(y)^{-1} )| |\Kern_{L^\times} f_{\vec{k}}(y)| \,dy\right)^2 \\
\leq \int_{G^\times} |T_{\vec{k}} \phi(x \,\upsilon^\times(y)^{-1} )|^2 w_{\vec{\beta},\vec{k}}(y) \,dy  \int_{G^\times} |\Kern_{L^\times} m_{\vec{k}}(y)|^2 w_{\vec{\beta}}^{-1}(y) \,dy\\
\leq C_{\vec{\beta},p,q} \|m_{\vec{k}}\|_{S^{\vec{\beta}}_{p,q}B(\R^{\vec{n}})}^2 \pi_1(w_{\vec{\beta},1,k_1}) \cdots \pi_\ell(w_{\vec{\beta},\ell,k_\ell}) (|T_{\vec{k}} \phi|^2)
\end{multline*}
thus
\begin{multline*}
\left\| \left( \sum_{\vec{k} \in \Z^\ell} |T_{\vec{k}} m(L^\flat)\phi|^2 \right)^{1/2} \right\|_r \\
\leq C_{\vec{\beta},p,q} \|m\|_{M_\gimel S^{\vec{\beta}}_{p,q}B} \left\| \sum_{\vec{k} \in \Z^\ell} \pi_1(w_{\vec{\beta},1,k_1}) \cdots \pi_\ell(w_{\vec{\beta},\ell,k_\ell}) (|T_{\vec{k}} \phi|^2) \right\|^{1/2}_{r/2}
\end{multline*}
for $2 \leq r < \infty$.

On the other hand, since $w_{\vec{\beta}} = w_{\vec{\beta}}^*$, for every $\psi \in L^{(r/2)'}(G)$ we have
\[\begin{split}
&\left|\int_G \left(\sum_{\vec{k} \in \Z^\ell} \pi_1(w_{\vec{\beta},1,k_1}) \cdots \pi_\ell(w_{\vec{\beta},\ell,k_\ell}) (|T_{\vec{k}} \phi|^2) \right) \psi \,d\mu_G\right| \\
&\phantom{MMMM}\leq \sum_{\vec{k} \in \Z^\ell} \int_G \left( \pi_1(w_{\vec{\beta},1,k_1}) \cdots \pi_\ell(w_{\vec{\beta},\ell,k_\ell}) (|T_{\vec{k}} \phi|^2) \right) |\psi| \,d\mu_G\\
&\phantom{MMMM}= \sum_{\vec{k} \in \Z^\ell} \int_G |T_{\vec{k}} \phi|^2 \left( \pi_\ell(w_{\vec{\beta},\ell,k_\ell}) \cdots \pi_1(w_{\vec{\beta},1,k_1}) (|\psi|)  \right) \,d\mu_G\\
&\phantom{MMMM}\leq \int_G \left(\sum_{\vec{k} \in \Z^\ell} |T_{\vec{k}} \phi|^2 \right) M_{\vec{\beta},\ell} \cdots M_{\vec{\beta},1} (|\psi|) \,d\mu_G \\
&\phantom{MMMM}\leq C_{\vec{\beta},r} \left\|\sum_{\vec{k} \in \Z^\ell} |T_{\vec{k}} \phi|^2 \right\|_{r/2} \|\psi\|_{(r/2)'},
\end{split}\]
that is,
\[\left\| \sum_{\vec{k} \in \Z^\ell} \pi_\ell(w_{\vec{\beta},\ell,k_\ell}) \cdots \pi_1(w_{\vec{\beta},1,k_1}) (|T_{\vec{k}} \phi|^2) \right\|_{r/2} \leq C_{\vec{\beta},r} \left\| \sum_{\vec{k} \in \Z^\ell} |T_{\vec{k}} \phi|^2 \right\|_{r/2},\]
and finally
\[\left\| \left( \sum_{\vec{k} \in \Z^\ell} |T_{\vec{k}} m(L^\flat)\phi|^2 \right)^{1/2} \right\|_r \leq C_{\vec{\beta},p,q,r} \|m\|_{M_\gimel S^{\vec{\beta}}_{p,q}B} \left\| \left(\sum_{\vec{k} \in \Z^\ell} |T_{\vec{k}} \phi|^2 \right)^{1/2}\right\|_r\]
which, by Lemma~\ref{lem:littlewoodpaley}, is equivalent to
\[\|m(L^\flat)\phi \|_r \leq C_{\vec{\beta},p,q,r} \|m\|_{M_\gimel S^{\vec{\beta}}_{p,q}B} \| \phi\|_r.\]

This gives the conclusion for $2 \leq r < \infty$. In order to obtain the same result for $1 < r < 2$, it should be noticed that, if $m$ is real-valued, then $m(L^\flat)$ is self-adjoint, so that boundedness for $2 < r < \infty$ implies boundedness for $1 < r < 2$. In the general case, one can decompose $m$ in its real and imaginary parts and then apply the previous result to each part.
\end{proof}

\begin{rem}
A multi-variate analogue of Propositions~\ref{prp:cancellation} and \ref{prp:kernelinfiniteorder} can be proved for the kernel $\Kern_{L^\times} m$ on the product $G^\times$, thus showing that, if $m$ satisfies a Marcinkiewicz condition of infinite order, then $\Kern_{L^\times} m$ is a \emph{smooth Calder\'on-Zygmund product kernel} (see \cite{nagel_singular_2001}, \S2.1, and \cite{mller_marcinkiewicz_1995}, Theorem~1.4).
\end{rem}

\begin{rem}
Although the convolution kernel $\Kern_{L^\flat} m$ of the operator $m(L^\flat)$ on the group $G$ is never used directly in the previous arguments, from Proposition~\ref{prp:convolutionproduct} it follows that this kernel is the push-forward of the kernel $\Kern_{L^\times} m$ of $m(L^\times)$ on $G^\times$ via the map $\upsilon^\times : G^\times \to G$, i.e.,
\[\langle \Kern_{L^\flat} m, \phi \rangle = \langle \Kern_{L^\times} m, \phi \circ \upsilon^\times \rangle\]
for $\phi \in \D(G)$ (cf.\ \S17.4.5 of \cite{dieudonn_treatise_analysis_1972} and Corollary~2.5.3 of \cite{nagel_singular_2001}).
\end{rem}

%% file: applications.tex
\section{Applications}

The previous Theorems~\ref{thm:mihlinhoermander} and \ref{thm:marcinkiewicz} can be applied, e.g., to the groups and systems of operators considered in \S\ref{section:examples}; the corresponding results are collected in Table~\ref{tbl:applications}. Here we focus on some few cases, which are probably the most interesting, also because they permit a comparison with results already present in the literature.

\begin{table}[p]%
\centering
\begin{tabular}{l|l|@{}c@{}|c|c|c|c@{ $+$ }c}
\multirow{2}{*}{group} & \multirow{2}{*}{operators} & \multirow{2}{*}{\phantom{.}$\dim G$\phantom{.}} & \multirow{2}{*}{$Q_G$} & \multicolumn{2}{c|}{\HPK{s}{d}}  & \multicolumn{2}{@{}c@{}}{required} \\
                     &                         &          &        & \phantom{..}$s$\phantom{..}             & $d$           & \multicolumn{2}{@{}c@{}}{smoothness} \\
 \hline
\TS\BS $H_n$           & $L_1,\dots,L_n,-iT$                 & \clap{$\scriptstyle 2n+1$}   & \clap{$\scriptstyle 2n+2$} & \clap{$\frac{2n+1}{2}$} & $1$           & $\frac{2n+1}{2}$ & $\frac{n}{p}$ \\
\TS\BS $H_n$           & $L,-iT$                 & \clap{$\scriptstyle 2n+1$}   & \clap{$\scriptstyle 2n+2$} & \clap{$\frac{2n+1}{2}$} & $1$           & $\frac{2n+1}{2}$ & $\frac{1}{p}$ \\
\TS\BS $\mathbb{H}H_n$ & $L_1,\dots,L_n,-iT_1,-iT_2,-iT_3$  & \clap{$\scriptstyle 4n+3$}   & \clap{$\scriptstyle 4n+6$} & \clap{$\frac{4n+3}{2}$} & $1$           & $\frac{4n+3}{2}$ & $\frac{n+2}{p}$ \\
\TS\BS $\mathbb{H}H_n$ & $L,-iT_1,-iT_2,-iT_3$  &  \clap{$\scriptstyle 4n+3$}   & \clap{$\scriptstyle 4n+6$} & \clap{$\frac{4n+3}{2}$} & $1$           & $\frac{4n+3}{2}$ & $\frac{3}{p}$ \\
\TS\BS $N_{3,2}$  & $L,D,-iT_1,-iT_2,-iT_3$ & $6$      & $9$    & $\frac{9}{2}$    & $4$           & $\frac{9}{2}$ & $\frac{1}{p}$    \\
\TS\BS $N_{3,2}$  & $L,D,\Delta$            & $6$      & $9$    & $\frac{9}{2}$    & $2$           & $\frac{9}{2}$ & $\frac{1}{p}$    \\
\TS\BS $N_{3,2}$  & $L,-iT_1,-iT_2,-iT_3$   & $6$      & $9$    & $\frac{9}{2}$    & $\frac{7}{2}$ & $\frac{9}{2}$ & $\frac{1}{2p}$   \\
\TS\BS $N_{3,2}$  & $L,\Delta$              & $6$      & $9$    & $\frac{9}{2}$    & $\frac{3}{2}$ & $\frac{9}{2}$ & $\frac{1}{2p}$   \\
\TS\BS $N_{3,2}$  & $L,D$                   & $6$      & $9$    & $\frac{9}{2}$    & $\frac{3}{2}$ & $\frac{9}{2}$ & $\frac{1}{2p}$   \\
\TS\BS $G_{5,2}$  & $L,-iX_2,-iX_1$         & $5$      & $7$    & $\frac{6}{2}$    & $2$           & $\frac{6}{2}$ & $\frac{1}{p}$    \\
\TS\BS $N_{4,2}$  & $L,D,-iT_{12},\dots,-iT_{34}$ & $10$ & $16$ & $\frac{16}{2}$   & $2$           & $\frac{16}{2}$ & $\frac{6}{p}$    \\
\TS\BS $N_{4,2}$  & $L,-iT_{12},\dots,-iT_{34}$   & $10$ & $16$ & $\frac{16}{2}$   & $2$           & $\frac{16}{2}$ & $\frac{5}{p}$    \\
\TS\BS $N_{4,2}$  & $L,D,P,\Delta$                & $10$ & $16$ & $\frac{16}{2}$   & $2$           & $\frac{16}{2}$ & $\frac{2}{p}$    \\
\TS\BS $N_{4,2}$  & $L,P,\Delta$                  & $10$ & $16$ & $\frac{16}{2}$   & $2$           & $\frac{16}{2}$ & $\frac{1}{p}$    \\
\TS\BS $N_{2,3}$  & $L,D,-iT_1,-iT_2$       & $5$      & $10$   & $\frac{10}{2}$   & $2$           & $\frac{10}{2}$ & $\frac{2}{p}$   \\
\TS\BS $N_{2,3}$  & $L,-iT_1,-iT_2$         & $5$      & $10$   & $\frac{10}{2}$   & $2$           & $\frac{10}{2}$ & $\frac{1}{p}$   \\
\TS\BS $N_{2,3}$  & $L,D,\Delta$            & $5$      & $10$   & $\frac{10}{2}$   & $1$           & $\frac{10}{2}$ & $\frac{2}{p}$   \\
\TS\BS $G_{6,19}$ & $L,-iX_2,-iX_1$         & $6$      & $10$   & $\frac{8}{2}$    & $1$           & $\frac{8}{2}$ & $\frac{2}{p}$    \\
\TS\BS $G_{6,23}$ & $L,-iX_2,-iX_1$         & $6$      & $11$   & $\frac{9}{2}$    & $1$           & $\frac{9}{2}$ & $\frac{2}{p}$    
\end{tabular}
\caption{Multiplier theorems applied to the groups and operators of \S\ref{section:examples}\\ \footnotesize(the required smoothness is expressed in terms of $L^p$ Besov norms with $p \geq 2$)}
\label{tbl:applications}
\end{table}

\subsection{The abelian case}

As a nilpotent Lie group, one can certainly take $\R^n$. In this case, commutativity of translation-invariant differential operators is automatically guaranteed.

For $j=1,\dots,n$, let $X_j = -\frac{\partial}{\partial x_j}$ be the partial derivative in the $j$-th component. Then the operators
\[-iX_1, \dots, -iX_n\]
form a weighted subcoercive system, since, e.g., the Laplacian $-(X_1^2 + \dots + X_n^2)$ is contained in the algebra generated by them (which in fact is the full algebra $\Diff(\R^n)$ of translation-invariant differential operators).

One should notice that, with respect to the abelian group structure of $\R^n$, every family of (possibly non-isotropic) dilations
\[\delta_t(x_1,\dots,x_n) = (t^{\lambda_1} x_1,\dots,t^{\lambda_n} x_n)\]
is automorphic. Moreover, the operators $-iX_1,\dots,-iX_n$ are homogeneous with respect to such a family of dilations.

Notice that, by the properties of the Fourier transform,
\[m(-iX_1,\dots,-iX_n) f = \Four^{-1} ( m (\Four f)) = f * \Four^{-1} m,\]
so that $\breve m = \Four^{-1} m$, and then, by the classical Plancherel formula,
\[\int_{\R^n} |m|^2 \,d\sigma = \int_{\R^n} |\breve m(x)|^2 \,dx = (2\pi)^{-n} \int_{\R^n} |m(\xi)|^2 \,d\xi,\]
i.e., the Plancherel measure $\sigma$ associated to the system $-iX_1,\dots,-iX_n$ is (apart from a constant factor) the Lebesgue measure on $\R^n$, which is clearly locally $n$-bounded on the whole $\R^n$.

This means that the aforementioned system satisfies the hypothesis \HPK{n/2}{n} (in this case, the degree of polynomial growth clearly coincides with the topological dimension) and we recover from Theorem~\ref{thm:mihlinhoermander} the classical Mihlin-H\"ormander theorem for Fourier multipliers on $\R^n$, i.e., Theorem~\ref{thmi:euclmihlin} of the introduction.

One can also think of each $-iX_j$ as a self-adjoint Rockland operator on the $j$-th factor of $\R^n$, i.e., as a Rockland system by itself, satisfying the hypothesis \HPK{1/2}{1}; by applying Theorem~\ref{thm:marcinkiewicz} to these systems, we recover Theorem~\ref{thmi:euclmarcinkiewicz} of the introduction.

\subsection{H-type groups}

Let $G$ be a H-type group, with a fixed stratification. Let $L$ be a homogeneous sublaplacian on $G$, and let $T_1,\dots,T_d$ be a basis of the center of the Lie algebra $\lie{g}$ of $G$. Then we have a Rockland system
\begin{equation*}
L,-iT_1,\dots,-iT_d,
\end{equation*}
to which our multiplier theorems can be applied, at least in two different ways.

Namely, this Rockland system on the H-type group $G$ satisfies the hypothesis \HPK{(\dim G)/2}{1} by Proposition~\ref{prp:HPKsatisfaction}, so that, by Theorem~\ref{thm:mihlinhoermander}, we get that, for $p \geq 2$, the condition
\[\|m\|_{M_\epsilon B_{p,p}^s} < \infty \quad\text{for some $s > \frac{\dim G}{2} + \frac{d}{p}$},\]
where $\epsilon_t$ are the dilations associated to the Rockland system, implies that the operator $m(L,-iT_1,\dots,-iT_d)$ is of weak type $(1,1)$ and is bounded on $L^r(G)$ for $1 < r < \infty$. 

On the other hand, we can consider each of the operators $L,-iT_1,\dots,-iT_d$ as forming a Rockland system by itself: the sublaplacian $L$ is Rockland on $G$, so that it satisfies the hypothesis \HPK{(\dim G)/2}{1}, whereas each central derivative $-iT_j$ is Rockland on a $1$-dimensional subgroup of the center of $G$, thus it satisfies the hypothesis \HPK{1/2,1}. Therefore, by Theorem~\ref{thm:marcinkiewicz}, for $p \geq 2$, the condition
\[\|m\|_{M_* S_{p,p}^{\vec{s}}B} < \infty \quad\text{for some $\vec{s} > \left(\frac{\dim G}{2},\frac{1}{2},\dots,\frac{1}{2}\right)$}\]
implies that $m(L,-iT_1,\dots,-iT_d)$ is bounded on $L^r(G)$ for $1 < r < \infty$.

If we compare these two results in terms of the required smoothness, we see that, for $p=\infty$, neither of the conditions is contained in the other: in fact, in the former conditions, only derivatives up to (something more than) the order $\frac{\dim G}{2}$ are required, but any kind of derivative; in the latter, instead, the order $\frac{\dim G + d}{2}$ is reached, but only with mixed derivatives, and in particular pure derivatives with respect to the spectral variables corresponding to $-iT_1,\dots,-iT_d$ are only required up to the order $\frac{1}{2}$. This comparison is illustrated in Figure~\ref{fig:graficoHtype} in the case of the Heisenberg groups $H_n$ (for which $d = 1$).

We notice that the result of M\"uller, Ricci and Stein \cite{mller_marcinkiewicz_1996} involving this system of operators is sharper than ours. In fact their condition, which is of Marcinkiewicz type, is expressed in terms of a ``mixed'' multi-parameter $L^2$ Sobolev norm
\[\|f\|_{S_{\mathrm{mix}}^{\vec{t}}H(\R^{1+n})}^2 = \int_{\R^{1+n}} (1+|\xi_0|)^{2t_0} \prod_{j=1}^n (1 + |\xi_0| + |\xi_j|)^{2t_j} |\Four f(\xi)|^2 \,d\xi,\]
with regularity threshold
\[\vec{t} > \left(\frac{\dim G-d}{2},\frac{1}{2},\dots,\frac{1}{2}\right);\]
therefore, with respect to the required smoothness, the condition of \cite{mller_marcinkiewicz_1996} is exactly the ``intersection'' of our two. On the other hand, our results are more general, since, for instance, the sublaplacian $L$ can be replaced by any Rockland operator (with respect to any homogeneous structure) on $G$, without any change in the condition on the multiplier.

On the Heisenberg group $H_n$ we can also consider the Rockland system
\[L_1,\dots,L_n,-iT\]
made of the partial sublaplacians and the central derivative. There are several ways in which our multiplier theorems can be applied to this system, since every subfamily of partial sublaplacians form (with or without the central derivative) a Rockland system on some subgroup of $H_n$. Analogously as before, the ``intersection'' of all our conditions corresponds to the sharper result of Veneruso \cite{veneruso_marcinkiewicz_2000} about this system of operators.

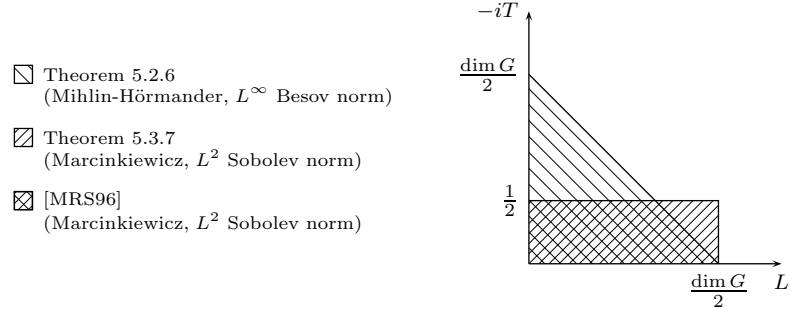
\begin{figure}%
\centering
\begin{pdfpic}
\psset{unit=12mm,linewidth=.5pt,hatchwidth=.2pt,hatchsep=2.5pt,labelsep=4pt}
\begin{pspicture}(-5.7,-.2)(2.9,2.8)
\pspolygon[fillstyle=vlines,hatchcolor=black,linestyle=none,hatchsep=3.9pt,hatchwidth=.3pt](2.1,0)(0,2.1)(0,0)
\psline(2.1,0)(0,2.1)
\uput[180](0,2.1){$\frac{\dim G}{2}$}\uput[-90](2.1,0){$\frac{\dim G}{2}$}
\pspolygon[fillstyle=vlines,hatchcolor=black,hatchsep=3.9pt,hatchwidth=.3pt](-5.7,2.2)(-5.7,2.0)(-5.5,2.0)(-5.5,2.2)
\uput[0](-5.5,2.1){\scriptsize Theorem~\ref{thm:mihlinhoermander}}
\uput[0](-5.5,1.85){\scriptsize (Mihlin-H\"ormander, $L^\infty$ Besov norm)}
\pspolygon[fillstyle=hlines,hatchcolor=black,linestyle=none](0,0)(2.1,0)(2.1,.7)(0,.7)
\psline(2.1,0)(2.1,.7)(0,.7)
\uput[180](0,.7){$\frac{1}{2}$}
\pspolygon[fillstyle=hlines,hatchcolor=black](-5.7,1.5)(-5.7,1.3)(-5.5,1.3)(-5.5,1.5)
\uput[0](-5.5,1.4){\scriptsize Theorem~\ref{thm:marcinkiewicz}}
\uput[0](-5.5,1.15){\scriptsize (Marcinkiewicz, $L^2$ Sobolev norm)}
\pspolygon[fillstyle=vlines,hatchcolor=black,hatchsep=3.9pt,hatchwidth=.3pt](-5.7,0.8)(-5.7,0.6)(-5.5,0.6)(-5.5,0.8)
\pspolygon[fillstyle=hlines,hatchcolor=black](-5.7,0.8)(-5.7,0.6)(-5.5,0.6)(-5.5,0.8)
\uput[0](-5.5,0.7){\scriptsize \cite{mller_marcinkiewicz_1996}}
\uput[0](-5.5,0.45){\scriptsize (Marcinkiewicz, $L^2$ Sobolev norm)}
\psaxes[ticks=none,labels=none]{->}(0,0)(2.8,2.8)\uput[180](0,2.8){\footnotesize $-iT$}\uput[-90](2.8,0){\footnotesize $L$}
\end{pspicture}
\end{pdfpic}
\caption{Comparison of multiplier results on the Heisenberg groups}
\label{fig:graficoHtype}
\end{figure}

\subsection{Non-nilpotent groups}

Theorem~\ref{thm:marcinkiewicz} allows one to obtain spectral multiplier theorems also on groups which are not homogeneous, even not nilpotent. Clearly, the case $\ell = 1$ gives transference of multiplier results from a homogeneous group to a larger, possibly non-homogeneous one. A less trivial example comes by considering an action of a torus $\T^d = \R^d/\Z^d$ on a homogeneous group $N$ by automorphisms which commute with dilations, and the corresponding semidirect product $\T^d \ltimes N$ (or alternatively its universal covering group $\R^d \ltimes N$).

Take for instance a diamond group $G = \T^d \ltimes H_n$ (see \cite{ludwig_dual_1995}). If $L$ is a $\T^d$-invariant homogeneous sublaplacian on $H_n$ (which always exists), and if $U_1,\dots,U_d$ are the partial derivatives on the torus $\T^d$, then
\[L,-iU_1,\dots,-iU_d\]
is a weighted subcoercive system on $G$, since these operators commute on $G$, and the algebra generated by them contains the operator
\[\Delta = L + (-iU_1)^2 + \dots + (-iU_d)^2,\]
which is a sublaplacian and therefore is weighted subcoercive on $G$. Each of the operators $L,-iU_1,\dots,-iU_d$ can be considered as a Rockland system in itself: $L$ is Rockland on $H_n$, and therefore satisfies \HPK{(\dim H_n)/2}{1}, whereas $-iU_j$ comes from the corresponding derivative on the $j$-th factor of $\R^d$, which satisfies \HPK{1/2,1}. By applying Theorem~\ref{thm:marcinkiewicz}, we then obtain that, if
\[\|m\|_{M_* S^{\vec{s}}_{2,2}B(\R^{1+d})} < \infty \qquad\text{for } \vec{s} > \left(\frac{\dim H_n}{2},\frac{1}{2},\dots,\frac{1}{2}\right),\]
then $m(L,-iU_1,\dots,-iU_d)$ is bounded on $L^r(G)$ for $1 < r < \infty$.

Thanks to Proposition~\ref{prp:homogeneousextension} and Corollary~\ref{cor:mihlinmarcinkiewiczsum}, this result in turn yields a multiplier theorem for the sublaplacian $\Delta$: if
\[\|m\|_{M_* B^{s}_{2,2}(\R)} < \infty \qquad\text{for } s > \frac{\dim H_n + d}{2} = \frac{\dim G}{2},\]
then $m(\Delta)$ is bounded on $L^r(G)$ for $1 < r < \infty$. We remark that:
\begin{itemize}
\item this condition is sharper than the one following by the general result of Alexopoulos \cite{alexopoulos_spectral_1994}, which instead requires $s > \frac{\dim G + 1}{2}$ in terms of an $L^\infty$ Besov norm (in fact, $\dim G + 1$ is the local dimension associated to the sublaplacian $\Delta$);
\item this is an example of a group in which the regularity threshold for a multiplier theorem can be lowered to half the topological dimension, which is neither a M\'etivier group (or a direct product of Euclidean and M\'etivier groups), nor $SU_2$;
\item by Proposition~\ref{prp:homogeneousextension}, Corollary~\ref{cor:mihlinmarcinkiewiczsum} and Proposition~\ref{prp:mihlinmarcinkiewiczproduct}, the sublaplacian $\Delta$ can be replaced by any operator of the form
\[L^{k_0} + -(iU_1)^{2k_1} + \dots + (-iU_d)^{2k_d} \qquad\text{or}\qquad L^{k_0} (-iU_1)^{k_1} \cdots (-iU_d)^{k_d}\]
for some $k_0,k_1,\dots,k_d \in \N \setminus \{0\}$, obtaining an analogous multiplier result with identical smoothness requirement.
\end{itemize}

Operators such as the complete Laplacian
\[\Delta_c = L + (-iT)^2 + (-iU_1)^2 + \dots + (-iU_d)^2,\]
where $T$ is the central derivative on $H_n$, can also be studied. By considering $L,-iT$ together as a Rockland system on $H_n$, and each of the $-iU_j$ separately as before, one obtains, by Theorem~\ref{thm:marcinkiewicz}, Proposition~\ref{prp:homogeneousextension} and Corollary~\ref{cor:mihlinmarcinkiewiczsum}, a multiplier theorem for the Laplacian $\Delta_c$, with a regularity threshold of order $\frac{\dim G}{2}$ in terms of an $L^\infty$ Besov norm.

Analogous considerations hold if one replaces $H_n$ by any H-type (or M\'etivier) group, and also if one takes the universal covering group $\R^d \ltimes H_n$; this last case comprises, for $d=1$, the oscillator groups. Notice that the previous result about the Laplacian $\Delta_c$, when stated on the universal covering group, is sharper than \cite{alexopoulos_spectral_1994}, since the degree of growth of the group is greater than its topological dimension.

Further examples include the plane motion group $\T \ltimes \R^2$, and also the semidirect product $\T \ltimes N_{2,3}$ determined by the action of $SO_2$ on the free $3$-step nilpotent group $N_{2,3}$ considered in \S\ref{subsection:example3step}. In these last cases, for some distinguished sublaplacians, we still get a sharpening of the result by Alexopoulos: although the required order of smoothness is the same, our condition is expressed in terms of an $L^2$ instead of an $L^\infty$ Besov norm.

%% file: appendix.tex
\phantomsection
\addcontentsline{toc}{chapter}{Appendix}
\chapter*{Appendix}
\markboth{Appendix}{}
\renewcommand{\thechapter}{A}

In the following we give a brief account of the theory of spectral integrals, including the spectral theorem for multiple operators and the connections with the Gelfand theory of commutative C$^*$-algebras. In doing so, we also summarize some results from the theory of Banach $*$-algebras which are used throughout the work.

\section{\texorpdfstring{Banach $*$-algebras and C$^*$-algebras}{Banach *-algebras and C*-algebras}}\label{section:staralgebras}

The following results can be found, together with a more extensive presentation, in \cite{bonsall_complete_1973}, \cite{dixmier_algebras_1982}, \cite{palmer_banach_1994}, \cite{palmer_banach_2001}.

\subsection{$*$-algebras}
A \emph{$*$-algebra}\index{algebra!$*$-algebra} is an associative algebra over $\C$ (i.e., a complex vector space $\Alg$ with a bilinear, associative product) endowed with a conjugate-linear \emph{involution}
\[\Alg \ni a \mapsto a^* \in \Alg\]
such that $(a^*)^* = a$, $(ab)^* = b^* a^*$. $\Alg$ is \emph{unital} if the product has an identity.

If $\Alg_1$ and $\Alg_2$ are $*$-algebras, a linear map $T : \Alg_1 \to \Alg_2$ which preserves multiplication and involution is called a \emph{$*$-homomorphism}; if moreover $\Alg_1$ and $\Alg_2$ are unital $*$-algebras, and $T$ maps the identity of $\Alg_1$ to the identity of $\Alg_2$, then $T$ is called a \emph{unital $*$-homomorphism}.

A \emph{Banach $*$-algebra} (with isometric involution) is a $*$-algebra $\Alg$ endowed with a norm $\|\cdot\|$ which makes it into a Banach space and such that
\[\|a_1 a_2\| \leq \|a_1\| \, \|a_2\|, \qquad \|a^*\| = \|a\|,\]
\[\|e\| = 1 \qquad\text{if $e \in \Alg$ is an identity.}\]
If moreover
\[\|a^* a\| = \|a\|^2,\]
then $\Alg$ is called a C$^*$-algebra\index{algebra!C$^*$-algebra}.

If $\HH$ is a Hilbert space, then the space $\Bdd(\HH)$ of bounded linear operators on $\HH$, with composition, adjunction and the operator norm, is a unital C$^*$-algebra. In fact, by the Gelfand-Naimark theorem, every C$^*$-algebra is (isometrically) $*$-isomorphic to a closed $*$-subalgebra of $\Bdd(\HH)$ for some Hilbert space $\HH$.

If a $*$-algebra $\Alg$ is not unital, then we embed it in a unital $*$-algebra
\[\Alg_U = \Alg \oplus \C,\]
called the \emph{unitization} of $\Alg$, with operations
\[(a_1,\lambda_1) (a_2,\lambda_2) = (a_1 a_2 + \lambda_1 a_2 + \lambda_2 a_1, \lambda_1 \lambda_2), \qquad (a,\lambda)^* = (a^*, \overline{\lambda});\]
in fact, the element $(0,1)$ is an identity for $\Alg_U$, and the embedding of $\Alg$ into $\Alg_U$ is given by $a \mapsto (a,0)$. If $\Alg$ is a Banach $*$-algebra, then $\Alg_U$ can be made into a Banach $*$-algebra with the norm
\[\|(a,\lambda)\| = \|a\| + |\lambda|;\]
if moreover $\Alg$ is a C$^*$-algebra, then the norm
\[\|(a,\lambda)\|' = \sup_{\substack{x \in A \\ \| x\| \leq 1}} \|ax + \lambda x\|\]
makes $\Alg_U$ into a C$^*$-algebra.

\subsection{Spectrum of an element}\label{subsection:spectrum}
For a unital $*$-algebra $\Alg$, we define the \emph{spectrum}\index{spectrum!of an element of a $*$-algebra} $\spec_\Alg(a)$ of an element $a \in \Alg$ as the set of the $\lambda \in \C$ such that $a - \lambda$ is not invertible in $\Alg$; if $\Alg$ is not unital, then the spectrum $\spec_\Alg(a)$ of an element $a \in \Alg$ is defined as its spectrum in the unitization $\Alg_U$ of $\Alg$. The \emph{spectral radius} $\rho_\Alg(a)$ of an element $a \in \Alg$ is given by
\[\rho_\Alg(a) = \sup \{ |\lambda| \tc \lambda \in \spec_\Alg(a)\}.\]
If $\Alg$ is a Banach $*$-algebra, then $\spec_\Alg(a)$ is a non-empty compact subset of $\C$ for every $a \in \Alg$, and
\[\rho_\Alg(a) \leq \|a\|;\]
moreover the \emph{spectral radius formula} holds:
\[\rho_\Alg(a) = \lim_{n \to \infty} \|a^n\|^{1/n} = \inf_n \|a^n\|^{1/n}.\]
If $\Alg'$ is a closed $*$-subalgebra of a Banach $*$-algebra $\Alg$, then, for every $b \in \Alg'$,
\[\partial \spec_{\Alg'}(b) \subseteq \partial \spec_{\Alg}(b) \subseteq \spec_{\Alg}(b) \subseteq \spec_{\Alg'}(b) \cup \{0\},\]
(see \cite{bonsall_complete_1973}, Proposition~I.5.12), and in particular
\[\rho_{\Alg}(b) = \rho_{\Alg'}(b).\]
If $\Alg$ is a C$^*$-algebra, then, for every $a \in \Alg$,
\[\|a\| = \rho_{\Alg}(a^* a)^{1/2}\]

\subsection{$*$-representations}
If $\Alg$ is a $*$-algebra, a \emph{$*$-representation}\index{representation!$*$-representation} $T$ of $\Alg$ on a Hilbert space $\HH$ is a $*$-homomorphism $T : \Alg \to \Bdd(\HH)$. For $a \in \Alg$, we set
\[\gamma_\Alg(a) = \sup \{\|T(a)\| \tc T \text{ is a $*$-representation of } \Alg\}.\]
The function $\gamma_{\Alg}$, which is called the \emph{Gelfand-Naimark seminorm} of $\Alg$, is a (possibly infinite-valued) submultiplicative seminorm on $\Alg$, which satisfies
\[\gamma_{\Alg}(a^* a) = \gamma_{\Alg}(a)^2, \qquad \gamma_{\Alg}(a) \leq \rho_{\Alg}(a^* a)^{1/2}.\]
In particular, if $\Alg$ is a Banach $*$-algebra, then
\[\gamma_{\Alg}(a) \leq \|a\|,\]
so that every $*$-representation is automatically continuous.

\subsection{Hermitian $*$-algebras}\label{subsection:hermitianalgebras}
An element $a$ of a $*$-algebra $\Alg$ is called \emph{hermitian} if $a = a^*$. A $*$-algebra $\Alg$ is called \emph{hermitian}\index{algebra!$*$-algebra!hermitian} if $\spec_{\Alg}(a) \subseteq \R$ for every hermitian $a \in \Alg$. In fact, for a Banach $*$-algebra $\Alg$, the following conditions are equivalent:
\begin{itemize}
\item $\Alg$ is hermitian;
\item $\spec_{\Alg}(a^* a) \subseteq \left[0,+\infty\right[$ for every $a \in \Alg$;
\item $\gamma_{\Alg}(a) = \rho_{\Alg}(a^* a)^{1/2}$ for every $a \in \Alg$.
\end{itemize}
The equivalence of the first two conditions is the \emph{Shirali-Ford theorem} (see \cite{bonsall_complete_1973}, Theorem~41.5), whereas the third condition is called \emph{Raikov's criterion} (cf.\ \cite{palmer_banach_2001}, Theorem~11.4.1).

Every C$^*$-algebra is hermitian.

\subsection{Commutative $*$-algebras}
A \emph{character} of a commutative $*$-algebra $\Alg$ is a non-null linear functional $\phi : \Alg \to \C$ which is multiplicative, i.e., such that $\phi(a_1 a_2) = \phi(a_1) \phi(a_2)$; the set $\GS(\Alg)$ of the characters of $\Alg$ is called the \emph{Gelfand spectrum}\index{spectrum!Gelfand spectrum} of $\Alg$.

If $\Alg$ is a commutative $*$-algebra, then the \emph{Gelfand transform} of $a$ is the function $\hat a : \GS(\Alg) \to \C$ defined by
\[\hat a(\phi) = \phi(a).\]
The \emph{Gelfand topology} is the weakest topology on $\GS(\Alg)$ such that all the maps $\hat a$ (for $a \in \Alg$) are continuous.

If $\Alg$ is a commutative Banach $*$-algebra, then, for all $a \in \Alg$,
\[\hat a(\GS(\Alg)) \subseteq \spec_\Alg(a) \subseteq \hat a(\GS(\Alg)) \cup \{0\},\]
and in particular
\[\|\hat a\|_\infty = \rho_\Alg(a) \leq \|a\|;\]
moreover $\GS(\Alg)$ is a locally compact Hausdorff topological space, and $\hat a \in C_0(\GS(\Alg))$ for every $a \in \Alg$. In fact, if $A$ is unital, then $\Gamma(\Alg)$ is compact and
\[\hat a(\GS(\Alg)) = \spec_\Alg(a)\]
for every $a \in \Alg$.

For a commutative Banach $*$-algebra $\Alg$, the map
\[\Alg \ni a \mapsto \hat a \in C_0(\GS(\Alg)),\]
which is called the \emph{Gelfand transform}\index{transform!Gelfand}, is a homomorphism of algebras, i.e., it is a linear map which preserves the product. In fact, we have:
\begin{itemize}
\item $\Alg$ is a hermitian Banach $*$-algebra if and only if the Gelfand transform is a $*$-homomorphism;
\item $\Alg$ is a C$^*$-algebra if and only if the Gelfand transform is an (isometric) $*$-isomorphism.
\end{itemize}

\section{Linear operators on Hilbert spaces}

While bounded operators on a Hilbert space form a C$^*$-algebra, and consequently are included in the theory set forth in \S\ref{section:staralgebras}, unbounded operators require a specific treatment. Moreover, among bounded operators, specific subclasses with peculiar properties can be considered. General references for the following are \cite{dunford_linear_1963}, \cite{rudin_functional_1973}, \cite{reed_methods_1980}.

\subsection{Operators as graphs}
Let $\HH_1, \HH_2$ be (complex) Hilbert spaces. A (possibly not everywhere defined) linear operator $L$ from $\HH_1$ to $\HH_2$ can be identified with its graph, which is a linear subspace of the direct product $\HH_1 \times \HH_2$. In particular, the \emph{domain} $D(A)$ and the \emph{range} $R(A)$ of $A$ can be simply thought of as the projections of $A$ on the two factors $\HH_1$ and $\HH_2$.

\subsection{Sum and product}
If $A_1,A_2$ are linear operators from $\HH_1$ to $\HH_2$, then their \emph{sum} $A_1 + A_2$ is defined pointwise on the intersection of their domains:
\[D(A_1 + A_2) = D(A_1) \cap D(A_2).\]
If $A$ is an operator from $\HH_1$ to $\HH_2$, and $B$ is an operator from $\HH_2$ to another Hilbert space $\HH_3$, then their \emph{product} (or \emph{composition}) $BA$ is an operator from $\HH_1$ to $\HH_3$ defined as a composition of relations, so that
\[D(BA) = \{ x \in D(A) \tc Ax \in D(B)\}.\]

\subsection{Closable and closed operators}
An operator $A$ from $\HH_1$ to $\HH_2$ is called \emph{closable} if its closure $\overline{A}$ (as a subspace of $\HH_1 \times \HH_2$) is still (the graph of) an operator, i.e., if $\overline{A}$ is univocal. $A$ is said to be \emph{closed} if $A = \overline{A}$. Clearly, if an operator $A$ admits a closed extension, i.e., a closed operator $B$ such that $A \subseteq B$, then $A$ is closable. By the closed map theorem, a closed operator which is everywhere defined is bounded.

\subsection{Adjoint of an operator}
Let $J_{\HH_1,\HH_2} : \HH_1 \times \HH_2 \to \HH_2 \times \HH_1$ be the isometric isomorphism defined by
\[J_{\HH_1,\HH_2}(x_1,x_2) = (-x_2,x_1).\]
We define the \emph{adjoint} $A^*$ of an operator $A$ from $\HH_1$ to $\HH_2$ as
\[A^* = (J_{\HH_1,\HH_2}(A))^\perp = J_{\HH_1,\HH_2}(A^\perp).\]
Then one can see that $A^*$ is an operator from $\HH_2$ to $\HH_1$ if and only if $D(A)$ is dense in $\HH_1$. Moreover, the domain of the adjoint $A^*$ of $A$ is given by
\[D(A^*) = \{ x_2 \in \HH_2 \tc x_1 \mapsto \langle A x_1, x_2 \rangle_{\HH_2} \text{ is continuous} \},\]
and for $x_2 \in D(A^*)$ we have that $A^* x_2$ is the unique element of $\HH_1$ such that
\[\langle x_1, A^* x_2 \rangle_{\HH_1} = \langle A x_1, x_2 \rangle_{\HH_2}.\]
From the definition, one has immediately $A^*$ is closed, and that $A^{**} = \overline{A}$. In particular, if $A^*$ is a densely defined operator, then $A$ is closable. Moreover
\[\ker A^* = R(A)^\perp.\]

\subsection{Symmetric and positive operators}
From now on, let us suppose that $\HH_1 = \HH_2 = \HH$. An operator $A$ on $\HH$ is said to be \emph{symmetric} if $A \subseteq A^*$. In particular, a densely defined symmetric operator is closable. Notice that $A$ is symmetric if and only if
\[\langle A x, y \rangle_{\HH} = \langle x, Ay \rangle_{\HH} \qquad\text{for all $x,y \in D(A)$}\]
if and only if
\[\langle A x, x \rangle_{\HH} \in \R \qquad\text{for all $x \in D(A)$.}\]
We say that an operator $A$ is \emph{positive} if
\[\langle A x, x \rangle_{\HH} \geq 0 \qquad\text{for all $x \in D(A)$.}\]
A closable operator $A$ is symmetric or positive if and only if $\overline{A}$ is.

\subsection{Self-adjoint and essentially self-adjoint operators, normal operators}
An operator $A$ on $\HH$ is said to be \emph{self-adjoint} if $A^* = A$; in particular, a self-adjoint operator is closed and densely defined.

A closable operator is said to be \emph{essentially self-adjoint} if its closure $\overline{A}$ is self-adjoint, i.e., if $A^* = \overline{A}$.

A closed and densely defined operator is said to be \emph{normal} if $A^* A = A A^*$.

\subsection{Spectrum}
If $A$ is a densely defined operator on $\HH$, the \emph{resolvent set} of $A$ is the set of the $\lambda \in \C$ such that $A - \lambda$ is injective and its inverse $(A - \lambda)^{-1}$ is densely defined and bounded; the complementary $\spec(A)$ of the resolvent set is called the \emph{spectrum}\index{spectrum!of a linear operator} of $A$.

A subset of $\spec(A)$ is the \emph{point spectrum} of $A$, which is the set of the eigenvalues of $A$, i.e., the $\lambda \in \C$ such that $A - \lambda$ is not injective. Clearly, the point spectrum coincides with the whole spectrum when $\HH$ is finite-dimensional.

The spectrum $\spec(A)$ of a densely defined operator $A$ on $\HH$ is a closed subset of $\C$. If $A$ is self-adjoint, then $\spec(A) \subseteq \R$. 

If $A$ is bounded, then $\spec(A) = \spec_{\Bdd(\HH)}(A)$.

\subsection{Compact operators, Hilbert-Schmidt operators and trace class operators}

A bounded operator $A$ on $\HH$ is said to be:
\begin{itemize}
\item \emph{compact}, if $A$ maps bounded subsets of $\HH$ to relatively compact ones;
\item \emph{Hilbert-Schmidt}, if
\[\|A\|_{\HS} = \left(\sum_{\alpha \in I} \|A v_\alpha\|^2\right)^{1/2} < \infty\]
for any (hence for every) complete orthonormal system $(v_\alpha)_{\alpha \in I}$ in $\HH$;
\item \emph{trace class}, if
\[\tr |A| = \sum_{\alpha \in I} \langle (A^* A)^{1/2} v_\alpha, v_\alpha \rangle < \infty\]
for any (hence for every) complete orthonormal system $(v_\alpha)_{\alpha \in I}$ in $\HH$.
\end{itemize}
The quantities $\|A\|_{\HS}$ and $\tr |A|$ are independent on the choice of the system $(v_\alpha)_{\alpha \in I}$, and define complete norms on the spaces of Hilbert-Schmidt operators and trace class operators respectively.

The trace class and the Hilbert-Schmidt operators form $*$-ideals of $\Bdd(\HH)$. The compact operators form a closed $*$-ideal of $\Bdd(\HH)$. Moreover, for a bounded operator on $\HH$, the following implications hold:
\[\text{finite rank} \,\Rightarrow\, \text{trace class} \,\Rightarrow\, \text{Hilbert-Schmidt} \,\Rightarrow\, \text{compact}.\]

For a trace class operator $A$, the \emph{trace}
\[\tr A = \sum_{\alpha \in I} \langle A v_\alpha, v_\alpha \rangle\]
is well defined (and finite). Moreover, the expression
\[\tr (B^* A) = \sum_{\alpha \in I} \langle A v_\alpha, B v_\alpha \rangle\]
defines an inner product, inducing the Hilbert-Schmidt norm $\|\cdot\|_{\HS}$.

\section{Spectral integration}\label{section:spectralintegral}

For the present and the next section, we refer mainly to \cite{berberian_notes_1966}, \cite{rudin_functional_1973},  \cite{fell_representations_1988} for detailed expositions and proofs.

\subsection{Orthogonal projections}
Let $\HH$ be a Hilbert space. An \emph{orthogonal projection} of $\HH$ is a bounded linear operator $P$ on $\HH$ such that $P^2 = P = P^*$.

An orthogonal projection $P$ is uniquely determined by its range $R(P)$, which is a closed subspace of $\HH$, and in fact every closed subspace of $\HH$ is the range of an orthogonal projection.

We define a (partial) order on the orthogonal projections of $\HH$ by setting
\[P \leq Q \quad \iff \quad R(P) \subseteq R(Q).\]
With respect to this order, the orthogonal projections of $\HH$ form a complete lattice; in particular, the \emph{supremum} of a family $\mathcal{P}$ of orthogonal projections is the orthogonal projection $\sup \mathcal{P}$ with range
\[R(\sup \mathcal{P}) = \overline{\Span \bigcup \{ R(P) \tc P \in \mathcal{P} \}}.\]

\subsection{Resolutions of the identity}
Let $X$ be a locally compact Hausdorff topological space. A \emph{resolution of the identity}\index{resolution of the identity} of $\HH$ on $X$ is a correspondence $E$ which maps Borel subsets $A$ of $X$ to orthogonal projections $E(A)$ of $\HH$, and which satisfies the following properties:
\begin{itemize}
\item $E(X) = \id_\HH$;
\item $E(A_1 \cup A_2) = E(A_1) + E(A_2)$ whenever $A_1 \cap A_2 = \emptyset$;
\item $E(A) = \lim_n E(A_n)$ whenever $\{A_n\}_{n \in \N}$ is an increasing sequence with union $A$, and where the limit is meant to be in the strong sense;
\item $E(A_1 \cap A_2) = E(A_1) E(A_2) = E(A_2) E(A_1)$.
\end{itemize}
A resolution of the identity is also called a \emph{projection-valued measure}. A resolution $E$ of the identity of $\HH$ on $X$ is said to be \emph{regular} if, for every Borel $A \subseteq X$,
\[E(A) = \sup \{ E(K) \tc K \text{ is a compact subset of } A\}.\]

Let $E$ be a resolution of the identity of $\HH$ on $X$. For every $u,v \in \HH$, the equality
\[E_{u,v}(A) = \langle E(A) u, v \rangle\]
defines a complex-valued Borel measure $E_{u,v}$ on $X$, which satisfies
\[|E_{u,v}|(X) \leq \|u\| \|v\|.\]
In fact, in the case $u = v$, we have that $E_{u,u}$ is a positive Borel measure, with
\[E_{u,u}(X) = \|u\|^2.\]
It can be shown that $E$ is regular if and only if $E_{u,u}$ is regular for every $u \in \HH$; in particular, if $X$ is second-countable, then every resolution of the identity on $X$ is regular (see \cite{fell_representations_1988}, Proposition II.11.10).

The \emph{support} $\supp E$ of a projection-valued measure $E$ on $X$ is defined by
\[\supp E = X \setminus \bigcup \{A \subseteq X \tc A \text{ is open and } E(A) = 0\}.\]
In fact, $\supp E$ is the closure of the union of the supports of the measures $E_{u,u}$ for $u \in \HH$. If $E$ is regular, then $E(\supp E) = \id_\HH$.

\subsection{Integral of bounded functions}
Let $B(X)$ be the set of (complex-valued) bounded Borel functions on a locally compact Hausdorff topological space $X$; in fact, $B(X)$ is a C$^*$-algebra with pointwise operations and the supremum norm.

If $E$ is a resolution of the identity of a Hilbert space $\HH$ on $X$, then there exists a unique continuous linear map $B(X) \to \Bdd(\HH)$ which maps $\chr_A$ to $E(A)$ for every Borel $A \subseteq X$; the image of a function $f \in B(X)$ via this map is denoted by
\[\int_X f \,dE \qquad\text{or}\qquad E[f],\]
and is called the \emph{integral} of the function $f$ with respect to the projection-valued measure $E$. We clearly have
\[\langle E[f] u, v\rangle = \int_X f \,dE_{u,v}\]
for every $f \in B(X)$, $u,v \in \HH$, and in particular
\[\|E[f] u\|^2 = \int_X |f|^2 \,dE_{u,u}.\]

In fact, the map $E[\cdot]$ is a $*$-homomorphism, so that
\[\|E[f]\| \leq \|f\|_\infty.\]
More precisely, if
\[N_E = \{f \in B(X) \tc E(\{f \neq 0\}) = 0\},\]
then $N_E$ is a closed $*$-ideal of $B(X)$, and the map $E[\cdot] : B(X) \to \Bdd(\HH)$ induces an isometric embedding of the quotient C$^*$-algebra
\[L^\infty(X,E) = B(X)/N_E\]
into the C$^*$-algebra $\Bdd(\HH)$; moreover, the norm on $L^\infty(X,E)$ is given by the $E$-essential supremum:
\[\|f\|_{L^\infty(X,E)} = \min \{\lambda \in \left[0,+\infty\right[ \tc E(\{|f| > \lambda\}) = 0\}\]
(see \cite{rudin_functional_1973}, \S12.20).

\subsection{Spectral integrals and Gelfand transform}\label{subsection:integralgelfand}
If $X$ is a locally compact Hausdorff topological space and $E$ is a resolution of the identity of a Hilbert space $\HH$ on $X$, then we can restrict the map $E[\cdot]$ to the sub-C$^*$-algebra $C_0(X)$ of $B(X)$, thus obtaining a $*$-representation $T$ of $C_0(X)$ on $\HH$. If $E$ is a regular resolution of the identity, then the $*$-representation $T$ is \emph{regular}, in the sense that
\[\bigcap \{\ker T(f) \tc f \in C_0(X)\} = \{0\}.\]
In fact, if $T$ is a regular $*$-representation of $C_0(X)$ on a Hilbert space $\HH$, then there exists a unique regular resolution $E$ of the identity of $\HH$ on $X$ such that $T(f) = E[f]$ for every $f \in C_0(X)$ (see \cite{fell_representations_1988}, Theorem~II.12.8).

In particular, if $\HH$ is a Hilbert space and $A$ is a closed commutative $*$-subalgebra of $\Bdd(\HH)$ containing the identity of $\HH$, then the inverse of the Gelfand transform is a regular $*$-representation of $C(\GS(A))$ on $\HH$, so that there exists a unique regular resolution $E$ of the identity of $\HH$ on the compact space $\GS(A)$ such that $E[\hat a] = a$ for every $a \in A$; this resolution $E$ will be called the \emph{spectral resolution} of the algebra $A$.

On the other hand, if $E$ is a regular resolution of the identity of a Hilbert space $\HH$ on a locally compact Hausdorff topological space $X$, then the C$^*$-algebra $C_0(\supp E)$, identified with a sub-C$^*$-algebra of $B(X)$ via extension by zero, is isometrically embedded in the quotient $L^\infty(X,E)$; moreover, by the Tietze-Urysohn extension theorem (see \cite{bourbaki_topology2}, \S IX.4.2), its image in $\Bdd(\HH)$ via the map $E[\cdot]$ coincides with the image $E[C_0(X)]$ of $C_0(X)$. Consequently, the Gelfand spectrum of $E[C_0(X)]$ can be identified with $\supp E$ via a homeomorphism $\Psi : \GS(E[C_0(X)]) \to \supp E$ such that
\[\phi(E[f])  = f(\Psi(\phi))\]
for every $f \in C_0(X)$, $\phi \in \GS(E[C_0(X)])$.

\subsection{Integral of unbounded functions}\label{subsection:spectralunbounded}
Let $E$ be a resolution of the identity of $\HH$ on $X$. We extend the definition of $E[f]$ to general Borel functions $f : X \to \C$, so that $E[f]$ is the (possibly unbounded) linear operator on $\HH$, with domain
\[D(E[f]) = \left\{ u \in \HH \tc \int_X |f|^2 \,dE_{u,u} < \infty\right\},\]
which is uniquely determined by the condition
\[\langle E[f] u, v \rangle = \int_X f \,dE_{u,v}\]
for all $u \in D(E[f])$, $v \in \HH$; in fact, it can be shown that, for every Borel $f : X \to \C$, the set $D(E[f])$ is a dense subspace of $\HH$, and that
\[\| E[f] u \|^2 = \int_X |f|^2 \,dE_{u,u}\]
for every $u \in D(E[f])$. Moreover:
\begin{itemize}
\item $\overline{E[f_1] + E[f_2]} = E[f_1 + f_2]$;
\item $\overline{E[f_1] E[f_2]} = E[f_1 f_2]$, and $D(E[f_1] E[f_2]) = D(E[f_1 f_2]) \cap D(E[f_2])$;
\item $E[\overline{f}] = E[f]^*$, and $E[f]^* E[f] = E[|f|^2] = E[f] E[f]^*$.
\end{itemize}
(cf.\ \cite{rudin_functional_1973}, Theorem~13.24). In particular, for every Borel $f : X \to \C$, $E[f]$ is a (closed) normal operator, and in fact $E[f]$ is self-adjoint if $f$ is real-valued; moreover, the spectrum of $E[f]$ is given by the $E$-essential image of $f$, i.e.,
\[\spec(E[f]) = \{\lambda \in \C \tc E(\{ |f - \lambda| < \epsilon \}) \neq 0 \text{ for all } \epsilon > 0\},\]
and, for every $\lambda \in \C$, the eigenspace of $E[f]$ relative to the eigenvalue $\lambda$ is characterized by
\[\{ u \in D(E[f]) \tc E[f]u = \lambda u \} = R(E\{ f = \lambda \})\]
(see \cite{rudin_functional_1973}, Theorem~13.27).

\subsection{Push-forward measure}\label{subsection:pushforwardspectral}
Let $X_1,X_2$ be locally compact Hausdorff topological spaces and $p : X_1 \to X_2$ be a Borel map. If $\HH$ is a Hilbert space and $E$ is a resolution of the identity of $\HH$ on $X_1$, then we can consider the \emph{push-forward} of $E$ via $p$, which is the resolution $p(E)$ of the identity of $\HH$ on $X_2$ defined by
\[p(E)(A) = E(p^{-1}(A)) \qquad\text{for all Borel $A \subseteq X_2$.}\]
The support of $p(E)$ is the $E$-essential image of $p$:
\[\supp p(E) = \{x \in X_2 \tc E(p^{-1}(U)) \neq 0 \text{ for all neighborhoods $U$ of $x$}\};\]
in particular, if $E$ is regular, then
\[\supp p(E) \subseteq \overline{p(\supp E)},\]
with equality if $p$ is continuous. Moreover, the following change-of-variable theorem holds: for every Borel function $f : X_2 \to \C$,
\[p(E)[f] = E[f \circ p]\]
(see \cite{rudin_functional_1973}, Theorem~13.28).

\subsection{Product measure}\label{subsection:productspectral}
Let $X_1,X_2$ be locally compact Hausdorff topological spaces; suppose moreover that $X_1,X_2$ are second-countable (thus metrizable by the Urysohn metrization theorem, see \S IX.9 of \cite{dugundji_topology_1966}). Let $\HH$ be a Hilbert space, and let $E_j$ be a (necessarily regular) resolution of the identity of $\HH$ on $X_j$, for $j=1,2$. If $E_1$ and $E_2$ commute, i.e.,
\[E_1(A_1) E_2(A_2) = E_2(A_2) E_1(A_1) \quad\text{for all Borel $A_1 \subseteq X_1$, $A_2 \subseteq X_2$,}\]
then there exists a unique (regular) resolution $E$ of the identity of $\HH$ on the product $X_1 \times X_2$ such that
\[E(A_1 \times A_2) = E_1(A_1) E_2(A_2) \quad\text{for all Borel $A_1 \subseteq X_1$, $A_2 \subseteq X_2$,}\]
which is called the \emph{product} of $E_1$ and $E_2$ (see \cite{berberian_notes_1966}, Theorem~33 and Corollary). Clearly, if $p_j : X_1 \times X_2 \to X_j$ is the canonical projection, then
\[p_j(E) = E_j,\]
for $j=1,2$. Moreover (see \cite{berberian_notes_1966}, Theorem~35)
\[\supp E \subseteq \supp E_1 \times \supp E_2.\]

\section{The spectral theorem}\label{section:spectraltheory}
\subsection{Spectral decomposition for a single normal operator}
Let $T$ be a (possibly unbounded) normal operator on a Hilbert space $\HH$. Then there exists (see \cite{rudin_functional_1973}, Theorem~13.33) a unique resolution $E$ of the identity of $\HH$ on $\C$ such that
\begin{equation}\label{eq:spectraltheorem}
T = \int_E \lambda \,dE(\lambda),\
\end{equation}
and in particular $\supp E = \spec(T)$. The projection-valued measure $E$ is called the \emph{spectral resolution}\index{spectral resolution} (or \emph{spectral measure}) of $T$. 

In terms of the spectral measure, a functional calculus is defined for the operator $T$: for every Borel function $f : \C \to \C$, we set
\[f(T) = \int_\C f \,dE.\]

\subsection{Compact normal operators}\label{subsection:spectralcompact}

Suppose that the normal operator $T$ on the Hilbert space $\HH$ is (bounded and) compact.
Then there exist a complete orthonormal system $(v_\alpha)_{\alpha \in I}$ in $\HH$ and a sequence $(\lambda_\alpha)_{\alpha \in I}$ of complex numbers such that
\[T v_\alpha = \lambda_\alpha v_\alpha \qquad\text{for all $\alpha \in I$,}\]
and
\[\text{for all $\varepsilon > 0$, the set } \{ \alpha \in I \tc |\lambda_\alpha| > \varepsilon\} \text{ is finite}\]
(see \cite{fell_representations_1988}, Theorem~VI.15.10); notice that the last condition implies that at most countably many $\lambda_\alpha$ are not null. This means that, if $E$ is the spectral resolution of $T$, then
\[R(E(\{\lambda\})) = \overline{\Span\{v_\alpha \tc \alpha \in I, \, \lambda_\alpha = \lambda\}},\]
so that
\[\dim R(E(\{\lambda\})) \text{ is finite for $\lambda \neq 0$,}\]
and
\[\{\lambda_\alpha\}_{\alpha \in I} \subseteq \spec T \subseteq \{\lambda_\alpha\}_{\alpha \in I} \cup \{0\}.\]
Therefore, for a compact normal operator, the spectral decomposition \eqref{eq:spectraltheorem} becomes a sum
\[T = \sum_{\substack{\alpha \in I \\ \lambda_\alpha \neq 0}} \lambda_\alpha E(\{\lambda_\alpha\})\]
(in the sense of strong convergence of operators).

\subsection{Strong commutativity}
Two normal operators $T_1,T_2$ are said to \emph{commute strongly} if their spectral resolutions commute: more precisely, if $E_1,E_2$ are the spectral resolutions of $T_1,T_2$ respectively, then we ask that
\[E_1(A) E_2(B) = E_2(B) E_1(A) \qquad\text{for all Borel $A,B \subseteq \C$.}\]
Notice that, in the case of bounded normal operators, strong commutativity is equivalent to commutativity (cf.\ \cite{rudin_functional_1973}, Theorem~12.16).

\subsection{Joint spectral theory for strongly commuting normal operators}
Suppose that $T_1,\dots,T_n$ are normal operators on $\HH$ which commute strongly pairwise, and let $E_1,\dots,E_n$ be their respective spectral resolutions. Then we can consider the product of $E_1,\dots,E_n$, i.e., the unique resolution $E$ of the identity of $\HH$ on $\C^n$ such that
\[E(A_1 \times \dots \times A_n) = E_1(A_1) \cdots E_n(A_n)\]
(see \S\ref{subsection:productspectral}), and we have
\[T_j = \int_{\C^n} \lambda_j \,dE(\lambda_1,\dots,\lambda_n) \qquad\text{for $j=1,\dots,n$.}\]
$E$ is called the \emph{joint spectral resolution}\index{spectral resolution!joint} of $T_1,\dots,T_n$ and its support, which is contained in $\spec(T_1) \times \dots \times \spec(T_n)$, is called the \emph{joint spectrum}\index{spectrum!joint spectrum} of $T_1,\dots,T_n$ and denoted by $\spec(T_1,\dots,T_n)$. 

Via the joint spectral resolution, a joint functional calculus for the operators $T_1,\dots,T_n$ is defined: for a Borel function $f : \C^n \to \C$, we set
\[f(T_1,\dots,T_n) = \int_{\C^n} f \,dE.\]